\documentclass[12pt]{book}
\usepackage[a4paper,top=25mm,bottom=25mm]{geometry}
\usepackage{import}

\usepackage{makeidx}
\makeindex

\usepackage{etex}
\usepackage{latexsym}
\usepackage{amsfonts,amsmath,amssymb,amscd,amsxtra,amsthm,verbatim,calc}
\usepackage[all]{xy}

\usepackage[toc,symbols,nomain,stylemods=longextra]{glossaries-extra}
\newcommand{\listofsymbolsname}{List of Symbols and Notation}
\makeglossaries

\usepackage{thmtools}
\usepackage{thm-restate}
\usepackage{hyperref}
\usepackage{cleveref}
\usepackage{todonotes}
\usepackage{verbatim}
\usepackage[]{mdframed}

\setuptodonotes{tickmarkheight= 4}

\usepackage{eucal}

\usepackage{tikz}
\usetikzlibrary{cd}
\usetikzlibrary{patterns}
\usetikzlibrary{backgrounds}

\usepackage{amsmath}
\usepackage{amstext}
\usepackage{amssymb}
\usepackage{amsthm}
\usepackage{amscd}

\usepackage[labelfont=bf]{caption}

\theoremstyle{plain}
\newtheorem{thm}{Theorem}[chapter]
\newtheorem{thm*}{Theorem}
\newtheorem{lem}[thm]{Lemma}
\newtheorem{prop}[thm]{Proposition}
\newtheorem{prop*}[thm*]{Proposition}
\newtheorem{cor}[thm]{Corollary}

\theoremstyle{definition}
\newtheorem{defn}[thm]{Definition}

\newtheorem{ex}[thm]{Example}
\newtheorem{qu}[thm]{Question}
\newtheorem{notation}[thm]{Notation}

\newtheorem{meth}[thm]{Method}
\newtheorem{conj}[thm]{Conjecture}

\theoremstyle{remark}

\newtheorem{rem}[thm]{Remark}

\DeclareMathOperator{\height}{ht}
\DeclareMathOperator{\depth}{depth}

\DeclareMathOperator{\Supp}{Supp}
\DeclareMathOperator{\reg}{reg}
\DeclareMathOperator{\Hred}{\widetilde{H}}
\DeclareMathOperator{\HK}{HK}
\DeclareMathOperator{\link}{link}
\DeclareMathOperator{\pdim}{pdim}
\DeclareMathOperator{\hcf}{hcf}

\DeclareMathOperator{\id}{id}
\DeclareMathOperator{\Sym}{Sym}
\DeclareMathOperator{\Skel}{Skel}
\DeclareMathOperator{\Cl}{Cl}
\DeclareMathOperator{\Ind}{Ind}
\DeclareMathOperator{\Star}{star}
\DeclareMathOperator{\codim}{codim}
\DeclareMathOperator{\free}{free}
\DeclareMathOperator{\SimpLink}{Simp-Link}

\newcommand{\isomto}{\overset{\sim}{\rightarrow}} 
\newcommand{\arrowX}[1]{\overset{#1}{\rightarrow}}

\newcommand{\ZZ}{\mathbb{Z}}
\newcommand{\QQ}{\mathbb{Q}}
\newcommand{\RR}{\mathbb{R}}

\newcommand{\KK}{\mathbb{K}}

\newcommand{\PP}{\mathbb{P}}

\newcommand{\hh}{\widehat{h}}

\newcommand{\calB}{\mathcal{B}}
\newcommand{\calBtilde}{\widetilde{\mathcal{B}}}
\newcommand{\calC}{\mathcal{C}}
\newcommand{\calD}{\mathcal{D}}
\newcommand{\calH}{\mathcal{H}}
\newcommand{\calV}{\mathcal{V}}

\newcommand{\II}{\mathbb{I}}
\newcommand{\JJ}{\mathbb{J}}
\newcommand{\WW}{\mathcal{W}}
\newcommand{\Vn}{\mathcal{V}_n}

\newcommand{\fC}{\mathfrak{C}}

\newcommand{\calI}{\mathfrak{I}}
\newcommand{\calP}{\mathcal{P}}
\newcommand{\barP}{\widehat{\mathcal{P}}}

\newcommand{\tbeta}{\widetilde{\beta}}
\newcommand{\tDelta}{\widetilde{\Delta}}

\newcommand{\tV}{\widetilde{V}}
\newcommand{\BDelta}{\mathcal{B}(\Delta)}
\newcommand{\LDelta}{\mathcal{L}(\Delta)}
\newcommand{\LiDelta}[1][i]{\mathcal{L}_{#1}(\Delta)}
\newcommand{\barDelta}{\widehat{\Delta}}

\newcommand{\eseq}[2]{\mathbf{e}^{#1}_{#2}} 
\newcommand{\ba}{\mathbf{a}}
\newcommand{\bfb}{\mathbf{b}}
\newcommand{\bc}{\mathbf{c}}
\newcommand{\bd}{\mathbf{d}}

\newcommand{\bm}{\mathbf{m}}
\newcommand{\bp}{\mathbf{p}}
\newcommand{\bs}{\mathbf{s}}
\newcommand{\bk}{\mathbf{k}}

\newcommand{\nth}[2][th]{{#2}^\text{#1}} 

\usepackage{xparse}
\NewDocumentCommand{\Cn}{O{n}}{\mathcal{C}_{#1}}
\NewDocumentCommand{\Dn}{O{n}}{\mathcal{D}_{#1}}
\NewDocumentCommand{\Dntilde}{O{n}}{\widetilde{\mathcal{D}}_{#1}}
\NewDocumentCommand{\Cnh}{O{n} O{h}}{\mathcal{C}_{#1}^{#2}}
\NewDocumentCommand{\Dnh}{O{n} O{h}}{\mathcal{D}_{#1}^{#2}}
\NewDocumentCommand{\Dnhtilde}{O{n} O{h}}{\widetilde{\mathcal{D}}_{#1}^{#2}}

\NewDocumentCommand{\lkds}{O{\sigma} O{\Delta}}{\link_{#2} #1}
\NewDocumentCommand{\Idstar}{O{\Delta}}{I_{{#1}^*}}

\glsaddstoragekey{group}{}{\grouplabel}
\glsxtrsetgrouptitle{alg}{Rings and Modules}
\glsxtrsetgrouptitle{comb}{Simplicial Complexes}
\glsxtrsetgrouptitle{graph}{Graphs}
\glsxtrsetgrouptitle{gen}{General}

\glsxtrnewsymbol[description={The set \ensuremath{\{1,\dots,n\}}}]{nset}{\ensuremath{[n]}}
\glsxtrnewsymbol[description={The rational vector space $\bigoplus_{d\in \ZZ} \QQ^{n+1}$ containing all Betti diagrams of $R$-modules}]{Vn}{\ensuremath{\Vn}}

\glsxtrnewsymbol[description={The projective dimension of the $R$-module $M$}]{pdimM}{\ensuremath{\pdim M}}
\glsxtrnewsymbol[description={The Castelnuovo-Mumford Regularity of the $R$-module $M$}]{regM}{\ensuremath{\reg M}}
\glsxtrnewsymbol[description={The Krull-dimension of the $R$-module $M$ (\textit{not} the dimension of $M$ as a $\KK$-vector space)}]{dimM}{\ensuremath{\dim M}}
\glsxtrnewsymbol[description={The codimension of the $R$-module $M$ (i.e. $n-\dim M$)}]{codimM}{\ensuremath{\codim M}}
\glsxtrnewsymbol[description={The depth of the $R$-module $M$}]{depthM}{\ensuremath{\depth M}}
\glsxtrnewsymbol[description={The expression $\sum_{i,d}(-1)^i d^j\beta_{i,d}$, from the Herzog-K\"{u}hl equations}]{HK}{\ensuremath{\HK_j(\beta)}}
\glsxtrnewsymbol[description={The ideal in $R$ generated by the polynomials $g_1,\dots,g_m$}]{genIdeal}{\ensuremath{\langle g_1,\dots,g_m\rangle}}
\glsxtrnewsymbol[description={The height of the ideal $I$}]{ht}{\ensuremath{\height I}}

\glsxtrnewsymbol[description={The vertex set of the complex $\Delta$}]{VD}{\ensuremath{V(\Delta)}}
\glsxtrnewsymbol[description={The dimension of the complex $\Delta$}]{dimD}{\ensuremath{\dim \Delta}}
\glsxtrnewsymbol[description={The codimension of the complex $\Delta$ (i.e. $n-\dim \Delta - 1$)}]{codimD}{\ensuremath{\codim \Delta}}
\glsxtrnewsymbol[description={The link of the face $\sigma$ in the complex $\Delta$}]{lkds}{\ensuremath{\lkds}}
\glsxtrnewsymbol[description={The Alexander Dual of the complex $\Delta$}]{AD}{\ensuremath{\Delta^*}}
\glsxtrnewsymbol[description={The join of the complexes $\Delta_1$ and $\Delta_2$}]{join}{\ensuremath{\Delta_1 \ast \Delta_2}}
\glsxtrnewsymbol[description={The join of the complexes $\Delta_1,\dots,\Delta_m$}]{multijoin}{\ensuremath{\circledast_i^m \Delta_i}}
\glsxtrnewsymbol[description={The cone of the complex $\Delta$}]{cone}{\ensuremath{C\Delta}}
\glsxtrnewsymbol[description={The suspension of the complex $\Delta$}]{susp}{\ensuremath{S\Delta}}
\glsxtrnewsymbol[description={The complex generated by the sets $F_1,\dots,F_m$}]{genComplex}{\ensuremath{\langle F_1,\dots,F_m\rangle}}
\glsxtrnewsymbol[description={The complex generated by the sets of the form $gF_i$ for $1\leq i \leq m$ and $g$ an element of the group $G$}]{genComplexGroup}{\ensuremath{\langle F_1,\dots,F_m\rangle_G}}
\glsxtrnewsymbol[description={The induced subcomplex of $\Delta$ on vertex set $U$}]{sc}{\ensuremath{\Delta_U, \Delta|_U}}
\glsxtrnewsymbol[description={The deletion of the face $g$ from the complex $\Delta$ (i.e. the largest subcomplex of $\Delta$ that does not contain $g$)}]{del}{\ensuremath{\Delta-g}}
\glsxtrnewsymbol[description={The complex obtained by starring the vertex $v$ on to the complex $\Delta$ in dimension $d$, as laid out in Definition \ref{Definition: Starring Operation}, after \cite{Klivans}}]{star}{\ensuremath{\Delta\Star_d v}}
\glsxtrnewsymbol[description={The product $\prod_{i\in \sigma}x_i$ in $R$ of variables indexed by the set $\sigma\subseteq [n]$}]{x^sigma}{\ensuremath{\mathbf{x}^\sigma}}
\glsxtrnewsymbol[description={The Stanley-Reisner ideal of the complex $\Delta$}]{IDelta}{\ensuremath{I_\Delta}}
\glsxtrnewsymbol[description={The dual Stanley-Reisner ideal of the complex $\Delta$}]{IDeltaStar}{\ensuremath{I_\Delta^*}}
\glsxtrnewsymbol[description={The Stanley-Reisner ring of the complex $\Delta$}]{KDelta}{\ensuremath{\KK[\Delta]}}

\glsxtrnewsymbol[description={The complete $j$-simplex}]{jsimp}{\ensuremath{\Delta^j}}
\glsxtrnewsymbol[description={The boundary of the $j$-simplex}]{jsimpbound}{\ensuremath{\partial\Delta^j}}
\glsxtrnewsymbol[description={The $r$-skeleton of the complex $\Delta$}]{Skel}{\ensuremath{\Skel_r(\Delta)}}
\glsxtrnewsymbol[description={The $d$-dimensional cross polytope}]{Orth}{\ensuremath{O^d}}

\glsxtrnewsymbol[description={The $\nth{i}$ reduced homolgy of $\Delta$ with coefficients in $\KK$}]{Hom}{\ensuremath{\Hred_i(\Delta)}}
\glsxtrnewsymbol[description={The set of degrees at which the complex $\Delta$ has non-trivial homologies over $\KK$}]{homIndex}{\ensuremath{h(\Delta)}}
\glsxtrnewsymbol[description={The set of degrees at which the complex $\lkds$ has non-trivial homologies over $\KK$}]{homIndexLink}{\ensuremath{h(\Delta,\sigma)}}
\glsxtrnewsymbol[description={The union of the sets $h(\Delta,\sigma)$ for which $\sigma$ has cardinality $m$}]{totalHomIndex}{\ensuremath{\hh(\Delta,m)}}

\glsxtrnewsymbol[description={The vertex set of the graph $G$}]{VG}{\ensuremath{V(G)}}
\glsxtrnewsymbol[description={The edge set of the graph $G$}]{EG}{\ensuremath{E(G)}}
\glsxtrnewsymbol[description={The complement of the graph $G$}]{Gc}{\ensuremath{G^c}}
\glsxtrnewsymbol[description={The complex of cliques of the graph $G$}]{ClG}{\ensuremath{\Cl(G)}}
\glsxtrnewsymbol[description={The independence complex of the graph $G$}]{IngG}{\ensuremath{\Ind(G)}}
\glsxtrnewsymbol[description={The graph consisting of a single edge}]{L}{\ensuremath{L}}
\glsxtrnewsymbol[description={The complete graph on $m$ vertices}]{Km}{\ensuremath{K_m}}
\glsxtrnewsymbol[description={The empty graph on $m$ vertices}]{Em}{\ensuremath{E_m}}
\glsxtrnewsymbol[description={The cyclic graph of order $m$}]{Cm}{\ensuremath{C_m}}
\glsxtrnewsymbol[description={The edge ideal of the graph $G$}]{I(G)}{\ensuremath{I(G)}}

\glsxtrnewsymbol[description={The Betti cone generated by all diagrams of Cohen-Macaulay $R$-modules of codimension $h$ whose shifts are bounded by the sequences $\ba$ and $\bfb$}]{C(a,b)}{\ensuremath{\calC(\ba,\bfb)}}
\glsxtrnewsymbol[description={The smallest integer-valued Betti diagram of shift type $\bc$ belonging to a Cohen-Macaulay module}]{pi(c)}{\ensuremath{\pi(\bc)}}
\glsxtrnewsymbol[description={The Betti cone generated by Stanley-Reisner ideals in $R$}]{Dn}{\ensuremath{\Dn}}
\glsxtrnewsymbol[description={The Betti cone generated by Stanley-Reisner ideals in $R$ of height $h$}]{Dnh}{\ensuremath{\Dnh}}
\glsxtrnewsymbol[description={The Betti cone generated by Stanley-Reisner ideals in $R$ of height $h$ and degree at least $2$}]{Dntilde}{\ensuremath{\Dntilde}}
\glsxtrnewsymbol[description={The Betti cone generated by Stanley-Reisner ideals in $R$ of height $h$ and degree at least $2$}]{Dnhtilde}{\ensuremath{\Dnhtilde}}
\glsxtrnewsymbol[description={The Betti cone generated by edge ideals in $R$}]{Cn}{\ensuremath{\Cn}}
\glsxtrnewsymbol[description={The Betti cone generated by edge ideals in $R$ of height $h$}]{Cnh}{\ensuremath{\Cnh}}
\glsxtrnewsymbol[description={The set of indices $(i,d)$ for which the Betti cone $\calC$ contains diagrams $\beta$ with $\beta_{i,d}\neq 0$}]{I(C)}{\ensuremath{\II(\calC)}}
\glsxtrnewsymbol[description={The minimal subspace of $\Vn$ containing the cone $\calC$}]{W(C)}{\ensuremath{\WW(\calC)}}

\glsxtrnewsymbol[description={The shape of the Betti diagram $\beta$ (i.e. the set of indices for which we have $\beta_{i,d}\neq 0$)}]{shape}{\ensuremath{S(\beta)}}
\glsxtrnewsymbol[description={The link poset of the complex $\Delta$, as given in Definition \ref{Definition: Link Poset}}]{P_Delta}{\ensuremath{P_\Delta}}
\glsxtrnewsymbol[description={The complexes $\delta_1$ and $\delta_2$ in the link poset $P_\Delta$ satisfy $\delta_2=\link_{\delta_1} \tau$ }]{links-in-PDelta}{\ensuremath{\delta_1 \arrowX{\tau} \delta_2}}

\glsxtrnewsymbol[description={The symmetric group on the set $V$}]{Sym}{\ensuremath{\Sym(V)}}
\glsxtrnewsymbol[description={The symmetric group on the set $[n]$}]{Sn}{\ensuremath{S_n}}
\glsxtrnewsymbol[description={The symmetric group on the set $\{0,\dots,p\}$}]{S0p}{\ensuremath{S^0_p}}

\glsxtrnewsymbol[description={The cycle complex of type $(a,b)$ as given in Definition \ref{Definition: Cycle Complexes}}]{cycle}{\ensuremath{\fC_{a,b}}}
\glsxtrnewsymbol[description={The intersection complex associated to the sequence $\bm$ as given in Definition \ref{Definition: Intersection Complexes}}]{int}{\ensuremath{\calI(\bm)}}
\glsxtrnewsymbol[description={The partition complex associated to the sequence $(a,p,m)$ as given in Definition \ref{Definition: Partition Complexes}}]{part}{\ensuremath{\calP(a,p,m)}}
\glsxtrnewsymbol[description={The closed partition complex associated to the sequence $(a,p,m)$ as given in Definition \ref{Definition: Closed Partition Complexes}}]{closedPart}{\ensuremath{\barP(a,p,m)}}
\glsxtrnewsymbol[description={The vertex set for the partition complex $\calP(a,p,m)$}]{Vap}{\ensuremath{V^a_p}}
\glsxtrnewsymbol[description={The set of boundary vertices for the partition complex $\calP(a,p,m)$}]{Xp}{\ensuremath{X_p}}
\glsxtrnewsymbol[description={The set of partition vertices for the partition complex $\calP(a,p,m)$}]{Yap}{\ensuremath{Y^a_p}}
\glsxtrnewsymbol[description={The set of partitions of $a+i-2$ into $i$ parts}]{P(a,i)}{\ensuremath{\PP(a,i)}}
\glsxtrnewsymbol[description={The $\lambda$-generating set of the partition complex $\calP(a,p,m)$ for some partition $\lambda \in \PP(a,i)$, as given in Definition \ref{Definition: Partition Complex Generating Sets}}]{Gpi-lam}{\ensuremath{G_{p,i}^\lambda}}
\glsxtrnewsymbol[description={The set of boundary vertices in the subset $\sigma\subseteq V^a_p$}]{S-X}{\ensuremath{\sigma_X}}
\glsxtrnewsymbol[description={The set of partition vertices in the subset $\sigma\subseteq V^a_p$}]{S-Y}{\ensuremath{\sigma_Y}}
\glsxtrnewsymbol[description={The set of indices $0\leq i \leq p$ such that $\sigma$ contains a vertex of the form $x_i$ or $y_i^j$}]{supp}{\ensuremath{\Supp(\sigma)}}

\glsxtrnewsymbol[description={The set of symbols of the form $u_\sigma$ where $\sigma$ is a face of the complex $\Delta$ of size at least $\dim\Delta+2-i$}]{SiDelta}{\ensuremath{S^i_\Delta}}
\glsxtrnewsymbol[description={The barycentric subdivision of the complex $\Delta$}]{BDelta}{\ensuremath{\BDelta}}

\glsxtrnewsymbol[description={The disjoint union of the structures $A$ and $B$}]{disunion}{\ensuremath{A + B,A\sqcup B}}
\glsxtrnewsymbol[description={A deformation retraction between the complexes $\Delta_1$ and $\Delta_2$}]{defRetract}{\ensuremath{\Delta_1 \rightsquigarrow \Delta_2}}
\glsxtrnewsymbol[description={A bijection between the sets $A$ and $B$}]{bij}{\ensuremath{A \isomto B}}
\glsxtrnewsymbol[description={An isomorphism between the structures $A$ and $B$}]{iso}{\ensuremath{A \cong B}}
\glsxtrnewsymbol[description={A map sending the element $a$ in one structure to the element $b$ in another}]{mapsto}{\ensuremath{a \mapsto b}}

\begin{document}
	\frontmatter
	
	\begin{titlepage}
		\begin{center}
			\vspace*{1cm}
			\Huge
			\textbf{Betti Cones of Stanley-Reisner Ideals}
			
			\vspace{1.5cm}
			\LARGE
			\textbf{David Carey}
			
			\vfill
			
			A thesis presented for the degree of\\
			Doctor of Philosophy
			
			\vspace{1.2cm}
			\Large
			School of Mathematics and Statistics\\
			University of Sheffield\\
			August 2023
		\end{center}
	\end{titlepage}
	
	\clearpage
	
	\thispagestyle{empty}
	
	
	\chapter*{Acknowledgements}
	I would like to thank a number of people for their help during the course of my PhD.
	
	Firstly, I am grateful to the Engineering and Physical Sciences Research Council for supporting my work.
	
	During the course of my PhD I was fortunate enough to attend the virtual Research Encounters in Algebraic and Combinatorial Topics conference in 2021; the Graduate Meeting on Combinatorial Commutative Algebra at the University of Minnesota in 2022; and the 4th Graduate Meeting on Applied Algebra and Combinatorics at KTH in 2023. I would like to thank the coordinators of all of these conferences, both for the hard work they put into organising the events, and for giving me the opportunity to present my research there. I am also very grateful to the many mathematicians I met across these events who took the time to tell me about their own work (or other work that interested them), for providing countless fascinating discussions about combinatorics and algebra. In particular I am indebted to Alex Lazar for an invaluable conversation about Cohen-Macaulay complexes.
	
	I would also like to thank my parents for providing a comfortable space to finish writing this thesis (and for keeping me supplied with cups of tea and great meals throughout); my brother for his assistance with some of my coding problems; Joseph Martin, Lewis Combes and David Williams for attending a mock run-through of my first ever academic talk, and offering some extremely helpful feedback on how to improve it; the Soho Scribblers for providing a regular, relaxed and friendly environment to write whenever I find myself down in London; and all of the friends and family who have helped to keep me grounded and boost my morale when the workload was most stressful.
	
	Lastly, and most importantly, I would like to thank my supervisor Moty Katzman for his endless supply of support, guidance and patience throughout all aspects of this project. I am very lucky to have had him as my supervisor.
	
	\chapter*{Abstract}
	The aim of this thesis is to investigate the Betti diagrams of squarefree monomial ideals in polynomial rings. Betti diagrams encode information about the minimal free graded resolutions of these ideals, and are therefore important algebraic invariants.
	
	Computing resolutions is a difficult task in general, but in our case there are tools we can use to simplify it. Most immediately, the Stanley-Reisner Correspondence assigns a unique simplicial complex to every squarefree monomial ideal, and Hochster’s Formula (\cite{Hoc} Theorem 5.1) allows us to compute the Betti diagrams of these ideals from combinatorial properties of their corresponding complexes. This reduces the algebraic problem of computing resolutions to the (often easier) combinatorial problem of computing homologies. As such, most of our work is combinatorial in nature.
	
	The other key tool we use in studying these diagrams is Boij-S\"{o}derberg Theory. This theory views Betti diagrams as vectors in a rational vector space, and investigates them by considering the convex cone they generate. This technique has proven very instructive: it has allowed us to classify all Betti diagrams up to integer multiplication. This thesis applies the theory more narrowly, to the cones generated by diagrams of squarefree monomial ideals.
	
	In Chapter \ref{Chapter: Introduction} we introduce all of these concepts, along with some preliminary results in both algebra and combinatorics we will need going forward. Chapter \ref{Chapter: Dimensions} presents the dimensions of our cones, along with the vector spaces they span.
	
	Chapters \ref{Chapter: PR Complexes and Degree Types} and \ref{Chapter: Families of PR Complexes} are devoted to the pure Betti diagrams in these cones (motivated in part by \cite{Bruns-Hibi} and \cite{Bruns-Hibi-earlier}), and the combinatorial properties of their associated complexes. Finally Chapter \ref{Chapter: Generating Degree Types} builds on these results to prove a partial analogue of the first Boij-S\"{o}derberg conjecture for squarefree monomial ideals, by detailing an algorithm for generating pure Betti diagrams of squarefree monomial ideals of any degree type.

	\tableofcontents
	\pagebreak
	
	
	
	
	\mainmatter
	
	\chapter{Introduction}
	
	Throughout this thesis we fix a field $\KK$ of arbitrary characteristic and a positive integer $n$, and we set $R$ to be the polynomial ring $\KK[x_1,\dots, x_n]$. Our aim is to investigate both the squarefree monomial ideals of $R$ and simplicial complexes on the vertex set $[n]=\{1,\dots,n\}$, which are related via the Stanley-Reisner Correspondence. More specifically we wish to study an important algebraic invariant of squarefree monomial ideals: their Betti diagrams.
	
	Betti diagrams are an invariant of any graded $R$-module $M$; they are defined as the matrix of exponents $\beta_{i,d}$ in a minimal free graded resolution of $M$,
	$$0\rightarrow \bigoplus_{d\in \ZZ} R(-d)^{\beta_{p,d}} \rightarrow ... \rightarrow \bigoplus_{d\in \ZZ} R(-d)^{\beta_{1,d}} \rightarrow \bigoplus_{d\in \ZZ} R(-d)^{\beta_{0,d}}\rightarrow M \rightarrow 0.$$
	and they encode a lot of important information about their respective modules: in particular, the Krull-dimension, projective dimension, Hilbert polynomial, and Castelnuovo-Mumford regularity of a module can all be computed from its Betii diagram.
	
	Unfortunately there does not, in general, exist an efficient way of computing Betti diagrams for arbitrary modules. However, in the case of monomial ideals, there are often combinatorial properties of the ideal which we can exploit to help us compute its Betti diagram. 
	
	More specifically, in the case of \textit{squarefree} monomial ideals, the Stanley-Reisner correspondence assigns the ideal a unique simplicial complex; and a seminal theorem of Hochster's (\cite{Hoc} Theorem 5.1) reframes the Betti numbers of the ideal in terms of homological data from its corresponding complex. Hochster's Formula has proven to be an incredibly powerful tool in the study of Betti diagrams. It has been successfully used to compute diagrams of many families of squarefree monomial ideals -- in particular, the edge ideals of certain families of graphs, such as in \cite{Ramos} Theorem 2.3.3 (complements of cyclic graphs) or \cite{Jacques} Theorem 5.3.8 (complete multipartite graphs) -- and to find combinatorial bounds for other algebraic invariants (such as Theorem 6.7 in \cite{Reg}, which gives a bound on the regularity of an edge ideal in terms of matchings in its corresponding graph).
	
	Note that every Betti diagram of a monomial ideal can be realised as the Betti diagram of a squarefree monomial ideal, via the process of polarisation (see \cite{polar} Definition 2.1). Thus, restricting our attention to the squarefree case does not in fact require us to neglect any Betti diagrams of arbitrary monomial ideals.
	
	This thesis is particularly focussed on those complexes whose Stanley-Reisner ideals have \textit{pure} Betti diagrams -- these are Betti diagrams arising from \textit{pure resolutions} of the form
	$$0\rightarrow R(-c_p)^{\beta_{p,c_p}} \rightarrow ... \rightarrow R(-c_1)^{\beta_{1,c_1}} \rightarrow R(-c_0)^{\beta_{0,c_0}}\rightarrow I \rightarrow 0.$$
	Put simply, a Betti diagram is pure if it has only a single nonzero entry in each column. Stanley-Reisner ideals with pure resolutions have been investigated extensively in the literature. Most notably, \cite{Bruns-Hibi-earlier}, \cite{Fro} and \cite{Bruns-Hibi} all include a number of results in this area. Between them they classify all the Stanley-Reisner ideals with pure resolutions corresponding to $1$-dimensional complexes (Proposition 2.2 of \cite{Bruns-Hibi-earlier}); Cohen-Macaulay posets (Theorem 3.6 of \cite{Bruns-Hibi-earlier}); and clique complexes of graphs (Theorem 1 in \cite{Fro} and Theorem 2.1 in \cite{Bruns-Hibi}). In the case of clique complexes, \cite{E-S} builds on Theorem 1 of \cite{Fro} to show that any 2-linear diagrams arising from a clique complex has a unique corresponding threshold graph. 
	
	Our specific interest is in the possible \textit{shapes} of pure diagrams of Stanley-Reisner ideals -- that is, the positions of their nonzero entries. This is equivalent to the notion of \textit{shift type} given in \cite{Bruns-Hibi} (page 1203), defined as the sequence $(c_p,\dots,c_0)$ of shifts in the resolution. The authors of this paper give a number of examples of, and restrictions on, the possible shift types of pure resolutions of Stanley-Reisner ideals (see e.g. Theorem 3.1 and Propositions 3.2 and 3.3).
	
	One reason pure diagrams have enjoyed particular attention in the literature comes from Boij-S\"{o}derberg Theory, which is a central motivation for our work. This theory (originally expounded in \cite{BS-Original-2006}) envisages all Betti diagrams of $R$-modules as vectors lying in the infinite-dimensional rational vector space $\Vn=\bigoplus_{d\in \ZZ} \QQ^{n+1}$, and studies the cone generated by the positive rational rays spanned by these diagrams. The central insight of Boij and S\"{o}derberg was that every ray in this cone can be decomposed into the sum of rays corresponding to pure diagrams of Cohen-Macaulay modules (and these pure rays can be easily computed). This result amounts to a complete classification of all Betti diagrams of $R$-modules up to multiplication by a positive rational.
	
	In the case of monomial ideals, there has been considerable interest in obtaining combinatorial descriptions of the Boij-S\"{o}derberg decomposition described above. This avenue has proven very fruitful in certain cases (see e.g. \cite{Nagel} Theorem 3.6, which looks at Ferrers ideals, or \cite{BS-Ver} Theorem 3.2, which looks at Borel ideals). However, there is a significant obstacle to this goal in general: while we know that there exists a pure Betti diagram of any given shift type (this is the first Boij-S\"{o}derberg conjecture, \cite{Floy} Theorem 1.9), this pure diagram need not be the diagram of a monomial ideal. Hence, the decomposition of the diagram of a monomial ideal $\beta(I)$ will often contain pure diagrams which do not correspond to monomial ideals themselves (and thus do not lend themselves readily to combinatorial manipulation).
	
	This raises a natural question: to what extent does the first Boij-S\"{o}derberg conjecture apply to the cone generated by diagrams of squarefree monomial ideals? Or put more simply, \textit{for which shift types does there exist a squarefree monomial ideal with a pure resolution of that shift type}? This question is the central motivation for our study of pure Betti diagrams of Stanley-Reisner ideals in Chapters \ref{Chapter: PR Complexes and Degree Types}, \ref{Chapter: Families of PR Complexes} and \ref{Chapter: Generating Degree Types}. 
	
	Another way to navigate the obstacle above would be to obtain a decomposition of $\beta(I)$ into the rational sum of diagrams of other squarefree monomial ideals. This would require a description of the extremal rays of the cone generated by Betti diagrams of squarefree monomial ideals. We study this cone in Chapter \ref{Chapter: Dimensions}, along with three notable subcones. While a complete description of these cones may currently be out of reach, we present some important properties, which may be of use in future classifications: in particular, we present and prove formulae for their dimensions, and describe the minimal subspaces of $\Vn$ that they span.
	
	For the avoidance of ambiguity we now present a brief overview of our substantial new results, broken down by chapter. It is worth noting that none of these results depend upon the characteristic of $\KK$.
	
	\section*{Chapter \ref{Chapter: Introduction}: Background Material}
	Our first chapter is a survey of relevant algebraic and combinatorial background material from the literature, to provide context and motivation for our main results, as well as a number of key tools that will help us in proving them. This includes a brief overview of topics surrounding Betti diagrams, Boij-S\"{o}derberg Theory, Stanley-Reisner and edge ideals, and Hochster's Formula. As such, none of the results in this chapter are new (with the possible exception of Lemma \ref{Lemma: Deformation Retract} -- a result detailing a condition under which a face may be deleted from a complex while preserving homotopy -- which we have been unable to find in the literature).
	
	\section*{Chapter \ref{Chapter: Dimensions}: Dimensions of Betti Cones}
	This chapter details our work on the dimensions of Betti cones generated exclusively by diagrams of squarefree monomial ideals. The chapter is built around the proofs of four key theorems.
	\begin{itemize}
		\item Theorem \ref{Theorem: dimDn} on the dimension of the Betti cone $\Dn$ generated by the diagrams of all squarefree monomial ideals in $R$.
		\item Theorem \ref{Theorem: dimCn} on the dimension of the Betti cone $\Cn$ generated by the diagrams of edge ideals in $R$.
		\item Theorem \ref{Theorem: dimDnh} on the dimension of the Betti cone $\Dnh$ generated by the diagrams of squarefree monomial ideals in $R$ of a given height $h$.
		\item Theorem \ref{Theorem: dimCnh} on the dimension of the Betti cone $\Cnh$ generated by the diagrams of edge ideals in $R$ of a given height $h$.
	\end{itemize}
	While proving Theorems \ref{Theorem: dimDn} and \ref{Theorem: dimDnh} we also demonstrate how they can be altered slightly to compute the dimensions of the cones $\Dntilde$ and $\Dnhtilde$, generated by diagrams squarefree monomial ideals of degree at least $2$ (equivalently, these are the diagrams corresponding to simplicial complexes with no missing vertices).
	
	Our proofs for each of these theorems involve constructing the minimal vector spaces containing the cone. We construct these spaces in each case by compiling a list of restrictions that the diagrams in the cone must satisfy, most of which can be found in the literature. Hence, the first part of each of our proofs (the `\textit{upper bound}' section) consists predominantly of a compilation of existing results. The \textit{new} part of each theorem, and the bulk of each proof (the `\textit{lower bound}' section), is that the vector spaces carved out by these restrictions are the \textit{minimal} spaces containing their respective cones. This demonstrates that (up to linear combination) these restrictions represent the \textit{only} linear relations satisfied by every diagram in the cone, and in that sense they are best possible.

	\section*{Chapter \ref{Chapter: PR Complexes and Degree Types}: PR Complexes and Degree Types}
	This chapter begins our investigation into pure resolutions of squarefree monomial ideals by introducing the family of PR complexes -- simplicial complexes $\Delta$ whose dual Stanley-Reisner ideals $\Idstar$ have pure Betti diagrams. We present an entirely combinatorial description of these complexes (Corollary \ref{Corollary: PR Complexes}) which follows from the Alexander Dual variant of Hochster's Formula, ADHF (Proposition 1 in \cite{Eagon-Reiner}). While this description is an elementary consequence of ADHF, we have not yet come across a treatment of these complexes in the literature.
	
	A central result of Eagon and Reiner in \cite{Eagon-Reiner} (Theorem 3) is that the dual Stanley-Reisner ideal $I_\Delta^*$ of a complex $\Delta$ has a linear resolution if and only if the Stanley-Reisner ring $\KK[\Delta]$ is Cohen-Macaulay. Thus PR complexes are a generalisation of the family of Cohen-Macaulay complexes. Combinatorially speaking, Cohen-Macaulay complexes are defined as complexes whose links have homology only at the highest possible degree (this is Reisner's criterion, in e.g. \cite{CCA} Theorem 5.53). The PR condition is a relaxation of the Cohen-Macaulay one, requiring only that the homology of each non-acyclic link is a function of the size of its corresponding face. Cohen-Macaulay complexes are the subject of a large amount of literature, ranging from largely algebraic investigations (see e.g. \cite{Iye}) to almost entirely combinatorial ones (see e.g. \cite{Lazar}).
	
	Of particular note in this chapter is the concept of maximal intersections in Section \ref{Subsection: Maximal Intersections}, particularly Proposition \ref{Proposition: beta_1 Purity Condition}, which details the conditions under which the second column of the Betti diagram $\beta(\Idstar)$ is pure; and the link poset (Definition \ref{Definition: Link Poset}), which is a poset structure that can be imposed on the links of a simplicial complex to help us identify whether its dual Stanley-Reisner ideal admits a pure resolution.
	
	\section*{Chapter \ref{Chapter: Families of PR Complexes}: Some Families of PR Complexes}
	This chapter presents a number of highly symmetric families of PR complexes, and proves that they satisfy the PR property. These include the following.
	\begin{itemize}
		\item The family of \textit{cycle complexes} $\fC_{a,b}$ given in Definition \ref{Definition: Cycle Complexes}; Theorem \ref{Theorem: Cycle Complexes are PR} states that the complex $\fC_{a,b}$ is PR with degree type $(a,b)$.
		\item The family of intersection complexes $\calI(\bm)$ given in Definition \ref{Definition: Intersection Complexes}; Theorems \ref{Theorem: Intersection Complexes Degree Type} and \ref{Theorem: Intersection Complexes Betti Numbers} state that these complexes are PR, and detail their degree types and Betti diagrams.
		\item The family of \textit{partition complexes} $\calP(a,p,m)$ given in Definition \ref{Definition: Partition Complexes}; Theorem \ref{Theorem: Partition Complex Degree Type} states that the complex $\calP(a,p,m)$ is PR with degree types of the form $(\overbrace{1,\dots,1,\underbrace{a,1,\dots,1}_{m}}^{p})$.
	\end{itemize}
	
	\section*{Chapter \ref{Chapter: Generating Degree Types}: Pure Resolutions of Any Degree Type}
	This chapter is devoted to the proof of our most significant result, which is the following.
	\begin{itemize}
		\item Theorem \ref{Theorem: PR Complexes of Any Degree Type} states that there exist squarefree monomial ideals with pure resolutions of any given degree type.
	\end{itemize}
	We define the \textit{degree type} of a pure resolution explicitly in Chapter \ref{Chapter: PR Complexes and Degree Types} -- the terminology is borrowed from \cite{Bruns-Hibi} (page 1203), although it should be noted that our definition differs slightly from theirs. In essence, the degree type contains all the information of the shift type $(c_p,\dots,c_0)$ except for the value of the initial shift $c_0$. Thus, Theorem \ref{Theorem: PR Complexes of Any Degree Type} can be seen as a slightly weaker variant of the first Boij-S\"{o}derberg conjecture, for the cone generated by diagrams of squarefree monomial ideals. 
	
	In proving this result, we construct a number of operations on simplicial complexes which preserve the purity of the corresponding Betti diagrams, while altering the diagram's degree type.
	\begin{itemize}
		\item The operations $\{\phi_i : i\in \ZZ^+\}$ in Definition \ref{Definition: Phi_i} map PR complexes of degree type $(d_p,\dots,d_i,1,\dots,1)$ to PR complexes of degree type $(d_p,\dots,d_{i+1},d_i+1,1,\dots,1)$ (this is Theorem \ref{Theorem: Phi_i Operations Degree Type})
		\item The operations $f^\lambda$ and $f^{\free}$ in Definitions \ref{Definition: f-lambda} and \ref{Definition: f-free} have similar PR preserving properties, as shown in Corollaries \ref{Corollary: f-lambda preserves PR} and \ref{Corollary: f-free preserves PR}.
	\end{itemize}.
	Our work in this chapter also involves a brief investigation in to the barycentric subdivision $\BDelta$ of arbitrary simplicial complexes $\Delta$. There are a number of results in the literature on the combinatorial effects of the barycentric subdivision process (for instance, \cite{bary} contains a number of results on its effects on $f$- and $h$-vectors, such as Theorems 2.2 and 3.1). The central results of Section \ref{Subsection: Links in BDelta}, which we believe to be new, demonstrate the effects of the process on links, homologies of links, and both the PR and Cohen-Macaulay properties.
	\begin{itemize}
		\item Proposition \ref{Proposition: Links in BDelta} is a description of the links in $\BDelta$ in terms of links in $\Delta$.
		\item Corollary \ref{Corollary: Homology of Links in BDelta} is a description of the homologies of these links in terms of the homologies of the links in $\Delta$.
		\item Proposition \ref{Proposition: When BDelta is PR} is an equivalence statement about the conditions under which the barycentric subdivision of $\Delta$ is PR (in particular, this occurs if and only if $\Delta$ is Cohen-Macaulay).
	\end{itemize}
	
	\section*{Chapter \ref{Chapter: Future Directions}: Future Directions}
	This final chapter is a brief account of possible future research directions that could be explored based on our main results. As such, it does not contain any substantial new results, but it presents a number of open questions and conjectures.

	\chapter{Background Material}\label{Chapter: Introduction}
	We begin by presenting some preliminary results from the literature which we will need going forward.
	
	Our most central tool, the Stanley-Reisner correspondence, provides a bridge between the fields of algebra and combinatorics; as such, this introductory chapter will be split into two sections: \textit{Results from Algebra} and \textit{Results from Combinatorics}. The former of these will include a brief exposition of Betti diagrams, before moving on to Boij-S\"{o}derberg theory, and detailing how this theory allows us to classify all Betti diagrams up to multiplication by an integer. The latter will describe the Stanley-Reisner Correspondence, as well as presenting some important preliminary results concerning simplicial complexes, and graph theory.
	
	The latter section also reviews some key results from simplicial homology, but does not contain an overview of the topic; an overview can be found in \cite{Hatcher} Chapter 2.1. All of the homology we consider in this thesis will be reduced homology over our arbitrary field $\KK$, and we use the notation $\Hred_i(\Delta)$ to denote the $\nth{i}$ reduced homology group of $\Delta$ with coefficients in $\KK$. None of our results depend upon the characteristic of $\KK$.
	
	
	
	Wherever possible we define all notation we use; any exceptions to this can be found in the \textit{List of Symbols and Notation} section at the end of the thesis.
	
	\section{Results from Algebra}\label{Subsection: Results from Algbebra}
	\subsection{Free Resolutions and Betti Diagrams}\label{Subsection: Resolutions}
	We fix a finitely generated graded $R$-module $M$ (this thesis is only concerned with graded modules, so from now on we will often simply use the term `\textit{module}' to mean `\textit{graded module}'). Consider a minimal free resolution of $M$, given by
	\begin{equation}\label{Equation: Free Resolution}
		0 \rightarrow R^{\beta_p}\rightarrow \dots \rightarrow R^{\beta_0} \rightarrow M \rightarrow 0 \, .
	\end{equation}
	
	
	\begin{rem}\label{Remark: Proj-dim}
		The length $p$ of this resolution is equal to the \textit{projective dimension} of $M$, $\pdim M$. By Hilbert's Syzygy Theorem (\cite{Eisenbud} Corollary 19.7), we have $\pdim M\leq n$.
	\end{rem}
	
	\begin{defn}\label{Definition: Total Betti numbers}
		We call the ranks $\beta_i$ of the free modules in this resolution the \textit{total Betti numbers} of $M$, and often write them as $\beta_i(M)$.
	\end{defn}
	
	Note that while, in general, $M$ will admit many distinct free resolutions, the Betti numbers of $M$ are an invariant. They also encode a lot of information about $M$. In particular, $\beta_0$ is the minimal size of a generating set for $M$, because the image of any basis of $R^{\beta_0}$ under the map $R^{\beta_0}\rightarrow M$ in (\ref{Equation: Free Resolution}) generates $M$. Similarly, $\beta_1$ is equal to the minimal number of generators for the first syzygy module of $M$; and more generally, $\beta_i$ is equal to the minimal number of generators of the $\nth{i}$ syzygy module.
	
	Informally, we can think of this as follows: $\beta_0$ tells us the number of generators of the module, $\beta_1$ tells us the number of relations between those generators, $\beta_2$ tells us the number of relations between those relations, and so on.
	
	One significant limitation of total Betti numbers is that they contain no information about the grading of $M$. Ideally, we would like an invariant that encodes not just the \textit{number} of generators of $M$ and its syzygy modules, but also the \textit{degrees} of those generators.
	
	To construct such an invariant we now consider a minimal \textit{graded} free resolution of $M$, given by
	\begin{equation}\label{Equation: Graded Free Resolution}
		0\rightarrow \bigoplus_{d\in \ZZ} R(-d)^{\beta_{p,d}} \rightarrow ... \rightarrow \bigoplus_{d\in \ZZ} R(-d)^{\beta_{1,d}} \rightarrow \bigoplus_{d\in \ZZ} R(-d)^{\beta_{0,d}}\rightarrow M \rightarrow 0
	\end{equation}
	where $R(-d)$ denotes the graded $R$-module with grading given by $[R(-d)]_a = R_{a-d}$ for $a\in \ZZ$. For each $i$, the integers $d$ for which $\beta_{i,d}\neq 0$ are called \textit{shifts} (at homological degree $i$). The resolution (\ref{Equation: Graded Free Resolution}) is graded in the sense that the maps in the resolution send elements of degree $d$ in one module to elements of degree $d$ in the next. Note that the nonzero entries in the matrices representing these maps all have positive degree.
	
	\begin{defn}\label{Definition: Graded Betti numbers}
		We call the exponents $\beta_{i,d}$ in (\ref{Equation: Graded Free Resolution}) the \textit{graded Betti numbers} of $M$ (or in this thesis, usually simply the \textit{Betti numbers} of $M$). We often write them as $\beta_{i,d}(M)$. 
	\end{defn}

	\begin{ex}\label{Example: Graded Free Resolution}
		Suppose $n=4$. The squarefree monomial ideal $I$ of $R$ generated by $x_1x_4$ and $x_1x_2x_3$ admits the following minimal graded free resolution.
		$$\begin{tikzcd}[ampersand replacement=\&]
			0 \rar \& R(-4) \rar{ \begin{pmatrix}
					x_2x_3 \\ -x_4
			\end{pmatrix} }\&[1em] R(-2)\oplus R(-3) \rar{ \begin{pmatrix}
					x_1x_4 & x_1x_2x_3
			\end{pmatrix} } \&[3em] I\rar \& 0.
		\end{tikzcd}$$
		The nonzero Betti numbers of $I$ are $\beta_{0,2}(I)=\beta_{0,3}(I)=\beta_{1,4}(I)=1$.
	\end{ex}
	
	Just as the total Betti number $\beta_0(M)$ is equal to the minimum number of generators of $M$, the graded Betti number $\beta_{0,d}(M)$ is equal to the minimum number of generators of $M$ of degree $d$. Similarly $\beta_{i,d}(M)$ is equal to the minimum number of generators of the $\nth{i}$ syzygy module of $M$ of degree $d$.
	
	The total Betti numbers and graded Betti numbers of $M$ are related via the equations
	$$\beta_i = \sum_{d\in \ZZ} \beta_{i,d} \text{ for each } 0\leq i \leq n\, .$$
	
	\begin{defn}\label{Definition: Betti Diagrams}
		The \textit{Betti diagram} of $M$, denoted $\beta(M)$, is a matrix containing the graded Betti numbers of $M$. In order to reduce the number of rows of this matrix, the standard notational convention (which we adopt for this thesis) is to write $\beta(M)$ as a matrix $(a_{ij})$ with $a_{ij}=\beta_{j,i+j}$, as shown.
		$$\begin{bmatrix}
			\vdots& \vdots& &\vdots \\
			\beta_{0,0} & \beta_{1,1} & \dots & \beta_{n,n} \\
			\beta_{0,1} & \beta_{1,2} & \dots & \beta_{n,n+1}\\
			\vdots & \vdots & & \vdots\\ 
		\end{bmatrix}$$
	\end{defn}
	Note that while this matrix is infinite, only finitely many of the entries are nonzero.
	
	We also occasionally write Betti diagrams in table notation as follows.
	\begin{notation}\label{Notation: Betti diagrams Table Format}	Suppose $\beta$ is a Betti diagram of an $R$-module. We can write $\beta$ as a table in the following way.
		$$\begin{array}{c | cccc}
			& 0 & 1 & \dots & n\\
			\hline
			0&\beta_{0,0}&\beta_{1,1}&\dots&\beta_{n,n}\\
			1&\beta_{0,1}&\beta_{1,2}&\dots&\beta_{n,n+1}\\
			\vdots&\vdots&\vdots&\vdots&\vdots
		\end{array}$$
	\end{notation}
	
	\begin{ex}\label{Example: Table Betti Diagram}
		The table $\begin{array}{c | ccc}
			& 0 & 1 & 2\\
			\hline
			4 & 8 & 9 & 1\\
			5 & . & . & 1\\ 
		\end{array}$ denotes a Betti diagram $\beta$ for which $\beta_{0,4}=8$, $\beta_{1,5}=9$, $\beta_{2,6}=\beta_{2,7}=1$, and all other Betti numbers are zero.
	\end{ex}
	
	Just as the projective dimension of a module can be seen as the ``\textit{length}'' of its Betti diagram, the following property can be seen as the Betti diagram's ``\textit{width}''.
	\begin{defn}\label{Definition: Regularity}
		Let $M$ be an $R$-module with Betti diagram $\beta(M)$. The \textit{(Castelnuovo-Mumford) regularity} of $M$ is given by $\reg(M)=\max\{d-i: \beta_{i,d}(M)\neq 0\}$.
	\end{defn}
	
	We have a number of important constraints on the Betti numbers of $R$-modules. The following is perhaps one of the most famous, and the most important for our work going forward. It holds for all $R$-modules of a given codimension $h$ (for a proof, see e.g. \cite{Floy} Section 1.3).
	
	\begin{lem}[Herzog-K\"{u}hl Equations]\label{Lemma: HK Equations}
		Let $M$ be a graded $R$-module of codimension $h$, and let $\beta=\beta(M)$. We have
		\begin{equation*}
			\sum_{i,d}(-1)^i d^j \beta_{i,d} = 0 \text{ for each } 0\leq j\leq h-1.
		\end{equation*}
	\end{lem}
	
	\begin{notation}\label{Notation: HK Equations}
		For convenience, for each $0\leq j\leq h-1$, we use the notation $\HK_j(\beta)$ to represent the expression $\sum_{i,d}(-1)^i d^j \beta_{i,d}$. In this notation, the Herzog-K\"{u}hl equations for modules of codimension $h$ can be rephrased as the statement $\HK_j(\beta)=0$ for each $0\leq j\leq h-1$.
	\end{notation}
	
	All of the $R$-modules we care about in this thesis are ideals, so we end this section with two key observations about the Betti diagrams of ideals.
	\begin{rem}\label{Remark: HK Equations}
		Let $I$ be an ideal in $R$ of height $h$. This means that the quotient $R/I$ has codimension $h$, and therefore, setting $M=R/I$ in Lemma \ref{Lemma: HK Equations} gives us $\HK_j(\beta(R/I))=0$ for each $0\leq j \leq h-1$. We can use this to find similar linear dependency relations for the diagram $\beta(I)$.
		
		Specifically, if
		$$0\rightarrow \bigoplus_{d\in \ZZ} R(-d)^{\beta_{p,d}} \rightarrow ... \rightarrow \bigoplus_{d\in \ZZ} R(-d)^{\beta_{1,d}} \rightarrow \bigoplus_{d\in \ZZ} R(-d)^{\beta_{0,d}}\rightarrow I \rightarrow 0$$
		is a minimal graded free resolution of $I$, then
		$$0\rightarrow \bigoplus_{d\in \ZZ} R(-d)^{\beta_{p,d}} \rightarrow ... \rightarrow \bigoplus_{d\in \ZZ} R(-d)^{\beta_{1,d}} \rightarrow \bigoplus_{d\in \ZZ} R(-d)^{\beta_{0,d}}\rightarrow R(0) \rightarrow R/I \rightarrow 0$$
		is a minimal graded free resolution of $R/I$.
		
		This means that the Betti diagrams of $R$ and $R/I$ are directly related via
		\begin{align*}
			\beta_{i,d}(R/I)=\begin{cases}
				\beta_{i-1,d}(I) &\text{ if } i > 0\\
				1  &\text{ if } (i,d)=(0,0) \\
				0 & \text{ otherwise.}
			\end{cases}
		\end{align*}
		
		Hence, for any $j\geq 0$, we have
		\begin{align*}
			\HK_j(\beta(R/I)) &= \sum_{i,d}(-1)^i d^j \beta_{i,d}(R/I)\\
			&= 0^j + \sum_{i,d}(-1)^i d^j \beta_{i-1,d}(I)\\
			&= 0^j - \sum_{i,d}(-1)^i d^j \beta_{i,d}(I)\\
			&=\begin{cases}
				1 - \HK_j(\beta(I)) & \text{ if } j=0\\
				- \HK_j(\beta(I))  & \text{ otherwise.}
			\end{cases}
		\end{align*}
		
		Thus, for any ideal $I$ of height $h$, the Herzog-K\"{u}hl equations tell us that $\HK_0(\beta(I))=1$ and $\HK_j(\beta(I))=0$ for each $1\leq j\leq h-1$.
	\end{rem}
	
	\begin{rem}\label{Remark: Betti diagrams of ideals only depend on generators}
		Suppose $I$ is an ideal in $R$, and let
		$$F_\bullet: 0\rightarrow \bigoplus_{d\in \ZZ} R(-d)^{\beta_{p,d}} \rightarrow ... \rightarrow \bigoplus_{d\in \ZZ} R(-d)^{\beta_{1,d}} \rightarrow \bigoplus_{d\in \ZZ} R(-d)^{\beta_{0,d}}\rightarrow I \rightarrow 0$$
		be a minimal free graded resolution for $I$ as an $R$-module.
		
		Let $S$ denote the polynomial ring $R[x_{n+1}]$, and consider the ideal $IS$ in $S$. Because $S$ is a free $R$-module, it is also faithfully flat, and hence $$F_\bullet\otimes_R S: 0\rightarrow \bigoplus_{d\in \ZZ} S(-d)^{\beta_{p,d}} \rightarrow ... \rightarrow \bigoplus_{d\in \ZZ} S(-d)^{\beta_{1,d}} \rightarrow \bigoplus_{d\in \ZZ} S(-d)^{\beta_{0,d}}\rightarrow IS \rightarrow 0$$ is a minimal free graded resolution for $IS$ as an $S$-module.
		
		In particular, any generating set for $I$ as an $R$-ideal is also a generating set for $IS$ as an $S$-ideal. Hence the Betti diagram of $I$ is dependent only on its generators, and not on the number of variables in its ambient polynomial ring.
	\end{rem}
	
	\subsection{Boij-S\"{o}derberg Theory and Betti Cones}\label{Subsection: Boij-Soderberg}
	
	One of the most significant advances in our understanding of Betti Diagrams came from a paper put on the Arxiv by Boij and S\"{o}derberg (and later published in 2008 as \cite{BS-Original-2006}), which laid out two important conjectures about the Betti diagrams of Cohen-Macaulay modules, both of which have since been proven. In this section we give a brief exposition of these two conjectures, and the surrounding theory; a more comprehensive overview can be found in \cite{Floy} Section 1.
	
	The aim of \cite{BS-Original-2006} was a classification result for Betti diagrams of $R$-modules; but instead of trying to classify all such Betti diagrams directly, the authors set themselves the task of classifying them \textit{up to multiplication by a positive rational} (or equivalently, by a positive \textit{integer}, because all of the diagrams have integer values). Their central insight was that any Betti diagram of an $R$-module may be viewed as a vector in the infinite-dimensional rational vector space $\Vn=\bigoplus_{d\in \ZZ} \QQ^{n+1}$, and thus their goal was to classify the positive rational rays spanned by the diagrams in this space.
	
	Note in particular that for any two $R$-modules $M_1$ and $M_2$, and any two positive integers $q_1$ and $q_2$, we have $q_1 \beta(M_1)+q_2 \beta(M_2)=\beta(M_1^{q_1}\oplus M_2^{q_2})$. It follows that the set of positive rational rays spanned by the Betti diagrams in $\Vn$ form a cone. The task, therefore, was to describe the cone.

	Boij and S\"{o}derberg initially restricted their attention to cones generated by diagrams of Cohen-Macaulay $R$-modules of a given codimension $h$, for some fixed $h\leq n$.  Recall that Cohen-Macualay modules are defined in terms of the Krull-dimension and depth of a module as follows (see \cite{C-M} Section 2.1 for more details).
	\begin{defn}\label{Definition: CM Modules}
		An $R$-module $M$ is \textit{Cohen-Macaulay} if we have $\dim M = \depth M$.
	\end{defn}
	
	We can find the projective dimensions of these modules from the Auslander-Buchsbaum formula for graded rings (see e.g. \cite{Eisenbud} Exercise 19.8), which is the following.
	\begin{thm}[Auslander-Buchsbaum Formula for Graded Rings]\label{Theorem: Auslander-Buchsbaum}
		Let $S$ be a graded ring and let $M$ be a finitely generated graded $S$-module of finite projective dimension. We have $$\pdim M = \depth S - \depth M.$$
	\end{thm}
	\begin{cor}[Auslander-Buchsbaum Formula for the polynomial Ring $R$]\label{Corollary: Auslander-Buchsbaum for R}
		Let $M$ be a finitely generated graded $R$-module. We have $$\pdim M = n - \depth M.$$
	\end{cor}
	\begin{proof}
		The polynomial ring $R$ is Cohen-Macaulay with depth and Krull-dimension both equal to $n$, and all finitely generated $R$-modules have finite projective dimension. 
	\end{proof}
	
	In the case where $M$ is a Cohen-Macaulay $R$-module of codimension $h$ we have $$\pdim M = n - \dim M=\codim M = h.$$ More generally, if $M$ is an arbitrary $R$-module we have $\depth M \leq \dim M$, and hence $\pdim M \geq \codim M$. Thus the diagrams of Cohen-Macaulay modules are the ones with minimal length for their given codimension.
	
	The cone generated by \textit{all} Cohen-Macaulay $R$-modules of codimension $h$ is infinite-dimensional, which makes it difficult to work with. For this reason, Boij and S\"{o}derberg restricted their attention further to finite dimensional subcones generated by Betti diagrams within a given window. We lay out how to build these cones below.
	
	For the rest of this section we fix two strictly decreasing sequences of integrs $\ba =(a_h,\dots,a_0)$ and $\bfb=(b_h,\dots,b_0)$ in $\ZZ^{h+1}$, such that $a_i\leq b_i$ for each $0\leq i \leq h$. Note that this condition imposes a partial order on the sequences in $\ZZ^{h+1}$, which we will henceforth denote by $\ba \leq \bfb$.
	
	\begin{defn}\label{Definition: Boij-Soderberg Cone and Indexing Window}
		We define
		\begin{enumerate}
			\item $\II(\ba,\bfb)=\left\{(i,d)\in \{0,\dots,h\}\times \ZZ : a_i\leq d \leq b_i \right \}$.
			\item $\calV(\ba,\bfb)= \bigoplus_{(i,d)\in \II(\ba,\bfb)} \QQ \subset \Vn$.
			\item $\calC(\ba,\bfb)$ to be the cone of positive rational rays $t \cdot\beta$ for Betti diagrams $\beta$ of Cohen-Macaulay $R$-modules of codimension $h$ which lie inside the subspace $\calV(\ba,\bfb)$.
		\end{enumerate}
	\end{defn}
	
	Boij and S\"{o}derberg's key aim was to find the extremal rays of the cone $\calC(\ba,\bfb)$. These are vectors $v_1,\dots, v_m$ in the cone such that
	\begin{itemize}
		\item Every vector in the cone can be written as a positive rational sum of the form $\lambda_1 v_1 + \dots \lambda_m v_m$.
		\item For each $1\leq i \leq m$ there do not exist linearly independent vectors $u$ and $w$ in the cone such that $v_i = u + w$.
	\end{itemize}
	
	It turns out that these extremal rays come from so-called \textit{pure} Betti diagrams, defined as follows.
	\begin{defn}\label{Definition: Pure Diagrams and Shift Type}
		Let $M$ be an $R$-module. We say the Betti diagram $\beta(M)$ is \textit{pure} if for every $0\leq i \leq n$ there is at most a single value of $c$ such that the diagram $\beta_{i,c}(M)$ is nonzero.
		
		In this case, $M$ admits a graded resolution of the form $$0\rightarrow R(-c_p)^{\beta_{p,c_p}} \rightarrow ... \rightarrow R(-c_1)^{\beta_{1,c_1}} \rightarrow R(-c_0)^{\beta_{0,c_0}}\rightarrow M.$$ We call such a resolution a \textit{pure resolution}, and we say it has \textit{shift type} $(c_p,..., c_0)$ (after the terminology of \cite{Bruns-Hibi} page 1203).
	\end{defn}
	
	Let $\bc=(c_h,\dots,c_0)$ be a strictly decreasing sequence such that $\ba \leq \bc \leq \bfb$. Any pure Betti diagram $\beta$ of shift type $\bc$ corresponding to a Cohen-Macaulay module of codimension $h$ must satisfy the Herzog-K\"{u}hl equations $\HK_0(\beta)=\dots = \HK_{h-1}(\beta)=0$. This gives us a system of $h$ linearly independent equations in the $h+1$ variables $\beta_{0,c_0},\dots,\beta_{h,c_h}$, and hence there is only a single ray of solutions in $\Vn$. We denote the smallest integer solution on this ray by $\pi(\bc)$, and we can compute it as follows (see \cite{BS-Original-2006}, Section 2.1).
	\begin{prop}\label{Proposition: pi(c)}
		Let $\pi=\pi(\bc)\in \Vn$ be the smallest pure integer solution to the Herzog-K\"{u}hl equations $\HK_0(\beta)=\dots=\HK_{h-1}(\beta)=0$. We have $$\pi_{i,c_i}=\Pi_{0\neq k \neq i} \frac{c_k - c_0}{c_k - c_i}.$$
	\end{prop}
	
	To show that the rays spanned by these pure diagrams are the extremal rays of the cone $\calC(\ba,\bfb)$, we need to prove two things. First, we need to show that they actually lie in the cone: that is, for every decreasing sequence $\bc$ in $\ZZ^{h+1}$ with $\ba\leq \bc \leq \bfb$ the ray $t\cdot \pi(\bc)$ actually contains a Betti diagram of a Cohen-Macaulay module of codimension $h$ (this is not immediate from the above discussion; Proposition \ref{Proposition: pi(c)} shows that any pure diagram of shift type $\bc$ corresponding to a Cohen-Macaulay module of height $h$ must lie on the ray $t\cdot \pi(\bc)$ \textit{if it exists}, but does not guarantee the existence of such a diagram).
	
	Second, we must show that the diagrams $\pi(\bc)$ satisfy the conditions of extremal rays. It is easy to see that $\pi(\bc)$ cannot be decomposed into the sum of two linearly independent vectors in the cone (indeed, if $\pi(\bc)=\beta^1+\beta^2$ then $\beta^1$ and $\beta^2$ must also lie on the ray $t\cdot \pi(\bc)$ by considering the position of zero entries), so what remains to show is that the diagram of any Cohen-Macaulay module of codimension $h$ can be decomposed into a rational sum of pure diagrams.
	
	The two claims above are the Boij-S\"{o}derberg conjectures (Theorems 1.9 and 1.11 in \cite{Floy}). The first conjecture was subsequently proven by Eisenbud, Fl\o{}ystad and Weyman in \cite{BS1-Proof-Paper},and the second by Eisenbud and Schreyer in \cite{BS2-Proof-Paper}.
	\begin{thm}[First Boij-S\"{o}derberg Conjecture]\label{Theorem: BS Conjecture 1}
		For any strictly decreasing sequence of integers $\bc=(c_h,\dots,c_0)$, there exists a Cohen-Macaulay graded $R$-module $M$ of codimension $h$ with a pure resolution of shift type $\bc$.
	\end{thm}
	
	\begin{thm}[Second Boij-S\"{o}derberg Conjecture]\label{Theorem: BS Conjecture 2}
		Let $M$ be a Cohen-Macaulay graded $R$-module of codimension $h$ whose Betti diagram lies inside $\calV(\ba,\bfb)$. There exist positive rational numbers $q_i$ and a chain of decreasing sequences $\bc^1 < \dots < \bc^r$ in $\ZZ^{h+1}$ such that
		$$\beta(M) = q_1 \pi(\bc^1)+\dots + q_r \pi (\bc^r).$$
	\end{thm}
	
	Boij and S\"{o}derberg later extended this second result to the non-Cohen-Macaulay case (\cite{BS-Non-CM} Theorem 2). In order to state this revised result, we will use the notation $\ZZ^{\leq n+1}$ to denote the set of integer sequences of length at most $n+1$. We can extend our earlier partial order on $\ZZ^{h+1}$ to this set by stipulating that $\ba \leq \bfb$ for sequences $\ba=(a_{p_1},\dots,a_0)$ and $\bfb=(b_{p_2},\dots,b_0)$ in $\ZZ^{\leq n+1}$ whenever $p_1\geq p_2$ and $a_i\leq b_i$ for each $0\leq i \leq p_1$ . 
	\begin{thm}[Second Boij-S\"{o}derberg Conjecture for Arbitrary Modules]\label{Theorem: BS Conjecture 2 non-CM}
		Let $M$ be an arbitrary $R$-module. There exist positive rational numbers $q_i$ and a chain of decreasing sequences $\bc^1<\dots<\bc^r$ in $\ZZ^{\leq n+1}$ such that $$\beta(M) = q_1 \pi(\bc^1)+\dots + q_r \pi (\bc^r).$$
	\end{thm}
	
	Together, Theorem \ref{Theorem: BS Conjecture 1} and Theorem \ref{Theorem: BS Conjecture 2 non-CM} provide a complete classification of all Betti diagrams of $R$-modules up to multiplication by a positive integer. Theorem \ref{Theorem: BS Conjecture 1} also allows us to classify the vector space spanned by the cone $\calC(\ba,\bfb)$.
	
	\begin{defn}\label{Definition: Boij-Soderberg Space}
		We define
		\begin{enumerate}
			\item $\calV^{\HK }(\ba,\bfb)= \left\{ \beta \in \calV(\ba,\bfb) : \HK_0(\beta)=\dots=\HK_{h-1}(\beta)=0 \right\}$.
			\item $\WW(\ba,\bfb)$ to be the subspace of $\calV(\ba,\bfb)$ generated by the diagrams in $\calC(\ba,\bfb)$.
		\end{enumerate}
	\end{defn}
	
	The former space $\calV^{\HK }(\ba,\bfb)$ has a basis consisting of pure diagrams $\pi(\bc)$ (this is Proposition 1.8 in \cite{Floy}).
	\begin{prop}\label{Proposition: Basis for V-HK}
		For any maximal chain $\ba = \bc^1 < \dots < \bc^r = \bfb$ in $\ZZ^{h+1}$, the diagrams $\pi(\bc^1),\dots, \pi(\bc^r)$ form a basis for $\calV^{\HK }(\ba,\bfb)$.
	\end{prop}
	The latter space $\WW(\ba,\bfb)$ must be a subspace of $\calV^{\HK }(\ba,\bfb)$, because all of its generating Betti diagrams satisfy the Herzog-K\"{u}hl Equations. In fact, it is a consequence of the first Boij-S\"{o}derberg conjecture that these two spaces are equal (this is Corollary 1.10 in \cite{Floy}).
	\begin{cor}\label{Corollary: BS Conjecture 1 Vector Space}
		The spaces $\WW(\ba,\bfb)$ and $\calV^{\HK}(\ba,\bfb)$ are equal.
	\end{cor}
	\begin{proof}
		By Theorem \ref{Theorem: BS Conjecture 1}, the space $\WW(\ba,\bfb)$ contains all the diagrams $\pi(\bc)$ for strictly decreasing sequences $\bc$ in $\ZZ^{h+1}$ with $\ba \leq \bc \leq \bfb$, and these diagrams generate $\calV^{\HK}(\ba,\bfb)$ by Proposition \ref{Proposition: Basis for V-HK}
	\end{proof}
	
	The key results of this thesis (Theorems \ref{Theorem: dimDn}, \ref{Theorem: dimCn},  \ref{Theorem: dimDnh} and \ref{Theorem: dimDnh} im Chapter \ref{Chapter: Dimensions}, and Theorem \ref{Theorem: PR Complexes of Any Degree Type} in Chapter \ref{Chapter: PR Complexes and Degree Types}) can be seen as analogues of either the first Boij-S\"{o}derberg conjecture or its above corollary, for Betti cones generated exclusively by diagrams of squarefree monomial ideals.
	
	In order to study the diagrams of squarefree monomial ideals, we must now delve into the world of combinatorics.
	
	\section{Results from Combinatorics}\label{Subsection: Results from Combinatorics}
	
	\subsection{Simplicial Complexes and Stanley-Reisner Ideals}\label{Subsection: Simplicial Complexes}
	Simplicial complexes are important combinatorial and topological objects, defined as follows.
	
	\begin{defn}\label{Definition: Simplicial Complex}
		A \textit{simplicial complex} $\Delta$ on vertex set $V$ is a set of subsets of $V$ such that for any $G \subseteq F \subseteq V$, if $F$ is in $\Delta$ then $G$ is also in $\Delta$.
		
		We often denote the vertex set of a complex $\Delta$ as $V(\Delta)$.
	\end{defn}
	\begin{rem}
		For our purposes the vertex set $V$ will always be finite.
	\end{rem}
	
	\begin{defn}\label{Definition: Simplicial Complex Key Terminology and Constructions}
		Let $\Delta$ be a simplicial complex on a vertex set $V$ of size $n$.
		\begin{enumerate}
			\item We refer to the elements of $\Delta$ as \textit{simplices} or \textit{faces}.
			\item A \textit{facet} of $\Delta$ is a face of $\Delta$ that is maximal with respect to inclusion.
			\item A \textit{minimal nonface} of $\Delta$ is a subset of $V$ that is not a face of $\Delta$ and is minimal with this property with respect to inclusion.  Note that $\Delta$ need not contain every singleton subset of $V$. A singleton minimal nonface of $\Delta$ is called a \textit{missing vertex}.
			\item A vertex in $V$ which is contained in only one facet of $\Delta$ is called a \textit{free vertex}.
			\item The \textit{dimension} of a face $\sigma$ in $\Delta$, denoted $\dim \sigma$, is $|\sigma|-1$. The dimension of $\Delta$, denoted $\dim \Delta$, is the dimension of its largest facet. If all facets of $\Delta$ have the same dimension, we say $\Delta$ is \textit{pure}.
			\item The \textit{codimension} of $\Delta$, denoted $\codim \Delta$, is $n-\dim \Delta -1$.
			\item For a face $\sigma\in \Delta$ we define the \textit{link} of $\sigma$ in $\Delta$ to be the complex $$\lkds=\{\tau \in \Delta : \sigma \cap \tau = \emptyset, \sigma \cup \tau \in \Delta\}.$$
			\item We define the \textit{Alexander Dual} of $\Delta$ to be the complex $$\Delta^* = \{F \subseteq V : V - F \notin \Delta\}.$$
			\item For two complexes $\Delta_1$ and $\Delta_2$ on disjoint vertex sets $V_1$ and $V_2$ we define the \textit{join} of $\Delta_1$ and $\Delta_2$ on vertex set $V_1 \sqcup V_2$ as $$ \Delta_1 \ast \Delta_2 = \{\sigma_1 \sqcup \sigma_1 : \sigma_1 \in \Delta_1, \sigma_2 \in \Delta_2 \}.$$
			\item The join of $\Delta$ with a single vertex is called the \textit{cone} over $\Delta$. We denote it by  $C\Delta$.
			\item The join of $\Delta$ with two disjoint vertices is called the \textit{suspension} of $\Delta$. We denote it by $S\Delta$.
		\end{enumerate}
	\end{defn}
	
	We often write simplicial complexes in \textit{facet notation}, as follows. 
	\begin{notation}\label{Notation: Facet Notation for Complexes}
		For subsets $F_1$,...,$F_m$ in a vertex set $V$, we use $\langle F_1,...,F_m\rangle$ to denote the smallest simplicial complex on $V$ containing $F_1$,...,$F_m$. Note that if the set $\{F_1,...,F_m\}$ is irredundant (i.e. no pair of subsets $F_i$ and $F_j$ satisfy $F_i \subseteq F_j$), then this is the same as the complex on $V$ with facets $F_1,...,F_m$.
	\end{notation}
	
	\begin{rem}
		Suppose $\sigma$ is a face in a simplicial complex $\Delta$, and $F_1,\dots,F_m$ are all the facets of $\Delta$ which contain $\sigma$. We have $\lkds = \langle F_1 -\sigma, \dots, F_m- \sigma \rangle$.
	\end{rem}
	
	\begin{ex}
		Suppose $\Delta$ is the complex  $\langle \{1,2,3\},\{3,4\} \rangle$ on vertex set $[4]$. We can draw $\Delta$ as follows.
		
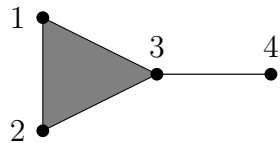
\begin{center}
			\begin{tikzpicture}[scale = 0.75]
				\tikzstyle{point}=[circle,thick,draw=black,fill=black,inner sep=0pt,minimum width=4pt,minimum height=4pt]
				
				\node (a)[point,label=left:$1$] at (0,1) {};
				\node (b)[point,label=left:$2$] at (0,-1) {};
				\node (c)[point,label=above:$3$] at (2,0) {};
				\node (d)[point,label=above:$4$] at (4,0) {};
				
				\begin{scope}[on background layer]
					\draw[fill=gray] (a.center) -- (b.center) -- (c.center) -- cycle;
					\draw[fill=gray] (c.center) -- (d.center);
				\end{scope}
			\end{tikzpicture}
		\end{center}
		\begin{itemize}
			\item This complex has dimension $2$, but it is not pure because it also has a facet of dimension $1$.
			\item The minimal nonfaces of $\Delta$ are $\{1,4\}$ and $\{2,4\}$, and hence the Alexander dual $\Delta^*$ has facets $[4]-\{1,4\}=\{2,3\}$ and $[4]-\{2,4\}=\{1,3\}$.
			\item 
			The cone $C\Delta$ looks like the following.
			$$
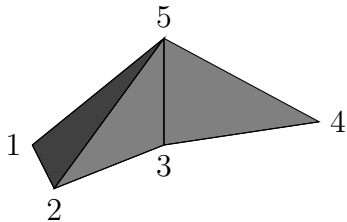
\begin{tikzpicture}[scale=1][line join = round, line cap = round]
				
				
				\coordinate [label=above:5] (5) at ({.5*sqrt(3)},{sqrt(2)},-.5);
				\coordinate [label=left:1] (1) at ({-.5*sqrt(3)},0,-.5);
				\coordinate [label=below:2] (2) at (0,0,1);
				\coordinate [label=below:3] (3) at ({.5*sqrt(3)},0,-.5);
				\coordinate [label=right:4] (4) at ({1.5*sqrt(3)},0,-1.3);
				
				\begin{scope}
					\draw (1)--(3);
					\draw[fill=darkgray,fill opacity=.5] (2)--(1)--(5)--cycle;
					\draw[fill=gray,fill opacity=.5] (3)--(2)--(5)--cycle;
					\draw[fill=gray,fill opacity=.5] (3)--(4)--(5)--cycle;
					\draw (2)--(1);
					\draw (2)--(3);
					\draw (3)--(5);
					\draw (2)--(5);
					\draw (1)--(5);
					\draw (3)--(4);
				\end{scope}
			\end{tikzpicture}$$
		\end{itemize}
	\end{ex}

	We now review some key subcomplex constructions.
	\begin{defn}\label{Definition: Subcomplexes}
		Let $\Delta$ be a simplicial complex on vertex set $V$.
		\begin{enumerate}
			\item For a subset $U\subset V$, we define the \textit{induced subcomplex} of $\Delta$ on $U$ to be the complex $$\Delta_U=\{\sigma \in \Delta : \sigma \subseteq U\}.$$
			\item For a face $g$ in $\Delta$ we define the \textit{deletion} of $g$ from $\Delta$ to be the complex $$\Delta-g = \{\sigma \in \Delta : \sigma \nsupseteq g \}.$$
		\end{enumerate}
	\end{defn}
	\begin{rem}
		We will sometimes denote the induced subcomplex $\Delta_U$ as $\Delta|_U$ to avoid ambiguity. 
	\end{rem}
	
	The complex $\Delta - g$ is the largest subcomplex of $\Delta$ that does not contain $g$. Note that, except in the case where $g$ is a vertex, this is larger than the induced subcomplex $\Delta_{V-g}$, because it contains all faces of $\Delta$ which intersect with $g$ strictly along its boundary. The following example illustrates the difference between the two constructions.
	
	\begin{ex}
		Let $\Delta$ be the complex
		
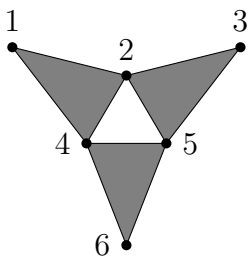
\begin{center}
			\begin{tikzpicture}[scale = 0.75]
				\tikzstyle{point}=[circle,thick,draw=black,fill=black,inner sep=0pt,minimum width=3pt,minimum height=3pt]
				\node (a)[point,label=above:$1$] at (0,3.5) {};
				\node (b)[point,label=above:$2$] at (2,3) {};
				\node (c)[point,label=above:$3$] at (4,3.5) {};
				\node (d)[point,label=left:$4$] at (1.3,1.8) {};
				\node (e)[point,label=right:$5$] at (2.7,1.8) {};
				\node (f)[point,label=left:$6$] at (2,0) {};	
				
				\begin{scope}[on background layer]
					\draw[fill=gray] (a.center) -- (b.center) -- (d.center) -- cycle;
					\draw[fill=gray] (b.center) -- (c.center) -- (e.center) -- cycle;
					\draw[fill=gray]   (d.center) -- (e.center) -- (f.center) -- cycle;
				\end{scope}
			\end{tikzpicture}
		\end{center}
		on vertex set $[6]$, and let $g$ denote the edge \{2,4\}.
		
		
		The complex $\Delta - g$ is equal to
		
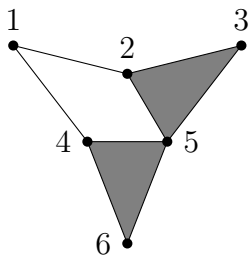
\begin{center}
			\begin{tikzpicture}[scale = 0.75]
				\tikzstyle{point}=[circle,thick,draw=black,fill=black,inner sep=0pt,minimum width=3pt,minimum height=3pt]
				\node (a)[point,label=above:$1$] at (0,3.5) {};
				\node (b)[point,label=above:$2$] at (2,3) {};
				\node (c)[point,label=above:$3$] at (4,3.5) {};
				\node (d)[point,label=left:$4$] at (1.3,1.8) {};
				\node (e)[point,label=right:$5$] at (2.7,1.8) {};
				\node (f)[point,label=left:$6$] at (2,0) {};	
				
				\begin{scope}[on background layer]
					\draw (d.center) -- (a.center) -- (b.center);
					\draw[fill=gray] (b.center) -- (c.center) -- (e.center) -- cycle;
					\draw[fill=gray]   (d.center) -- (e.center) -- (f.center) -- cycle;
				\end{scope}
			\end{tikzpicture}
		\end{center}
		whereas the induced subcomplex $\Delta_{[6]-g}$ is equal to
		
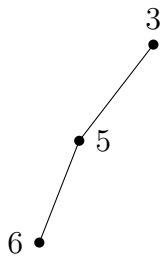
\begin{center}
			\begin{tikzpicture}[scale = 0.75]
				\tikzstyle{point}=[circle,thick,draw=black,fill=black,inner sep=0pt,minimum width=3pt,minimum height=3pt]
				\node (a)[point,label=above:$1$] at (0,3.5) {};
				\node (c)[point,label=above:$3$] at (4,3.5) {};
				\node (e)[point,label=right:$5$] at (2.7,1.8) {};
				\node (f)[point,label=left:$6$] at (2,0) {};	
				
				\begin{scope}[on background layer]
					\draw(c.center) -- (e.center);
					\draw (e.center) -- (f.center);
				\end{scope}
			\end{tikzpicture}
		\end{center}
	\end{ex}
	
	We now review some examples of important complexes, which will be crucial to our work going forward.
	\begin{defn}\label{Definition: Trivial Complexes}
		There are two trivial simplicial complexes on any vertex set.
		\begin{enumerate}
			\item The \textit{irrelevant complex} is the set $\{\emptyset\}$.
			\item The \textit{void complex} is the empty set $\emptyset$.
		\end{enumerate}
	\end{defn}
	\begin{rem}\label{Remark: void does not equal irrelevant}
		Note that the irrelevant complex and the void complex are \textit{not} the same complex. The former has a single face, of dimension $-1$, while the latter has no faces at all. In particular, the former complex has $\nth[st]{(-1)}$ reduced homology, while the latter is acyclic.
	\end{rem}
	\begin{defn}\label{Definition: Simplex and Boundary of Simplex}
		Let $j\geq -1$ be an integer, and let $V$ be a vertex set of size $j+1$.
		\begin{enumerate}
			\item The \textit{full simplex} on vertex set $V$ is the complex whose faces are every subset of $V$. We often call $\Delta$ the \textit{$j$-simplex}, and denote it by $\Delta^j$.
			\item The \textit{boundary of the simplex} on vertex set $V$ is the complex whose faces are every proper subset of $V$. We often call $\Delta$ the \textit{boundary of the $j$-simplex}, and denote it by $\partial \Delta^j$.
		\end{enumerate}
	\end{defn}
	\begin{rem}\label{Remark: j-Simplex Isomorphism classes}
		Technically speaking, the $j$-simplex and its boundary are actually isomorphsim classes of complexes rather than individual complexes, because they can be defined on any vertex set of size $j+1$. By convention, we will assume they are defined on vertex set $[j+1]$ unless otherwise stated.
	\end{rem}
	\begin{defn}\label{Definition: Skeleton Complex}
		Let $r$ be an integer greater than or equal to $-1$, and let $\Delta$ be any complex. We define the \textit{$r$-skeleton of $\Delta$} to be the complex
		$$\Skel_r(\Delta) = \{\sigma \in \Delta : \dim \sigma \leq r\}$$
		If $\Delta$ is a full simplex on vertex set $V$ we often denote this as $\Skel_r(V)$.
	\end{defn}
	\begin{rem}\label{Remark: Skeleton Complexes}
		\begin{enumerate}
			\item The $j$-simplex $\Delta^j$ is equal to $\Skel_j([j+1])$.
			\item The boundary of the $j$-simplex $\partial \Delta^j$ is equal to $\Skel_{j-1}([j+1])$.
			\item For any complex $\Delta$ and any $r\geq \dim \Delta$ we have $\Skel_r(\Delta) = \Delta$.
			\item For any complex $\Delta$ we have $\Skel_{-1}(\Delta)$ is the irrelevant complex $\{\emptyset\}$ on $V(\Delta)$, and for $r<-1$ we have that $\Skel_{r}(\Delta)$ is the void complex $\emptyset$.
		\end{enumerate}
	\end{rem}
	\begin{ex}
		The $3$-simplex $\Delta^3$ on vertex set $[4]$ has the following nontrivial skeleton complexes.
		\begin{center}
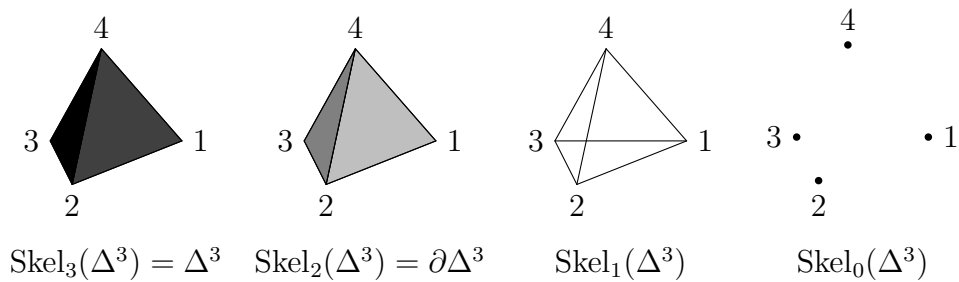
			\begin{tabular}{ c c c c }
				
				\begin{tikzpicture}[line join = round, line cap = round]
					\coordinate [label=above:4] (4) at (0,{sqrt(2)},0);
					\coordinate [label=left:3] (3) at ({-.5*sqrt(3)},0,-.5);
					\coordinate [label=below:2] (2) at (0,0,1);
					\coordinate [label=right:1] (1) at ({.5*sqrt(3)},0,-.5);
					
					\begin{scope}
						\draw (1)--(3);
						\draw[fill=darkgray,fill opacity=1] (2)--(1)--(4)--cycle;
						\draw[fill=black,fill opacity=1] (3)--(2)--(4)--cycle;
						\draw (2)--(1);
						\draw (2)--(3);
						\draw (3)--(4);
						\draw (2)--(4);
						\draw (1)--(4);
					\end{scope}
				\end{tikzpicture} & \begin{tikzpicture}[line join = round, line cap = round]
					
					\coordinate [label=above:4] (4) at (0,{sqrt(2)},0);
					\coordinate [label=left:3] (3) at ({-.5*sqrt(3)},0,-.5);
					\coordinate [label=below:2] (2) at (0,0,1);
					\coordinate [label=right:1] (1) at ({.5*sqrt(3)},0,-.5);
					
					\begin{scope}
						\draw (1)--(3);
						\draw[fill=lightgray,fill opacity=.5] (2)--(1)--(4)--cycle;
						\draw[fill=gray,fill opacity=.5] (3)--(2)--(4)--cycle;
						\draw (2)--(1);
						\draw (2)--(3);
						\draw (3)--(4);
						\draw (2)--(4);
						\draw (1)--(4);
					\end{scope}
				\end{tikzpicture}  & \begin{tikzpicture}[line join = round, line cap = round]
					
					\coordinate [label=above:4] (4) at (0,{sqrt(2)},0);
					\coordinate [label=left:3] (3) at ({-.5*sqrt(3)},0,-.5);
					\coordinate [label=below:2] (2) at (0,0,1);
					\coordinate [label=right:1] (1) at ({.5*sqrt(3)},0,-.5);
					
					\begin{scope}
						\draw (1)--(3);
						\draw (2)--(1);
						\draw (2)--(3);
						\draw (3)--(4);
						\draw (2)--(4);
						\draw (1)--(4);
					\end{scope}
				\end{tikzpicture}
				& 	\begin{tikzpicture}[line join = round, line cap = round]
					
					\tikzstyle{point}=[circle,thick,draw=black,fill=black,inner sep=0pt,minimum width=2pt,minimum height=2pt]
					\node (4) [point,label=above:4] at (0,{sqrt(2)},0) {};
					\node (3) [point,label=left:3] at ({-.5*sqrt(3)},0,-.5) {};
					\node (2) [point,label=below:2] at (0,0,1) {};
					\node (1) [point,label=right:1] at ({.5*sqrt(3)},0,-.5) {};
					
				\end{tikzpicture} \\
				$\Skel_3(\Delta^3)=\Delta^3$& $\Skel_2(\Delta^3) = \partial \Delta^3$& $\Skel_1(\Delta^3)$ & $\Skel_0(\Delta^3)$
			\end{tabular}
		\end{center}
		The darker shading of the image of $\Skel_3(\Delta^3)$ indicates that this complex contains the $3$-dimensional face $\{1,2,3,4\}$, while the lighter shading for $\Skel_2(\Delta^3)$ indicates that this complex only contains the bounding $2$-simplices of the face $\{1,2,3,4\}$.
	\end{ex}
	
	\begin{defn}\label{Definition: Cross Polytope}
		The $d$-dimensional \textit{cross polytope} (also known as the $d$-dimensional \textit{orthoplex}) is the simplicial complex $O^d$, defined recursively for $d\geq -1$ as follows: 
		\begin{align*}
			O^{-1} &= \{\emptyset\}\\
			O^{d+1}&= S(O^d).
		\end{align*}
	\end{defn}
	\begin{ex}\label{Example: Cross Polytopes}
		The cross polytopes of dimensions $-1$, $0$, $1$ and $2$ are shown below.
		\begin{center}
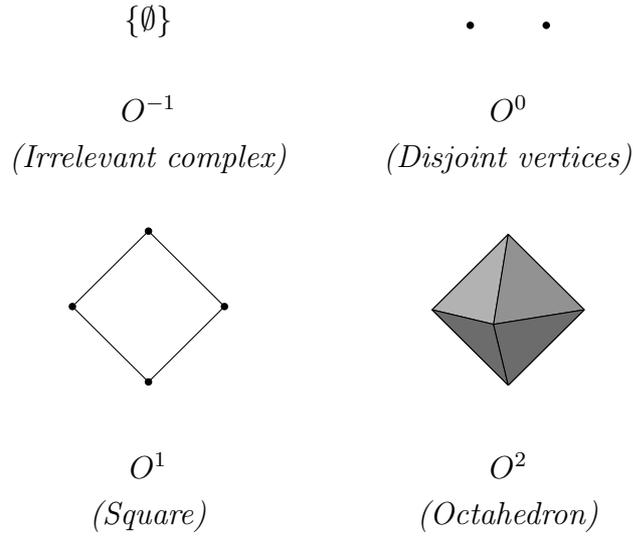
			\begin{tabular}{ c c c c }
				$\{\emptyset\}$& 
				&
				& \begin{tikzpicture}
					\tikzstyle{point}=[circle,thick,draw=black,fill=black,inner sep=0pt,minimum width=2pt,minimum height=2pt]
					\node (a)[point] at (0,0) {};
					\node (b)[point] at (1,0) {};
				\end{tikzpicture}\\
				&&&\\
				$O^{-1}$& & & $O^0$\\
				\textit{(Irrelevant complex)} & & & \textit{(Disjoint vertices)} \\
				&&&\\
				\begin{tikzpicture}[scale=1]
					\tikzstyle{point}=[circle,thick,draw=black,fill=black,inner sep=0pt,minimum width=2pt,minimum height=2pt]
					\node (a)[point] at (1,0) {};
					\node (b)[point] at (0,1) {};
					\node (c)[point] at (1,2) {};
					\node (d)[point] at (2,1) {};
					
					\draw (a.center) -- (b.center) -- (c.center) -- (d.center) -- cycle;
				\end{tikzpicture}&
				&
				& 	\begin{tikzpicture}[line join=bevel,z=-5.5]
					\coordinate (A1) at (0,0,-1);
					\coordinate (A2) at (-1,0,0);
					\coordinate (A3) at (0,0,1);
					\coordinate (A4) at (1,0,0);
					\coordinate (B1) at (0,1,0);
					\coordinate (C1) at (0,-1,0);
					
					\draw (A1) -- (A2) -- (B1) -- cycle;
					\draw (A4) -- (A1) -- (B1) -- cycle;
					\draw (A1) -- (A2) -- (C1) -- cycle;
					\draw (A4) -- (A1) -- (C1) -- cycle;
					\draw [fill opacity=0.7,fill=lightgray!80!gray] (A2) -- (A3) -- (B1) -- cycle;
					\draw [fill opacity=0.7,fill=gray!70!lightgray] (A3) -- (A4) -- (B1) -- cycle;
					\draw [fill opacity=0.7,fill=darkgray!30!gray] (A2) -- (A3) -- (C1) -- cycle;
					\draw [fill opacity=0.7,fill=darkgray!30!gray] (A3) -- (A4) -- (C1) -- cycle;
				\end{tikzpicture}\\
				&&&\\
				$O^1$& & & $O^2$\\
				\textit{(Square)} & & & \textit{(Octahedron)}
			\end{tabular}
		\end{center}
	\end{ex}
	
	Having reviewed some key examples of simplicial complexes, we now move on to the Stanley-Reisner correspondence. For any simplicial complex on vertex set $[n]$, we may assign a unique corresponding squarefree monomial ideal. 
	\begin{defn}\label{Definition: SR Ideals}
		Let $\Delta$ be a simplicial complex on vertex set $[n]$. We define the \textit{Stanley-Reisner ideal} of $\Delta$ to be $$I_\Delta=\langle \mathbf{x}^\sigma : \sigma \subseteq [n], \sigma \notin \Delta \rangle$$
		where $\mathbf{x}^\sigma$ denotes the squarefree monomial $\prod_{i\in \sigma}x_i$ in $R$.
		
		We define the \textit{Stanley-Reisner ring} of $\Delta$ to be $$\KK[\Delta]=R/I_\Delta.$$
	\end{defn}
	\begin{rem}\label{Remark: SR Ideals Arbitrary Vertex Sets}
		We also extend this definition to complexes $\Delta$ on arbitrary vertex set $V$ of size $n$ by first choosing an ordering on $V$, which allows us to define an isomorphism from $\Delta$ to a complex on $[n]$. Technically this makes the Stanley-Reisner ideal dependent on the specific ordering we choose, but we still use the phrase `\textit{the}' Stanley-Reisner ideal anyway, because all possible orderings give us the same ideal up to isomorphism.
	\end{rem}
	
	For a simplicial complex $\Delta$ the minimal generators of $I_\Delta$ correspond to the minimal nonfaces of $\Delta$. In particular, note that every squarefree monomial ideal is uniquely determined by its minimal generators, and similarly every simplicial complex is uniquely determined by its minimal nonfaces. Thus, the assignment $\Delta \mapsto I_\Delta$ gives us a one-to-one correspondence from simplicial complexes on $[n]$ to squarefree monomial ideals in $R$ (this is Theorem 1.7 in \cite{CCA}).
	\begin{prop}\label{Proposition: SR 1-1 Correspondence}
		The assignment $\Delta \mapsto I_\Delta$ gives a bijection $$\{\text{Simplicial complexes on } [n]\}\isomto \{\text{Squarefree monomials in } R\}.$$
	\end{prop}

	This one-to-one correspondence allows us to reframe algebraic properties of Stanley-Reisner ideals and rings in terms of combinatorial properties of their corresponding complexes, and vice versa. For example, we can compute the height of the ideal $I_\Delta$, and hence the Krull-dimension of the ring $\KK[\Delta]$, respectively from the codimension and dimension of $\Delta$.
	
	\begin{lem}\label{Lemma: height I-Delta}
		For any simplicial complex $\Delta$ on vertex set $[n]$, we have 
		$$\height I_\Delta = n - \dim \Delta - 1 = \codim \Delta.$$
	\end{lem}
	\begin{proof}
		For a face $F$ of $\Delta$ of size $d+1$, we define $P_F$ to be the ideal in $R$ generated by the vertices $x_i$ where $i\in [n] - F$. This is a monomial prime ideal in $R$ of height $n-d-1$, and because every nonface of $\Delta$ shares at least one vertex with $[n]-F$ we must have $P_F \supseteq I_\Delta$.
		
		Conversely, any monomial prime ideal $P$ in $R$ containing $I_\Delta$ is of the form $\langle x_{i_1},\dots,x_{i_h} \rangle$ for some $1\leq i_1 < \dots < i_h \leq n$ such that each nonface of $\Delta$ contains at least one of the vertices $x_{i_1},\dots,x_{i_h}$. Thus the set $F_P = [n]-\{x_{i_1},\dots,x_{i_h}\}$ is a face of $\Delta$.
		
		This gives us a one-to-one correspondence between faces of $\Delta$ and monomial prime ideals in $R$ containing $I_\Delta$, which sends faces of dimension $d$ to ideals of height $n-d-1$. Under this correspondence, facets of maximum dimension are sent to monomial primes containing $I_\Delta$ of minimal height. Because $I_\Delta$ is a squarefree monomial ideal then all of the minimal primes of $I_\Delta$ are monomial (see \cite{MonIdeals} Corollary 1.3.6), and the result follows.
	\end{proof}
	
	\begin{cor}\label{Corollary: dim k[Delta]}
		For any simplicial complex $\Delta$, we have
		$$\dim \KK[\Delta] = \dim \Delta + 1.$$
	\end{cor}
	\begin{proof}
		For any ideal $I$ in $R$ we have $\dim R/I = n - \height I$, so this follows directly from Lemma \ref{Lemma: height I-Delta}.
	\end{proof}
	
	For our purposes the most important example of this combinatorial reframing is Hochster's Formula, which is the subject of the following section.
	
	\subsection{Hochster's Formula and Simplicial Homology}\label{Subsection: Hochster's Formula}
	As we have seen already, the Stanley-Reisner construction allows us to translate algebraic problems into combinatorial ones, and vice versa, by reframing the algebraic invariants of squarefree monomial ideals in terms of combinatorial invariants of their corresponding complexes.
	
	In this section we present a crucial result of Hochster's, which reframes the Betti diagrams of Stanley-Reisner ideals in terms of simplicial homology. As mentioned already, all of our homology is over $\KK$, and we use the notation $\Hred_i(\Delta)$ to denote the $\nth{i}$ reduced homology group of $\Delta$ with coefficients in $\KK$.
	
	\begin{thm}[Hochster's Formula]\label{Theorem: Hochster's Formula}
		Let $\Delta$ be a simplicial complex on a vertex set $V$ of size $n$. For any integers $0\leq i \leq n$ and $d$, we have
		\begin{equation*}
			\beta_{i,d}(I_\Delta) =\sum_{U\in {V \choose d}} \dim_\KK \Hred_{d-i-2}(\Delta_U)
		\end{equation*}
		where ${V \choose d}$ denotes the subsets of $V$ of size $d$.
	\end{thm}
	For a proof of Hochster's Formula see \cite{MonIdeals}, Theorem 8.1.1.
	
	\begin{rem}\label{Remark: Missing Vertices Row 0 of Betti Table}
		Hochster's Formula shows us that the Betti numbers of the form $\beta_{i,i+1}(I_\Delta)$ come from induced subcomplexes of $\Delta$ with $\nth[st]{(-1)}$ reduced homology. The only simplicial complex which has $\nth[st]{(-1)}$ reduced homology is the irrelevant complex $\{\emptyset\}$, and for any nonempty subset $U\subseteq V$, we have $\Delta_U=\{\emptyset\}$ if and only if $U$ consists solely of missing vertices of $\Delta$. In particular, if $\Delta$ has $m$ missing vertices, then for any $0\leq i \leq n-1$ we have
		\begin{align*}
			\beta_{i,i+1}(I_\Delta) &= \sum_{U\in {V \choose i+1}} \dim_\KK \Hred_{-1}(\Delta_U)\\
			&= \#\left\{U\in {V \choose i+1} : \Delta_U = \{\emptyset\}\right\}\\
			&= {m \choose i+1}\, .
		\end{align*}
	\end{rem}
	
	Sometimes it will be more useful for us to study squarefree monomial ideals using a \textit{dual} Stanley-Reisner construction, which we lay out below.
	
	Note that every simplicial complex $\Delta$ is equal to the Alexander dual of its own Alexander dual $\Delta^*$ (i.e. we have $(\Delta^*)^*=\Delta$). In particular, this means that every squarefree monomial ideal in $R$ can be viewed uniquely as the dual Stanley-Reisner ideal $I_{\Delta^*}$ of some complex $\Delta$.
	
	\begin{prop}\label{Proposition: Dual Stanley Reisner Ideal}
		Let $\Delta$ be a simplicial complex on vertex set $[n]$. The dual Stanley-Reisner ideal $\Idstar$ is given by $$\Idstar=\langle \mathbf{x}^\sigma : \sigma \subseteq [n], [n]-\sigma \in \Delta \rangle$$ where $\mathbf{x}^\sigma$ denotes the squarefree monomial $\prod_{i\in \sigma}x_i$ in $R$.
	\end{prop}
	\begin{proof}
		By construction, the nonfaces of the complex $\Delta^*$ are the complements of the faces of $\Delta$. The result follows from Definition \ref{Definition: SR Ideals}.
	\end{proof}
	
	\begin{notation}\label{Notation: Dual Stanley-Reisner Ideals}
		We sometimes denote the \textit{dual Stanley-Reisner ideal} $I_{\Delta^*}$ of $\Delta$ by $I^*_{\Delta}$.
	\end{notation}
	
	There is an Alexander dual variant of Hochster's Formula, due to Eagon and Reiner, which allows us to compute the Betti numbers of the dual Stanley-Reisner ideal $I_{\Delta^*}$ directly using combinatorial and homological data from the complex $\Delta$ (this is \cite{Eagon-Reiner} Proposition 1, rephrased in terms of Betti numbers of Stanley-Reisner ideals as opposed to Betti polynomials of Stanley-Reisner rings; see this article for a proof).
	
	\begin{thm}[Hochster's Formula, Alexander Dual Variant] \label{Theorem: ADHF}
		Let $\Delta$ be a simplicial complex on a vertex set $V$ of size $n$. For any integers $0\leq i\leq n$ and $d$, we have
		\begin{equation*}
			\beta_{i,d}(\Idstar)=\sum_{\sigma\in \Delta, |\sigma|=n-d} \dim_\KK \Hred_{i-1}(\lkds) 
		\end{equation*}
	\end{thm}
	\begin{rem}\label{Remark: Link of Emptyset and Facets}
		Let $\Delta$ be a simplicial complex.
		\begin{enumerate}
			\item We have $\lkds[\emptyset]=\Delta$, which tells us that  $\beta_{i,n}(\Idstar)=\Hred_{i-1}(\Delta)$ for any integer $1\leq i \leq n$.
			\item For any facet $F$ of $\Delta$ we have $\lkds[F]=\{\emptyset\}$, which has only $\nth[st]{(-1)}$ homology of dimension $1$. Thus for any integer $d$, the Betti number $\beta_{0,d}(\Idstar)$ is equal to the number of facets of $\Delta$ of size $n-d$. In particular, the minimum degree of a generator of $\Idstar$ is equal to the codimension of $\Delta$. 
		\end{enumerate}
	\end{rem}

	Both Hochster's Formula and its Alexander Dual variant will be fundamental to our work going forwards, and for this reason most of our proofs centre around the computation of simplicial homology. We devote the rest of this section to some important results which can simplify these computations.
	
	Firstly, note that for any $j\geq -1$, the definition of the homology group $\Hred_j(\Delta)$ (as given in e.g. \cite{Hatcher}, pages 104-6) relies solely on the faces of $\Delta$ of dimension $j+1$ and $j$, and this data is contained in the skeleton complex $\Skel_r(\Delta)$ for any $r>j$. This gives us the following lemma.
	
	\begin{lem}\label{Lemma: Hj Delta = Hj Skel(Delta)}
		Let $\Delta$ be a simplicial complex, and let $r>j\geq -1$ be two integers. We have $\Hred_j(\Delta)=\Hred_j(\Skel_r(\Delta))$.
	\end{lem}
	
	Next we introduce two important results which allow us to compute the homology of certain complexes from the homologies of their subcomplexes. Both of these results have more general topological analogues, regarding the singular homology of any topological space; but we present only the cases where the space in question is a simplicial complex. For a proof of the Mayer-Vietoris Sequence see \cite{Hatcher} pages 149-150; for a proof of the K\"{u}nneth Formula see \cite{Hatcher} page 276, Corollary 3B.7.
	
	\begin{prop}[Reduced Mayer-Vietoris Sequence for Simplicial Complexes]\label{Proposition: MVS}
		Let $\Delta$ be a simplicial complex and suppose we have $\Delta = A \cup B$ for two subcomplexes $A$ and $B$ with nonempty intersection. There is a long exact sequence
		\begin{align*}
			\dots \rightarrow \Hred_{r+1}(\Delta) \rightarrow \Hred_r(A\cap B)\rightarrow \Hred_r(A)\oplus \Hred_r(B) \rightarrow &\Hred_r(\Delta)\rightarrow \dots\\
			&\dots \rightarrow \Hred_0(\Delta)\rightarrow 0.
		\end{align*}
	\end{prop}
	
	\begin{prop}[K\"{u}nneth Formula for Joins of Simplicial Complexes]\label{Proposition: Kunneth Formula}
		Let $\Delta$ be a simplicial complex and suppose we have $\Delta = A \ast B$ for two subcomplexes $A$ and $B$. For any integer $r\geq -1$ we have an isomorphism
		\begin{equation*}
			\Hred_{r+1}(\Delta)\cong \sum_{i+j=r}\Hred_i(A)\otimes\Hred_j(B).
		\end{equation*}
	\end{prop}

	Many of the results in this thesis are about the \textit{shapes} of Betti diagrams of Stanley-Reisner ideals (i.e. the indices for which the Betti numbers are nonzero) rather than the values of specific Betti numbers. Thus we will often be more interested in the degrees at which a complex has nontrivial homology than in what those nontrivial homologies are. For this reason, we sometimes make use of the following notation.
	
	\begin{defn}\label{Definition: Homology Index Sets}
		Let $\Delta$ be a simplicial complex. We define the \textit{homology index set of} $\Delta$ as $h(\Delta)=\{i\in \ZZ : \Hred_i(\Delta) \neq 0\}$. 
	\end{defn}
	\begin{rem}
		We may add homology index sets $A$ and $B$ together using the rule $A+B=\{i+j: i\in A, j\in B\}$. If $A$ and $B$ are both singletons, then this is the same as adding the elements of those singletons together; and if either set is empty then the resulting sum is also empty.
	\end{rem}
	
	Using this notation we have the following corollary to the K\"{u}nneth Formula.
	\begin{cor}\label{Corollary: homology index set of joins}
		Let $\Delta_1,\dots,\Delta_m$ be simplicial complexes, and let $\circledast_{j=1}^m \Delta_j$ denote the join of all of them. We have that $h(\circledast_{j=1}^m \Delta_j)=\sum_{j=1}^m h(\Delta_j) + \{m-1\}$.
	\end{cor}
	\begin{proof}
		We prove this by induction on $m\geq 1$. For the base case $m=1$, there is nothing to prove. For the inductive step, assuming that $h(\circledast_{j=1}^m \Delta_j)=\sum_{j=1}^m h(\Delta_j) + \{m-1\}$, we have
		\begin{align*}
			h(\circledast_{j=1}^{m+1} \Delta_i) &= h((\circledast_{j=1}^m \Delta_j)\ast \Delta_{m+1}) &\\
			&= h(\circledast_{j=1}^m \Delta_j)+ h(\Delta_{m+1}) +\{1\} &\text{by Prop. \ref{Proposition: Kunneth Formula}}\\
			&= (\sum_{j=1}^m h(\Delta_j) + \{m-1\})+h(\Delta_{m+1}) +\{1\} &\text{by ind. hyp.}\\
			&= \sum_{j=1}^{m+1} h(\Delta_j) + \{m\}.
		\end{align*}
	\end{proof}
	
	Another crucial tool for computing homology comes from the following result (see e.g. \cite{Hatcher}, Corollary 2.11).
	\begin{prop}\label{Proposition: Homotopy Equiv => Homology Equiv}
		If two topological spaces are homotopy equivalent, they are homology equivalent.
	\end{prop}
	
	In particular, every deformation retraction is a homotopy equivalence, and for this reason we will often make use of deformation retractions when computing homology. As an example, this allows us to find the homology of cones and suspensions.
	
	\begin{cor}\label{Corollary: Homology of Cone}
		Let $\Delta$ be a simplicial complex. The cone $C\Delta$ is acyclic.
	\end{cor}
	\begin{proof}
		The cone $C\Delta$ is the join of $\Delta$ with a single vertex $v$, and it deformation retracts on to the vertex $v$, which is acyclic.
	\end{proof}
	\begin{cor}\label{Corollary: Homology of Suspension}
		Let $\Delta$ be a simplicial complex. For any integer $i\geq -1$ the suspension $S\Delta$ has homology $\Hred_i(S\Delta)=\Hred_{i-1}(\Delta)$.
	\end{cor}
	\begin{proof}
		The suspension $S\Delta$ is the join of $\Delta$ with the complex consisting of two disjoint vertices $u$ and $v$. We can decompose $\Delta$ into the subcomplexes $A\cup B$ where $\Delta \ast \{u\}$ and $B=\Delta \ast\{v\}$. Both of these subcomplexes are cones and are hence acyclic by Corollary \ref{Corollary: Homology of Cone}. They intersect at $\Delta$ itself, and hence for each $i\geq -1$, the Mayer-Vietoris sequence gives us the isomorphism $$0\rightarrow \Hred_i(S\Delta)\rightarrow \Hred_{i-1}(\Delta)\rightarrow 0.$$
	\end{proof}
	
	We end this section by presenting a particularly useful lemma for finding deformation retracts. It provides sufficient conditions for ensuring that a complex $\Delta$ deformation retracts on to $\Delta - g$, for a given face $g\in \Delta$ (recall that $\Delta-g$ is the deletion of $g$ from $\Delta$, as defined in Definition \ref{Definition: Subcomplexes}). We suspect this lemma exists in the literature, but have been unable to find it.
	\begin{lem}\label{Lemma: Deformation Retract}
		Let $\Delta$ be a simplicial complex, and let $g\subsetneqq f$ be faces of $\Delta$ such that every facet of $\Delta$ that contains $g$ also contains $f$. There is a deformation retraction $\Delta \rightsquigarrow \Delta - g$, obtained by identifying $g$ with a vertex in $f-g$.
	\end{lem}
	\begin{proof}
		Suppose $\Delta$ has $n$ vertices $v_1,\dots,v_n$ with $g=\{v_1,\dots,v_k\}$ and $v_{k+1}$ in $f-g$.
		
		Let $X=X_{\Delta}$ be the geometric realization of $\Delta$ in $\RR^n$ and $X_{\Delta-g}$ the geometric realization of $\Delta -g$. For each nonempty face $\sigma =\{v_{i_1},\dots,v_{i_r}\}$ of $\Delta$, we define $X_\sigma$ to be the set $\left\{\lambda_1 e_{i_1}+\dots \lambda_r e_{i_r} : \lambda_1, \dots, \lambda_r > 0, \sum_{j=1}^r \lambda_j = 1\right\} \subset X$, where $\{e_1,\dots,e_n\}$ is the canonical basis of $\RR^n$. Note that $X=\bigcup_{\sigma \in \Delta} X_\sigma$, while $X_{\Delta-g}=\bigcup_{\sigma \in \Delta, \sigma \nsupseteq g} X_\sigma$.
		
		Let $\bp$ be a point in $X$. We may write it as $\bp = \sum_{j=1}^n \lambda_j e_j$ where the coefficients $\lambda_j$ are all nonnegative and sum to $1$. In particular, $\bp$ lies inside $X_\sigma$ for some $\sigma\in \Delta$ containing $g$ if and only if all of $\lambda_1,\dots,\lambda_k$ are positive.
		
		We define $\lambda = \min \{\lambda_1, \dots, \lambda_k\}$. This allows us to rewrite the point $\bp$ as
		\begin{equation}\label{Equation: bp}
			\bp = \lambda(e_1 + \dots + e_k) + \sum_{i=1}^{k} (\lambda_i - \lambda) e_i + \lambda_{k+1} e_{k+1} + y
		\end{equation}
		for some $y$ in the span of $\{e_{k+2},\dots,e_n\}$. Note that the coefficients $(\lambda_i-\lambda)$ are all nonnegative, and (by the definition of $\lambda$) at least one of them is zero. Note also that $\lambda$ itself is nonzero if and only if all of $\lambda_1,\dots,\lambda_k$ are positive, which occurs if and only if $\bp$ is in $X_\sigma$ for some $\sigma$ containing $g$. In other words, we have $\lambda = 0$ if and only if $\bp$ lies inside $X_{\Delta-g}$.
		
		Using this notation for the points in $X$, we can define a function $\varphi: X \times [0,1] \rightarrow \RR^n$ as follows. For $\bp$ as in Equation (\ref{Equation: bp}) and $0\leq t \leq 1$ we define 
		\begin{equation}\label{Equation: varphi}
			\varphi(\bp, t) = \lambda (1-t)(e_1 + \dots + e_k) + \sum_{i=1}^{k} (\lambda_i - \lambda) e_i + (\lambda_{k+1} + k \lambda t) e_{k+1} + y.
		\end{equation}
		
		We claim that $\varphi$ is a deformation retraction from $X$ to $X_{\Delta-g}$.
		
		First, note that $\varphi$ is continuous. Indeed, the function $\lambda=\min \{\lambda_1,\dots,\lambda_k\}$ is continuous in the variables $\lambda_i$, and hence so is each summand in Equation (\ref{Equation: varphi}). Each summand is also continuous in $t$.
		
		Next, we note that $\varphi(*,0)$ is the identity on $X$. Moreover, for $\bp$ in $X_{\Delta-g}$ we have $\lambda=0$, and hence $\varphi(\bp,t)=\bp$ for every $0\leq t \leq 1$.
		
		It remains to show that the image of $\varphi(*,t)$ lies inside $X$ for every value of $t$, and in particular that the image of $\varphi(*,1)$ is $X_{\Delta-g}$.
		
		For the former claim, we note that the sum of the coefficients of $e_1,\dots,e_n$ in the decomposition of $\varphi(\bp,t)$ is the same as the sum of these coefficients in the decomposition of $\bp$ (which is $1$), and all of these coefficients are nonnegative. We may assume that $\bp$ lies inside $X_{\sigma}$ for some face $\sigma$ in $\Delta$ containing $g$ (otherwise $\bp$ lies inside $X_{\Delta-g}$ and we are already done). By our assumption on $g$ and $f$ we have that $\sigma\cup \{v_{k+1}\}$ is also a face of $\Delta$. If $\sigma = \{v_1,\dots,v_k,v_{i_1},\dots,v_{i_r}\}$ for some $k<i_1<\dots<i_r\leq n$, then the vectors in the decomposition of $\varphi(\bp,t)$ with strictly positive coefficients are all contained in $\{e_1,\dots,e_k,e_{i_1},\dots,e_{i_r}\}\cup\{e_{k+1}\}$. We conclude that $\varphi(\bp,t)$ lies inside $X_{\tau}$ for some $\tau \subseteq \sigma \cup \{v_{k+1}\}$, and hence inside $X$.
		
		For the latter claim, we note that the coefficients of the vectors $e_1,\dots, e_k$ in $\varphi(\bp,1)$ are $\lambda_1-\lambda,\dots,\lambda_k-\lambda$, and at least one of these must be zero by the definition of $\lambda$. Thus $\varphi(\bp,1)$ lies inside $X_{\Delta-g}$.
	\end{proof}
	
	In particular we have the following corollary, which will be sufficient for our needs in most (but not all) cases.
	\begin{cor}\label{Corollary: Deformation Retract}
		Let $\Delta$ be a simplicial complex, and let $a$ and $b$ be distinct vertices of $\Delta$ satisfying the following conditions.
		\begin{enumerate}
			\item There is at least one facet of $\Delta$ containing $a$ (i.e. it is not a missing vertex).
			\item Every facet of $\Delta$ containing $a$ also contains $b$.
		\end{enumerate}
		There is a deformation retraction $\Delta \leadsto \Delta - \{a\}$ given by the map $a\mapsto b$.
	\end{cor}
	\begin{proof}
		This comes from Lemma \ref{Lemma: Deformation Retract}, setting $g=\{a\}$ and $f=\{a,b\}$. Note that $g$ is a face of $\Delta$ by assumption (1), and $f$ is a face of $\Delta$ by assumption (2).
	\end{proof}
	
	\subsection{Graphs and Edge Ideals}\label{Subsection: Graphs}
	A $1$- or $0$-dimensional simplicial complex with no missing vertices contains the same data as a graph. The following graph-theoretic definitions are standard, and fundamental to our work going forwards.
	\begin{defn}\label{Definition: Graph}
		A \textit{graph} $G=(V,E)$ consists of a set of \textit{vertices} $V$ and a set of \textit{edges} $E$ consisting of subsets of $V$ of size 2.
		
		We often denote the vertex set and edge set of a graph $G$ as $V(G)$ and $E(G)$ respectively.
	\end{defn}
	
	\begin{defn}\label{Definition: Graph Theory Key Terminology and Constructions}
		Let $G=(V,E)$ be a graph, and $m$ any positive integer.
		\begin{enumerate}
			\item We define the \textit{complement} $G^c$ of $G$ to be the graph on vertex set $V$ with edge set $\{\{x,y\}: x, y \in V \{x,y\}\notin E\}$.
			\item For a subset $U\subset V$ we define the \textit{induced subgraph} of $G$ to be the graph on vertex set $U$ with edge set $\{\{x,y\}:x,y\in U,\{x,y\}\in E\}$.
			\item A \textit{vertex cover} for $G$ is a subset $U\subset V$ such that for every edge $e\in E$ we have $e\cap U \neq \emptyset$.
			\item A \textit{matching} in $G$ is a collection $M$ of pairwise disjoint edges in $G$. We say $M$ is a \textit{maximal matching} if it is maximal with respect to inclusion.
			\item A \textit{cycle of length} $m$ in $G$ (also called an $m$\textit{-cycle}) is a sequence of edges in $G$ of the form $\{v_1,v_2\},\dots,\{v_{m-1},v_m\},\{v_m,v_1\}$ for some vertices $v_1,\dots,v_m\in V$.
			\item We say a vertex $v\in V$ is \textit{isolated} in $G$ if it is contained in no edges of $G$. We say $v$ is \textit{universal} in $G$ if it is contained in every edge of $G$.
		\end{enumerate}
	\end{defn}
	\begin{ex}
		Let $G$ be the graph on vertex set $[5]$ with edges $\{1,2\}$, $\{1,4\}$, $\{2,3\}$, $\{2,4\}$ and $\{3,4\}$. We can draw $G$ as follows.
		
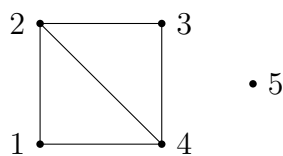
\begin{center}
			\begin{tikzpicture}[scale = 0.8]
				\tikzstyle{point}=[circle,thick,draw=black,fill=black,inner sep=0pt,minimum width=2pt,minimum height=2pt]
				\node (a)[point, label=left: $1$] at (0,0) {};
				\node (b)[point, label=left: $2$] at (0,2) {};
				\node (c)[point, label=right: $3$] at (2,2) {};
				\node (d)[point, label=right: $4$] at (2,0) {};
				\node (e)[point, label=right: $5$] at (3.5,1) {};
				
				\draw (a.center) -- (b.center) -- (c.center) -- (d.center) -- cycle;
				\draw (b.center) -- (d.center);
			\end{tikzpicture}
		\end{center}
		\begin{itemize}
			\item Every edge in $G$ contains either $2$ or $4$ so the set $\{2,4\}$ is a vertex cover for $G$.
			\item The edges $\{1,2\}$ and $\{3,4\}$ form a maximal matching in $G$, as do the edges $\{1,4\}$ and $\{2,3\}$.
			\item $G$ contains two 3-cycles (the cycles $\{1,2\},\{2,4\},\{4,1\}$ and $\{2,3\},\{3,4\},\{4,2\}$) and one 4-cycle (the cycle $\{1,2\},\{2,3\},\{3,4\},\{4,1\}$).
			\item 	The complement $G^c$ of $G$ is the graph with edges $\{1,5\},\{2,5\},\{3,5\},\{4,5\}$ and $\{1,3\}$, which we can draw as follows.
			
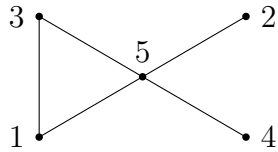
\begin{center}
				\begin{tikzpicture}[scale = 0.8]
					\tikzstyle{point}=[circle,thick,draw=black,fill=black,inner sep=0pt,minimum width=2pt,minimum height=2pt]
					\node (1)[point, label=left: $1$] at (0,0) {};
					\node (2)[point, label=right: $2$] at (3.4,2) {};
					\node (3)[point, label=left: $3$] at (0,2) {};
					\node (4)[point, label=right: $4$] at (3.4,0) {};
					\node (5)[point, label=above: $5$] at (1.7,1) {};
					
					\draw (1.center) -- (3.center);
					\draw (1.center) -- (5.center);
					\draw (2.center) -- (5.center);
					\draw (3.center) -- (5.center);
					\draw (4.center) -- (5.center);
				\end{tikzpicture} 
			\end{center}
			\item Note that the vertex $5$ is isolated in $G$, and therefore universal in $G^c$.
		\end{itemize}
	\end{ex}
	
	As mentioned above, a graph may be viewed as a $1$-dimensional simplicial complex with no missing vertices. However, there are (at least) two other important ways of associating complexes to graphs, as laid out below.
	
	\begin{defn}\label{Definition: Complexes from Graphs}
		Let $G=(V,E)$ be a graph.
		\begin{enumerate}
			\item We define the \textit{complex of cliques} of $G$, $\Cl(G)$, to be the complex on $V$ whose faces are the \textit{cliques} in $V$ - i.e. those subsets $\{v_1,\dots,v_r\}\subset V$ such that for each $1\leq i < j \leq r$ we have $\{v_i,v_j\}\in E$.
			\item We define the \textit{independence complex} of $G$, $\Ind(G)$, to be the complex on $V$ whose faces are all the \textit{independent subsets} of $V$ - i.e. those subsets $\{v_1,\dots,v_r\}\subset V$ such that for each $1\leq i < j \leq r$ we have $\{v_i,v_j\}\notin E$.
		\end{enumerate}
	\end{defn}
	\begin{rem}\label{Remark: Ind(G)=Cl(G^c)}
		The independence complex of $G$ is equal to the complex of cliques of $G^c$ (i.e. $\Ind(G)=\Cl(G^c)$).
	\end{rem}
	
	We now review some examples of important graphs, which will be crucial to our work going forward.
	\begin{defn}\label{Definition: Important Graphs}
		Let $m$ be any positive integer.
		\begin{enumerate}
			\item The \textit{complete graph} on $m$ vertices, denoted $K_m$, is the graph on vertex set $[m]$ with all possible edges (i.e. its edge set is ${[m] \choose 2} = \{U\subseteq [m] : |U| = 2\}$).
			We sometimes denote the graph $K_2$, consisting of a single edge, as $L$.
			\item The \textit{empty graph} on $m$ vertices, denoted $E_m$, is the graph on vertex set $[m]$ with no edges (i.e. its edge set is $\emptyset$).
			\item The \textit{cyclic graph of order} $m$, denoted $C_m$, is the graph on vertex set $[m]$ with edge set $\{\{1,2\},\{2,3\},\dots,\{m-1,m\},\{m,1\}\}$ (i.e. the edges of $C_m$ form a single cycle of length $m$).
		\end{enumerate}
	\end{defn}
	\begin{rem}\label{Remark: Graph Isomorphism Classes}
		Just as with our convention for the vertex sets of $\Delta^j$ and $\partial \Delta^j$ in Remark \ref{Remark: j-Simplex Isomorphism classes}, our convention for the vertex sets of the graphs $K_m$, $E_m$ and $C_m$ is arbitrary. Again, these graphs are better understood as isomorphism classes, and we could have chosen a graph on any vertex set of size $m$ as our representative for these classes. Similarly, we could have chosen any $m$-cycle for the edges of $C_m$.
	\end{rem}
	\begin{ex}
		For the case $m=5$ we have
		\begin{center}
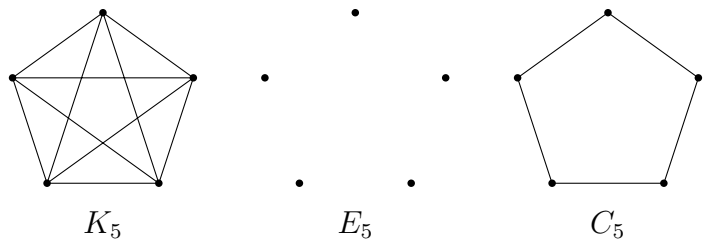
			\begin{tabular}{ c c c c c }
				
				\begin{tikzpicture}[scale = 1.25]
					\tikzstyle{point}=[circle,thick,draw=black,fill=black,inner sep=0pt,minimum width=2pt,minimum height=2pt]
					\node (a)[point] at (0,1) {};
					\node (b)[point] at (0.951,0.309) {};
					\node (c)[point] at (0.588,-0.809) {};
					\node (d)[point] at (-0.588,-0.809) {};
					\node (e)[point] at (-0.951,0.309) {};
					
					\draw (a.center) -- (b.center) -- (c.center) -- (d.center) -- (e.center) -- cycle;
					\draw (a.center) -- (c.center) -- (e.center) -- (b.center) -- (d.center) -- cycle;
				\end{tikzpicture} &  & \begin{tikzpicture}[scale = 1.25]
					\tikzstyle{point}=[circle,thick,draw=black,fill=black,inner sep=0pt,minimum width=2pt,minimum height=2pt]
					\node (a)[point] at (0,1) {};
					\node (b)[point] at (0.951,0.309) {};
					\node (c)[point] at (0.588,-0.809) {};
					\node (d)[point] at (-0.588,-0.809) {};
					\node (e)[point] at (-0.951,0.309) {};
					
				\end{tikzpicture}
				& & \begin{tikzpicture}[scale = 1.25]
					\tikzstyle{point}=[circle,thick,draw=black,fill=black,inner sep=0pt,minimum width=2pt,minimum height=2pt]
					\node (a)[point] at (0,1) {};
					\node (b)[point] at (0.951,0.309) {};
					\node (c)[point] at (0.588,-0.809) {};
					\node (d)[point] at (-0.588,-0.809) {};
					\node (e)[point] at (-0.951,0.309) {};
					
					\draw (a.center) -- (b.center) -- (c.center) -- (d.center) -- (e.center) -- cycle;
				\end{tikzpicture}\\
				$K_5$& & $E_5$ & & $C_5$
			\end{tabular}
		\end{center}
	\end{ex}
	
	Just as the Stanley-Reisner correspondence allows us to assign a unique squarefree monomial ideal to every simplicial complex, there is a similar assignment of ideals to graphs.
	\begin{defn}\label{Definition: Edge Ideals}
		Let $G$ be a graph on vertex set $[n]$. We define the \textit{edge ideal} of $G$ to be the ideal $$I(G)=\langle x_ix_j : \{i,j\}\in E(G) \rangle$$
	\end{defn}
	\begin{rem}\label{Remark: Edge Ideals Arbitrary Vertex Sets}
		As with Stanley-Reisner ideals we extend this definition to graphs $G$ on arbitrary vertex sets $V$ of size $n$ by first choosing an ordering on $V$. Again, all possible orderings give us the same ideal up to isomorphism, which we refer to as `\textit{the}' edge ideal.
	\end{rem}
	
	While Stanley-Reisner ideals account for all squarefree monomial ideals, edge ideals account for all squarefree monomial ideals generated entirely in degree $2$. Because the edge ideal of a graph \textit{is} a squarefree monomial ideal, Proposition \ref{Proposition: SR 1-1 Correspondence} tells us that it must be the Stanley-Reisner ideal of some complex. In fact, it is the Stanley-Reisner ideal of the independence complex of the graph.
	\begin{prop}\label{Proposition: I(G) = I_Ind(G)}
		Let $G$ be a graph on vertex set $[n]$. We have $$I(G) = I_{\Ind(G)}$$
	\end{prop}
	\begin{proof}
		The nonfaces of $\Ind(G)$ are all the sets $U\subseteq V$ which contain an edge of $G$. Thus the minimal nonfaces of $\Ind(G)$ are the edges of $G$. The result follows.
	\end{proof}
	
	
	\chapter{Dimensions of Betti Cones}\label{Chapter: Dimensions}
	In this chapter we study the dimensions of the Betti cones on Stanley-Reisner ideals and Edge ideals, which we denote by $\Dn$ and $\Cn$ respectively. We also fix a positive integer $h < n$, and introduce the subcones $\Dnh$ and $\Cnh$, generated respectively by Stanley-Reisner ideals and edge ideals of height $h$. Recall that these cones live inside the infinite-dimensional vector space $\Vn=\bigoplus_{d\in \ZZ} \QQ^{n+1}$.
	
	We make use of some notational short-hands throughout: for a complex $\Delta$ we use $\beta(\Delta)$ to denote the Betti diagram $\beta(I_\Delta)$; and similarly for a graph $G$ we use $\beta(G)$ to denote the Betti diagram $\beta(I(G))$.
	\begin{rem}\label{Remark: Ambiguity in beta notation}
		As observed at the start of Section \ref{Subsection: Graphs}, every graph is a $1$- or $0$-dimensional simplicial complex. Nevertheless, using the notation $\beta(G)$ to represent the diagram of the edge ideal of $G$ (as opposed to the Stanley-Reisner ideal obtained from viewing $G$ as a simplicial complex) should not result in ambiguity. It will usually be clear both from context and from our choice of notation whether we are viewing the structure in question as a graph or a complex, and hence whether its corresponding ideal is obtained using the edge ideal or Stanley-Reisner construction. In the few cases where there is a genuine risk of misinterpretation we will use the notation $\beta(I_\Delta)$ and $\beta(I(G))$ for clarity.
	\end{rem}
	
	Specifically we prove the following results about our four cones.
	
	\begin{thm}\label{Theorem: dimDn}
		Let $\Dn$ be the Betti cone generated by all Betti diagrams of Stanley-Reisner ideals in $R$. We have
		$$\dim \Dn = \frac{n(n+1)}{2}.$$
	\end{thm}
	
	\begin{thm}\label{Theorem: dimCn}
		Let $\Cn$ be the Betti cone generated by all Betti diagrams of edge ideals in $R$. We have
		$$\dim \Cn = \begin{cases}
			r^2  &\text{ if } n=2r \\
			r^2 + r &\text{ if } n=2r+1.\\
		\end{cases}$$
	\end{thm}
	
	\begin{thm}\label{Theorem: dimDnh}
		Let $\Dnh$ be the Betti cone generated by all Betti diagrams of Stanley-Reisner ideals in $R$ of height $h$. We have
		$$\dim \Dnh = \frac{n(n-1)}{2}-\frac{h(h-1)}{2}+1.$$
	\end{thm}
	
	\begin{thm}\label{Theorem: dimCnh}
		Let $\Cnh$ be the Betti cone generated by all Betti diagrams of edge ideals in $R$ of height $h$. We have $$\dim \Cnh = h(n-h-1)+1.$$
	\end{thm}
	
	\begin{rem}\label{Remark: Dntilde and Dnhtilde Dimension}
		Some definitions of simplicial complexes require that the complexes have no missing vertices. This is equivalent to requiring that the generators of the corresponding Stanley-Reisner ideals have degree at least two. We do not use this convention in this thesis, and as such, our cones $\Dn$ and $\Dnh$ contain diagrams corresponding to Stanley-Reisner ideals with generators of degree one.
		
		However, for the benefit of readers who prefer this convention, we denote the cone generated by diagrams of Stanley-Reisner ideals whose generators have degree at least two by $\Dntilde$. The dimension of this slightly smaller cone can be computed as $$\dim \Dntilde = \dim \Dn - n = \frac{n(n-1)}{2}.$$ Similarly, we let $\Dnhtilde$ denote the cone generated by diagrams of Stanley-Reisner ideals of height $h$ whose generators have degree at least two. The dimension of this cone is given by $$\dim \Dnhtilde = \dim \Dnh - h = \frac{n(n-1)}{2}-\frac{h(h+1)}{2}+1.$$
		
		Note that the generators of edge ideals all have degree exactly two, and thus we have the following inclusion of cones.
		\begin{center}
			\begin{tikzcd}
				\Cn \arrow[r, phantom, "\subset"]
				& \Dntilde \arrow[r, phantom, "\subset"]
				& \Dn \\
				\Cnh  \arrow[u, phantom, sloped, "\subset"] \arrow[r, phantom, "\subset"]
				& \Dnhtilde \arrow[u, phantom, sloped, "\subset"] \arrow[r, phantom, "\subset"]
				& \Dnh \arrow[u, phantom, sloped, "\subset"]
			\end{tikzcd}
		\end{center}
		
		We will show how our proof for the dimension of $\Dn$ can be modified to account for the subcone $\Dntilde$ in remarks along the way; and we will present our proofs for the cones $\Dnh$ and $\Dnhtilde$ together.
	\end{rem}
	
	\section{Motivation}\label{Subsection: Motivation Dimensions}
	There are a number of motivations for finding the dimensions of these cones. Most immediately, computing dimension is a crucial first step in understanding any cone: not only because the cone's dimension is a fundamental property in itself, but also because it may prove useful in further classifications. Specifically, any cone $\calC$ in $\Vn$ can be classified as the intersection of a number of defining halfspaces (see \cite{convex}, Theorem 4.5). These are halfspaces $\calH$ of $\Vn$ of the form $\{\beta \in \Vn : f(\beta) \geq 0 \}$ for some linear form $f(\beta)$, such that
	\begin{itemize}
		\item $\calC$ is completely contained in $\calH$.
		\item The intersection of $\calC$ with the boundary of $\calH$ (that is, the hyperplane $\{\beta \in \Vn : f(\beta) = 0 \}$)  has dimension $\dim \calC - 1$.
	\end{itemize}
	To determine whether a given halfspace $\calH$ of $\Vn$ is a defining halfspace for $\calC$, we need to check whether it satifies these two properties; and for the latter property, this is far easier to do if we know the dimension of $\calC$.
	
	Moreover, our proofs for these results will, in each case, require us to find the cone's minimal ambient vector space (i.e. the minimal subspace of $\Vn$ containing the cone). In this sense our four theorems above can be seen as analogues to Corollary \ref{Corollary: BS Conjecture 1 Vector Space} of the first Boij-S\"{o}derberg Conjecture, which told us the minimal subspace of $\Vn$ containing the cone $\calC(\ba,\bfb)$.
	
	Corollary \ref{Corollary: BS Conjecture 1 Vector Space} is, in essence, a result about the linear dependency relations satisfied by the diagrams in the cone $\calC(\ba,\bfb)$. Indeed, it tells us that (up to linear combination) the Herzog-K\"{u}hl Equations are the \textit{only} relations satisfied by all diagrams of Cohen-Macaulay modules of codimension $h$. Similarly, in finding the minimal subspaces for our cones, we will present (again, up to linear combination) every linear dependency relation which is satisfied by all diagrams in the cone.
	
	\section{Key Tools}
	Our proofs for all four results will proceed in roughly the same way: first, we bound the dimension from above by finding a finite-dimensional subspace of $\Vn=\bigoplus_{d\in \ZZ} \QQ^{n+1}$ containing the cone; then we bound it from below by exhibiting an appropriately sized linearly independent set of Betti diagrams lying in the cone (thus, in the process, showing that the subspace we found is in fact minimal). We begin by presenting some key tools which will help us in constructing our linearly independent sets of Betti diagrams.
	
	\subsection{Diagrams of Specific Families of Complexes and Graphs}\label{Subsection: Betti diagrams of Complexes and Graphs}
	
	In this section we present the Betti diagrams of some specific families of complexes and graphs, along with an important proposition and corollary (Lemma \ref{Lemma: Betti S-Delta} and Corollary \ref{Corollary: Betti G+L}) which help us construct Betti diagrams of slightly more complicated complexes and graphs. 
	
	In what follows, we use the notation $\Delta^j$, $\partial \Delta^j$ and $\Skel_r(V)$ from Definitions \ref{Definition: Simplex and Boundary of Simplex} and \ref{Definition: Skeleton Complex}, for integers $j\geq 0$ and $r\geq -1$ and vertex set $V$; and the notation $K_m$, $E_m$, $C_m$ and $L$ from Definition \ref{Definition: Important Graphs}, for positive integers $m$. We also write the disjoint union of graphs and complexes in additive notation, so for example we would use $C_5+2L$ to denote the graph
	
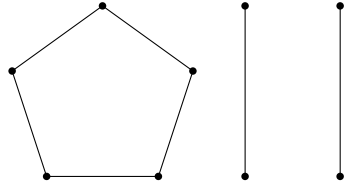
\begin{center}
		\begin{tikzpicture}[scale = 1.25]
			\tikzstyle{point}=[circle,thick,draw=black,fill=black,inner sep=0pt,minimum width=2pt,minimum height=2pt]
			\node (a)[point] at (0,1) {};
			\node (b)[point] at (0.951,0.309) {};
			\node (c)[point] at (0.588,-0.809) {};
			\node (d)[point] at (-0.588,-0.809) {};
			\node (e)[point] at (-0.951,0.309) {};
			
			\node (f)[point] at (1.5,1) {};
			\node (g)[point] at (1.5,-0.809) {};
			
			\node (h)[point] at (2.5,1) {};
			\node (i)[point] at (2.5,-0.809) {};
			
			\draw (a.center) -- (b.center) -- (c.center) -- (d.center) -- (e.center) -- cycle;
			\draw (f.center) -- (g.center);
			\draw (h.center) -- (i.center);
		\end{tikzpicture}
	\end{center}
	
	We begin by finding the homology (and hence the Betti diagrams) of skeleton complexes, which will allow us to derive the diagrams of a number of other families of complexes and graphs.
	\begin{lem}\label{Lemma: Homology of Skeleton Complexes}
		Let $m$ and $r$ be integers and set $\Delta = \Skel_r([m])$ with $m\geq 1$ and $r\geq -1$. For any integer $i\geq -1$ we have $$\dim_{\KK} \Hred_i(\Delta)=\begin{cases}
			{m-1 \choose r+1}	& \text{ if } i=r\\
			0 					& \text{ otherwise.}
		\end{cases}$$
	\end{lem}
	\begin{proof}
		We proceed by double induction on $r\geq -1$ and $m\geq 1$. Note that for any $m\geq 1$, the complex $\Skel_{-1}([m])$ is equal to the irrelevant complex $\{\emptyset\}$, which has only $\nth[st]{(-1)}$ homology of dimension ${m - 1 \choose 0}=1$. And for any $r\geq 0$, the complex $\Skel_{r}([1])$ is a single point. This is acyclic, which means it has $\nth{r}$ homology of dimension ${0 \choose r+1}=0$. This proves the result in the base cases $r=-1$ and $m=1$.
		
		Now suppose that $m>1$,  and define, for $\Delta=\Skel_r([m])$,
		\begin{align*}
			A &= \langle \sigma \in \Delta : m \in \sigma \rangle \\
			B &= \langle \sigma \in \Delta : m \notin \sigma \rangle\
		\end{align*}
		so that $\Delta$ decomposes into the union $A\cup B$. Note that $A$ is a cone over the vertex $m$, and is therefore acyclic. Also $B$ contains every subset of $[m-1]$ of dimension less than or equal to $r$, so it is equal to $\Skel_r([m-1])$. By induction this has only $\nth{r}$ homology of dimension ${m-2 \choose r+1}$. Meanwhile the intersection $A\cap B$ contains every subset of $[m-1]$ of dimension less than or equal to $r-1$, so it is equal to $\Skel_{r-1}([m-1])$, which has only $\nth[st]{(r-1)}$ homology of dimension ${m-2 \choose r}$ by induction.
		
		The Mayer-Vietoris Sequence (Proposition \ref{Proposition: MVS}) gives us the following short exact sequence
		\begin{equation*}
			0 \rightarrow \Hred_r(B) \rightarrow \Hred_r(\Delta) \rightarrow \Hred_{r-1}(A\cap B)\rightarrow 0\, .
		\end{equation*}
		Thus $\Delta$ has only $\nth{r}$ homology, and the dimension of this homology is ${m-2 \choose r+1}+{m-2 \choose r}={m-1 \choose r+1}$.
	\end{proof}
	
	\begin{lem}\label{Lemma: Betti Skeleton Complexes}
		Let $m$ and $r$ be integers and set $\Delta = \Skel_r([m])$. For any integers $i$ and $d$, we have
		$$\beta_{i,d}(\Delta)=
		\begin{cases}
			{m \choose i+r+2}	{i+r+1 \choose r+1}	& \text{if } 0 \leq i \leq m-r-2 \text{ and } d=i+r+2\\
			0 		& \text{otherwise.}
		\end{cases}$$
		In particular, the diagram $\beta(\Skel_r([m]))$ has the following shape.
		\begin{equation*}
			\begin{bmatrix}
				\beta_{0,r+2} & \dots & \beta_{m-r-2,m}\\
			\end{bmatrix}
		\end{equation*}
	\end{lem}
	\begin{proof}
		For any subset $U\subseteq [m]$ of size $d$, the induced subcomplex $\Delta_U$ is isomorphic to $\Skel_r([d])$. Thus Hochster's Formula tells us that $$\beta_{i,d}(I_\Delta) =\sum_{U\in {[m] \choose d}} \dim_\KK \Hred_{d-i-2}(\Skel_r([d])).$$ By Lemma \ref{Lemma: Homology of Skeleton Complexes}, this must be zero unless $d=r+i+2$, in which case we have ${m \choose r+i+2}$ subsets $U$ of $[m]$ of size $d$, each of which gives us an induced subcomplex with homology of dimension ${i+r+1 \choose r+1}$. The result follows.
	\end{proof}
	
	\begin{cor}\label{Corollary: Betti Diagram of Simplex}
		Let $j$ be a nonnegative integer, and let $\Delta^j$ and $\partial \Delta^j$ denote, respectively, the $j$-simplex and its boundary, on vertex set $[j+1]$. We have
		\begin{enumerate}
			\item $\beta(\Delta^j)=0$.
			\item $\beta_{i,d}(\partial \Delta^j)=\begin{cases}
				1 & \text{ if } (i,d) = (0,j+1)\\
				0 & \text{ otherwise.}
			\end{cases}$
			
			In particular, $\beta(\partial \Delta^j)$ has the shape $\begin{bmatrix} \beta_{0,j+1} \end{bmatrix}$.
		\end{enumerate}
	\end{cor}
	\begin{proof}
		These results follow immediately from Lemma \ref{Lemma: Betti Skeleton Complexes}, noting that $\Delta^j=\Skel_j([j+1])$ and $\partial \Delta^j = \Skel_{j-1}([j+1])$.
	\end{proof}
	\begin{rem}
		These results can also be seen from the fact that $\Delta^j$ has no missing faces and hence its Stanley-Reisner ideal is the zero ideal; and $\partial \Delta^j$ has a single missing face of degree $j+1$ and hence its Stanley-Reisner ideal admits the resolution $R(-(j+1))\rightarrow I_{\partial \Delta^j}$.
	\end{rem}
	
	We now move on to the Betti diagrams of our three important families of graphs.
	\begin{prop}\label{Proposition: Betti-En}
		Let $m$ be a positive integer, and let $E_m$ denote the empty graph on $m$ vertices. We have $\beta(E_m)=0$.
	\end{prop}
	\begin{proof}
		The independence complex of $E_m$ is $\Delta^{m-1}$, so the result follows from Corollary \ref{Corollary: Betti Diagram of Simplex}.
	\end{proof}
	\begin{rem}
		Again, this result can also be seen directly from the fact that the edge ideal of $E_m$ is the zero ideal.
	\end{rem}
	
	\begin{prop}\label{Proposition: Betti-Kn}
		Let $m$ be a positive integer, and let $K_m$ denote the complete graph on $m$ vertices. For any integers $i$ and $d$, we have
		$$\beta_{i,d}(K_m)=
		\begin{cases}
			(i+1){m\choose i+2} & \text{if } 0\leq i\leq m-2 \text{ and } d=i+2\\
			0 & \text{otherwise.}
		\end{cases}$$
		In particular, the diagram $\beta(K_m)$ has the following shape.
		\begin{equation*}
			\begin{bmatrix}
				\beta_{0,2} & \dots & \beta_{m-2,m}\\
			\end{bmatrix}
		\end{equation*}
	\end{prop}
	\begin{proof}
		The independence complex of $K_m$ is equal to $\Skel_0([m])$, so the result follows from Lemma \ref{Lemma: Betti Skeleton Complexes}. 
	\end{proof}
	
	\begin{prop}\label{Proposition: Betti-Cn}
		Let $m$ be a positive integer, and suppose $G$ is the complement of the cyclic graph on $m$ vertices, $C_m$. For any integers $i$ and $d$, we have
		\begin{equation*}
			\beta_{i,d}(G)=
			\begin{cases}
				\frac{m(i+1)}{m-i-2}{m-2\choose i+2}& \text{if } 0\leq i \leq m-4 \text{ and } d=i+2\\
				1 & \text{if } (i,d) = (m-3,m) \\
				0 & \text{otherwise.}
			\end{cases}
		\end{equation*}
		In particular, the diagram $\beta((C_m)^c)$ has the following shape.
		\begin{equation*}
			\begin{bmatrix}
				\beta_{0,2} & \dots & \beta_{m-4,m-2} & \\
				& & &\beta_{m-3,m}\\ 
			\end{bmatrix}
		\end{equation*}
	\end{prop}
	\begin{proof}
		See Theorem 2.3.3 in \cite{Ramos}.
	\end{proof}
	
	We end with some results about how the operations of coning and suspension affect (or rather, in the former case do \textit{not} affect) Betti diagrams. In particular, Lemma \ref{Lemma: Betti S-Delta} and Corollary \ref{Corollary: Betti G+L} will be crucial to our construction of Betti diagrams in this section, as they allow us to take a diagram we understand already and increase its regularity by exactly 1.
	
	\begin{lem}\label{Lemma: Betti C-Delta}
		Let $\Delta$ be a simplicial complex and $C\Delta$ the cone over $\Delta$. We have $\beta(\Delta)=\beta(C\Delta)$.
	\end{lem}
	\begin{proof}
		The minimal nonfaces of $C\Delta$ are the same as the minimal nonfaces of $\Delta$. This means that their Stanley-Reisner ideals have the same generators, and thus have the same Betti diagrams by Remark \ref{Remark: Betti diagrams of ideals only depend on generators}.
	\end{proof}
	
	\begin{cor}\label{Corollary: Betti Graphs With Isolated Vertices}
		Let $G$ be a graph and let $G+v$ denote the graph obtained by adding a single isolated vertex $v$ to $G$. We have $\beta(G+v) = \beta(G)$.
	\end{cor}
	\begin{proof}
		The independence complex $\Ind(G+v)$ is equal to $C\Ind(G)$, so this follows from Lemma \ref{Lemma: Betti C-Delta}. 
	\end{proof}
	
	\begin{lem}\label{Lemma: Betti S-Delta}
		Let $\Delta$ be a simplicial complex and $S\Delta$ the suspension of $\Delta$. For any integers $0\leq i\leq n$ and $d$, we have $$\beta_{i,d}(S\Delta)=\begin{cases}
			\beta_{0,2}(\Delta) + 1 &\text{ if } (i,d)=(0,2)\\
			\beta_{i,d}(\Delta)+\beta_{i-1,d-2}(\Delta) & \text{ otherwise.}
		\end{cases}$$
	\end{lem}
	\begin{proof}
		Let $i$ and $d$ be integers with $0\leq i \leq n$, and let $V$ denote the vertex set of $\Delta$. We can denote the vertex set of $S\Delta$ as $\tV=V\cup \{x,y\}$ for some additional vertices $x$ and $y$.
		
		Suppose $U$ is a subset of $\tV$ of size $d$. If exactly one of the vertices $x$ or $y$ is in $U$, then $(S\Delta)|_U$ is a cone over this vertex, and is hence acyclic. Thus in order for $(S\Delta)|_U$ to have homology, it must either contain neither of the vertices $x$ and $y$, or both of them. In the former case, the induced subcomplex $(S\Delta)|_U$ is equal to $\Delta_U$. In the latter, it is equal to $S(\Delta_{U-\{x,y\}})$.
		
		Thus, by Hochster's Formula and Corollary \ref{Corollary: Homology of Suspension}, we have
		\begin{align*}
			\beta_{i,d}(S\Delta)&=\sum_{U\in {V \choose d}} \dim_\KK \Hred_{d-i-2}(\Delta_U)+ \sum_{U\in {V \choose d-2}} \dim_\KK \Hred_{d-i-2}(S(\Delta_U))\\
			&= \beta_{i,d}(\Delta) +\sum_{U\in {V \choose d-2}}\dim_\KK \Hred_{d-i-3}(\Delta_U)\\
			&= \beta_{i,d}(\Delta) +\sum_{U\in {V \choose d-2}}\dim_\KK \Hred_{(d-2)-(i-1)-2}(\Delta_U)\\
			&=\begin{cases}
				\beta_{0,2}(\Delta) + 1 &\text{ if } (i,d)=(0,2)\\
				\beta_{i,d}(\Delta)+\beta_{i-1,d-2}(\Delta) & \text{ otherwise.}
			\end{cases}
		\end{align*}
	\end{proof}
	
	\begin{cor}\label{Corollary: Betti G+L}
		Let $G$ be a graph, and $L$ the graph consisting of a single edge between two vertices. For any integers $i$ and $d$, we have
		$$\beta_{i,d}(G+L)=\begin{cases}
			\beta_{0,2}(G) + 1 &\text{ if } (i,d)=(0,2)\\
			\beta_{i,d}(G)+\beta_{i-1,d-2}(G) & \text{ otherwise.}
		\end{cases}$$
	\end{cor}
	\begin{proof}
		The independence complex $\Ind(G+L)$ is equal to $S \Ind(G)$, so this follows from Lemma \ref{Lemma: Betti S-Delta}.
	\end{proof}

	\subsection{Indexing Sets and Initiality}\label{Subsection: Initiality}
	In the following sections, we establish formulae for the dimensions of our cones. Our proofs proceed by showing that the formulae given are both upper and lower bounds for the dimensions of the cones.
	
	To find a lower bound $l$ for the dimension of a convex cone $\calC$, it suffices to find a linearly independent set of $l$ vectors lying in $\calC$, as this shows that the smallest vector space containing $\calC$ must have dimension at least $l$. In this section, for ease of explanation, we present terminology for a simple condition that ensures linear independence. We then describe our general method for finding the dimensions of our cones, using this condition.
	
	Suppose $\calC$ lives inside the rational vector space $\calV=\bigoplus_{i\in \II} \QQ$ for some finite indexing set $\II$. For a vector $v$ in $\calV$ and an index $i\in \II$, let $v_i$ denote the $i^\text{th}$ coordinate of $v$. Also suppose we have a strict total ordering $\prec$ on $\II$.
	
	\begin{defn}\label{Definition: Initial}
		Let $v\in \calV$ and $i\in \II$. We say $v$ is \textit{$i$-initial with respect to $\prec$} (which we often write as \textit{$i_\prec$-initial}, or just \textit{$i$-initial} when doing so does not result in ambiguity) if
		\begin{enumerate}
			\item The component $v_i$ is nonzero;
			\item For every $j\in \II$ such that $i\prec j$, the component $v_j$ equals zero.
		\end{enumerate}
	\end{defn}
	\begin{ex}
		The rational vector space $\QQ^3$ can be thought of as $\bigoplus_{i\in [3]}\QQ e_i$, where $\{e_1,e_2,e_3\}$ is the canonical basis of $\QQ^3$. Ordering the indexing set $[3]$ in the standard way with $1\prec 2 \prec 3$, we have that the vector $(1,0,0)$ is $1_\prec$-initial, the vector $(1,1,0)$ is $2_\prec$-initial, and the vector $(1,1,1)$ is $3_\prec$-initial.
	\end{ex}
	If $X=\{v^i\}_{i\in \II}$ is a set of vectors lying in $\calC$ such that for each $i\in \II$, $v^i$ is $i_\prec$-initial, then $X$ must be linearly independent. So to find a linearly independent set of vectors in $\calC$, it suffices to define an order $\prec$ on some appropriately sized subset $\JJ \subseteq \II$, and find an $i_\prec$-initial vector lying in $\calC$ for each $i$ in $\JJ$.
	
	In our case, all of our cones live inside the infinite dimensional vector space $\Vn = \bigoplus_{d\in \ZZ} \QQ^{n+1}$. Thus we use the following method to find their dimensions.
	
	
	\begin{mdframed}
		\begin{meth}\label{Method: Dimension of Cones}
			Let $\calC$ be a cone lying inside $\Vn$. The following suffices to demonstrate that $\dim \calC = D$.
			\begin{enumerate}
				\item Find a finite indexing set $\II(\calC)\subset \{0,\dots,n\}\times \ZZ$ such that $\calC \subset \bigoplus_{(i,d)\in \II(\calC)} \QQ$.
				\item Find a subspace $\WW(\calC)\leq\bigoplus_{(i,d)\in \II(\calC)} \QQ$ containing $\calC$, with $\dim \WW(\calC)=D$ (this shows that $\dim \calC \leq D$).
				\item Define an ordering $\prec$ on $\II(\calC)$.
				\item Find a subset $\JJ\subseteq \II(\calC)$ of size $D$, and a set of diagrams $\{B^{i,d}:(i,d)\in \JJ\}$ inside $\calC$ such that for each $(i,d)\in \JJ$, the diagram $B^{i,d}$ is $(i,d)_\prec$-initial (this shows that $\dim \calC \geq D$).
			\end{enumerate}
		\end{meth}
	\end{mdframed}
	
	\section{Dimension of $\Dn$}\label{Subsection: dimDn}
	In this section, we prove Theorem \ref{Theorem: dimDn} on the dimension of the cone $\Dn$, generated by all diagrams of Stanley-Reisner ideals.
	
	\subsection{Upper Bound}\label{Subsection: ub Dn}
	We start by bounding the dimension from above, by following the first two steps of Method \ref{Method: Dimension of Cones}. In particular we work towards finding a finite indexing set $\II(\Dn)\subset \{0,...,n\}\times\ZZ$ of size $\frac{n(n+1)}{2}$ such that for every diagram $\beta\in \Dn$, and any pair of integers $0\leq i\leq n$ and $d$ with $(i,d)\notin \II(\Dn)$, we have $\beta_{i,d}=0$. This will demonstrate that $\Dn$ actually lies inside the finite-dimensional vector space $\WW(\Dn) = \bigoplus_{(i,d)\in \II(\Dn)}\QQ$, and hence we have $\dim \Dn\leq \dim \WW(\Dn) = \frac{n(n+1)}{2}$.
	
	We begin with the following well-known restrictions on the Betti diagrams of Stanley-Reisner ideals.
	
	\begin{prop}\label{Proposition: Inequalities Dn}
		Let $\Delta$ be a simplicial complex on vertex set $[n]$ and fix $\beta=\beta(\Delta)$. Suppose $i$ and $d$ are integers with $0\leq i\leq n$. We have $\beta_{i,d}=0$ if either of the following conditions hold:
		\begin{enumerate}
			\item $d > n$
			\item $d \leq i$
		\end{enumerate}
	\end{prop}
	\begin{proof}
		Both of these conditions are immediate results of Hochster's Formula.
		\begin{enumerate}
			\item There are no subsets of $[n]$ of size greater than $n$.
			\item If $d\leq i$ then $d-i-2\leq -2$, so no induced subcomplex of $\Delta$ has homology at degree $d-i-2$.
		\end{enumerate}
	\end{proof}
	\begin{rem}
		Part (2) of this result can also be seen from more elementary considerations. Specifically, the degrees of the generators of a squarefree monomial ideal are at least one, so the result must hold for $i=0$; and the degrees of the maps in any minimal free resolution are at least one, so the general result follows by induction on $i$.
	\end{rem}
	
	These inequalities give us a much clearer picture of the shape of the Betti diagrams in $\Dn$. In particular they look like the following.
	\begin{equation}\label{Equation: Dn Matrix}
		\begin{bmatrix}
			\beta_{0,1} & \beta_{1,2} & \dots & \dots & \dots & \beta_{n-1,n}  \\
			\beta_{0,2} & \beta_{1,3} & \dots & \dots & \beta_{n-2,n} & \\
			\beta_{0,3} & \beta_{1,4} & \dots & \beta_{n-3,n} &\\
			\vdots &\vdots & \vdots & & &\\
			\beta_{0,n-1} &\beta_{1,n} & & & & \\
			\beta_{0,n}
		\end{bmatrix}
	\end{equation}
	
	Thus, we may define our indexing set $\II(\Dn)$ and subspace $\WW(\Dn)$ as follows.
	\begin{defn}\label{Definition: II(Dn) and WW(Dn)}
		We define
		\begin{enumerate}
			\item $\II(\Dn) := \left\{(i,d)\in \{0,...,n\}\times \ZZ: 0\leq i < d\leq n \right\}$.
			\item $\WW(\Dn) := \bigoplus_{(i,d)\in \II(\Dn)}\QQ$.
		\end{enumerate}
	\end{defn}
	
	By Proposition \ref{Proposition: Inequalities Dn}, the cone $\Dn$ must lie in $\WW(\Dn)$ as desired.
	
	For ease of explanation, it will sometimes be useful for us to refer to individual rows of $\II(\Dn)$.
	\begin{defn}\label{Definition: Rows of Index Set}
		Let $(i,d)\in \II(\Dn)$. We say $(i,d)$ is in row $\rho$ if we have $d-i-1=\rho$.
	\end{defn}
	
	We can arrange the elements of $\II(\Dn)$ in rows as in Equation (\ref{Equation: Dn Matrix}).
	\begin{equation*}\label{Equation: I(Dn) Matrix}
		\begin{matrix}
			(0,1) & (1,2) & \dots & \dots & \dots & (n-1,n)\\
			(0,2) & (1,3) & \dots & \dots & (n-2,n)&\\
			(0,3)& (1,4) & \dots & (n-3,n) & &\\
			\vdots& \vdots & \vdots & & &\\
			(0,n-1)& (1,n) & & & & \\
			(0,n)& & & & &
		\end{matrix}
	\end{equation*}
	
	We can see that row 0 of $\II(\Dn)$ has $n$ elements, row 1 has $n-1$ elements, and so on. In general, for each $0\leq i \leq n-1$, row $i$ of $\II(\Dn)$ has $n-i$ elements. Hence we have
	\begin{align*}
		|\II(\Dn)|&= \sum_{i=0}^{n-1} (n-i)\\
		&= \frac{n(n+1)}{2}
	\end{align*}
	which means that the expression in Theorem \ref{Theorem: dimCn} is an upper bound for $\dim \Dn$.
	
	\begin{rem}\label{Remark: dimDntilde Upper Bound}
		We can modify the above argument to find an indexing set for the cone $\Dntilde$, generated by diagrams of complexes with no missing vertices. Specifically, if $\Delta$ is a complex with no missing vertices, then for any nonempty subset $\emptyset \neq U \subset V(\Delta)$, the induced subcomplex $\Delta_U$ cannot be equal to $\{\emptyset\}$, and hence $\Hred_{-1}(\Delta_U)=0$. By Hochster's formula this means that $\beta_{0,1}(\Delta)=\beta_{1,2}(\Delta)=...=\beta_{n-1,n}(\Delta)=0$. Thus for diagrams in the cone $\Dntilde$, the top row of Betti numbers in Equation (\ref{Equation: Dn Matrix}) are all zero. In particular the indexing set $\II(\Dntilde)$ is equal to $\II(\Dn)-\{(0,1),\dots,(n-1,n)\}$, which has size $\frac{n(n-1)}{2}$.
	\end{rem}
	
	\subsection{Lower Bound}\label{Subsection: lb Dn}
	To complete our proof of Theorem \ref{Theorem: dimDn}, we need to show that the space $\WW(\Dn)$ is in fact the \textit{minimal} subspace of $\Vn$ containing $\Dn$, by following the last two steps of Method \ref{Method: Dimension of Cones}.
	
	In particular we present an ordering $\prec$ on $\II(\Dn)$ and a set of $(i,d_\prec)$-initial diagrams in $\Dn$ for each $(i,d)$ in $\II(\Dn)$. The ordering we choose is as follows. 
	\begin{defn}\label{Definition: Order Dn}
		For any two pairs $(i,d)$ and $(i',d')$ in $\II(\Dn)$ we write $(i,d)\prec (i',d')$ if $d-i< d'-i'$, or $d-i= d'-i'$and $i<i'$.
	\end{defn}
	\begin{rem}
		In other words we say $(i,d)\prec(i',d')$ if $(i,d)$ lies in a lower numbered row, or if they both lie in the same row with $i<i'$.
	\end{rem}
	
	For convenience, we extend the terminology of Definition \ref{Definition: Initial} by declaring a complex $\Delta$ to be $(i,d)_\prec$-initial if the diagram $\beta(\Delta)$ is $(i,d)_\prec$-initial.
	
	Thus we need to find a set of $(i,d)_\prec$-initial complexes on vertex set $[n]$. The following lemma will be helpful in this, because it allows us to broaden our search from complexes with exactly $n$ vertices to complexes with at \textit{most} $n$ vertices.
	\begin{lem}\label{Lemma: Dm in Dn}
		For any positive integer $m<n$, we have $\calD_m\subset \Dn$.
	\end{lem}
	\begin{proof}
		If $\Delta$ is a complex on vertex set $[m]$, then by Lemma \ref{Lemma: Betti C-Delta}, we can extend it to a graph on $[n]$ by taking a cone over it a total of $n-m$ times, without affecting its Betti diagram. This means the diagram $\beta(\Delta)$ lies inside $\Dn$, and the result follows.
	\end{proof}
	
	There is now an obvious candidate for our family of $(i,d)_\prec$-initial complexes.
	\begin{prop}\label{Proposition: Skeleton complexes are (i,d)-initial}
		Let the ordering $\prec$ on $\II(\Dn)$ be as in Definition \ref{Definition: Order Dn}, and let $(i,d)\in \II(\Dn)$. The complex $\Skel_{d-i-2}([d])$ is $(i,d)_\prec$-initial.
	\end{prop}
	\begin{proof}
		By Lemma \ref{Lemma: Betti Skeleton Complexes} the complex $\Skel_r([m])$ is $(m-r-2,m)$-initial. The result follows from substituting $d=m$ and $i=m-r-2$.
	\end{proof}
	\begin{ex}
		Suppose we have $n=4$. The following diagram depicts, for each index $(i,d)\in \II(\Dn[4])$, an $(i,d)$-initial complex on up to $4$ vertices (we use a $\times$ symbol to denote missing vertices).
		\begin{center}
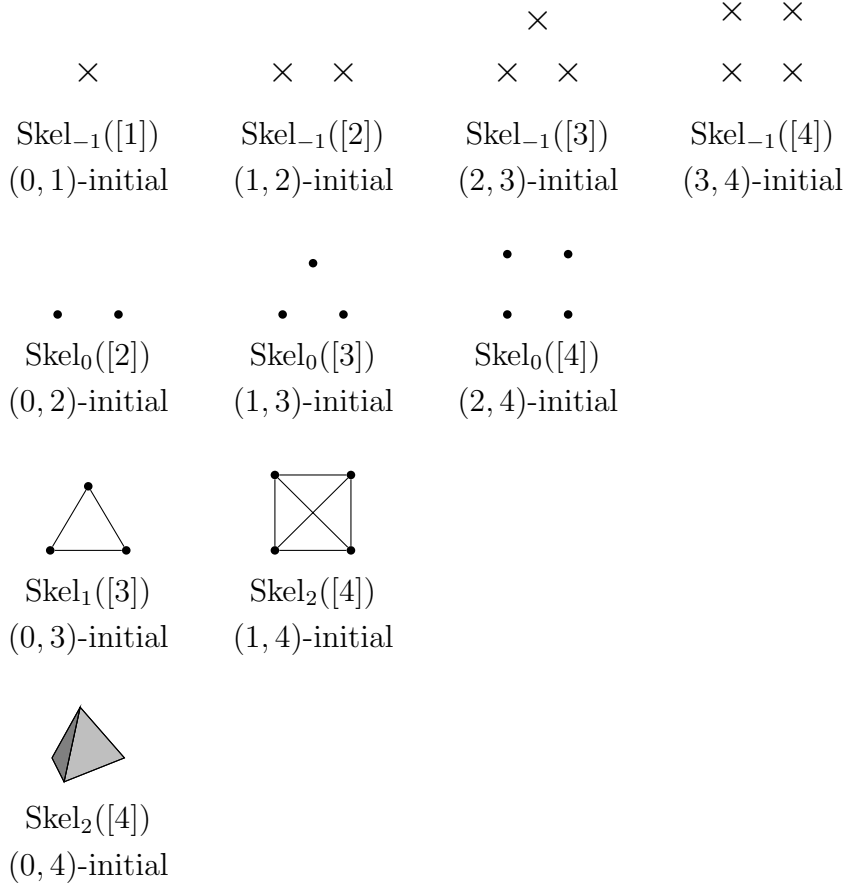
			\begin{tabular}{ c c c c c c c }
				\begin{tikzpicture}[scale=0.4]
					\tikzstyle{point}=[circle,thick,draw=black,fill=black,inner sep=0pt,minimum width=2pt,minimum height=2pt]
					\node[scale=1.2] at (0,0) {$\times$};
				\end{tikzpicture}& 
				& \begin{tikzpicture}[scale=0.4]
					\tikzstyle{point}=[circle,thick,draw=black,fill=black,inner sep=0pt,minimum width=2pt,minimum height=2pt]
					\node[scale=1.2] at (0,0) {$\times$};
					\node[scale=1.2] at (2,0) {$\times$};
				\end{tikzpicture} & & \begin{tikzpicture}[scale=0.4]
					\tikzstyle{point}=[circle,thick,draw=black,fill=black,inner sep=0pt,minimum width=2pt,minimum height=2pt]
					\node[scale=1.2] at (0,0) {$\times$};
					\node[scale=1.2] at (2,0) {$\times$};
					\node[scale=1.2] at (1,1.7) {$\times$};
				\end{tikzpicture} & & \begin{tikzpicture}[scale=0.4]
					\tikzstyle{point}=[circle,thick,draw=black,fill=black,inner sep=0pt,minimum width=2pt,minimum height=2pt]
					\node[scale=1.2] at (0,0) {$\times$};
					\node[scale=1.2] at (2,0) {$\times$};
					\node[scale=1.2] at (0,2) {$\times$};
					\node[scale=1.2] at (2,2) {$\times$};
				\end{tikzpicture} \\
				$\Skel_{-1}([1])$& & $\Skel_{-1}([2])$ & & $\Skel_{-1}([3])$ & & $\Skel_{-1}([4])$\\
				$(0,1)$-initial& & $(1,2)$-initial & & $(2,3)$-initial& & $(3,4)$-initial\\
				
				& & & & & & \\
				
				\begin{tikzpicture}[scale=0.4]
					\tikzstyle{point}=[circle,thick,draw=black,fill=black,inner sep=0pt,minimum width=2pt,minimum height=2pt]
					\node[point, scale=1.2] at (0,0) {};
					\node[point, scale=1.2] at (2,0) {};
				\end{tikzpicture}
				& & \begin{tikzpicture}[scale=0.4]
					\tikzstyle{point}=[circle,thick,draw=black,fill=black,inner sep=0pt,minimum width=2pt,minimum height=2pt]
					\node[point, scale=1.2] at (0,0) {};
					\node[point, scale=1.2] at (2,0) {};
					\node[point, scale=1.2] at (1,1.7) {};
				\end{tikzpicture} & & \begin{tikzpicture}[scale=0.4]
					\tikzstyle{point}=[circle,thick,draw=black,fill=black,inner sep=0pt,minimum width=2pt,minimum height=2pt]
					\node[point, scale=1.2] at (0,0) {};
					\node[point, scale=1.2] at (2,0) {};
					\node[point, scale=1.2] at (0,2) {};
					\node[point, scale=1.2] at (2,2) {};
				\end{tikzpicture} & & \\
				$\Skel_0([2])$& & $\Skel_0([3])$ & & $\Skel_0([4])$ & & \\
				$(0,2)$-initial& & $(1,3)$-initial & & $(2,4)$-initial& &\\
				
				& & & & & & \\
				
				\begin{tikzpicture}[scale=0.5]
					\tikzstyle{point}=[circle,thick,draw=black,fill=black,inner sep=0pt,minimum width=2pt,minimum height=2pt]
					\node[point, scale=1.2] (1) at (0,0) {};
					\node[point, scale=1.2] (2) at (2,0) {};
					\node[point, scale=1.2] (3) at (1,1.7) {};
					
					\draw (1) -- (2) -- (3) -- (1);
				\end{tikzpicture} & & \begin{tikzpicture}[scale=0.5]
					\tikzstyle{point}=[circle,thick,draw=black,fill=black,inner sep=0pt,minimum width=2pt,minimum height=2pt]
					\node[point, scale=1.2] (1) at (0,0) {};
					\node[point, scale=1.2] (2) at (2,0) {};
					\node[point, scale=1.2] (3) at (0,2) {};
					\node[point, scale=1.2] (4) at (2,2) {};
					
					\draw (1) -- (2) -- (3) -- (4) -- (1);
					\draw (1) -- (3);
					\draw (2) -- (4);
				\end{tikzpicture} & & & & \\
				$\Skel_1([3])$& & $\Skel_2([4])$ & & & & \\
				$(0,3)$-initial& & $(1,4)$-initial & & & &\\
				
				& & & & & & \\
				
				\begin{tikzpicture}[scale=0.55][line join = round, line cap = round]
					
					\coordinate (4) at (0,{sqrt(2)},0);
					\coordinate (3) at ({-.5*sqrt(3)},0,-.5);
					\coordinate (2) at (0,0,1);
					\coordinate (1) at ({.5*sqrt(3)},0,-.5);
					
					\begin{scope}
						\draw (1)--(3);
						\draw[fill=lightgray,fill opacity=.5] (2)--(1)--(4)--cycle;
						\draw[fill=gray,fill opacity=.5] (3)--(2)--(4)--cycle;
						\draw (2)--(1);
						\draw (2)--(3);
						\draw (3)--(4);
						\draw (2)--(4);
						\draw (1)--(4);
					\end{scope}
				\end{tikzpicture}& & & & & & \\
				$\Skel_2([4])$& & & & & & \\
				$(0,4)$-initial& & & & & &
			\end{tabular}
		\end{center}
		Note that while many of these complexes have fewer than $4$ vertices, their Betti diagrams all still lie inside $\Dn[4]$ by Lemma \ref{Lemma: Dm in Dn}.
	\end{ex}
	
	Putting these results together we can now prove Theorem \ref{Theorem: dimDn}.
	\begin{proof}[Proof of Theorem \ref{Theorem: dimDn}]
		We have already seen that $$\dim \Dn \leq 
		|\II(\Dn)| = \frac{n(n+1)}{2}.$$
		We can show that $\dim \Dn \geq |\II(\Dn)|$ by finding a linearly independent set of diagrams in $\Dn$ of size $|\II(\Dn)|$.
		
		To that end, let $(i,d)$ be an index in $\II(\Dn)$ and define the ordering $\prec$ on $\II(\Dn)$ as in Definition \ref{Definition: Order Dn}. Proposition \ref{Proposition: Skeleton complexes are (i,d)-initial} tells us that the complex $\Skel_{d-i-2}([d])$ is $(i,d)_{\prec}$-initial. It also has $d$ vertices, and because $d\leq n$, Lemma \ref{Lemma: Dm in Dn} tells us that its diagram lies inside $\Dn$. Thus we have an $(i,d)$-initial diagram in $\Dn$ for each $(i,d)$ in $\II(\Dn)$. This completes the proof.
	\end{proof}
	
	\begin{rem}\label{Remark: dimDntilde Lower Bound}
		If we restrict the ordering $\prec$ to the subset $\II(\Dntilde)\subset \II(\Dn)$ (as defined in Remark \ref{Remark: dimDntilde Upper Bound}), the above proof also shows that the cone $\Dntilde$ contains an $(i,d)$-initial diagram for every index $(i,d)$ in $\II(\Dntilde)$, and hence we have $\dim \Dntilde = |\II(\Dntilde)|=\frac{n(n-1)}{2}$. This is because all of the $(i,d)$-initial complexes chosen for the indices $(i,d)$ in $\II(\Dntilde)$ are $r$-skeletons for some $r\geq 0$, and therefore have no missing vertices, so their diagrams also lie inside $\Dntilde$.
	\end{rem}
	
	Before we move on to finding the dimension of our three subcones of $\Dn$ we note the following important lemma, which we can use to construct $(i,d)_\prec$-initial diagrams in $\Dn$ recursively. While we did not require this lemma to demonstrate the dimension of $\Dn$, analogous results will be crucial in demonstrating the dimensions of the three subcones we are about to consider.
	
	\begin{lem}\label{Lemma: SDelta initiality}
		Let $\prec$ be as in Definition \ref{Definition: Order Dn}, suppose $n>2$, and let $(i,d)$ be an index in $\II(\Dn)$ with $d-i\geq 2$. If $\Delta$ is an $(i-1,d-2)_\prec$-initial complex then the complex $S\Delta$ is $(i,d)_\prec$-initial.
	\end{lem}
	\begin{proof}
		By Lemma \ref{Lemma: Betti S-Delta}, we have $\beta_{i,d}(S\Delta) = \beta_{i,d}(\Delta) + \beta_{i-1,d-2}(\Delta)$. By the $(i-1,d-2)$-initiality of $\beta(\Delta)$, we have $\beta_{i-1,d-2}(\Delta) \neq 0$, so $\beta_{i,d}(S\Delta)$ must be nonzero too.
		
		Now let $(i',d')\in \II(\Dn)$ with $(i,d)\prec (i',d')$. Again, by Lemma \ref{Lemma: Betti S-Delta} we have $\beta_{i',d'}(S\Delta) = \beta_{i',d'}(\Delta) + \beta_{i'-1,d'-2}(\Delta)$. We must have $(i-1,d-2)\prec (i'-1,d'-2)$, and also $(i'-1,d'-2)\prec (i',d')$ because they are in different rows. Hence, by the $(i-1,d-2)$-initiality of $\beta(\Delta)$, both the terms $\beta_{i',d'}(\Delta)$ and $\beta_{i'-1,d'-2}(\Delta)$ are zero, and $\beta_{i',d'}(S\Delta)$ is zero too.
		
		This shows that $\beta(S\Delta)$ is $(i,d)_\prec$-initial as required.
	\end{proof}
	
	\section{Dimension of $\Cn$}\label{Subsection: dimCn}
	In this section, we prove Theorem \ref{Theorem: dimCn}, on the dimension of the subcone $\Cn\subset \Dn$ generated by diagrams of edge ideals.
	
	\subsection{Upper Bound}\label{Subsection: ub Cn}
	As before, we start by bounding the dimension from above, by finding an appropriately sized indexing set $\II(\Cn)$ at which the diagrams of $\Cn$ are all nonzero, and hence a vector space $\WW(\Cn)$ containing $\Cn$ of the appropriate dimension. Note that we must have $\II(\Cn)\subseteq \II(\Dn)$ and hence $\WW(\Cn)\leq \WW(\Dn)$ because every edge ideal is a Stanley-Reisner ideal. Moreover, every generator of an edge ideal has degree two, and hence we have that $\II(\Cn)$ is actually a subset of $\II(\Dntilde)$, so it does not contain the indices $(0,1),\dots,(n-1,n)$.
	
	To find our indexing set $\II(\Cn)$, we need to obtain further restrictions on the positions of the nonzero values of the diagrams in $\Cn$. The following result turns out to be sufficient.
	\begin{prop}\label{Proposition: Inequalities Cn}
		Let $G$ be a graph on vertex set $[n]$, and set $\beta=\beta(G)$. For any integers $i$ and $d$ with $d > 2i +2$, we have $\beta_{i,d} = 0$.
	\end{prop}
	\begin{proof}
		This is Lemma 2.2 in \cite{Katz}.
	\end{proof}
	
	This additional inequality gives us a clearer picture of the shape of the Betti diagrams in $\Cn$. Specifically, if $n = 2r$ is even, then the diagrams $\beta\in \Cn$ look like this.
	\begin{equation}\label{Equation: Cn Even Matrix}
		\begin{bmatrix}
			\beta_{0,2} & \beta_{1,3} & \beta_{2,4} & \dots & \dots & \beta_{n-3,n-1} & \beta_{n-2,n}\\
			& \beta_{1,4} & \beta_{2,5} & \dots & \dots & \beta_{n-3,n} &\\
			& & \ddots & & & &\\
			& & & \beta_{r-1,2r} & & & \\
		\end{bmatrix}
	\end{equation}
	If $n = 2r+1$ is odd, then they look like this.
	\begin{equation}\label{Equation: Cn Odd Matrix}
		\begin{bmatrix}
			\beta_{0,2} & \beta_{1,3} & \beta_{2,4} & \dots & \dots & \dots & \beta_{n-3,n-1} & \beta_{n-2,n}\\
			& \beta_{1,4} & \beta_{2,5} & \dots & \dots & \dots & \beta_{n-3,n} &\\
			& & \ddots & & & & &\\
			& & & \beta_{r-1,2r} & \beta_{r, 2r+1} & & & \\
		\end{bmatrix}
	\end{equation}
	
	Thus, we may define our indexing set $\II(\Cn)$ and subspace $\WW(\Cn)$ as follows.
	\begin{defn}\label{Definition: II(Cn) and WW(Cn)}
		We define
		\begin{enumerate}
			\item $\II(\Cn) := \left\{(i,d)\in \{0,...,n\}\times \ZZ: i+2\leq d\leq \min\{2i+2,n\} \right\}$.
			\item $\WW(\Cn) := \bigoplus_{(i,d)\in \II(\Cn)}\QQ$.
		\end{enumerate}
	\end{defn}
	
	By Proposition \ref{Proposition: Inequalities Cn}, the cone $\Cn$ must lie in $\WW(\Cn)$ as desired.
	
	Just as with $\II(\Dn)$ we can arrange the elements of $\II(\Cn)$ in rows to match the above Betti diagrams (we number the rows of $\II(\Cn)$ using the same convention we used for the rows of $\II(\Dn)$ in Definition \ref{Definition: Rows of Index Set}).
	So if $n=2r$ then the rows of $\II(\calC_{2r})$ look like the following.
	\begin{equation*}
		\begin{matrix}
			(0,2) & (1,3) & (2,4) & \dots & \dots & (n-3,n-1) & (n-2,n)\\
			& (1,4) & (2,5) & \dots & \dots & (n-3,n) &\\
			& & \ddots & & & &\\
			& & & (r-1,2r) & & & \\
		\end{matrix}
	\end{equation*}
	And if $n=2r+1$ then the rows of $\II(\calC_{2r+1})$ look like the following.
	\begin{equation*}
		\begin{matrix}
			(0,2) & (1,3) & (2,4) & \dots & \dots & \dots & (n-3,n-1) & (n-2,n)\\
			& (1,4) & (2,5) & \dots & \dots & \dots & (n-3,n) &\\
			& & & \ddots & & & & &\\
			& & & & (r-1,2r) & (r,2r+1) &  & & \\
		\end{matrix}
	\end{equation*}
	
	We can see that row 1 of $\II(\Cn)$ has $n-1$ elements, row 2 has $n-3$ elements, and so on. In general, for each $1\leq i \leq r$, row $i$ of $\II(\Cn)$ has $n-2i+1$ elements.
	
	Hence we have
	\begin{align*}
		|\II(\calC_{2r})|&= \sum_{i=1}^r (2r-2i+1)\\
		&= \sum_{i=1}^r (2r) - \sum_{i=1}^{r}(2i-1)\\
		&= 2r^2 - r^2\\
		&= r^2
	\end{align*}
	and
	\begin{align*}
		|\II(\calC_{2r+1})|&=\sum_{i=1}^r (2r+1-2i+1)\\
		&= \sum_{i=1}^r (2r+1) - \sum_{i=1}^{r}(2i-1)\\
		&= (2r^2 + r) - r^2\\
		&= r^2 + r \, .
	\end{align*}
	
	Therefore, the expressions in Theorem \ref{Theorem: dimCn} are upper bounds for $\dim \Cn$.
	
	\subsection{Lower Bound}\label{Subsection: lb Cn}
	We are left with the task of showing that the space $\WW(\Cn)$ is in fact the minimal subspace of $\WW(\Dn)$ containing the cone $\Cn$, which we do by finding an appropriately sized linearly independent set of diagrams in $\Cn$, following Method \ref{Method: Dimension of Cones}. Many of the constructions and lemmas we need are analogues of the ones we employed for the cone $\Dn$.
	
	For our ordering $\prec$ on $\II(\Cn)$ we simply take the restriction of the ordering on $\II(\Dn)$ given in Definition \ref{Definition: Order Dn}. And just as we did with complexes, we now extend the terminology of Definition \ref{Definition: Initial} to graphs by declaring a graph $G$ to be $(i,d)_\prec$-initial if the diagram $\beta(G)$ is $(i,d)_\prec$-initial.
	
	The cone $\Cn$ also admits a direct analogue of Lemma \ref{Lemma: Dm in Dn}, which allows us to broaden our search from graphs with exactly $n$ vertices to graphs with at most $n$ vertices.
	\begin{lem}\label{Lemma: Cm in Cn}
		For any positive integer $m < n$, we have $\Cn[m] \subset \Cn$.
	\end{lem}
	\begin{proof}
		If $G$ is a graph on vertex set $[m]$, then by Corollary \ref{Corollary: Betti Graphs With Isolated Vertices}, we can extend it to a graph on $[n]$ by adding some isolated vertices, without affecting its Betti diagram. This means the diagram $\beta(G)$ lies inside $\Cn$, and the result follows.
	\end{proof}
	
	Unfortunately we cannot use the diagrams of skeleton complexes as our $(i,d)$-initial diagrams in $\Cn$, because not all skeleton complexes are independence complexes of graphs, and hence the majority of the $(i,d)$-initial diagrams we found in the last section do not lie in the cone $\Cn$.
	
	However, the diagrams of the 0-skeletons $\Skel_0([m])$  \textit{do} lie in $\Cn$, because we have $\Ind(K_m)=\Skel_0([m])$. Thus we can find $(i,d)$-initial diagrams for every $(i,d)$ in row 1 of $\II(\Cn)$. To find the rest, we use the following analogue of Lemma \ref{Lemma: SDelta initiality}.
	\begin{lem}\label{Lemma: G+L initiality}
		Let $\prec$ be as in Definition \ref{Definition: Order Dn}, suppose $n>2$, and let $(i,d)$ be an index in $\II(\Cn)$ with with $d-i\geq 3$. If $\Delta$ is an $(i-1,d-2)_\prec$-initial graph then the graph $G+L$ is $(i,d)_\prec$-initial.
	\end{lem}
	\begin{proof}
		This follows directly from Lemma \ref{Lemma: SDelta initiality} because $\Ind(G+L)=S\Ind(G)$.
	\end{proof}
	
	Before we explain the procedure for finding $(i,d)$-initial graphs in the general case, we present a specific example to illustrate the basic principle.
	\begin{ex}
		The set $\II(\calC_6)$ has size $3^2 = 9$, and it looks like the following.
		\begin{equation*}
			\begin{matrix}
				(0,2) & (1,3) & (2,4) & (3,5) & (4,6)\\
				& (1,4) & (2,5) & (3,6) &\\
				& & (2,6) & &
			\end{matrix}
		\end{equation*}
		
		So we want to find nine linearly independent diagrams in $\mathcal{C}_6$, one for each $(i,d)\in \II(\calC_6)$. The ordering $\prec$ on $\II(\calC_6)$ is $(0,2)\prec (1,3) \prec (2,4) \prec (3,5) \prec (4,6) \prec (1,4) \prec (2,5) \prec (3,6) \prec (2,6)$.
		
		By Proposition \ref{Proposition: Betti-Kn}, we see that the complete graph on $2$ vertices, $K_2$, is $(0,2)$-initial. Similarly, $K_3$ is $(1,3)$-initial, $K_4$ is $(2,4)$-initial, $K_5$ is $(3,5)$-initial and $K_6$ is $(4,6)$-initial.
		
		From the above, and Lemma \ref{Lemma: G+L initiality}, we also find that $K_2+L$ is $(1,4)$-initial, $K_3+L$ is $(2,5)$-initial and $K_4+L$ is $(3,6)$-initial. Similarly, we can see that $K_2+2L$ is $(2,6)$-initial.
		
		So placing each graph in its corresponding position in $\II(\Cn)$, we get the following.
		\begin{center}
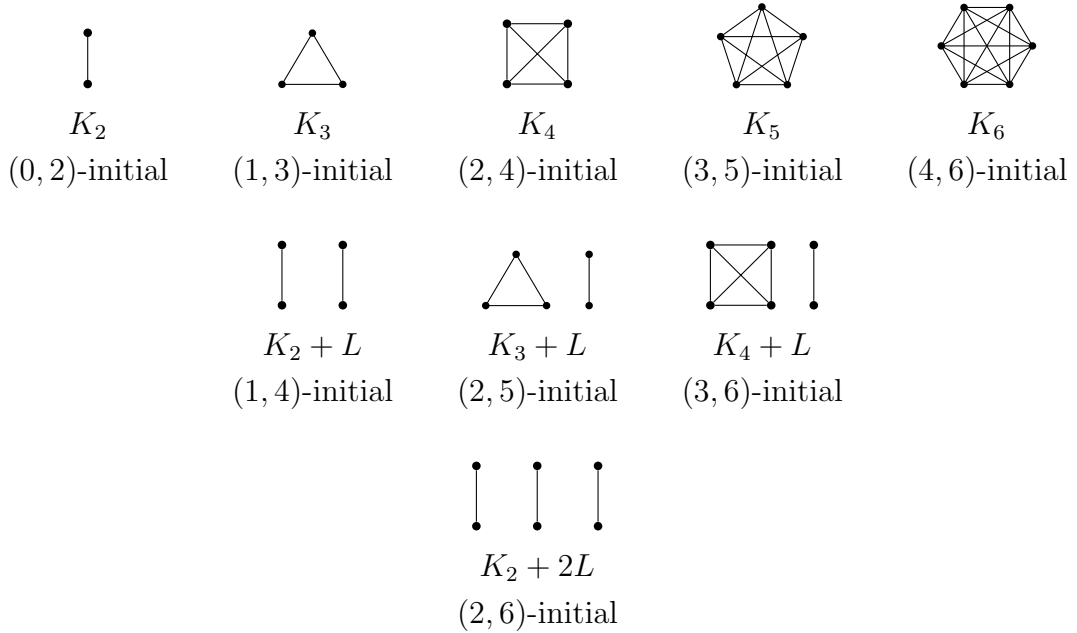
			\begin{tabular}{ c c c c c c c c c }
				\begin{tikzpicture}[scale = 0.4]
					\tikzstyle{point}=[circle,thick,draw=black,fill=black,inner sep=0pt,minimum width=2pt,minimum height=2pt]
					\node (a)[point, scale=1.2] at (1,0) {};
					\node (b)[point, scale=1.2] at (1,1.7) {};
					
					\draw (a.center) -- (b.center);
				\end{tikzpicture} && \begin{tikzpicture}[scale = 0.4]
					\tikzstyle{point}=[circle,thick,draw=black,fill=black,inner sep=0pt,minimum width=2pt,minimum height=2pt]
					\node (a)[point] at (0,0) {};
					\node (b)[point] at (2,0) {};
					\node (c)[point] at (1,1.7) {};
					
					\draw (a.center) -- (b.center) -- (c.center) -- cycle;
				\end{tikzpicture} && \begin{tikzpicture}[scale=0.4]
					\tikzstyle{point}=[circle,thick,draw=black,fill=black,inner sep=0pt,minimum width=2pt,minimum height=2pt]
					\node[point, scale=1.2] (1) at (0,0) {};
					\node[point, scale=1.2] (2) at (2,0) {};
					\node[point, scale=1.2] (3) at (0,2) {};
					\node[point, scale=1.2] (4) at (2,2) {};
					
					\draw (1) -- (2) -- (3) -- (4) -- (1);
					\draw (1) -- (3);
					\draw (2) -- (4);
				\end{tikzpicture} && \begin{tikzpicture}[scale = 0.57]
					\tikzstyle{point}=[circle,thick,draw=black,fill=black,inner sep=0pt,minimum width=2pt,minimum height=2pt]
					\node (a)[point] at (0,1) {};
					\node (b)[point] at (0.951,0.309) {};
					\node (c)[point] at (0.588,-0.809) {};
					\node (d)[point] at (-0.588,-0.809) {};
					\node (e)[point] at (-0.951,0.309) {};
					
					\draw (a.center) -- (b.center) -- (c.center) -- (d.center) -- (e.center) -- cycle;
					\draw (a.center) -- (c.center) -- (e.center) -- (b.center) -- (d.center) -- cycle;
				\end{tikzpicture} && \begin{tikzpicture}[scale = 0.3]
					\tikzstyle{point}=[circle,thick,draw=black,fill=black,inner sep=0pt,minimum width=2pt,minimum height=2pt]
					\node (1)[point] at (0,0) {};
					\node (2)[point] at (2,0) {};
					\node (3)[point] at (3,1.7) {};
					\node (4)[point] at (2,3.4) {};
					\node (5)[point] at (0,3.4) {};
					\node (6)[point] at (-1,1.7) {};
					
					\draw (1.center) -- (2.center) -- (3.center) -- (4.center) -- (5.center) -- (6.center) -- cycle;
					\draw (1.center) -- (3.center) -- (5.center) -- cycle;
					\draw (2.center) -- (4.center) -- (6.center) -- cycle;
					\draw (1.center) -- (4.center);
					\draw (2.center) -- (5.center);
					\draw (3.center) -- (6.center);
				\end{tikzpicture} \\
				$K_2$&& $K_3$ && $K_4$ && $K_5$ && $K_6$  \\
				$(0,2)$-initial && $(1,3)$-initial && $(2,4)$-initial && $(3,5)$-initial && $(4,6)$-initial\\
				& & & & & & & & \\
				
				&& \begin{tikzpicture}[scale=0.4]
					\tikzstyle{point}=[circle,thick,draw=black,fill=black,inner sep=0pt,minimum width=2pt,minimum height=2pt]
					\node[point, scale=1.2] (1) at (0,0) {};
					\node[point, scale=1.2] (2) at (2,0) {};
					\node[point, scale=1.2] (3) at (0,2) {};
					\node[point, scale=1.2] (4) at (2,2) {};
					
					\draw (1) -- (3);
					\draw (2) -- (4);
				\end{tikzpicture} && \begin{tikzpicture}[scale = 0.4]
					\tikzstyle{point}=[circle,thick,draw=black,fill=black,inner sep=0pt,minimum width=2pt,minimum height=2pt]
					\node (a)[point] at (0,0) {};
					\node (b)[point] at (2,0) {};
					\node (c)[point] at (1,1.7) {};
					
					\node (d)[point] at (3.4,0) {};
					\node (e)[point] at (3.4,1.7) {};
					
					\draw (a.center) -- (b.center) -- (c.center) -- cycle;
					\draw (d.center) -- (e.center);
				\end{tikzpicture}  && \begin{tikzpicture}[scale=0.4]
					\tikzstyle{point}=[circle,thick,draw=black,fill=black,inner sep=0pt,minimum width=2pt,minimum height=2pt]
					\node[point, scale=1.2] (1) at (0,0) {};
					\node[point, scale=1.2] (2) at (2,0) {};
					\node[point, scale=1.2] (3) at (0,2) {};
					\node[point, scale=1.2] (4) at (2,2) {};
					
					\node[point, scale=1.2] (5) at (3.4,0) {};
					\node[point, scale=1.2] (6) at (3.4,2) {};
					
					\draw (1) -- (2) -- (3) -- (4) -- (1);
					\draw (1) -- (3);
					\draw (2) -- (4);
					\draw (5) -- (6);
				\end{tikzpicture} && \\
				&& $K_2+L$ && $K_3+L$ && $K_4+L$ &&  \\
				&& $(1,4)$-initial && $(2,5)$-initial && $(3,6)$-initial &&\\
				& & & & & & & & \\
				
				&& && \begin{tikzpicture}[scale=0.4]
					\tikzstyle{point}=[circle,thick,draw=black,fill=black,inner sep=0pt,minimum width=2pt,minimum height=2pt]
					\node[point, scale=1.2] (1) at (0,0) {};
					\node[point, scale=1.2] (2) at (2,0) {};
					\node[point, scale=1.2] (3) at (0,2) {};
					\node[point, scale=1.2] (4) at (2,2) {};
					\node[point, scale=1.2] (5) at (4,0) {};
					\node[point, scale=1.2] (6) at (4,2) {};
					
					\draw (1) -- (3);
					\draw (2) -- (4);
					\draw (5) -- (6);
				\end{tikzpicture} && && \\
				&& && $K_2+2L$ &&  &&  \\
				&&  && $(2,6)$-initial &&  &&\\
			\end{tabular}
		\end{center}
		
		All of the graphs $K_2,K_3,K_4,K_6,K_2+L,K_3+L,K_4+L$ and $K_2+2L$ have $6$ vertices or fewer, so by Lemma \ref{Lemma: Cm in Cn}, their diagrams all lie in $\mathcal{C}_6$ as required.\\
	\end{ex}
	
	In the above example, the graphs associated to the top row of $\II(\Cn)$ were the complete graphs on $n$ or fewer vertices, and we found graphs for each subsequent row by adding disjoint edges to the graphs we had already found. We can generalise this process to arbitrary values of $n$, and thus prove Theorem \ref{Theorem: dimCn}, as below.
	
	\begin{proof}[Proof of Theorem \ref{Theorem: dimCn}]
		We have already seen that $$\dim \Cn \leq 
		|\II(\Cn)| = \begin{cases}
			r^2  &\text{ if } n=2r \\
			r^2 + r &\text{ if } n=2r+1.\\
		\end{cases}$$
		
		We now show that $\dim\Cn \geq |\II(\Cn)|$ by exhibiting a set of graphs $\{G_{i,d}:(i,d)\in \II(\Cn)\}$ such that for each $(i,d)\in \II(\Cn)$, $G_{i,d}$ is $(i,d)$-initial (with respect to the ordering $\prec$ given in Definition \ref{Definition: Order Dn}) and has $d$ vertices (by Lemma \ref{Lemma: Cm in Cn} this is sufficent to ensure that the diagrams $\beta(G_{i,d})$ lie in $\Cn$, because for each $(i,d)\in \II(\Cn)$ we have $d\leq n$).
		
		We proceed by induction on $n\geq 1$. The set $\II(\calC_1)$ is empty, so for the base case $n=1$ there is nothing to prove.
		
		For the inductive step, suppose that $n>1$ and that we have a set $\{G_{i,d}:(i,d)\in \II(\calC_{n-1})\}$ where each $G_{i,d}$ is an $(i,d)$-initial graph on $d$ vertices. The set $\II(\calC_{n-1})$ is a subset of $\II(\Cn)$, so we can extend our set of graphs to a set $\{G_{i,d}:(i,d)\in \II(\Cn)\}$ by adding graphs $G_{i,d}$ for the values of $(i,d)$ in $\II(\Cn)-\II(\calC_{n-1})$.
		
		By Proposition \ref{Proposition: Betti-Kn}, the complete graph $K_n$ is $(n-2,n)$-initial and has $n$ vertices, so we set $G_{n-2,n}=K_n$. For every other value of $(i,d)$ in $\II(\Cn)-\II(\calC_{n-1})$, the index $(i-1,d-2)$ is in $\II(\calC_{n-1})$, and hence we define $G_{i,d}=G_{i-1,d-2}+L$. This graph has $(d-2)+2=d$ vertices and by Lemma \ref{Lemma: G+L initiality}, we know it must be $(i,d)$-initial. This completes the proof.
	\end{proof}
	
	\section{Dimension of $\Dnh$}\label{Subsection: dimDnh}
	Now that we have found the dimension of our larger cones $\Dn$ and $\Cn$, we turn our attention to the subcones $\Dnh$ and $\Cnh$ generated by diagrams of ideals of height $h$.
	
	We begin with the cone $\Dnh$ generated by diagrams of Stanley-Reisner ideals of height $h$ (or equivalently, all complexes of codimension $h$), and work towards proving Theorem \ref{Theorem: dimDnh}.
	
	\subsection{Upper Bound}\label{Subsection: ub Dnh}
	As before, we begin by searching for an indexing set $\II(\Dnh)$ and a minimal subspace $\WW(\Dnh)$.
	
	Recall that the diagrams $\beta(\Delta)$ that generate this cone come from complexes of codimension $h$, by Lemma \ref{Lemma: height I-Delta}. This allows us to derive the two following results.
	
	\begin{prop}\label{Proposition: Inequalities Dnh 1}
		Let $\Delta$ be a complex on vertex set $[n]$ of codimension $h$, and set $\beta=\beta(\Delta)$. For any integers $i$ and $d$ with $d - i > n-h+1$, we have $\beta_{i,d} = 0$.
	\end{prop}
	\begin{proof}
		The fact that $\codim \Delta=h$ means that $\Delta$, along with all of its induced subcomplexes, has no faces of dimension higher than $n-h-1$. In particular, for any subset $U\subseteq [n]$ and any $j> n-h-1$, we must have $\Hred_j(\Delta_U) = 0$. Thus by Hochster's Formula, we have that $\beta_{i,d}=0$ whenever $d-i-2>n-h-1$. The result follows.
	\end{proof}
	\begin{prop}\label{Proposition: Inequalities Dnh 2}
		Let $\Delta$ be a complex on vertex set $[n]$ of codimension $h$, and set $\beta=\beta(\Delta)$. For any integers $h\leq i \leq n-1$ we have $\beta_{i,i+1}=0$.
	\end{prop}
	\begin{proof}
		Because $\Delta$ has codimension $h$ it must have a facet of size $n-h$. Thus it can have at most $h$ missing vertices, which means that the Betti numbers $\beta_{h,h+1},\dots,\beta_{n-1,n}$ must all be zero by Remark \ref{Remark: Missing Vertices Row 0 of Betti Table}.
	\end{proof}
	
	These inequalities give us a much clearer picture of what the diagrams in $\Dnh$ look like.
	\begin{equation}\label{Equation: Dnh Matrix}
		\begin{bmatrix}
			\beta_{0,1} & \dots & \beta_{h-1,h} &  & & &\\
			\beta_{0,2} & \dots & \beta_{h-1,h+1} & \dots & \dots &\dots & \beta_{n-2,n} \\
			\beta_{0,3} & \dots & \beta_{h-1,h+2} & \dots & \dots & \beta_{n-3,n} & \\
			\vdots &\vdots & \vdots & \vdots & \vdots & & &\\
			\beta_{0,n-h} &\dots & \beta_{h-1,n-1} & \beta_{h,n} & & & & \\
			\beta_{0,n-h+1} &\dots & \beta_{h-1,n} & & & & & \\
		\end{bmatrix}
	\end{equation}
	
	We have now found all the restriction we need in order to define our indexing set $\II(\Dnh)$.
	\begin{defn}\label{Definition: II(Dnh)}
		We define $$\II(\Dnh) := \left\{(i,d)\in \II(\Dn): d-i \leq n-h+1 \}-\{(h,h+1),\dots, (n-1,n)\right\}.$$
	\end{defn}
	
	Unlike in the cases of the cones $\Cn$ and $\Dn$, the minimal subspace $\WW(\Dnh)$ containing the cone $\Dnh$ is not simply the space carved out by this indexing set, $\bigoplus_{(i,d)\in \II(\Dnh)}\QQ$. This is because, as well as satisfying the conditions of Propositions \ref{Proposition: Inequalities Dnh 1} and \ref{Proposition: Inequalities Dnh 2}, we know from Remark \ref{Remark: HK Equations} that the diagrams $\beta$ in $\Dnh$ must also satisfy the Herzog-K\"{u}hl equations $\HK_1(\beta)=\dots = \HK_{h-1}(\beta)=0$. Thus the space $\WW(\Dnh)$ is actually equal to the following.
	\begin{defn}\label{Definition: WW(Dnh)}
		We define $$\WW(\Dnh) := \left\{\beta \in \bigoplus_{(i,d)\in \II(\Dnh)}\QQ : \HK_1(\beta)=\dots =\HK_{h-1}(\beta)=0\right\}.$$
	\end{defn}
	
	To prove that the formula given in Theorem \ref{Theorem: dimDnh} is an upper bound for $\dim \Dnh$, we need to show that $\dim \WW(\Dnh)=\frac{n(n-1)}{2}-\frac{h(h-1)}{2}+1$. This is the content of the following proposition.
	\begin{prop}\label{Proposition: dim WW(Dnh)}
		Let $\II(\Dnh)$ and $\WW(\Dnh)$ be as in Definitions \ref{Definition: II(Dnh)} and \ref{Definition: WW(Dnh)}. We have
		\begin{enumerate}
			\item $|\II(\Dnh)|=\frac{n(n-1)}{2}-\frac{h(h-1)}{2}+h$.
			\item $\dim \WW(\Dnh)=\frac{n(n-1)}{2}-\frac{h(h-1)}{2}+1$.
		\end{enumerate}
	\end{prop}
	\begin{proof}
		For part (1), we start by arranging the elements of $\II(\Dnh)$ in rows as in Equation (\ref{Equation: Dnh Matrix}).
		\begin{equation*}\label{Equation: I(Dnh) Matrix}
			\begin{matrix}
				(0,1) & \dots & (h-1,h) &  & & &\\
				(0,2) & \dots & (h-1,h+1) & \dots & \dots &\dots & (n-2,n) \\
				(0,3) & \dots & (h-1,h+2) & \dots & \dots & (n-3,n) & \\
				\vdots &\vdots & \vdots & \vdots & \vdots & & &\\
				(0,n-h) &\dots & (h-1,n-1) & (h,n) & & & & \\
				(0,n-h+1) &\dots & (h-1,n) & & & & & \\
			\end{matrix}
		\end{equation*}
		Labelling the rows of this set as in Definition \ref{Definition: Rows of Index Set} we see that row 0 of $\II(\Dnh)$ has $h$ elements. Also row 1 has $n-1$ elements, row 2 has $n-2$ elements, and so on. In general, for each $1\leq i \leq n-h$, row $i$ of $\II(\Dnh)$ has $n-i$ elements. This means we have
		\begin{align*}
			|\II(\Dnh)|&= h+\sum_{i=1}^{n-h} (n-i)\\
			&= h+\sum_{i=h}^{n-1} i\\
			&= h+ \sum_{i=1}^{n-1} i - \sum_{i=1}^{h-1} i\\
			&= \frac{n(n-1)}{2}-\frac{h(h-1)}{2}+h.
		\end{align*}
		
		For part (2), we need to show that the Herzog-K\"{u}hl equations satisfied by the diagrams in $\Dnh$ are linearly independent. To this end we define, for a diagram $\beta$ in $\Dnh$ and for each $1 \leq d \leq n$, the variable
		\begin{align*}
			t_d = \sum_{i} (-1)^i \beta_{i,d}.
		\end{align*}
		This allows us to express the relations $\HK_1(\beta)= ... = \HK_{h-1}(\beta)=0$ as
		\begin{equation*}
			(t_1,...,t_n)
			\begin{pmatrix}
				1 & \dots & 1\\
				2 & \dots & 2^{h-1}\\
				\vdots & \vdots & \vdots\\
				n & \dots & n^{h-1}
			\end{pmatrix}
			=0\, .
		\end{equation*}
		The matrix of coefficients given above is a Vandermonde matrix with distinct rows, which means in particular that all of its columns are linearly independent.
		
		Thus we have
		\begin{align*}
			\dim (\WW(\Dnh)) &= |\II(\Dnh)| - (h-1)\\
			&= \frac{n(n-1)}{2}-\frac{h(h-1)}{2}+1\, .
		\end{align*}
	\end{proof}
	
	\begin{rem}\label{Remark: dimDnhtilde Upper Bound}
		The above proof can be modified for the cone $\Dnhtilde$, generated by diagrams of complexes with no missing vertices and codimension $h$. As noted in Remark \ref{Remark: dimDntilde Upper Bound}, if $\Delta$ is a complex with no missing vertices, then the Betti numbers $\beta_{0,1}(\Delta),\dots,\beta_{h-1,h}(\Delta)$ are all zero. Thus the indexing set $\II(\Dnhtilde)$ is equal to $\II(\Dnh)-\{(0,1),\dots,(h-1,h)\}$, which has cardinality $|\II(\Dnh)|-h$. Consequently the dimension of the minimal subspace $\WW(\Dnhtilde)$ containing $\Dnhtilde$ is equal to $\dim\WW(\Dnh)-h$, which is $\frac{n(n-1)}{2}-\frac{h(h+1)}{2}+1$.
	\end{rem}
	
	\subsection{Lower Bound, $h=1$ Case}\label{Subsection: lb Dnh}
	Once again, we now need to show that the space $\WW(\Dnh)$ is the minimal subspace of $\WW(\Dn)$ containing $\Dnh$, by following the last two steps of Method \ref{Method: Dimension of Cones}. Thus we search for a linearly independent set of diagrams in $\Dnh$ of size $\frac{n(n-1)}{2}-\frac{h(h-1)}{2}+1$. 
	
	For reasons that will become apparent, we will treat the cases $h=1$ and $h>1$ separately. The majority of this section is devoted solely to the $h=1$ case. However we begin with two lemmas that hold for \textit{any} value of $h$, starting with the following analogue of Lemma \ref{Lemma: Dm in Dn}, which allows us to broaden our search from complexes with exactly $n$ vertices to complexes with at most $n$ vertices, just as we did with our larger cone $\Dn$.
	\begin{lem}\label{Lemma: Dmh in Dnh}
		For any positive integer $m<n$, we have $\Dnh[m]\subset \Dnh$.
	\end{lem}
	\begin{proof}
		If $\Delta$ is a complex on vertex set $[m]$, then by Lemma \ref{Lemma: Betti C-Delta} the cone $C^{n-m}(\Delta)$ is a complex on vertex set $[n]$ with the same Betti diagram as $\Delta$. Moreover we have
		\begin{align*}
			\codim C^{n-m}(\Delta) &= n - \dim C^{n-m}(\Delta) - 1\\
			&= n - (n-m + \dim \Delta) - 1\\
			&= m - \dim \Delta - 1\\
			&= \codim \Delta\, .
		\end{align*}
		Thus the diagram $\beta(\Delta)$ lies inside $\Dnh$.
	\end{proof}
	
	An identical proof shows that for an integer $m<n$ the cone $\widetilde{\calD}_m^h$ lies inside the cone $\Dnhtilde$. In fact the following lemma allows us to restrict our attention solely to the cone $\Dnhtilde$.
	\begin{lem}\label{Lemma: Dimension of Dnhtilde => Dimension of Dnh}
		Let $\Dnhtilde$ denote the cone generated by diagrams of complexes on $n$ vertices, with codimension $h$ and no missing vertices. Let $\calBtilde$ be a linearly independent set of diagrams in the cone $\Dnhtilde$. We may extend $\calBtilde$ to a linearly independent set of diagrams in the cone $\Dnh$ of size $|\calB|=|\calBtilde|+h$.
	\end{lem}
	\begin{proof}
		We fix an arbitrary ordering $\prec$ on $\II(\Dnhtilde)$ and extend it to an odering on $\II(\Dnh)= \II(\Dnhtilde)\cup \{(0,1),\dots,(h-1,h)\}$ by stipulating that for any index $(i,d)\in \II(\Dnhtilde)$ we have $(i,d)\prec (0,1) \prec \dots \prec (h-1,h)$.
		
		It suffices to find an $(i,d)_\prec$-initial diagram in $\Dnh$ for each index $(i,d)$ in $\{(0,1),\dots,(h-1,h)\}$. To this end, we fix an integer $1\leq d \leq h$, and consider the diagram $\beta$ corresponding to the complex $\Delta^{n-h-1}+\Skel_0([d-h])$ on a vertex set of size $n$. This complex has codimension $h$. It also has $d$ missing vertices, so by Remark \ref{Remark: Missing Vertices Row 0 of Betti Table} the Betti numbers $\beta_{0,1},\dots,\beta_{d-1,d}$ are all nonzero, while the Betti numbers $\beta_{d,d+1},\dots,\beta_{h-1,h}$ are all zero. Thus $\beta$ is a $(d-1,d)_\prec$-initial diagram lying in the cone $\Dnh$, and the result follows.
	\end{proof}
	
	In particular Lemma \ref{Lemma: Dimension of Dnhtilde => Dimension of Dnh} allows us to extend any basis for the vector space $\WW(\Dnhtilde)$ to a basis for $\WW(\Dnh)$. Thus, to find an appropriately sized linearly independent set of diagrams in $\Dnh$, it suffices to find a linearly independent  set of diagrams in the \textit{subcone} $\Dnhtilde$ which spans $\WW(\Dnhtilde)$.
	
	We now proceed to finding these diagrams in the case where $h=1$. In this case, we define our ordering on the indexing set $\II(\Dnhtilde[n][1])$ to be the restriction of the ordering $\prec$ on $\II(\Dn)$ given in Definition \ref{Definition: Order Dn}.
	
	Note that the indexing set $\II(\Dnhtilde[n][1])$ is the same as the indexing set $\II(\Dntilde)$, which means it contains the following indices.
	\begin{equation*}\label{Equation: I(Dn1tilde) Matrix}
		\begin{matrix}
			(0,2) & (1,3) & \dots & \dots & (n-2,n)&\\
			(0,3)& (1,4) & \dots & (n-3,n) & &\\
			\vdots& \vdots & \vdots & & &\\
			(0,n-1)& (1,n) & & & & \\
			(0,n)& & & & &
		\end{matrix}
	\end{equation*}
	Note also that for any $h$ we have $\dim \WW(\Dnhtilde) = |\II(\Dnhtilde)|-(h-1)$, and so in particular we have that $\dim \WW(\Dnhtilde[n][1]) = |\II(\Dnhtilde[n][1])|$. Thus we are looking for  $(i,d)_\prec$-initial complexes corresponding to \textit{every} index $(i,d)$ in $\II(\Dnhtilde[n][1])$ (in the general case we will be searching for complexes corresponding to all but $h-1$ of the indices in $\II(\Dnhtilde)$).
	
	To build these complexes we use the following elegant construction from \cite{Klivans} page 3.
	\begin{defn}\label{Definition: Starring Operation}
		Let $r\geq -1$ be an integer, and let $\Delta$ be a simplicial complex on vertex set $V$. For an additional vertex $v$ outside of $V$, we define the \textit{starred complex} $$\Delta \Star_r v := \Delta \cup \{\sigma \cup \{v\} : \sigma \in \Delta, |\sigma| = r \}$$
		on vertex set $V\cup \{v\}$.
	\end{defn}
	\begin{rem}\label{Remark: Starring Operation}
		The operation $\huge{\textunderscore}\Star_r v$ is a generalisation of the coning operation. In particular, if $\dim \Delta = d$, then $\Delta\Star_{d+1} v$ is equal to the join $\Delta \ast v$, which is just the cone $C(\Delta)$. We list some other specific cases below for clarity.
		\begin{enumerate}
			\item The complex $\Delta \Star_{-1} v$ is obtained from $\Delta$ by adding a single missing vertex $v$ to $\Delta$ (i.e. the vertex $v$ is not a face of $\Delta \Star_{-1} v$ but it is an element of the vertex set of $\Delta \Star_{-1} v$).
			\item The complex $\Delta \Star_0 v$ is obtained by adding a single isolated vertex $v$ to $\Delta$ (i.e. we have $\Delta \Star_0 v = \Delta + v$).
			\item The complex $\Delta \Star_1 v$ is obtained by adding the vertex $v$ to $\Delta$, along with every edge $\{x,v\}$ for vertices $x$ in $\Delta$.
		\end{enumerate}
	\end{rem}
	\begin{ex}\label{Example: Starring Operation}
		Suppose $\Delta$ is the boundary of the $2$-simplex (which has dimension $1$). We depict the complexes $\Delta\Star_r v$ below, for $r\in \{-1,0,1,2\}$. As always, we denote missing vertices with a $\times$ symbol.
		\begin{center}
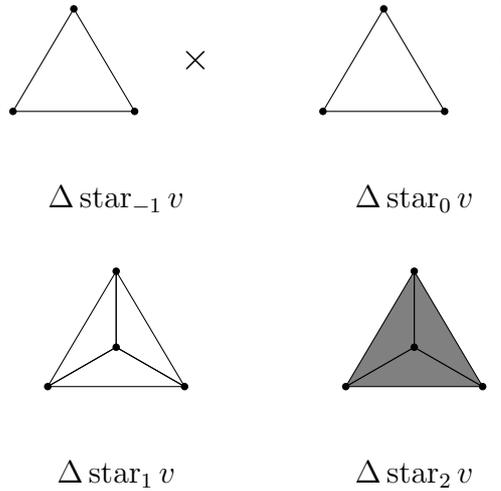
			\begin{tabular}{ c c c c }
				\begin{tikzpicture}[scale=0.8]
					\tikzstyle{point}=[circle,thick,draw=black,fill=black,inner sep=0pt,minimum width=2pt,minimum height=2pt]
					\node (a)[point] at (0,0) {};
					\node (b)[point] at (2,0) {};
					\node (c)[point] at (1,1.7) {};
					\node[scale=1.2] at (3,0.85) {$\times$};
					
					\draw (a.center) -- (b.center) -- (c.center) -- cycle;
				\end{tikzpicture}& 
				&
				& \begin{tikzpicture}[scale=0.8]
					\tikzstyle{point}=[circle,thick,draw=black,fill=black,inner sep=0pt,minimum width=2pt,minimum height=2pt]
					\node (a)[point] at (0,0) {};
					\node (b)[point] at (2,0) {};
					\node (c)[point] at (1,1.7) {};
					\node[point] at (3,0.85) {};
					
					\draw (a.center) -- (b.center) -- (c.center) -- cycle;
				\end{tikzpicture}\\
				&&&\\
				$\Delta \Star_{-1} v$& & & $\Delta \Star_0 v$\\
				&&&\\
				\begin{tikzpicture}[scale=0.9]
					\tikzstyle{point}=[circle,thick,draw=black,fill=black,inner sep=0pt,minimum width=2pt,minimum height=2pt]
					\node (a)[point] at (0,0) {};
					\node (b)[point] at (2,0) {};
					\node (c)[point] at (1,1.7) {};
					\node (d)[point] at (1,0.58) {};
					
					\draw (a.center) -- (b.center) -- (d.center) -- cycle;
					\draw (b.center) -- (c.center) -- (d.center) -- cycle;
					\draw (c.center) -- (a.center) -- (d.center) -- cycle;
				\end{tikzpicture}&
				&
				& \begin{tikzpicture}[scale=0.9]
					\tikzstyle{point}=[circle,thick,draw=black,fill=black,inner sep=0pt,minimum width=2pt,minimum height=2pt]
					\node (a)[point] at (0,0) {};
					\node (b)[point] at (2,0) {};
					\node (c)[point] at (1,1.7) {};
					\node (d)[point] at (1,0.58) {};
					
					\begin{scope}[on background layer]
						\draw[fill=gray] (a.center) -- (b.center) -- (d.center) -- cycle;
						\draw[fill=gray] (b.center) -- (c.center) -- (d.center) -- cycle;
						\draw[fill=gray] (c.center) -- (a.center) -- (d.center) -- cycle;
					\end{scope}
				\end{tikzpicture}\\
				&&&\\
				$\Delta \Star_1 v$& & & $\Delta \Star_2 v$
			\end{tabular}
		\end{center}
	\end{ex}
	
	To work out how the starring operation affects Betti diagrams, we must first look at how it affects homology. The following lemma allows us to compute the homology of $\Delta \Star_r v$ in particular cases, which will be sufficient for our needs.
	\begin{lem}\label{Lemma: Starring and Homology}
		Let $\Delta$ be a simplicial complex on vertex set $V$ and define $\Gamma = \Delta \Star_r v$ for some integer $r\geq 0$.
		\begin{enumerate}
			\item For any $j< r$ we have $\Hred_j(\Gamma) = 0$.
			\item For any $j> r$ we have $\Hred_j(\Gamma) = \Hred_j(\Delta)$.
			\item If $\Delta$ is acyclic we have $\Hred_r(\Gamma) = \Hred_{r-1}(\Skel_{r-1}(\Delta))$.
		\end{enumerate}
	\end{lem}
	\begin{proof}
		For part (1), fix some $j<r$. By Lemma \ref{Lemma: Hj Delta = Hj Skel(Delta)} we know that the $\nth{j}$ homology of $\Gamma$ is dependent only on the $r$-skeleton of $\Gamma$, which is the same as the $r$-skeleton of the cone $C\Delta$. Thus we have \begin{align*}
			\Hred_j(\Gamma) &= \Hred_j(\Skel_r(\Gamma))\\
			&= \Hred_j(\Skel_r(C\Delta))\\
			&= \Hred_j(C\Delta)\\
			&= 0.
		\end{align*}
		
		Now we move on to parts (2) and (3). We begin by decomposing $\Gamma$ into the union $\Delta \cup A$ where $A$ is the subcomplex of $\Gamma$ generated by all faces containing $v$. This latter subcomplex of $\Gamma$ is a cone over $v$ and is therefore acyclic. Moreover, the intersection $\Delta \cap A$ consists of all the faces $\sigma$ of $\Delta$ whose union with $v$ is a face of $\Gamma$. By the definition of the $\huge{\textunderscore}\Star_r v$ operation, $\sigma \cup \{v\}$ is contained in $\Gamma$ if and only if $\dim \sigma \leq r-1$. Thus we have $\Delta \cap A = \Skel_{r-1}(\Delta)$.
		
		The Mayer-Vietoris sequence gives us an exact sequence \begin{equation}\label{Equation: MVS}
			\rightarrow \Hred_j(\Delta) \rightarrow \Hred_j(\Gamma) \rightarrow \Hred_{j-1}(\Skel_{r-1}(\Delta)) \rightarrow \Hred_{j-1}(\Delta) \rightarrow \Hred_{j-1}(\Gamma) \rightarrow \dots
		\end{equation}
		
		For $j\geq r$, we have $\Hred_j(\Skel_{r-1}(\Delta))=0$. This means that for each $j>r$ we have the exact sequence $$0\rightarrow \Hred_j(\Delta) \rightarrow \Hred_j(\Gamma) \rightarrow 0$$ which proves part (2). The Mayer-Vietoris sequence in Equation (\ref{Equation: MVS}) also includes the maps $$\dots\rightarrow \Hred_r(\Delta) \rightarrow \Hred_r(\Gamma) \rightarrow \Hred_{r-1}(\Skel_{r-1}(\Delta)) \rightarrow \Hred_{r-1}(\Delta) \rightarrow \dots$$ so if $\Delta$ is acyclic this tells us that $\Hred_r(\Gamma)$ is isomorphic to $\Hred_{r-1}(\Skel_{r-1}(\Delta))$, proving part (3).
	\end{proof}
	
	Using this result we can find the homology, and hence the Betti diagrams, of the starred complex $\Delta^m \Star_r v$ for integers $0\leq r \leq m \leq n-2$ (where $\Delta^m$ denotes the $m$-simplex). These complexes will give us our $(i,d)_\prec$-initial diagrams for the cone $\Dnhtilde[n][1]$. 
	\begin{cor}\label{Corollary: Homology of Starred Simplex}
		Let $m$ and $r$ be integers with $0\leq r \leq m \leq n-2$, and define $\Gamma=\Delta^m \Star_r v$. For any integer $j\geq -1$ we have $$\dim_{\KK} \Hred_j(\Gamma)=\begin{cases}
			{m \choose r}& \text{ if } j=r\\
			0& \text{ otherwise.}
		\end{cases}$$
	\end{cor}
	\begin{proof}
		The complex $\Delta^m$ is acyclic, so by Lemma \ref{Lemma: Starring and Homology} we have that  $\Hred_r(\Gamma)=\Hred_{r-1}(\Skel_{r-1}(\Delta^m))$, and all other homogies are zero. Note that $\Skel_{r-1}(\Delta^m)$ can be rewritten as $\Skel_{r-1}([m+1])$, which has $\nth[st]{(r-1)}$ homology of dimension ${m \choose r}$ by Lemma \ref{Lemma: Homology of Skeleton Complexes}.
	\end{proof}
	
	\begin{cor}\label{Corollary: Betti Diagram of Starred Simplex}
		Let $m$ and $r$ be integers with $0\leq r \leq m \leq n-2$, and define $\Gamma=\Delta^m \Star_r v$. For any index $(i,d)\in \II(\Dnhtilde[n][1])$ we have
		\begin{enumerate}
			\item $\beta_{i,d}(\Gamma)=\begin{cases}
				{m + 1 \choose d-1}{d-2 \choose r}& \text{ if } 0\leq i \leq m-r \text{ and } d=i+r+2\\
				0& \text{ otherwise.}
			\end{cases}$
			
			In particular the diagram $\beta(\Gamma)$ has the following shape.
			\begin{equation*}
				\begin{bmatrix}
					\beta_{0,r+2} & \dots & \beta_{m-r,m+2}\\
				\end{bmatrix}
			\end{equation*}
			\item The Betti diagram of the complex $\Delta^{d-2}\Star_{d-i-2} v$ is an $(i,d)_\prec$-initial diagram lying inside $\Dnhtilde[n][1]$.
		\end{enumerate}
	\end{cor}
	\begin{proof}
		Fix some index $(i,d)\in \II(\Dnhtilde[n][1])$, and pick some vertex set $V$ for $\Delta^m$ of size $m+1$. Hochster's Formula gives us that $$\beta_{i,d}(\Gamma)=\sum_{U\in {V\cup \{v\} \choose d}} \dim_{\KK}\Hred_{d-i-2}(\Gamma_U).$$ In particular for any subset $U\subset V$ of size $d$ we have $\Gamma_U=\Delta^m|_U$, which is the full simplex on the vertices of $U$, and is hence acyclic.
		
		Thus the only subsets of $V\cup \{v\}$ which contribute to the Betti numbers of $\Gamma$ are the ones that contain $v$, so we have $$\beta_{i,d}(\Gamma)=\sum_{U\in {V \choose d-1}} \dim_{\KK}\Hred_{d-i-2}(\Gamma_{U\cup \{v\}}).$$ Note that for any subset $U\subseteq V$ of size $d-1$ we have $\Gamma_{U\cup \{v\}}=\Delta^m|_U \Star_r v$, and $\Delta^m|_U$ is isomorphic to the $(d-2)$-simplex $\Delta^{d-2}$.
		
		Hence by Corollary \ref{Corollary: Homology of Starred Simplex}, the induced subcomplex $\Gamma_U$ has homology only at degree $r$ and this homology has dimension ${d-2 \choose r}$. The number of subsets of $U\subseteq V$ of size $d-1$ is equal to ${m+1 \choose d-1}$. This proves claim (1).
		
		If our integers $m$ and $r$ satisfy $(m-r,m+2)=(i,d)$, then the Betti diagram of $\Gamma$ is $(i,d)_\prec$-initial. The values $m=d-2$ and $r=d-i-2$ satisfy this condition. Moreover the complex $\Delta^{d-2}\Star_{d-i-2} v$ has codimension $1$ and $d$ vertices (so in particular it has no more than $n$ vertices). Thus its Betti diagram lies inside $\Dnh[n][1]$. This proves claim (2).
	\end{proof}
	
	\begin{ex}\label{Example: (i,d)-initial Diagrams for Dn1tilde}
		In the case where $n=4$, we have the following $(i,d)_\prec$-initial diagrams in the cone $\Dnhtilde[4][1]$.
		\begin{center}
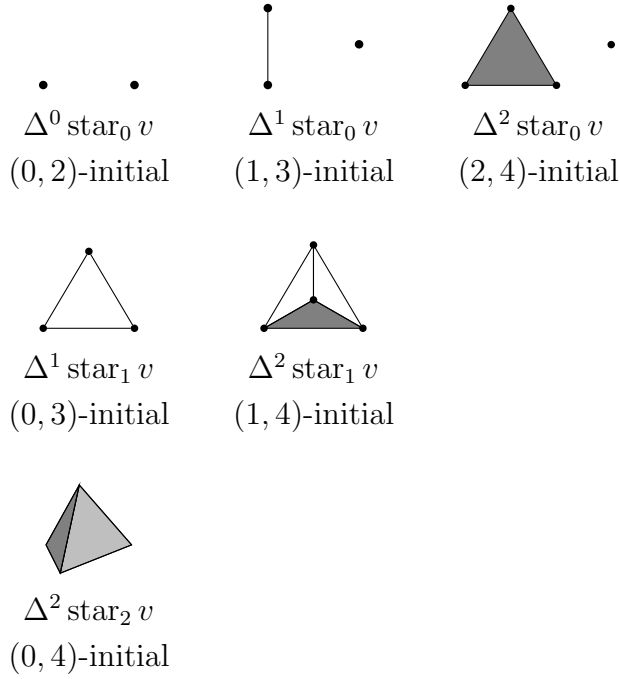
			\begin{tabular}{ c c c c c }
				\begin{tikzpicture}[scale=0.6]
					\tikzstyle{point}=[circle,thick,draw=black,fill=black,inner sep=0pt,minimum width=2pt,minimum height=2pt]
					\node[point, scale=1.2] at (0,0) {};
					\node[point, scale=1.2] at (2,0) {};
				\end{tikzpicture}
				& & \begin{tikzpicture}[scale = 0.6]
					\tikzstyle{point}=[circle,thick,draw=black,fill=black,inner sep=0pt,minimum width=2pt,minimum height=2pt]
					\node (a)[point, scale=1.2] at (1,0) {};
					\node (b)[point, scale=1.2] at (1,1.7) {};
					\node (e)[point, scale=1.2] at (3,0.9) {};
					
					\draw (a.center) -- (b.center);
				\end{tikzpicture} & & \begin{tikzpicture}[scale = 0.6]
					\tikzstyle{point}=[circle,thick,draw=black,fill=black,inner sep=0pt,minimum width=2pt,minimum height=2pt]
					\node (a)[point] at (0,0) {};
					\node (b)[point] at (2,0) {};
					\node (c)[point] at (1,1.7) {};
					\node (d)[point] at (3.2,0.9) {};
					
					\begin{scope}[on background layer]
						\draw[fill=gray] (a.center) -- (b.center) -- (c.center) -- cycle;
					\end{scope}
				\end{tikzpicture} \\
				$\Delta^0 \Star_0 v$& & $\Delta^1 \Star_0 v$ & & $\Delta^2 \Star_0 v$  \\
				$(0,2)$-initial& & $(1,3)$-initial & & $(2,4)$-initial\\
				
				& & & & \\
				
				\begin{tikzpicture}[scale=0.6]
					\tikzstyle{point}=[circle,thick,draw=black,fill=black,inner sep=0pt,minimum width=2pt,minimum height=2pt]
					\node[point] (1) at (0,0) {};
					\node[point] (2) at (2,0) {};
					\node[point] (3) at (1,1.7) {};
					
					\draw (1) -- (2) -- (3) -- (1);
				\end{tikzpicture} & & \begin{tikzpicture}[scale=0.65]
					\tikzstyle{point}=[circle,thick,draw=black,fill=black,inner sep=0pt,minimum width=2pt,minimum height=2pt]
					\node (a)[point] at (0,0) {};
					\node (b)[point] at (2,0) {};
					\node (c)[point] at (1,1.7) {};
					\node (d)[point] at (1,0.58) {};
					
					\begin{scope}[on background layer]
						\draw[fill=gray] (a.center) -- (b.center) -- (d.center) -- cycle;
						\draw (b.center) -- (c.center) -- (d.center) -- cycle;
						\draw (c.center) -- (a.center) -- (d.center) -- cycle;
					\end{scope}
				\end{tikzpicture} & &  \\
				$\Delta^1 \Star_1 v$& & $\Delta^2 \Star_1 v$ & & \\
				$(0,3)$-initial& & $(1,4)$-initial & &\\
				
				& & & &\\
				
				\begin{tikzpicture}[scale=0.65][line join = round, line cap = round]
					
					\coordinate (4) at (0,{sqrt(2)},0);
					\coordinate (3) at ({-.5*sqrt(3)},0,-.5);
					\coordinate (2) at (0,0,1);
					\coordinate (1) at ({.5*sqrt(3)},0,-.5);
					
					\begin{scope}
						\draw (1)--(3);
						\draw[fill=lightgray,fill opacity=.5] (2)--(1)--(4)--cycle;
						\draw[fill=gray,fill opacity=.5] (3)--(2)--(4)--cycle;
						\draw (2)--(1);
						\draw (2)--(3);
						\draw (3)--(4);
						\draw (2)--(4);
						\draw (1)--(4);
					\end{scope}
				\end{tikzpicture}& & & & \\
				$\Delta^2 \Star_2 v$& & & & \\
				$(0,4)$-initial& & & &
			\end{tabular}
		\end{center}
	\end{ex}
	
	We have now found $(i,d)_\prec$-initial diagrams for every index $(i,d)$ in $\II(\Dnhtilde[n][1])$, which is sufficient to prove Theorem \ref{Theorem: dimDnh} in the case where $h=1$. We proceed to the general case.
	
	\subsection{Lower Bound, $h>1$ Case}\label{Subsection: lb Dnh h>1}
	For this section we assume that $h>1$. We wish to find an ordering on the index set $\II(\Dnhtilde)$, and a set of $(i,d)$-initial diagrams for all but $h-1$ of the indices $(i,d)\in \II(\Dnhtilde)$ with respect to this ordering.
	
	However, in this case we will no longer be able to use the restriction of the ordering $\prec$ we used for $\II(\Dn)$. Instead we define the following refined ordering $\prec_h$ on $\II(\Dnhtilde)$. 
	\begin{defn}\label{Definition: Order Dnhtilde}
		For any two pairs $(i,d)$ and $(i',d')$ in $\II(\Dnhtilde)$ we write $(i,d)\prec_h (i',d')$ $(i,d)\prec_h (i',d')$ if and only if any one of the following conditions hold.
		\begin{enumerate}
			\item $d-i<d'-i'$.
			\item $d-i=d'-i'=2$ and $i<i'$
			\item $d-i=d'-i'>2$, $i'\geq h$, and $i<i'$.
			\item $d-i=d'-i'>2$, $i,i'<h$ and $i>i'$.
		\end{enumerate}
	\end{defn}
	\begin{rem}\label{Remark: prec1 = prec}
		\begin{enumerate}
			\item The ordering $\prec_h$ is the same as the ordering $\prec$ except that for all rows after row 1 we \textit{reverse} the ordering of the first $h$ elements (i.e. those for which $0\leq i \leq h-1$).
			
			For example, the indexing set $\II(\Dnhtilde[6][3])$ contains the elements $$\begin{matrix}
				(0,2)& (1,3) & (2,4) & (3,5) & (4,6)\\
				(0,3)& (1,4) & (2,5) & (3,6)\\
				(0,4)& (1,5) & (2,6) & \\
			\end{matrix}$$
			and the ordering $\prec_3$ agrees with the ordering $\prec$ for row 1, so we have $$(0,2)\prec_3 (1,3)\prec_3 (2,4) \prec_3 (3,5) \prec_3 (4,6).$$ Then for rows 2 and 3 we reverse the ordering of the first three elements, so we get $$(2,5) \prec_3 (1,4) \prec_3 (0,3) \prec_3 (3,6)$$ and $$(2,6) \prec_3 (1,5) \prec_3 (0,4).$$
			
			\item We also extend this ordering to the $h=1$ case. Note that in this case the ordering $\prec_1$ is in fact the same as the ordering $\prec$. Crucially, this means that the complexes we found in the last section are $(i,d)$-initial with respect to the ordering $\prec_1$.
		\end{enumerate}
	\end{rem}
	
	The ordering $\prec_h$ may initially seem a little counter-intuitive, but it is actually well-suited to our purposes. The main reason it is useful to us comes from a proposition later on in this section (Proposition \ref{Proposition: Betti adding Empty}) which details how we can adjust the shape of the Betti diagram $\beta(\Delta)$ by adding a number of isolated vertices to $\Delta$. This allows us to increase the codimension of $\Delta$ while only adding nonzero entries in the Betti diagram to the \textit{right} of pre-existing nonzero entries.
	
	For example suppose $\Delta$ is the codimension $2$ complex 
	$$
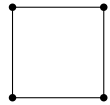
\begin{tikzpicture}[scale=0.6]
		\tikzstyle{point}=[circle,thick,draw=black,fill=black,inner sep=0pt,minimum width=2pt,minimum height=2pt]
		\node (a)[point] at (0,0) {};
		\node (b)[point] at (0,2) {};
		\node (c)[point] at (2,2) {};
		\node (d)[point] at (2,0) {};
		
		\draw (a.center) -- (b.center) -- (c.center) -- (d.center) -- cycle;
	\end{tikzpicture}$$ which has the following Betti diagram. $$\begin{array}{c | cc}
		& 0 & 1 \\
		\hline
		2& 2 & .\\
		3& . & 1
	\end{array}$$ The only nonzero entry in row 2 of this diagram is $\beta_{1,4}$, and there are no nonzero entries in lower rows. This means the diagram is $(1,4)$-initial with respect to every ordering we have considered so far.
	
	Now suppose we want to use $\Delta$ to construct a new $(1,4)$-initial complex of codimension $3$. We can obtain a complex of codimension $3$ by adding a single isolated vertex, giving us the following. $$
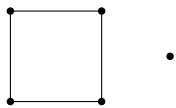
\begin{tikzpicture}[scale=0.6]
		\tikzstyle{point}=[circle,thick,draw=black,fill=black,inner sep=0pt,minimum width=2pt,minimum height=2pt]
		\node (a)[point] at (0,0) {};
		\node (b)[point] at (0,2) {};
		\node (c)[point] at (2,2) {};
		\node (d)[point] at (2,0) {};
		\node (e)[point] at (3.5,1) {};
		
		\draw (a.center) -- (b.center) -- (c.center) -- (d.center) -- cycle;
	\end{tikzpicture}$$ This has the Betti diagram $$\begin{array}{c | cccc}
		& 0 & 1 & 2 & 3 \\
		\hline
		2& 6 & 8 & 4 & 1\\
		3& . & 1 & 1 & .
	\end{array}$$ which means it is $(1,4)$-initial with respect to $\prec_3$ but \textit{not} with respect to our original ordering $\prec$.
	
	Proposition \ref{Proposition: Betti adding Empty} will make this idea precise. Before we present it, however, we first note the following two results. The latter result is an analogue of Lemma \ref{Lemma: SDelta initiality} for the ordering $\prec_h$.
	\begin{lem}\label{Lemma: Ordering Dnh Technical Result}
		Let $\prec_h$ be as in Definition \ref{Definition: Order Dnhtilde} and let $(i,d)$ be an index in $\II(\Dnhtilde)$ such that
		\begin{enumerate}
			\item $d-i\geq 3$.
			\item $i \neq 0$.
			\item $(i,d)\notin \{(1,4),(2,5),\dots, (h-1,h+2)\}$.
		\end{enumerate}
		For any other index $(i',d')$ in $\II(\Dnhtilde)$ such that $i'\neq 0$ and $(i,d)\prec_h (i',d')$, we have $(i-1,d-2)\prec_{h-1}(i'-1,d'-2)$
	\end{lem}
	\begin{proof}
		By the first condition we cannot have $d-i=d'-i'=2$, so this leaves us with three possible cases. 
		\begin{itemize}
			\item Suppose first that $d-i<d'-i'$. This means we have $(d-2)-(i-1)<(d'-2)-(i'-1)$, and hence $(i-1,d-2)\prec_{h-1} (i'-1,d'-2)$.
			\item Next, suppose we have $d-i=d'-i'\geq 3$, $i'\geq h$ and $i<i'$. In this case we have $(d-2)-(i-1)=(d'-2)-(i'-1)\geq 2$, $i'-1\geq h-1$ and $i-1 < i'-1$. Regardless of the precise value of $(d-2)-(i-1)$, we have $(i-1,d-2)\prec_{h-1} (i'-1,d'-2)$.
			\item Finally, suppose $d-i=d'-i'\geq 3$, $i,i'< h$ and $i>i'$. Because $(i,d)$ is not contained in the set $\{(1,4),\dots, (h-1,h+2)\}$, and we know $i<h$, we must in fact have $d-i\geq 4$. This gives us that $(d-2)-(i-1)=(d'-2)-(i'-1)\geq 3$, $i'-1\geq h-1$ and $i-1 < i'-1$, and therefore $(i-1,d-2)\prec_{h-1} (i'-1,d'-2)$ as desired.
		\end{itemize}
	\end{proof}
	
	\begin{lem}\label{Lemma: SDelta initiality Dnh}
		Let $\prec_h$ be as in Definition \ref{Definition: Order Dnhtilde} and let $(i,d)$ be an index in $\II(\Dnhtilde)$ satisfying the three conditions of Lemma \ref{Lemma: Ordering Dnh Technical Result}. If $\Delta$ is an $(i-1,d-2)_{\prec_{h-1}}$-initial complex of codimension $h-1$ then the complex $S\Delta$ is $(i,d)_{\prec_h}$-initial and has codimension $h$.
	\end{lem}
	\begin{proof}
		Taking the suspension of a complex increases its dimension by 1 but its number of vertices by 2, and hence it increases the codimension by 1. So $S\Delta$ must have codimension $h$.
		
		By the assumptions on $(i,d)$ we know $(i,d)\neq (0,2)$, and hence by Lemma \ref{Lemma: Betti S-Delta} we have $\beta_{i,d}(S\Delta)=\beta_{i-1,d-2}(\Delta)+\beta_{i,d}(\Delta)$. By the $(i-1,d-2)_{\prec_{h-1}}$-initality of $\Delta$, we have $\beta_{i-1,d-2}(\Delta)\neq 0$ and hence $\beta_{i,d}(\Delta)$ must be nonzero as well.
		
		Now fix some index $(i',d')$ in $\II(\Dnhtilde)$ with $(i,d)\prec_h (i',d')$. This condition ensures that $(i',d')\neq (0,2)$ and hence Lemma \ref{Lemma: Betti S-Delta} tells us that $\beta_{i',d'}(S\Delta)=\beta_{i'-1,d'-2}(\Delta)+\beta_{i',d'}(\Delta)$. By Lemma \ref{Lemma: Ordering Dnh Technical Result} we have $(i-1,d-2)\prec_{h-1} (i'-1,d'-2)\prec_{h-1}(i',d')$ (unless $i'$ is equal to zero, in which case $\beta_{i'-1,d'-2}(\Delta)$ is also equal to zero by definition). Thus by the $(i,d)_{\prec_{h-1}}$-initiality of $\Delta$, the Betti numbers $\beta_{i'-1,d'-2}(\Delta)$ and $\beta_{i',d'}(\Delta)$ are both zero, so their sum must also be zero.
	\end{proof}
	
	\begin{rem}
		A crucial aspect of Lemma \ref{Lemma: SDelta initiality Dnh} (which is perhaps easy to overlook on a first reading) is that the initiality properties of the complexes $\Delta$ and $S\Delta$ are with respect to two \textit{different} orderings. The complex $\Delta$ is $(i-1,d-2)$-initial with respect to $\prec_{h-1}$, while its suspension $S\Delta$ is $(i,d)$-initial with respect to $\prec_h$.
	\end{rem}

	We can use Lemma \ref{Lemma: SDelta initiality Dnh} to construct a linearly independent set of diagrams in $\Dnhtilde$ recursively from linearly independent diagrams in $\Dnhtilde[n][h-1]$, as long as we can also find $(i,d)_{\prec_h}$-initial diagrams for indices $(i,d)$ which do \textit{not} satisfy the three conditions of Lemma \ref{Lemma: Ordering Dnh Technical Result}.
	
	Before we demonstrate how to do this, recall that we do not need to find $(i,d)_{\prec_h}$-initial diagrams for \textit{every} index in $\II(\Dnhtilde)$. We need to find diagrams for all but $h-1$ of them (to show that $\dim \Dnhtilde = |\II(\Dnhtilde)|-(h-1)$).
	
	In fact, due to the way we have constructed the ordering $\prec_h$, the cone $\Dnhtilde$ contains no $(i,d)_{\prec_h}$-initial diagrams for the indices $(i,d)\in \{(0,2),(1,3),\dots,(h-1,h+1)\}$. This is because every diagram $\beta$ in the cone must have projective dimension of at least $h$ by the Auslander-Buchsbaum Formula (Corollary \ref{Corollary: Auslander-Buchsbaum for R}), and hence must have $\beta_{i,d}\neq 0$ for some $(i,d)\in \II(\Dnhtilde)$ with $(0,2) \prec_h (1,3)\prec_h\dots \prec_h (h-1,h+1) \prec_h (i,d)$. For this reason we search for $(i,d)_{\prec_h}$-initial diagrams for indices $(i,d)$ in $\II(\Dnhtilde)-\{(0,2),(1,3),\dots,(h-1,h+1)\}$.
	
	So the indices $(i,d)$ in $\II(\Dnhtilde)$ for which we still need to find $(i,d)_{\prec_h}$-initial diagrams are
	\begin{enumerate}
		\item The indices $(d-2,d)$ for $h+1 \leq d \leq n$, in row 1.
		\item The indices $(0,d)$ for $3\leq d \leq n-h+1$, in column 0.
		\item The indices $(d-3,d)$ for $4 \leq d \leq h+2$, in row 2.
	\end{enumerate}
	
	We now introduce Proposition \ref{Proposition: Betti adding Empty}, which we can use to construct diagrams for all three of the above cases.
	\begin{prop}\label{Proposition: Betti adding Empty}
		Let $m$ and $l$ be positive integers. Suppose $\Delta$ is a complex on a vertex set of size $m$, and let $\beta =\beta(I_\Delta)$. We define $\tDelta = \Delta+\Skel_0([l])$ and  $\tbeta=\beta(I_{\tDelta})$.
		\begin{enumerate}
			\item for every $0\leq i \leq m + l -2$, we have $\tbeta_{i,i+2}\neq 0$.
			\item for every $(i,d)$ in $\II(\Dn)$ with $d-i\geq 3$, we have $\tbeta_{i,d}=\sum_{j=0}^{i} {l \choose i - j} \beta_{j,j+d-i}$.
		\end{enumerate}
	\end{prop}
	\begin{rem}\label{Remark: Skel_0(l) is E_l}
		Note that $\Skel_0([l])$ is the complex consisting of $l$ isolated vertices. In other words, it contains the same data as the empty graph $E_l$. We use the notation $\Skel_0([l])$ to indicate that we are viewing the structure as a complex rather than a graph, and hence the corresponding ideal $I_{\tDelta}$ is obtained using the Stanley-Reisner construction rather than the edge ideal construction (i.e. it is obtained from $I_\Delta$ by adding generators corresponding to the \textit{nonfaces} of $\Skel_0([l])$).
	\end{rem}
	\begin{proof}
		We label the vertex set of $\Delta$ as $V$ and the vertex set of $\tDelta=\Delta+\Skel_0([l])$ as $\tV=V\sqcup [l]$.
		
		For part (1), we note that for any subset $U\subseteq \tV$ that contains both a vertex in $\Delta$ and a vertex in $\Skel_0([l])$, the complex $\tDelta_U$ must be disconnected. By Hochster's Formula, we get that the Betti numbers $\widetilde{\beta}_{0,2},...,\widetilde{\beta}_{m+l-2,m+l}$ are all nonzero.
		
		For part (2), suppose $(i,d)\in \II(\Dn)$ with $d-i\geq 3$. The addition of isolated vertices to $\Delta$ has no effect on homologies of degree greater than zero. This means that for a subset $U\subseteq \tV$, the only part of $U$ that contributes to the $(d-i-2)^\text{nd}$ homology of $\tDelta_U$ is $U\cap V$, and so we have $\Hred_{d-i-2}(\tDelta_U) = \Hred_{d-i-2}(\Delta_{U\cap V})$. This homology is zero if $r=|U\cap V|<d-i$. Meanwhile, for each $d-i\leq r \leq d$ there are precisely ${l \choose d-r}$ ways of extending a subset $S\subseteq V$ of size $r$ to a subset $\widetilde{S}\subseteq \tV$ of size $d$. Thus, by Hochster's Formula, we get
		\begin{align*}
			\tbeta_{i,d} &=\sum_{\substack{U\in {\tV\choose r}}} \dim_\KK \Hred_{d-i-2}(\tDelta_U)\\
			&=\sum_{r=d-i}^{d} \text{      } \sum_{\substack{U\in {\tV\choose d}\\ |U\cap V|=r}} \dim_\KK \Hred_{d-i-2}(\tDelta_U)\\
			&=\sum_{r=d-i}^d \text{      } \sum_{\substack{U\in {V\choose r}}} {l \choose d-r}\dim_\KK \Hred_{d-i-2}(\Delta_U)\\
			&=\sum_{r=d-i}^d \text{      } \sum_{\substack{U \in {V\choose r}}} {l \choose d-r} \dim_\KK \Hred_{r-(r-(d-i))-2}(\Delta_U)\\
			&=\sum_{r=d-i}^d {l \choose d-r} \beta_{r-(d-i),r}\\
			&=\sum_{j=0}^i {l \choose i-j}\beta_{j,j+d-i} \, .
		\end{align*}
	\end{proof}
	
	\begin{cor}\label{Corollary: Betti simplex-adding-empty}
		Suppose $d$ is a positive integer with $h+1 \leq d \leq n$. Let $\Delta^{d-h-1}$ be the $(d-h-1)$-simplex on a vertex set of size $d-h$. For $\tDelta=\Delta^{d-h-1}+\Skel_0([h])$ we have
		\begin{equation*}
			\beta(I_{\tDelta}) =
			\begin{bmatrix}
				\beta_{0,2} & \dots & \beta_{d-2,d}
			\end{bmatrix}\, .
		\end{equation*}
		
		In particular $\beta(I_{\tDelta})$ is $(d-2,d)_{\prec_h}$-initial.
	\end{cor}
	\begin{proof}
		By Corollary \ref{Corollary: Betti Diagram of Simplex} we know that the Betti diagram of $\Delta^{d-h-1}$ is zero. The result follows from Proposition \ref{Proposition: Betti adding Empty}.
	\end{proof}
	
	\begin{cor}\label{Corollary: Betti boundary-simplex-adding-empty}
		Suppose $d$ is a positive integer with $3 \leq d \leq n-h+1$. Let $\partial \Delta^{d-1}$ be the boundary of the $(d-1)$-simplex on a vertex set of size $d$. For $\tDelta=\partial \Delta^{d-1}+\Skel_0([h-1])$ we have
		\begin{equation*}
			\beta(I_{\tDelta}) = 
			\begin{bmatrix}
				\beta_{0,2} & \dots & \dots & \dots & \beta_{d+h-3,d+h-1}\\
				&  &  &  & \\
				&  &  &  & \\
				&  &  &  & \\
				\beta_{0,d}& \dots & \beta_{h-1,d+h-1} & &
			\end{bmatrix}\, .
		\end{equation*}
		In particular $\beta(I_{\tDelta})$ is $(0,d)_{\prec_h}$-initial.
	\end{cor}
	\begin{proof}
		By Corollary \ref{Corollary: Betti Diagram of Simplex} we know the only nonzero Betti number of $\partial \Delta^{d-1}$ is $\beta_{0,d}$. The result follows from Proposition \ref{Proposition: Betti adding Empty}.
	\end{proof}
	
	\begin{cor}\label{Corollary: Betti cyclic-adding-Em}
		Suppose that $h\leq n-2$ and that $d$ is a positive integer with $4\leq d \leq h+2$. Let $C_d$ be the cyclic graph on $d$ vertices. For $\tDelta = \Cl(C_d) + \Skel_0([h+2-d])$ we have
		\begin{equation*}
			\beta(I_{\tDelta}) =
			\begin{bmatrix}
				\beta_{0,2} & \dots & \dots & \dots & \beta_{h-1,h+1} & \beta_{h,h+2}\\
				& & \beta_{d-3,d} & \dots & \beta_{h-1,h+2} &
			\end{bmatrix}\, .
		\end{equation*}
		
		In particular $\beta(I_{\tDelta})$ is $(d-3,d)_{\prec_h}$-initial.
	\end{cor}
	\begin{rem}\label{Remark: Cl(Cm) as a complex not a graph}
		For $d\geq 4$, the clique complex $\Cl(C_d)$ contains no additional faces to the graph $C_d$, so its facets are simply the edges of $C_d$. As in Remark \ref{Remark: Skel_0(l) is E_l}, our reason for using the notation $\Cl(C_d)$ here rather than $C_d$ is to indicate that we are viewing this structure as a complex rather than a graph, and hence the corresponding ideal $I_{\tDelta}$ is constructed using the Stanley-Reisner construction rather than the edge ideal construction (i.e. it is obtained from $I_\Delta$ by adding generators corresponding to the \textit{nonfaces} of $\Cl(C_d)$). 
	\end{rem}
	\begin{proof}
		Suppose $G$ is the complement of the graph $C_d$. We have $\Cl(C_d)=\Ind(G)$ by Remark \ref{Remark: Ind(G)=Cl(G^c)}. From Proposition \ref{Proposition: Betti-Cn} we know that the edge ideal $I(G)$ has Betti diagram $$\begin{bmatrix}
			\beta_{0,2} & \dots & \beta_{d-4,d-2} & \\
			& & &\beta_{d-3,d}\\ 
		\end{bmatrix}$$ and hence so does the Stanley-Reisner ideal $I_{\Cl(C_d)}$ by Proposition \ref{Proposition: I(G) = I_Ind(G)}. The result follows from Proposition \ref{Proposition: Betti adding Empty}.
	\end{proof}
	
	\begin{ex}
		Consider the cone $\Dnhtilde[6][3]$. For this cone we need to find ten complexes $\Gamma^3_{i,d}$, for indices $(i,d)\in \II(\Dnhtilde[6][3])$ such that the diagram $\beta(\Gamma^3_{i,d})$ is $(i,d)_{\prec_3}$-initial. We need these complexes to have codimension $3$, and no more than $6$ vertices. Table \ref{Table: D63-tilde Diagrams Table} shows ten such complexes.
		
		For each index $(i,d)\in \II(\Dnhtilde[6][3])$ the table also shows the shape of the diagram $\beta(\Gamma^3_{i,d})$, chiefly to demonstrate that it is indeed $(i,d)$-inital as required.  In order to demonstrate this fact more clearly we denote the Betti number $\beta_{i,d}$ itself in bold. For most of our diagrams, there are indices $(i',d')\in \II(\Dnhtilde[6][3])$ for which the entry $\beta_{i',d'}$ is zero even though $(i',d')\prec_3 (i,d)$ (and hence $(i,d)_{\prec_3}$-initiality would permit $\beta_{i',d'}$ to be nonzero). We denote these entries in the Betti diagram by a $-$ sign.
		\begin{center}
			\begin{table}[p]
				\begin{tabular}{ c|c|c|l }
					$(i,d)$& Complex $\Gamma^3_{i,d}$ & Image of $\Gamma^3_{i,d}$ & Shape of Betti Diagram $\beta(\Gamma^3_{i,d})$ \\
					\hline
					&&&\\
					$(2,4)$&$\Delta^0 + \Skel_0([3])$ & \begin{tikzpicture}[scale = 0.4]
						\tikzstyle{point}=[circle,thick,draw=black,fill=black,inner sep=0pt,minimum width=2pt,minimum height=2pt]
						\node (a)[point] at (1,0.9) {};
						\node (d)[point] at (3,0.2) {};
						\node (e)[point] at (4,0.9) {};
						\node (f)[point] at (3,1.6) {};
					\end{tikzpicture} & $\begin{bmatrix}
						\beta_{0,2} & \beta_{1,3} & \boldsymbol{\beta_{2,4}}
					\end{bmatrix}$\\
					&&&\\
					$(3,5)$&$\Delta^1 + \Skel_0([3])$ & \begin{tikzpicture}[scale = 0.4]
						\tikzstyle{point}=[circle,thick,draw=black,fill=black,inner sep=0pt,minimum width=2pt,minimum height=2pt]
						\node (a)[point] at (1,0) {};
						\node (b)[point] at (1,1.7) {};
						\node (d)[point] at (3,0.2) {};
						\node (e)[point] at (4,0.9) {};
						\node (f)[point] at (3,1.6) {};
						
						\draw (a.center) -- (b.center);
					\end{tikzpicture} & $\begin{bmatrix}
						\beta_{0,2} & \beta_{1,3} & \beta_{2,4} & \boldsymbol{\beta_{3,5}}
					\end{bmatrix}$\\
					&&& \\
					$(4,6)$&$\Delta^2 + \Skel_0([3])$ & \begin{tikzpicture}[scale = 0.5]
						\tikzstyle{point}=[circle,thick,draw=black,fill=black,inner sep=0pt,minimum width=2pt,minimum height=2pt]
						\node (a)[point] at (0,0) {};
						\node (b)[point] at (2,0) {};
						\node (c)[point] at (1,1.7) {};
						\node (d)[point] at (3,0.2) {};
						\node (e)[point] at (4,0.9) {};
						\node (f)[point] at (3,1.6) {};
						
						\begin{scope}[on background layer]
							\draw[fill=gray] (a.center) -- (b.center) -- (c.center) -- cycle;
						\end{scope}
					\end{tikzpicture} & $\begin{bmatrix}
						\beta_{0,2} & \beta_{1,3} & \beta_{2,4} & \beta_{3,5}& \boldsymbol{\beta_{4,6}}
					\end{bmatrix}$\\
					&&&\\
					\hline
					&&&\\
					$(2,5)$& $\Cl(C_5)$ & \begin{tikzpicture}[scale = 0.6]
						\tikzstyle{point}=[circle,thick,draw=black,fill=black,inner sep=0pt,minimum width=2pt,minimum height=2pt]
						\node (a)[point] at (0,1) {};
						\node (b)[point] at (0.951,0.309) {};
						\node (c)[point] at (0.588,-0.809) {};
						\node (d)[point] at (-0.588,-0.809) {};
						\node (e)[point] at (-0.951,0.309) {};
						
						\draw (a.center) -- (b.center) -- (c.center) -- (d.center) -- (e.center) -- cycle;
					\end{tikzpicture} & $\begin{bmatrix}
						\beta_{0,2} & \beta_{1,3} & - & - & -\\
						& &\boldsymbol{\beta_{2,5}}
					\end{bmatrix}$\\
					&&& \\
					$(1,4)$&$\Cl(C_4) + \Skel_0([1])$ & \begin{tikzpicture}[scale = 0.4]
						\tikzstyle{point}=[circle,thick,draw=black,fill=black,inner sep=0pt,minimum width=2pt,minimum height=2pt]
						\node (a)[point] at (0,0) {};
						\node (b)[point] at (0,2) {};
						\node (c)[point] at (2,2) {};
						\node (d)[point] at (2,0) {};
						\node (e)[point] at (3.5,1) {};
						
						\draw (a.center) -- (b.center) -- (c.center) -- (d.center) -- cycle;
					\end{tikzpicture} & $\begin{bmatrix}
						\beta_{0,2} & \beta_{1,3} & \beta_{2,4} & \beta_{3,5} & -\\
						&\boldsymbol{\beta_{1,4}}& \beta_{2,5} &
					\end{bmatrix}$\\
					&&&\\
					$(0,3)$&$\partial \Delta^2 + \Skel_0([2])$ & \begin{tikzpicture}[scale = 0.5]
						\tikzstyle{point}=[circle,thick,draw=black,fill=black,inner sep=0pt,minimum width=2pt,minimum height=2pt]
						\node (a)[point] at (0,0) {};
						\node (b)[point] at (2,0) {};
						\node (c)[point] at (1,1.7) {};
						\node (d)[point] at (3,0.1) {};
						\node (e)[point] at (3,1.6) {};
						
						\draw (a.center) -- (b.center) -- (c.center) -- cycle;
					\end{tikzpicture} &  $\begin{bmatrix}
						\beta_{0,2} & \beta_{1,3} & \beta_{2,4} & \beta_{3,5} & -\\
						\boldsymbol{\beta_{0,3}} &\beta_{1,4}&\beta_{2,5} &
					\end{bmatrix}$\\
					&&&\\
					$(3,6)$&$S(\Delta^1 + \Skel_0([2]))$ & \begin{tikzpicture}[scale = 0.5]
						\tikzstyle{point}=[circle,thick,draw=black,fill=black,inner sep=0pt,minimum width=2pt,minimum height=2pt]
						\node (a)[point] at (0,1) {};
						\node (b)[point] at (1,1) {};
						\node (c)[point] at (2,1) {};
						\node (d)[point] at (3,1) {};
						\node (e)[point] at (1.5,0) {};
						\node (f)[point] at (1.5,2) {};
						
						\begin{scope}[on background layer]
							\draw[fill=gray] (e.center) -- (a.center) -- (b.center) -- cycle;
							\draw[fill=gray] (f.center) -- (a.center) -- (b.center) -- cycle;
							\draw (e.center) -- (c.center) -- (f.center) -- (d.center) -- cycle;
						\end{scope}
					\end{tikzpicture} & $\begin{bmatrix}
						\beta_{0,2} & \beta_{1,3} & \beta_{2,4}&\beta_{3,5}&\beta_{4,6}\\
						- &\beta_{1,4}&\beta_{2,5}&\boldsymbol{\beta_{3,6}}&
					\end{bmatrix}$\\
					&&&\\
					\hline
					&&&\\
					$(2,6)$&$S^2(\Delta^0 \Star_0 v)$ & \begin{tikzpicture}[scale=0.6][line join=bevel,z=-5.5]
						\coordinate (A1) at (0,0,-1);
						\coordinate (A2) at (-1,0,0);
						\coordinate (A3) at (0,0,1);
						\coordinate (A4) at (1,0,0);
						\coordinate (B1) at (0,1,0);
						\coordinate (C1) at (0,-1,0);
						
						\draw (A1) -- (A2) -- (B1) -- cycle;
						\draw (A4) -- (A1) -- (B1) -- cycle;
						\draw (A1) -- (A2) -- (C1) -- cycle;
						\draw (A4) -- (A1) -- (C1) -- cycle;
						\draw [fill opacity=0.7,fill=lightgray!80!gray] (A2) -- (A3) -- (B1) -- cycle;
						\draw [fill opacity=0.7,fill=gray!70!lightgray] (A3) -- (A4) -- (B1) -- cycle;
						\draw [fill opacity=0.7,fill=darkgray!30!gray] (A2) -- (A3) -- (C1) -- cycle;
						\draw [fill opacity=0.7,fill=darkgray!30!gray] (A3) -- (A4) -- (C1) -- cycle;
					\end{tikzpicture}& $\begin{bmatrix}
						\beta_{0,2} & - & - & - & -\\
						- &\beta_{1,4}& - & -\\
						& & \boldsymbol{\beta_{2,6}}
					\end{bmatrix}$\\
					&&& \\
					$(1,5)$& $S(\partial \Delta^2 + \Skel_0([1]))$ & \begin{tikzpicture}[scale = 0.6][line join = round, line cap = round]
						\tikzstyle{point}=[circle,thick,draw=black,fill=black,inner sep=0pt,minimum width=2pt,minimum height=2pt]
						
						\coordinate (3) at ({-.5*sqrt(3)},0,-.5);
						\coordinate (2) at (.2,0,1) {};
						\coordinate (1) at ({.5*sqrt(3)},0,-.5);
						
						\node [point] (a) at (1.7,0,-.5) {};
						
						\coordinate (4) at (0,{.8*sqrt(2)},0);
						\coordinate (5) at (0,-{.8*sqrt(2)},0);

						\begin{scope}
							\draw (1)--(3);
							\draw[fill=gray,fill opacity=.5] (2)--(1)--(4)--cycle;
							\draw[fill=darkgray,fill opacity=.5] (3)--(2)--(4)--cycle;
							\draw (2)--(1);
							\draw (2)--(3);
							\draw (3)--(4);
							\draw (2)--(4);
							\draw (1)--(4);
							
							\draw[fill=gray,fill opacity=.5] (2)--(1)--(5)--cycle;
							\draw[fill=darkgray,fill opacity=.5] (3)--(2)--(5)--cycle;
							\draw (3)--(5);
							\draw (2)--(5);
							\draw (1)--(5);
							
							\draw (4) -- (a);
							\draw (5) -- (a);
						\end{scope}
					\end{tikzpicture} & $\begin{bmatrix}
						\beta_{0,2} & \beta_{1,3} & \beta_{2,4}& -& -\\
						\beta_{0,3} & \beta_{1,4} & \beta_{2,5}& \beta_{3,6}&\\
						&\boldsymbol{\beta_{1,5}} & \beta_{2,6} & &
					\end{bmatrix}$\\
					&&& \\
					$(0,4)$&$\partial \Delta^3 + \Skel_0([2])$ & \begin{tikzpicture}[scale = 0.6][line join = round, line cap = round]
						\tikzstyle{point}=[circle,thick,draw=black,fill=black,inner sep=0pt,minimum width=2pt,minimum height=2pt]
						
						\coordinate (4) at (0,{sqrt(2)},0);
						\coordinate (3) at ({-.5*sqrt(3)},0,-.5);
						\coordinate (2) at (0,0,1) {};
						\coordinate (1) at ({.5*sqrt(3)},0,-.5);
						\node [point] (a) at (1.5,0,-.5) {};
						\node [point] (b) at (1.5,1,-.5) {};
						
						\begin{scope}
							\draw (1)--(3);
							\draw[fill=gray,fill opacity=.5] (2)--(1)--(4)--cycle;
							\draw[fill=darkgray,fill opacity=.5] (3)--(2)--(4)--cycle;
							\draw (2)--(1);
							\draw (2)--(3);
							\draw (3)--(4);
							\draw (2)--(4);
							\draw (1)--(4);
						\end{scope}
					\end{tikzpicture} & $\begin{bmatrix}
						\beta_{0,2} & \beta_{1,3} & \beta_{2,4}&\beta_{3,5}&\beta_{4,6}\\
						- & - & - & - &\\
						\boldsymbol{\beta_{0,4}} & \beta_{1,5} & \beta_{2,6} & &
					\end{bmatrix}$\\
				\end{tabular}
				\caption{\label{Table: D63-tilde Diagrams Table}\textit{$(i,d)_{\prec_3}$-initial complexes with diagrams in the cone $\Dnhtilde[6][3]$}}
			\end{table}
		\end{center}
	\end{ex}
	\pagebreak
	We now have all the ingredients we need to prove Theorem \ref{Theorem: dimDnh}.
	\begin{proof}[Proof of Theorem \ref{Theorem: dimDnh}]
		We have already seen that $$\dim \Dnh \leq 
		\dim (\WW(\Dnh)) = \frac{n(n-1)}{2}-\frac{h(h-1)}{2}+1.$$ We can show that $\dim \Dnh \geq \dim (\WW(\Dnh))$ by finding a linearly independent set of diagrams in $\Dnh$ of size $|\II(\Dnh)|-(h-1)$. In fact, by Lemma \ref{Lemma: Dimension of Dnhtilde => Dimension of Dnh} it suffices to find a linearly independent set of diagrams in $\Dnhtilde$ of size $|\II(\Dnhtilde)|-(h-1)$.
		
		To this end, we work towards exhibiting a set of complexes $\{\Gamma^h_{i,d} : (i,d)\in \II(\Dnhtilde)-\{(0,2),\dots,(h-2,h)\}\}$ such that for each $(i,d)\in\II(\Dnhtilde)-\{(0,2),\dots,(h-2,h)\}$, the complex $\Gamma^h_{i,d}$ has at most $n$ vertices, codimension $h$, and is $(i,d)$-initial with respect to the ordering $\prec_h$ given in Definition \ref{Definition: Order Dn}. In particular, for each index $(i,d)\in \II(\Dnhtilde)-\{(0,2),\dots,(h-2,h)\})$, we will show that the complex $\Gamma_{i,d}^h$ has a vertex set either of size $d$ (which is less than or equal to $n$ by Proposition \ref{Proposition: Inequalities Dn}) or of size $h+d-i-1$ (which is less than or equal to $n$ by Proposition \ref{Proposition: Inequalities Dnh 1}).
		
		We proceed by induction on $h\geq 1$. For any index $(i,d)$ in $\II(\Dnhtilde[n][1])$ the complex $\Delta^{d-2}\Star_{d-i-2} v$ is $(i,d)_{\prec_1}$-initial by Corollary \ref{Corollary: Homology of Starred Simplex}. Moreover it has $d$ vertices and codimension $1$. This proves the base case $h=1$.
		
		Now assume $h>1$, and pick some index $(i,d)\in \II(\Dnhtilde)-\{(0,2),\dots,(h-2,h)\}$. There are four cases to consider.
		
		First, if $(i,d)$ is in row 1 (i.e. $d-i=2$) with $h+1\leq d \leq n$, we define $\Gamma^h_{i,d}$ to be the complex $\Delta^{d-h-1}+\Skel_0([h])$. This has $d$ vertices and codimension $h$, and it is $(i,d)_{\prec_h}$-initial by Corollary \ref{Corollary: Betti simplex-adding-empty}.
		
		Next, if $(i,d)$ is in column 0 (i.e. $i=0$) with $3\leq d \leq n+h-1$, we define $\Gamma^h_{i,d}$ to be the complex $\partial \Delta^{d-1}+\Skel_0([h-1])$. This has $d+h-1$ vertices (which is the same as $d+h-i-1$ vertices because $i=0$) and codimension $h$. Moreover it is $(i,d)_{\prec_h}$-initial by Corollary \ref{Corollary: Betti boundary-simplex-adding-empty}.
		
		Next, if $(i,d)$ is in row 2 (i.e. $d-i=3$) with $4\leq d \leq h+2$, we define $\Gamma^h_{i,d}$ to be the complex $\Cl(C_d)+\Skel_0([h+2-d])$. This has $h+2$ vertices (which is the same as $d+h-i-1$ vertices because $d-i=3$) and codimension $h$. Moreover it is $(i,d)_{\prec_h}$-initial by Corollary \ref{Corollary: Betti cyclic-adding-Em}.
		
		Finally, for all other indices $(i,d)$ in $\II(\Dnhtilde)-\{(0,2),\dots,(h-2,h)\}$, the index $(i-1,d-2)$ lies inside the indexing set $\II(\Dnhtilde[n][h-1])-\{(0,2),\dots,(h-3,h-1)\}$, and hence by induction we have already found an $(i-1,d-2)_{\prec_{h-1}}$-initial complex $\Gamma^{h-1}_{i-1,d-2}$ with codimension $h-1$ on a vertex set whose size is either $d-1$ or $(h-1)+(d-2)-(i-1)-1$. We define $$\Gamma^h_{i,d}= S \Gamma^{h-1}_{i-1,d-2}.$$ By Lemma \ref{Lemma: SDelta initiality Dnh} this complex has codimension $h$ and is $(i,d)_{\prec_h}$-initial. Also it has exactly two more vertices than $\Gamma^{h-1}_{i-1,d-2}$ which means it has either $d$ vertices or $h+d-i-1$ vertices, as required. This completes the proof.
	\end{proof}
	
	\section{Dimension of $\Cnh$}\label{Subsection: dimCnh}
	We now arrive at our final cone $\Cnh$, generated by the diagrams of edge ideals of height $h$. Note that this is a subcone of all the other cones we've studied so far. We work towards proving Theorem \ref{Theorem: dimCnh} on the cone's dimension.
	
	\subsection{Upper Bound}\label{Subsection: ub Cnh}
	Once again, we begin by searching for an indexing subset $\II(\Cnh)\subset \II(\Cn)$, and hence a minimal subspace $\WW(\Cnh)$ containing $\Cnh$.
	
	To find our indexing subset, we need to obtain some additional restrictions on the positions of the nonzero values of the diagrams in $\Cnh$. To help us with this, we need two important lemmas (these are, respectively, Theorem 4.4 in \cite{SFMon} and Corollary 7.2.4 in \cite{Vil}). Before reading these lemmas, it may be helpful to recall the concepts of vertex covers and matchings as given in Definition \ref{Definition: Graph Theory Key Terminology and Constructions}, and of regularity as given in Definition \ref{Definition: Regularity}.
	
	
	\begin{lem}\label{Lemma: Height of Edge Ideals}
		Let $G$ be a graph. The following are equivalent.
		\begin{enumerate}
			\item The height of $I(G)$ is equal to $h$.
			\item $G$ has a minimally sized vertex cover of size $h$.
		\end{enumerate}
	\end{lem}
	\begin{lem}\label{Lemma: Regularity of Edge Ideals}
		Let $G$ be a graph, and let $\alpha$ be the minimum size of a maximal matching in $G$. We have $\reg I(G)\leq \alpha + 1$.
	\end{lem}
	
	Using these two lemmas, we can prove the following proposition.
	\begin{prop}\label{Proposition: Inequalities Cnh}
		Consider $\beta\in \Cnh$. For every $(i,d)\in S_n$ satisfying $d - i > h 
		+ 1$, we have $\beta_{i,d}=0$.
	\end{prop}
	\begin{proof}
		Let $G$ be a graph on $n$ vertices whose edge ideal has height $h$. We need to show that $\reg I(G)\leq h+1$. By Lemma \ref{Lemma: Height of Edge Ideals}, there is a minimal vertex cover $\{x_{i_1},...,x_{i_h}\}$ for $G$, which means that every edge in $E(G)$ contains at least one of $x_{i_1},...x_{i_h}$. Hence, no matching in $G$ can consist of more than $h$ edges, so the minimal size of a maximal matching in $G$ must be less than or equal to $h$. By Lemma \ref{Lemma: Regularity of Edge Ideals}, we have $\reg I(G)\leq h+1$.
	\end{proof}
	
	Because $\Cnh\subset \Dnh$, we already know from Proposition \ref{Proposition: Inequalities Dnh 1} that the diagrams $\beta$ in $\Cnh$ satisfy $\beta_{i,d}=0$ whenever $d-i>n-h+1$. Along with Proposition \ref{Proposition: Inequalities Cnh} this tells us that in fact $\beta_{i,d}=0$ whenever $d-i>\min\{h,n-h\}+1$. This gives us a much clearer picture of what the diagrams in $\Cnh$ look like. Setting $m=\min\{h,n-h\}$, the diagrams $\beta\in \Cnh$ look like the following.
	\begin{equation}
		\begin{bmatrix}\label{Equation: Cnh Matrix}
			\beta_{0,2} & \dots & \dots & \dots & \dots & \dots & \beta_{n-2,n}\\
			&\ddots & & & & &\\
			& & \beta_{m-1,2m} &\dots &\beta_{n-m-1,n} & \\
		\end{bmatrix}
	\end{equation}
	
	Thus, we may define our indexing set $\II(\Cnh)$ and our subspaces $(W_n^h)'$ and $W_n^h$ as follows.
	\begin{defn}\label{Definition: II(Cnh)}
		We define $$\II(\Cnh) := \left\{(i,d)\in \II(\Cn): d-i \leq \min\{h,n-h\}+1\right\}.$$
	\end{defn}
	
	As with the cone $\Dnh$, we know that as well as living inside the vector space carved out by this indexing set, the diagrams $\beta$ in $\Cnh$ must also satisfy the Herzog-K\"{u}hl equations $\HK_1(\beta)=\dots = \HK_{h-1}(\beta)=0$ by Remark \ref{Remark: HK Equations}. Thus we define our vector space $\WW(\Cnh)$ as follows.
	\begin{defn}\label{Definition: WW(Cnh)}
		We define $$\WW(\Cnh) := \left\{\beta \in \bigoplus_{(i,d)\in \II(\Cnh)}\QQ : \HK_1(\beta)=\dots =\HK_{h-1}(\beta)=0 \right\}.$$
	\end{defn}
	
	To prove that the fomula given in Theorem \ref{Theorem: dimCnh} is an upper bound for $\dim \Cnh$, we need to show that $\dim (\WW(\Cnh))=h(n-h-1)+1$, which we do below.
	\begin{prop}\label{Proposition: dim WW(Cnh)}
		Let $\II(\Cnh)$ and $\WW(\Cnh)$ be as in Definitions \ref{Definition: II(Cnh)} and \ref{Definition: WW(Cnh)}. We have
		\begin{enumerate}
			\item $|\II(\Cnh)|=h(n-h)$.
			\item $\dim \WW(\Dnh)=h(n-h-1)+1$.
		\end{enumerate}
	\end{prop}
	\begin{proof}
		For part (1), we set $m=\min\{h,n-h\}$, and arrange the elements of $\II(\Cnh)$ in rows as in Equation (\ref{Equation: Cnh Matrix}).
		\begin{equation*}\label{Equation: II(Cnh)}
			\begin{matrix}
				(0,2) & \dots & \dots & \dots & \dots & \dots & (n-1,n)\\
				&\ddots & & & & &\\
				& & (m-1,2m) &\dots &(n-m-1,n) & \\
			\end{matrix}
		\end{equation*}
		As noted in Section \ref{Subsection: ub Cn}, for each $1\leq i\leq m$, row $i$ of $\II(\Cnh)$ has $n-2i+1$ elements, and hence we have
		\begin{align*}
			|\II(\Cnh)|&= \sum_{i=1}^m (n-2i+1)\\
			&= \sum_{i=1}^m n - \sum_{i=1}^m(2i-1)\\
			&= nm - m^2\\
			&= m(n-m)\\
			&=h(n-h)\, .
		\end{align*}
		
		For part (2), as noted in the proof of Proposition \ref{Proposition: dim WW(Dnh)}, the Herzog-K\"{u}hl equations $\HK_1(\beta)= 0, ..., \HK_{h-1}(\beta)=0$ are linearly independent, and so we have
		\begin{align*}
			\dim (\WW(\Cnh)) &= |\II(\Cnh)|-(h-1)\\
			&= h(n-h)-h+1\\
			&= h(n-h-1)+1\, .
		\end{align*}
	\end{proof}
	
	\subsection{Lower Bound}\label{Subsection: lb Cnh}
	To complete our proof of Theorem \ref{Theorem: dimCnh}, it only remains to show that the space $\WW(\Cnh)$ is the \textit{minimal} subspace of $\WW(\Cn)$ containing $\Cnh$, by finding an appropriately sized linearly independent set of diagrams lying in the cone $\Cnh$.
	
	In fact we have already found these diagrams: for every index $(i,d)$ in $\II(\Cnh)-\{(0,2),\dots,(h-2,h)\}$ we will show that the $(i,d)_{\prec_h}$-initial diagram we found inside $\Dnh$ in the previous section actually lies inside the cone $\Cnh$ as well. We proceed to the proof immediately.
	
	\begin{proof}[Proof of Theorem \ref{Theorem: dimCnh}]
		We have already seen that $$\dim \Cnh \leq 
		\dim (\WW(\Cnh)) = h(n-h-1)+1.$$ We can show that $\dim \Cnh \geq \dim (\WW(\Cnh))$ by finding a linearly independent set of diagrams in $\Cnh$ of size $|\II(\Cnh)|-(h-1)$.
		
		In our proof on the dimension of the cone $\Dnh$, we constructed, for each index $(i,d)\in \II(\Dnh)-\{(0,2),\dots,(h-2,h)\}$, an $(i,d)_{\prec_h}$-initial complex $\Gamma^h_{i,d}$ of codimension $h$ on at most $n$ vertices. Thus to find our linearly independent diagrams in $\Cnh$ it suffices to show that for each index $(i,d)\in \II(\Cnh)\{(0,2),\dots,(h-1,h+1)\}$ the complex $\Gamma^h_{i,d}$ is in fact the independence complex of some graph $G^h_{i,d}$, and hence lies inside $\Cnh$. We will demonstrate this in some cases by showing that $\Gamma^h_{i,d}$ is the complex of cliques of some graph, which is the same as the independence complex of the graph's complement by Remark \ref{Remark: Ind(G)=Cl(G^c)}.
		
		We proceed by induction on $h\geq 1$. Suppose first that our index $(i,d)$ lies in $\{(h-1,h+1),\dots,(n-2,n)\}$. In this case we defined the complex $\Gamma^h_{i,d}$ to be $\Delta^{d-h-1}+\Skel_0([h])$, which is equal to the complex of cliques $\Cl(K_{d-h}+E_h)$, so we set $G^h_{i,d}=(K_{d-h}+E_h)^c$. Note that this is true even for the base case $h=1$, where we defined $\Gamma^1_{i,d}$ to be the starred complex $\Delta^{d-2} \Star_0 v$, because this complex is equal to the disjoint union $\Delta^{d-2}+\Skel_0([1])$. In particular, the indexing set $\II(\Cnh[n][1])$ contains only the indices $(0,2),\dots,(n-2,n)$, so this proves the base case.
		
		Next suppose we have $(i,d) \in \{(1,4), \dots, (h-1,h+2)\}$. Here we defined the complex $\Gamma^h_{i,d}$ to be $\Cl(C_d) + \Skel_0([h+2-d])$. This is equal to $\Cl(C_d + E_{h+2-d})$, so we set $G^h_{i,d}=(C_d + E_{h+2-d})^c$.
		
		For every other index $(i,d)$ in $\II(\Cnh)$, we defined the complex $\Gamma^h_{i,d}$ to be the suspension $S (\Gamma^{h-1}_{i-1,d-2})$. By induction the complex $\Gamma^{h-1}_{i-1,d-2}$ is equal to $\Ind(G^{h-1}_{i-1,d-2})$ for some graph $G^{h-1}_{i-1,d-2}$ and hence we may define $G^h_{i,d}=G^{h-1}_{i-1,d-2}+L$.
	\end{proof}
	\begin{rem}
		Note that the indices $(0,d)$ in $\II(\Dnhtilde)$ with $3\leq d \leq n-h+1$, for which we found the complexes $\Gamma^h_{i,d}=\partial \Delta^{d-1}+\Skel_0([h-1])$, lie outside the indexing set $\II(\Cnh)$ (by Proposition \ref{Proposition: Inequalities Cn}). Thus it does not present an issue to this proof that their corresponding complexes are not independence complexes of graphs.
	\end{rem}
	
	\section{Concluding Remarks}
	Our proofs for these results demonstrate, up to linear combination, all the linear dependency relations that are satisfied by every diagram in the cone. For each cone $\calC$ we started by finding the cone's indexing set $\II(\calC)$, which showed us the possible shapes of the diagrams in the cone, and thus gave us cofinitely many relations of the form $\beta_{i,d}=0$ for indices $(i,d)$ outside of $\II(\calC)$. For the cones $\Dn$, $\Dntilde$ and $\Cn$, these relations are the only ones we have. For the cones $\Dnh$, $\Dnhtilde$ and $\Cnh$ we also have the Herzog-K\"{u}hl equations $\HK_1(\beta)=\dots=\HK_{h-1}(\beta)=0$.
	
	
	
	
	
	\chapter{PR Complexes and Degree Types}\label{Chapter: PR Complexes and Degree Types}
	Our aim over the next three chapters is to investigate the possible shapes of pure Betti diagrams arising from Stanley-Reisner ideals. In particular, we work towards proving Theorem \ref{Theorem: PR Complexes of Any Degree Type}, stated below, which is a partial analogue to the first Boij-S\"{o}derberg Conjecture, for the cone $\Dn$.
	
	
	By the `\textit{shape}' of a Betti diagram, we mean the possible positions of its nonzero entries. Formally, we can define a Betti diagram's shape as follows.
	\begin{defn}\label{Definition: Shape of Betti Diagram}
		The \textit{shape} of a Betti diagram $\beta$ is the set $$S(\beta)=\left\{(i,d)\in \{0,\dots,n\}\times \ZZ : \beta_{i,d}\neq 0 \right\}.$$
	\end{defn}
	
	We have already found some restrictions on the shapes of the diagrams in $\Dn$ in Chapter \ref{Chapter: Dimensions} (specifically, they are all subsets of $\II(\Dn)$). Now we narrow our attention slightly to the \textit{pure} diagrams in the cone. Our goal is to answer the following question.
	\begin{qu}\label{Question: Shapes of Betti diagrams}
		What are the possible shapes of pure diagrams arising from Stanley-Reisner ideals?
	\end{qu}
	
	In the case of pure diagrams, the shape of the diagram contains the same information as its shift type (as given in Definition \ref{Definition: Pure Diagrams and Shift Type}). This allows us to revise Question \ref{Question: Shapes of Betti diagrams} as follows.
	\begin{qu}\label{Question: Shift Types}
		For a given strictly decreasing sequence $\bc=(c_p,\dots,c_0)$ of positive integers, can we find a Stanley-Reisner ideal $I$ such that $\beta(I)$ is pure with shift type $\bc$?
	\end{qu}
	
	As it turns out, the answer to this question is \textit{no}, in general. For example, the results of the previous chapter demonstrate that no Stanley-Reisner ideal can have a pure resolution with shift type $(5,2)$. Indeed, any such ideal $I$ would have to be generated in degree $2$, which would make it an edge ideal. But the index $(1,5)$ lies outside of $\II(\Cn)$, so we must have $\beta_{1,5}(I)=0$.

	However we can ask a broader question. We start by introducing some new terminology, loosely mirroring terminology found in \cite{Bruns-Hibi} (page 1203), which is also focussed on classifying the pure diagrams of Stanley-Reisner ideals.
	
	\begin{defn}
		Suppose we have a Stanley-Reisner ideal $I$ with a pure resolution
		$$0\rightarrow R(-c_p)^{\beta_{p,c_p}} \rightarrow ... \rightarrow R(-c_1)^{\beta_{1,c_1}} \rightarrow R(-c_0)^{\beta_{0,c_0}}\rightarrow I.$$
		We say that this pure resolution (and the corresponding Betti diagram) has \textit{degree type} $(c_p-c_{p-1},...,c_1-c_0)$.
	\end{defn}
	\begin{rem}
		Our notion of degree type differs slightly from that found in \cite{Bruns-Hibi}, which defines it as the sequence $(c_p-c_{p-1},...,c_1-c_0,c_0)$. This is the sequence of degrees of the maps in the resolution (i.e. the degrees of the elements in the matrices representing those maps). Our version of degree type records the degrees of all of these maps except for $R(-c_0)^{\beta_{0,c_0}}\rightarrow I$.
	\end{rem}
	
	The degree type of a pure resolution contains slightly less information than its shift type, but the shift type is uniquely determined by the degree type and the value of $c_0$. Thus Question \ref{Question: Shift Types} could be widened as follows.
	\begin{qu}\label{Question: Degree Types}
		For a given sequence $\bd=(d_p,\dots,d_1)$ of positive integers, can we find a Stanley-Reisner ideal $I$ such that $\beta(I)$ is pure with degree type $\bd$?
	\end{qu}
	
	Answering Question \ref{Question: Degree Types} would still go a significant way towards classifying the possible shapes of pure Betti diagrams of Stanley-Reisner ideals. We will show that the answer to this question is, in fact, \textit{yes}, by proving the following theorem.
	
	\begin{thm}\label{Theorem: PR Complexes of Any Degree Type}
		Let $\bd=(d_p,\dots,d_1)$ be any sequence of positive integers. There exists a simplicial complex $\Delta$ such that the Betti diagram $\beta(I_{\Delta})$ is pure with degree type $\bd$.
	\end{thm}
	
	Over the next three chapters we work towards proving Theorem \ref{Theorem: PR Complexes of Any Degree Type}. In this chapter we introduce the family of PR complexes, complexes whose dual Stanley-Reisner ideals have pure resolutions; in Chapter \ref{Chapter: Families of PR Complexes} we look at some interesting  subfamilies of PR complexes, along with their corresponding degree types and Betti diagrams; and finally, in Chapter \ref{Chapter: Generating Degree Types}, we prove Theorem \ref{Theorem: PR Complexes of Any Degree Type} by presenting an algorithm for generating a PR complex of any given degree type.
	
	\section{Motivation}\label{Subsection: Motivation PR}
	
	We came to study this topic through reading the results of \cite{Bruns-Hibi}, which is also focussed on the construction of Stanley-Reisner ideals with pure resolutions, and itself builds on earlier work from \cite{Fro} and \cite{Bruns-Hibi-earlier} (this latter article was completed in 1993 but later published in 1998). We hope that our work adds something to these efforts.
	
	One reason the study of pure Betti diagrams is of particular interest in general is because of the Boij-S\"{o}derberg conjectures: as we saw in Section \ref{Subsection: Boij-Soderberg}, the extremal rays of the cone $\calC(\ba,\bfb)$ are the pure diagrams corresponding to Cohen-Macaulay modules; and these are also the extremal rays of the wider Betti cone generated by \textit{all} diagrams of $R$-modules in the same window. 
	
	It should be noted, however, that the same is \textit{not} true for the cone $\Dn$. While the pure diagrams in $\Dn$ which correspond to Cohen-Macaulay modules must be extremal rays of $\Dn$ (because they are extremal in the wider cone generated by all $R$-modules, in some appropriately sized window containing the index set $\II(\Dn)$), many pure diagrams in $\Dn$ do \textit{not} correspond to Cohen-Macaulay modules. Moreover, as a proper subcone of the wider cone generated by all $R$-modules, many of the extremal rays of $\Dn$ are not extremal in the wider cone. In general, there are many extremal rays of $\Dn$ which are not pure (see Example $\ref{Example: non-pure extremal ray}$ in Section \ref{Subsection: Extremal Rays and Defining Halfspaces}), and many pure diagrams which are not extremal (see Example $\ref{Example: non-extremal pure diagram}$ in Section \ref{Subsection: Extremal Rays and Defining Halfspaces}).
	
	Nevertheless, classifying the possible shapes of pure diagrams in $\Dn$ can still help us in understanding its extremal rays. Most notably, if $\Dn$ contains a pure diagram of shape $S$ then it must contain an extremal ray of shape $S$.
	
	To see why, suppose we were to write a diagram $\beta$ in $\Dn$ as a sum of extremal rays $\beta=\sum_{j} \alpha^j$. This gives us that $S(\beta) = \bigcup_j S(\alpha^j)$. Hence, for every $j$, we have $S(\alpha_j)\subseteq S(\beta)$, and moreover, every index $(i,d)$ in $S(\beta)$ must be contained in at least one of the sets $S(\alpha^j)$. Suppose now that $\beta$ is pure, and choose $\alpha^j$ such that $S(\alpha^j)$ contains the index $(i,d)$ for which $i$ is maximal. It follows that $S(\alpha^j)=S(\beta)$.
	
	Thus, by finding the shapes of the pure diagrams in $\Dn$ we find the shapes of some of its extremal rays.
	
	Perhaps more significantly, studying the shapes of the pure diagrams of Stanley-Reisner ideals also allows us to investigate the extent to which the Boij-S\"{o}derberg conjectures hold true of the cones generated by these diagrams. As mentioned at the start of this chapter, our key theorem (Theorem \ref{Theorem: PR Complexes of Any Degree Type}) can be seen as a partial analogue to the first Boij-S\"{o}derberg Conjecture (Theorem \ref{Theorem: BS Conjecture 1}) for Betti diagrams of Stanley-Reisner ideals. However, there are two key differences between the two theorems. The first is that Theorem \ref{Theorem: PR Complexes of Any Degree Type} is a result about degree types rather than shift types (we have already noted that the shift type analogue of Theorem \ref{Theorem: PR Complexes of Any Degree Type} is \textit{not} true). The second is that Theorem \ref{Theorem: PR Complexes of Any Degree Type} places no conditions on the number of vertices of the complex $\Delta$, so it doesn't give us a specific value of $n$ for which the diagram $\beta(I_\Delta)$ is contained in the cone $\Dn$. What it tells us is that there is \textit{some} value of $n$ for which there exists a pure Betti diagram of degree type $\bd$ in $\Dn$ (and thus in all the cones $\Dn[m]$ for $m\geq n$).

	\section{An Introduction to PR Complexes}\label{Subsection: PR Complexes}
	Our aim going forward is to construct Stanley-Reisner ideals with pure resolutions of varying degree types. In particular we wish to construct simplicial complexes whose \textit{dual} Stanley-Reisner ideals have pure Betti diagrams.
	
	The Alexander Dual version of Hochster's Formula (Theorem \ref{Theorem: ADHF}, or ADHF for short) gives us a combinatorial description of these complexes, as shown below.
	
	\begin{cor}\label{Corollary: PR Complexes}
		Let $\Delta$ be a simplicial complex on $[n]$. The diagram $\beta=\beta(\Idstar)$ is pure if and only if $\Delta$ satisfies the following condition:
		
		For every $i\geq -1$, and every face $\sigma,\tau \in \Delta$, if $\Hred_i(\lkds)\neq 0 \neq \Hred_i(\link_\Delta \tau)$ then $|\sigma|=|\tau|$.
	\end{cor}
	\begin{proof}
		The diagram $\beta$ is pure if and only if for every $i$, there exists at most one $c_i$ such that $\beta_{i,c_i}\neq 0$. By ADHF, this holds if and only if there are no two faces of different sizes in $\Delta$ whose links both have nontrivial homology at the same degree.
	\end{proof}
	
	\begin{defn}\label{Definition: PR Complexes}
		We refer to complexes which satisfy the condition in Corollary \ref{Corollary: PR Complexes} as \textit{PR complexes} (over $\KK$), where \textit{PR} stands for \textit{Pure Resolution}.
	\end{defn}
	\begin{rem}
		All of our work in this thesis is done over the arbitrary field $\KK$. Hence, for the rest of this thesis we will simply use the phrase `\textit{PR}' to mean `\textit{PR over} $\KK$'. While the PR condition is generally dependent on our choice of field, it is worth remarking that every PR complex presented in this thesis satisfies the PR condition over any field.
	\end{rem}
	
	We will examine some key properties of PR examples in due course, but for now, we observe the following immediate result of the PR condition.
	
	\begin{lem}\label{Lemma: PR complexes are pure}
		All PR complexes are pure (i.e. their facets all have the same dimension).
	\end{lem}
	\begin{proof}
		The link of a facet is the irrelevant complex $\{\emptyset\}$ which has nontrivial $\nth[st]{(-1)}$ homology. Thus a PR complex cannot have two facets of different sizes.
	\end{proof}
	
	The following standard lemma will be particularly helpful to us when studying the links in PR complexes.
	\begin{lem}\label{Lemma: links in links}
		Let $\Delta$ be a simplicial complex, $\sigma$ a face of $\Delta$ and $\tau \subseteq \sigma$. We have $\lkds=\link_{\lkds[\tau]}(\sigma - \tau)$.
	\end{lem}
	\begin{proof}
		For any face $f$ of $\Delta$ we have
		\begin{align*}
			f\in \lkds &\Leftrightarrow f \sqcup \sigma \in \Delta\\
			&\Leftrightarrow f \sqcup (\sigma-\tau)\sqcup \tau \in \Delta\\
			&\Leftrightarrow f \sqcup (\sigma-\tau) \in \lkds[\tau]\\
			&\Leftrightarrow f \in \link_{\lkds[\tau]}(\sigma - \tau) \,.
		\end{align*}
	\end{proof}
	This lemma shows that links can be computed piecewise, and in any order we choose. More specifically, to compute the link of a face $\sigma$ in a complex $\Delta$, we can start by computing the link $L=\lkds[x]$ for some vertex $x\in\sigma$, and then compute the link of $\sigma-x$ in $L$. This technique lends itself readily to inductive arguments, and often allows us to restrict our attention to the links of vertices.\\
	
	The PR property is a condition on the homology of links. For this reason it will be useful for us to extend the concept of homology index sets (Definition \ref{Definition: Homology Index Sets}) in the following way.
	\begin{defn}\label{Definition: Homology Index Sets With Links}
		Let $\Delta$ be a simplicial complex.
		\begin{itemize}
			\item For a face $\sigma$ in $\Delta$ we define the \textit{homology index set of} $\Delta$ \textit{at} $\sigma$ to be the set $h(\Delta, \sigma)=\{i\in \ZZ : \Hred_i(\lkds) \neq 0\}$. Note that under this definition, we have $h(\Delta)=h(\Delta,\emptyset)$.
			\item For a natural number $m$ we define the \textit{complete homology index set} of $\Delta$ at $m$ as $\hh(\Delta,m)=\bigcup_{\sigma\in \Delta, |\sigma|=m} h(\Delta,\sigma)$.
		\end{itemize}
	\end{defn}
	
	Using this notation for homology index sets, we can present an alternate definition for PR complexes.
	\begin{prop}\label{Proposition: Alternate PR Definition}
		Let $\Delta$ be a simplicial complex. The following are equivalent.
		\begin{enumerate}
			\item $\Delta$ is PR.
			\item For any $\sigma$ and $\tau$ in $\Delta$ with $|\sigma|\neq |\tau|$, we have $h(\Delta,\sigma)\cap h(\Delta,\tau)=\emptyset$.
			\item For any distinct integers $m_1\neq m_2$, we have $\hh(\Delta,m_1)\cap \hh(\Delta,m_2) = \emptyset$.
		\end{enumerate}
	\end{prop}
	\begin{proof}
		This is a rephrasing of the PR condition in terms of homology index sets.
	\end{proof}
	
	We can also use ADHF to derive an entirely combinatorial description of the degree type of a PR complex. In order to do this thoroughly, it will be useful to examine the homology index sets of PR complexes in more depth.
	
	\subsection{Homology Index Sets and Degree Types}
	In this section we investigate the homology index sets of PR complexes as given in Definition \ref{Definition: Homology Index Sets With Links}. Our aim is to provide a combinatorial reframing of degree types, and then use this reframing to give a concrete description of the homology index sets of PR complexes with a given degree type. We also introduce the notion of the \textit{offset} of a PR complex, which is the minimum size of a face whose link has homology.
	
	We begin with the following elementary observation: for any face $\sigma$ in an arbitrary simplicial complex $\Delta$, the dimension of $\lkds$ is at most $\dim \Delta - |\sigma|$. This dimension gives us an upper bound on the degrees of nontrivial homologies for $\lkds$, which means, very roughly speaking, that as the size of $\sigma$ increases, the degrees of the nontrivial homologies of $\lkds$ tend to decrease. The following lemma makes this idea more precise.
	
	\begin{lem}\label{Lemma: Homology of descending links}
		Let $\Delta$ be any simplicial complex. Suppose we have some face $\sigma$ in $\Delta$ such that $\Hred_j(\lkds)\neq 0$ for some $j\geq 0$. There exists a chain of simplices $\sigma = \tau_j \subsetneqq \tau_{j-1} \subsetneqq \dots \subsetneqq \tau_0 \subsetneqq \tau_{-1}$ in $\Delta$ such that for each $-1\leq i\leq j$, we have $\Hred_i(\lkds[\tau]_i)\neq 0$.
	\end{lem}
	\begin{proof}
		
		
		We prove this algebraically, setting $\beta=\beta(\Idstar)$. It suffices to show that there is some face $\sigma \subsetneqq \tau \in \Delta$ for which $\Hred_{j-1}(\lkds[\tau])\neq 0$. The result then follows by induction on $j$.
		
		By replacing $\Delta$ with $\lkds$ (and using Lemma \ref{Lemma: links in links}), we may assume that $\sigma = \emptyset$, and thus we need only find a nonempty face $\tau$ in $\Delta$ whose link has $\nth[st]{(j-1)}$ homology. To find a candidate for $\tau$, note that by ADHF we have that $\beta_{j+1,n}\neq 0$. Hence there must be some $d<n$ such that $\beta_{j,d}\neq 0$. This means (again, by ADHF) that there exists some nonempty face $\tau$ in $\Delta$ of size $n-d$ for which $\Hred_{j-1}(\lkds)\neq 0$, as required.
	\end{proof}

	\begin{cor}\label{Corollary: PR Complex links have single homology}
		Let $\Delta$ be a PR complex. Every link in $\Delta$ has at most one nontrivial homology.
	\end{cor}
	\begin{proof}
		Let $\sigma$ be a face of $\Delta$ and suppose for contradiction that $\Hred_i(\lkds)\neq 0\neq \Hred_j(\lkds)$ for some $i < j$. By Lemma \ref{Lemma: Homology of descending links}, there exists some $\tau \supsetneqq \sigma$ such that $\Hred_i(\link_\Delta \tau)\neq 0$. But this contradicts the fact that $\Delta$ is PR, because $|\tau|>|\sigma|$. 
	\end{proof}
	\begin{rem}
		In particular, the complex $\Delta$ is equal to $\link_\Delta \emptyset$, so it must have at most one nontrivial homology itself.
	\end{rem}
	
	Another way of phrasing Corollary \ref{Corollary: PR Complex links have single homology} is that for PR complexes $\Delta$, the homology index sets $h(\Delta,\sigma)$ are all either empty or singletons. Moreover, as the next corollary demonstrates, the indices in the nonempty homology index sets $h(\Delta,\sigma)$ decrease as $|\sigma|$ increases (making our observation at the start of this section exact, for PR complexes).
	\begin{cor}\label{Corollary: Homology index sets decreasing sequence}
		Let $\Delta$ be a PR complex, and suppose $\sigma_1$ and $\sigma_2$ are faces of $\Delta$ with $h(\Delta,\sigma_1)=\{i_1\}$ and $h(\Delta,\sigma_2)=\{i_2\}$. If $|\sigma_1|<|\sigma_2|$ then $i_1>i_2$.
	\end{cor}
	\begin{proof}
		We cannot have $i_1=i_2$ as this directly contradicts the PR condition. If $i_1< i_2$, then by Lemma \ref{Lemma: Homology of descending links} we can find some face $\sigma_3$ of $\Delta$ strictly containing $\sigma_2$ such that $\Hred_{i_1}(\lkds_3)\neq 0$, which also contradicts the PR property because $|\sigma_3|>|\sigma_1|$.
	\end{proof}
	
	Using these results about the homology index sets of PR complexes, we are now able to give an entirely combinatorial description of the degree type of a PR complex $\Delta$, which agrees with the degree type of the Betti diagram $\beta(\Idstar)$.
	
	\begin{defn}\label{Definition: PR Complex Degree Types}
		Let $\Delta$ be a PR Complex, and let $p$ be the maximum index for which there exists some face $\sigma$ in $\Delta$ such that $\Hred_{p-1}(\lkds)\neq 0$.
		
		For each $0\leq i\leq p$, we define $s_i$ to be the size $|\sigma|$ of the faces $\sigma$ of $\Delta$ for which $\Hred_{i-1}(\lkds)\neq 0$, and for each $1\leq i\leq p$, we define $d_i=s_{i-1}-s_i$.
		
		We call the sequence $(d_p,\dots,d_1)$ the \textit{degree type} of $\Delta$.
	\end{defn}
	\begin{rem}
		The integers $s_{p-1},\dots,s_0$ are well-defined by Lemma \ref{Lemma: Homology of descending links}, and form a strictly decreasing sequence by Corollary \ref{Corollary: Homology index sets decreasing sequence}. Thus the degree type $(d_p,\dots,d_1)$ must consist of positive integers.
		
		To see why this notion of the degree type is the same as the degree type of the pure diagram $\beta(\Idstar)$, suppose that $$0\rightarrow R(-c_p)^{\beta_{p,c_p}} \rightarrow ... \rightarrow R(-c_1)^{\beta_{1,c_1}} \rightarrow R(-c_0)^{\beta_{0,c_0}}\rightarrow \Idstar$$
		is a minimal graded free resolution of $\Idstar$. By definition the degree type of this resolution is the sequence $(d_p,\dots,d_1)$ where for each $1\leq i \leq p$ we define $d_i=c_i-c_{i-1}$. By ADHF, for each $1\leq i \leq p$ we have $c_i=n-s_i$ and $c_{i-1}=n-s_{i-1}$, and hence we have $d_i=(n-s_i)-(n-s_{i-1})=s_{i-1}-s_i$.
	\end{rem}
	
	\begin{defn}\label{Definition: offset}
		Let $\Delta$ be a PR complex and let $s_0,\dots,s_p$ be as in Definition \ref{Definition: PR Complex Degree Types} above. We call the value $s_p$ (i.e. the minimum size of a face of $\Delta$ whose link has homology) the \textit{offset} of $\Delta$.
	\end{defn}
	
	Note that a PR complex $\Delta$ has offset 0 if and only if it has nontrivial homology itself (because $\lkds[\emptyset]$ is equal to $\Delta$). Taken together, the degree type and offset of a PR complex $\Delta$ determine its dimension, as shown below.
	
	\begin{prop}\label{Proposition: Dimension of PR Complexes}
		Let $\Delta$ be a PR complex with degree type $\bd$ and offset $s$. We have $\dim \Delta = s + \sum \bd - 1$.
	\end{prop}
	\begin{proof}
		Using the notation of Definition \ref{Definition: PR Complex Degree Types} we have $\sum \bd = \sum_{i=1}^p d_i = \sum_{i=1}^p (s_i - s_{i-1}) = s_0 - s_p$. The value $s_p$ is the offset of $\Delta$ (i.e. $s=s_p$), and the value $s_0$ is the size of the facets of $\Delta$ (all of which are the same by Lemma \ref{Lemma: PR complexes are pure}). Thus $\dim \Delta = s_0 - 1 = s_p + (s_0 - s_p) - 1 = s + \sum \bd - 1$.
	\end{proof}
	
	Proposition \ref{Proposition: Dimension of PR Complexes} shows us that PR complexes with offset 0 are minimal PR complexes, in the sense that they are the PR complexes of their given degree type with minimal dimension. In fact, they are minimal in an even stronger sense: namely, any PR complex with degree type $\bd$ contains a link with the same degree type and offset $0$. To explain why, we require the following lemma, which shows that any link in a PR complex is also PR.

	\begin{lem}\label{Lemma: Links in PR complexes also PR}
		Let $\Delta$ be a PR complex with degree type $\bd=(d_p,\dots,d_1)$ and offset $s$, and let $\sigma$ be a face of $\Delta$. The complex $\lkds$ is also PR, with degree type $\bd'$ and offset $s'$ such that
		\begin{enumerate}
			\item $\bd'=(d_j,\dots,d_1)$ is a subsequence of $\bd$.
			\item $s'$ satisfies $s' + \sum \bd' + |\sigma| = s +\sum \bd$.
		\end{enumerate}
	\end{lem}
	\begin{proof}
		We assume $\Delta$ has a vertex set of size $n$, and we set $\delta=\lkds$ for notational convenience. Let $0\leq i \leq p$ and $d$ be integers. From ADHF, we have $$\beta_{i,d}(I_\delta^*)=\sum_{\substack{\tau \in \delta\\ |\tau|=n-|\sigma|-d}}\dim_{\KK}\Hred_{i-1}(\link_\delta \tau).$$ By Lemma \ref{Lemma: links in links}, the complex $\link_\delta \tau$ is equal to $\link_\Delta (\tau \sqcup \sigma)$, and hence the sum on the right-hand side of this equation is equal to $$\sum_{\substack{f\in \Delta\\ \sigma \subseteq f\\ |f|=n-d}} \dim_{\KK} \Hred_{i-1}(\link_\Delta f)$$ which appears as a summand in the ADHF decomposition of $\beta_{i,d}(I_\Delta^*)$. We conclude that $\beta_{i,d}(I_\delta^*)\leq \beta_{i,d}(I_\Delta^*)$. Thus $\beta(I_\delta^*)$ is pure and its degree type is a subsequence of the degree type of $\beta(I_\Delta^*)$.
		
		For the second part of the lemma, suppose $\bd'$ and $s'$ are the degree type and offset of $\lkds$. By Proposition \ref{Proposition: Dimension of PR Complexes} we have
		\begin{align*}
			s'+\sum \bd' - 1 &= \dim (\lkds)\\
			&= \dim \Delta - |\sigma|\\
			&= s+ \sum \bd - 1 - |\sigma|
		\end{align*}
		and the result follows.
	\end{proof}
	
	\begin{cor}\label{Corollary: Offset 0 => Minimal}
		Let $\Delta$ be a PR complex with degree type $\bd$. There exists a face $\sigma \in \Delta$ such that $\lkds$ is PR with degree type $\bd$ and offset $0$.
	\end{cor}
	\begin{proof}
		Suppose $\Delta$ has offset $s$. By definition this means that there exists some face $\sigma \in \Delta$ of size $s$ such that $\Hred_{p-1}(\lkds)\neq 0$. By Lemma \ref{Lemma: Links in PR complexes also PR} the complex $\lkds$ is PR, and its degree type is some subsequence $\bd'$ of $\bd$. Also $\lkds$ has offset $0$ because it has homology. Thus, from the second part of Lemma \ref{Lemma: Links in PR complexes also PR}, we have $\sum \bd' + |\sigma|=s+\sum \bd$. Because $|\sigma|=s$, this gives us $\sum \bd'= \sum\bd$, and hence $\bd'=\bd$.
	\end{proof}
	
	We end this section by demonstrating how the complete homology sets of a PR complex can be computed from its degree type and offset.
	\begin{prop}\label{Proposition: Alternate PR Definition With Degree Type}
		Let $\Delta$ be a simplicial complex. The following are equivalent.
		\begin{enumerate}
			\item $\Delta$ is a PR complex with degree type $(d_p,\dots,d_1)$ and offset $s$.
			\item $\hh(\Delta,m) = \begin{cases}
				\{r-1\} & \text{ if } m = s + \sum_{j=r+1}^p d_j \text{ for some } 0\leq r \leq p\\
				\emptyset & \text{ otherwise.}\\
			\end{cases}$
		\end{enumerate}
	\end{prop}
	\begin{proof}	
		Let $\Delta$ be a PR complex of the specified degree type and offset. Let $s_p<\dots<s_0$ be as in Definition \ref{Definition: PR Complex Degree Types} (so that $s_p = s$). We know from Corollary \ref{Corollary: PR Complex links have single homology} that every homology index set of $\Delta$ is either a singleton or empty, and from Corollary \ref{Corollary: Homology index sets decreasing sequence} we deduce that the same is true of the \textit{complete} homology index sets. By definition of $s_p,\dots,s_0$, the nonempty complete homology index sets are precisely the sets $\hh(\Delta,s_p),\dots,\hh(\Delta,s_0)$.
		
		Specifically, for each $0\leq r \leq p$, the set $\hh(\Delta,s_r)$ is the singleton $\{r-1\}$, and we have $s_r = s_p + (s_{p-1} - s_p) + \dots +(s_r - s_{r+1}) = s + \sum_{j=r+1}^p d_j$.
		
		Conversely, if $\Delta$ satisfies condition (2), then $\Delta$ must be PR by Proposition \ref{Proposition: Alternate PR Definition}, and we can recover its degree type and offset from the values of $m$ for which $\hh(\Delta,m)$ are nonempty.
	\end{proof}
	
	Before we proceed to look at some specific examples of PR complexes, we take a brief detour to talk about the relationship between the PR condition and the Cohen-Macaulay condition due to Reisner.
	
	\subsection{PR Complexes and Cohen-Macaulay Complexes}
	Readers familiar with Reisner's Criterion for Cohen-Macaulay complexes may well notice a striking similarity between this criterion and the PR condition, and this is not a coincidence.
	
	Reisner's Criterion is the following (this is \cite{CCA} Theorem 5.53, rephrased in terms of homology index sets; for a proof see \cite{CCA} Theorem 13.37).
	\begin{thm}[Resiner's Criterion for Cohen-Macaulayness]\label{Theorem: Reisner's Criterion}
		Let $\Delta$ be a simplicial complex. The Stanley-Reisner ring $\KK[\Delta]$ is Cohen-Macaulay if and only if for any $\sigma$ in $\Delta$, the homology index set $h(\Delta,\sigma)$ is either empty or the singleton $\{\dim \Delta - |\sigma|\}$.
	\end{thm}
	
	\begin{defn}
		We refer to complexes that satisfy Reisner's Criterion as \textit{Cohen-Macaulay complexes} (over $\KK$).
	\end{defn}
	\begin{rem}
		Just as with PR complexes, we will henceforth omit the phrase `\textit{over} $\KK$' when talking about Cohen-Macaulay complexes.
	\end{rem}
	
	Every Cohen-Macaulay complex satisfies the PR condition, and hence Cohen-Macaulay complexes are a subfamily of PR complexes. We can see this both combinatorially and algebraically.
	
	From a combinatorial perspective, suppose we have two differently sized faces $\sigma$ and $\tau$ of a Cohen-Macaulay complex $\Delta$. We have $\dim \Delta - |\sigma|\neq \dim \Delta - |\tau|$, and hence the homology index sets $h(\Delta,\sigma)$ and $h(\Delta,\tau)$ are disjoint. Thus $\Delta$ is PR by Proposition \ref{Proposition: Alternate PR Definition}.
	
	From an algebraic perspective, the result comes from the following theorem of Eagon and Reiner (see \cite{Eagon-Reiner} Theorem 3).
	\begin{thm}[Eagon-Reiner Theorem]\label{Theorem: CM iff Linear}
		Let $\Delta$ be a simplicial complex. The following are equivalent.
		\begin{enumerate}
			\item $\KK[\Delta]$ is Cohen-Macaulay.
			\item $I_{\Delta^*}$ has a linear resolution.
		\end{enumerate}
	\end{thm}
	
	A \textit{linear resolution} is a pure resolution of degree type $(1,...,1)$. Hence we can rephrase the Eagon-Reiner Theorem as follows.
	\begin{cor}\label{Corollary: CM iff PR of deg type (1...1)}
		Let $\Delta$ be a simplicial complex. The following are equivalent.
		\begin{enumerate}
			\item $\Delta$ is Cohen-Macaulay.
			\item $\Delta$ is PR with degree type $(1,\dots,1)$.
		\end{enumerate}
	\end{cor}
	
	Using this description of Cohen-Macaulay complexes we can recover Reisner's Criterion from Proposition \ref{Proposition: Alternate PR Definition With Degree Type}.
	\begin{cor}\label{Corollary: PR (1...1) Definition}
		Let $\Delta$ be a simplicial complex. The following are equivalent.
		\begin{enumerate}
			\item $\Delta$ is a PR complex with degree type $\underbrace{(1,\dots,1)}_{p}$ and offset $s$.
			\item $\hh(\Delta,m) = \begin{cases}
				\{\dim \Delta-m\} & \text{ if } s\leq m \leq s+p\\
				\emptyset & \text{ otherwise.}\\
			\end{cases}$
		\end{enumerate}
	\end{cor}
	\begin{proof}
		From Proposition \ref{Proposition: Alternate PR Definition With Degree Type}, we know that condition (1) is satisfied if and only if the nonempty complete homology index sets $\hh(\Delta,m)$ occur only at the values $m= s + p - r$ for $0\leq r\leq p$, and are equal to $\{r-1\}$. Equivalently, they occur only for values of $m$ between $s$ and $s+p$, and are equal to $\{s+p-m-1\}$. By Proposition \ref{Proposition: Dimension of PR Complexes} we have $\dim \Delta = s + p -1$. The result follows.
	\end{proof}

	\begin{rem}\label{Remark: PR (dp,...,di,1,...,1) Definition}
		A substantially identical proof shows us that if $\Delta$ has offset $s$ and degree type $(d_p,\dots,d_i,1,\dots,1)$, then for any $m > s+\sum_{j=i}^p d_j$ we have $\hh(\Delta,m)=\{\dim \Delta - m\}$. This will be crucial to our proof of Theorem \ref{Theorem: PR Complexes of Any Degree Type} in Chapter \ref{Chapter: Generating Degree Types}.\\
		
	\end{rem}
	
	Corollary \ref{Corollary: CM iff PR of deg type (1...1)} tells us that we can determine whether a PR complex is Cohen-Macaulay purely from its degree type. In fact, if we also know the offset of the complex, we have enough information to recover both the depth and dimension of the Stanley-Reisner ring (note that we are talking about the Stanley-Reisner ring of the complex itself here, rather than the dual Stanley-Reisner ring).
	
	\begin{lem}\label{Lemma: PR depth and dimension}
		Let $\Delta$ be a PR complex with degree type $\bd=(d_p,\dots,d_1)$ and offset $s$. We have the following.
		\begin{enumerate}
			\item $\dim \KK[\Delta] = s + \sum \bd$.
			\item $\depth \KK[\Delta] = s + p$.
		\end{enumerate}
	\end{lem}
	\begin{proof}
		Recall from Corollary \ref{Corollary: dim k[Delta]} that $\dim \KK[\Delta]$ is equal to $\dim \Delta + 1$, so part (1) follows immediately from Proposition \ref{Proposition: Dimension of PR Complexes}.
		
		For part (2), suppose $\Delta$ has $n$ vertices. The Auslander-Buchsbaum Formula (Theorem \ref{Theorem: Auslander-Buchsbaum}) tells us that $\depth \KK[\Delta] = n - \pdim \KK[\Delta]$. Moreover, Theorem 5.99 of \cite{CCA} tells us that $\pdim \KK[\Delta] = \reg \Idstar$.
		
		Let $\beta = \beta(\Idstar)$. Because $\Delta$ is PR of degree type $(d_p,\dots,d_1)$, we know the nonzero Betti numbers of $\Idstar$ are $\beta_{0,c_0},\dots, \beta_{p,c_p}$ with the shifts $c_0,\dots,c_p$ given by $c_p=n-s$ and $c_{i-1}=c_i-d_i$ for each $p\geq i \geq 1$. Recall from Definition \ref{Definition: Regularity} that $\reg \Idstar$ is the maximum value of $c-i$ such that $\beta_{i,c}(\Idstar)\neq 0$, which is $c_p-p$. Hence we have $\depth \KK[\Delta] = n - (c_p-p) = n - (n - s - p)= s+ p$.
	\end{proof}
	
	In particular, the difference between the dimension and the depth of $\KK[\Delta]$ is $\sum \bd - p = \sum_{i=1}^p (d_i-1)$, which is wholly dependent on the degree type of $\Delta$. Hence, the degree type of a PR complex contains a measure of the extent to which that complex fails to be Cohen-Macaulay. Seen in this light, Theorem \ref{Theorem: PR Complexes of Any Degree Type} tells us that there exist PR complexes which are arbitrarily far away from satisfying the Cohen-Macaulay property.
	
	\subsection{Examples of PR Complexes}\label{Subsection: PR Examples}
	
	We now consider a few examples of PR complexes, along with their corresponding Betti diagrams and degree types. All of these examples will be integral in motivating the constructions of PR complexes going forward.
	
	\begin{ex}\label{Example: PR Tetrahedron}
		Let $\Delta$ be the boundary of the $3$-simplex on vertex set $[4]$, and let $\beta=\beta(\Idstar)$.
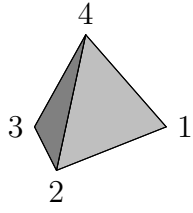
		\begin{center}
			\begin{tikzpicture}[line join = round, line cap = round]
				
				\coordinate [label=above:4] (4) at (0,{sqrt(2)},0);
				\coordinate [label=left:3] (3) at ({-.5*sqrt(3)},0,-.5);
				\coordinate [label=below:2] (2) at (0,0,1);
				\coordinate [label=right:1] (1) at ({.5*sqrt(3)},0,-.5);
				
				\begin{scope}
					\draw (1)--(3);
					\draw[fill=lightgray,fill opacity=.5] (2)--(1)--(4)--cycle;
					\draw[fill=gray,fill opacity=.5] (3)--(2)--(4)--cycle;
					\draw (2)--(1);
					\draw (2)--(3);
					\draw (3)--(4);
					\draw (2)--(4);
					\draw (1)--(4);
				\end{scope}
			\end{tikzpicture}
		\end{center}
		
		
		
		
		
		
		This complex has only $\nth[nd]{2}$ homology. It has four vertices with links of the form 
		\begin{tikzpicture}[scale = 0.2]
			\tikzstyle{point}=[circle,thick,draw=black,fill=black,inner sep=0pt,minimum width=2pt,minimum height=2pt]
			\node (a)[point] at (0,0) {};
			\node (b)[point] at (2,0) {};
			\node (c)[point] at (1,1.7) {};
			
			\draw (a.center) -- (b.center) -- (c.center) -- cycle;
		\end{tikzpicture}, six edges with links of the form \begin{tikzpicture}[scale = 0.2]
			\tikzstyle{point}=[circle,thick,draw=black,fill=black,inner sep=0pt,minimum width=2pt,minimum height=2pt]
			\node (a)[point] at (0,0) {};
			\node (c)[point] at (1.5,0) {};
		\end{tikzpicture}, and four 2-dimensional facets.
		
		Using ADHF we get that $\beta$ is equal to 
		$$\begin{array}{c | cccc}
			& 0 & 1 & 2 & 3\\
			\hline
			1 & 4 & 6 & 4 & 1\\ 
		\end{array}$$
		which means $\Delta$ is PR with shift type $(4,3,2,1)$ and degree type $(1,1,1)$.
	\end{ex}
	
	The above example generalises as follows.
	\begin{ex}\label{Example: PR Simplex}
		Let $\Delta$ be the boundary of the $n$-simplex $\partial \Delta^n$ on vertex set $[n+1]$, and let $\beta=\beta(\Idstar)$. For any integer $1\leq d \leq n+1$ there are ${n+1 \choose d}$ faces of $\Delta$ of size $n+1-d$, and for any such face $\sigma$, the complex $\lkds$ is the boundary of the $(n-|\sigma|)$-simplex on vertex set $[n+1]-\sigma$, which has homology only at degree $n-|\sigma|-1= d-2$ (and this homology has dimension $1$). Thus, $\beta$ is equal to
		$$\begin{array}{c | cccc }
			& 0 & \dots & n-1 & n \\
			\hline
			1 & {n+1 \choose 1} & \dots & {n+1 \choose n} & {n+1 \choose n+1}\\ 
		\end{array}$$
		which means $\Delta$ is PR with degree type $(\underbrace{1,\dots,1}_{n})$ (i.e. it is Cohen-Maculay).
	\end{ex}
	Hence we can obtain PR complexes with degree type $(\underbrace{1,\dots,1}_{n})$ for any $n$.
	\begin{rem}
		We can compute a minimal resolution of the ideal in Example \ref{Example: PR Simplex} directly using the Koszul complex construction, as discussed below.
		
		If $\Delta$ is the the boundary of the $(n-1)$-simplex $\partial \Delta^{n-1}$ on vertex set $[n]$, then its dual Stanley-Reisner ideal $\Idstar$ is the maximal ideal $\langle x_1,\dots,x_n\rangle$ in the polynomial ring $R=\KK[x_1,\dots,x_n]$.
		
		The sequence $x_1,\dots,x_n$ is regular, and so its corresponding Koszul complex is a minimal free resolution of $R/\Idstar$ (see \cite{C-M} Corollary 1.6.14). The corresponding resolution of $\Idstar$ is
		\[
		\begin{tikzcd}[row sep=1.5em,column sep=1.5em]
			0 \arrow{r} &R(-n) \arrow{r}{d} &R(-(n-1))^{{n \choose n-1}} \arrow{r}{d} &\dots \arrow{r}{d} &R(-1)^{n\choose 1} \arrow{r}{d} &\Idstar \arrow{r} &0.
		\end{tikzcd}
		\]
		The free modules $R(-i)^{n \choose i}$ are the exterior algebras $\bigwedge^i R^n$ of $R^n$. If$\{e_1,\dots,e_n\}$ is a basis for $R^n$, then $\bigwedge^i R^n$ has a basis consisting of elements of the form $e_{r_1}\wedge \dots \wedge e_{r_i}$ for $r_1<\dots<r_i$, and the differential maps $d$ are given by $$d(e_{r_1}\wedge \dots \wedge e_{r_i})=\sum_{j=1}^i (-1)^{j+1} e_{r_1}\wedge \dots \wedge \widehat{e_{r_j}}\wedge \dots \wedge e_{r_i}$$
		where the symbol $\widehat{e_{r_j}}$ represents that the element $e_{r_j}$ is omitted from this product. For more details on exterior algebras and Koszul complexes, see \cite{C-M} Section 1.6.
	\end{rem}
	
	\begin{rem}\label{Remark: Simplicial Spheres are CM}
		In fact Example \ref{Example: PR Simplex} is an instance of another even more general result, namely that every simplicial sphere is Cohen-Macaulay (see \cite{C-M} Section 5.4), and more specifically, any simplicial $n$-sphere has degree type $(\underbrace{1,\dots,1}_{n})$. 
	\end{rem}
	
	Next we turn our attention to a special case of Example \ref{Example: PR Simplex}: the boundary of the $1$-simplex. This is the complex consisting of two disjoint vertices, and it is PR with corresponding Betti diagram $\begin{array}{c | cc}
		& 0 & 1\\
		\hline
		1 	&	2	&	1
	\end{array}$ and degree type $(1)$. This special case admits another generalisation, as follows.
	
	\begin{ex}\label{Example: PR Disjoint Simplices}
		Let $\Delta$ be the complex $\Delta^n+\Delta^n$ consisting of two disjoint $n$-simplices, and let $\beta=\beta(\Idstar)$. The complex itself has only $\nth{0}$ homology, and all of its proper links are acyclic, except for the links of its two facets, each of which has size $n+1$. Thus $\beta$ is equal to
		$$\begin{array}{c | cc}
			& 0 & 1\\
			\hline
			n+1 	&	2		&	.	\\
			\vdots	&	.		&	.	\\
			2n+1	&	.		&	1
		\end{array}$$
		which means $\Delta$ is PR with degree type $(n+1)$
	\end{ex}
	Hence we can obtain PR complexes with degree type $(n)$ for any $n$.
	\begin{rem}
		Just as with Example \ref{Example: PR Simplex} we can compute the ideal in Example \ref{Example: PR Disjoint Simplices} directly using the Koszul complex. Specifically, if $\Delta$ is the disjoint union of two $(n-1)$-simplices $\Delta^{n-1}+\Delta^{n-1}$ on vertex set $[2n]$ then its dual Stanley-Reisner ideal $\Idstar$ is the ideal $\langle x_1\dots x_n, x_{n+1}\dots x_{2n} \rangle$ in the polynomial ring $R=\KK[x_1,\dots,x_{2n}]$. For notational convenience, we set $f=x_1\dots x_{n+1}$ and $g = x_{n+1}\dots x_{2n}$. Because $f,g$ is a regular sequence in $R$, its corresponding Koszul complex is a minimal free resolution of $R/\Idstar$. The corresponding resolution of $\Idstar$ is
		$$\begin{tikzcd}[ampersand replacement=\&]
			0 \rar \& R(-2n) \rar{ \begin{pmatrix}
					g \\ -f
			\end{pmatrix} }\&[1em] R(-n)^2 \rar{ \begin{pmatrix}
					f & g
			\end{pmatrix} } \&[3em] \Idstar \rar \& 0.
		\end{tikzcd}$$
	\end{rem}
	
	The following three examples are of PR complexes with more interesting degree types.
	
	\begin{ex}\label{Example: PR Zelda Symbol}
		Let $\Delta$ be the complex
		$$
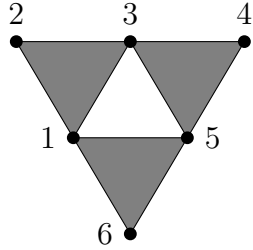
\begin{tikzpicture}[scale = 0.75]
			\tikzstyle{point}=[circle,thick,draw=black,fill=black,inner sep=0pt,minimum width=4pt,minimum height=4pt]
			\node (a)[point,label=above:$2$] at (0,3.4) {};
			\node (b)[point,label=above:$3$] at (2,3.4) {};
			\node (c)[point,label=above:$4$] at (4,3.4) {};
			\node (d)[point,label=left:$1$] at (1,1.7) {};
			\node (e)[point,label=right:$5$] at (3,1.7) {};
			\node (f)[point,label=left:$6$] at (2,0) {};	
			
			\begin{scope}[on background layer]
				\draw[fill=gray] (a.center) -- (b.center) -- (d.center) -- cycle;
				\draw[fill=gray] (b.center) -- (c.center) -- (e.center) -- cycle;
				\draw[fill=gray]   (d.center) -- (e.center) -- (f.center) -- cycle;
			\end{scope}
		\end{tikzpicture}$$
		on vertex set $[6]$, and let $\beta=\beta(\Idstar)$. 
		
		
			
			
			
			This complex has only $\nth[st]{1}$ homology. It has three vertices with links of the form \begin{tikzpicture}[scale = 0.2]
				\tikzstyle{point}=[circle,thick,draw=black,fill=black,inner sep=0pt,minimum width=2pt,minimum height=2pt]
				\node (a)[point] at (0,0) {};
				\node (b)[point] at (0,2) {};
				\node (c)[point] at (1.5,0) {};
				\node (d)[point] at (1.5,2) {};
				
				\draw[fill=gray] (a.center) -- (b.center);
				\draw[fill=gray] (c.center) -- (d.center);
			\end{tikzpicture}, three facets of size 3, and the links of all of its other faces are acyclic. Using ADHF we see that $\beta$ is equal to $$\begin{array}{c | ccc }
				& 0 & 1 & 2\\
				\hline
				3 & 3 & . & .\\ 
				4 & . & 3 & 1\\ 
			\end{array}$$
			which means $\Delta$ is PR, with shift type $(6,5,3)$ and degree type $(1,2)$.
		\end{ex}

		\begin{ex}\label{Example: PR 3D Zelda Symbol}
			Let $\Delta$ be the complex on vertex set $[9]$, consisting of three tetrahedra, connected together at vertices as shown below.
			
			
			$$
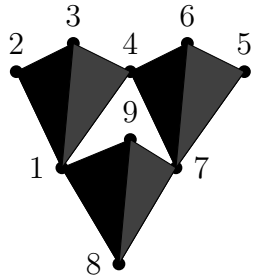
\begin{tikzpicture}[scale = 0.75]
				\tikzstyle{point}=[circle,thick,draw=black,fill=black,inner sep=0pt,minimum width=4pt,minimum height=4pt]
				\node (a)[point,label=above:$2$] at (0,3.4) {};
				\node (b)[point,label=above:$4$] at (2,3.4) {};
				\node (c)[point,label=above:$5$] at (4,3.4) {};
				\node (d)[point,label=left:$1$] at (0.8,1.7) {};
				\node (e)[point,label=right:$7$] at (2.8,1.7) {};
				\node (f)[point,label=left:$8$] at (1.8,0) {};
				
				\node (g)[point,label=above:$3$] at (1,3.9) {};
				\node (h)[point,label=above:$6$] at (3,3.9) {};
				\node (i)[point,label=above:$9$] at (2,2.2) {};
				
				\draw [fill=lightgray, fill opacity=.5] (a.center) -- (b.center) -- (d.center) -- cycle;
				\draw [fill=lightgray, fill opacity=.5] (b.center) -- (c.center) -- (e.center) -- cycle;
				\draw [fill=lightgray, fill opacity=.7] (d.center) -- (e.center) -- (f.center) -- cycle;
				
				\draw[fill=black, fill opacity=1] (a.center)--(d.center)--(g.center)--cycle;
				\draw[fill=darkgray,fill opacity=.7] (b.center)--(d.center)--(g.center)--cycle;
				
				\draw[fill=black,fill opacity=1] (b.center)--(e.center)--(h.center)--cycle;
				\draw[fill=darkgray,fill opacity=.7] (c.center)--(e.center)--(h.center)--cycle;
				
				\draw[fill=black,fill opacity=1] (d.center)--(f.center)--(i.center)--cycle;
				\draw[fill=darkgray,fill opacity=.7] (e.center)--(f.center)--(i.center)--cycle;
			\end{tikzpicture}$$
			Let $\beta=\beta(\Idstar)$. This complex has only $\nth[st]{1}$ homology. It has three vertices with links of the form \begin{tikzpicture}[scale = 0.2]
				\tikzstyle{point}=[circle,thick,draw=black,fill=black,inner sep=0pt,minimum width=2pt,minimum height=2pt]
				\node (a)[point] at (0,0) {};
				\node (b)[point] at (2,0) {};
				\node (c)[point] at (1,1.7) {};
				\node (d)[point] at (3,0) {};
				\node (e)[point] at (5,0) {};
				\node (f)[point] at (4,1.7) {};
				
				\draw[fill=gray] (a.center) -- (b.center) -- (c.center) -- cycle;
				\draw[fill=gray] (d.center) -- (e.center) -- (f.center) -- cycle;
			\end{tikzpicture}, three facets of size 4, and the links of all of its other faces are acyclic. Using ADHF we see that $\beta$ is equal to $$\begin{array}{c | ccc}
				&	0 & 1 & 2\\
				\hline
				5 & 3 & . & .\\
				6 & . & . & .\\
				7 & . & 3 & 1\\ 
			\end{array}$$
			which means $\Delta$ is PR, with shift type $(9,8,5)$ and degree type $(1,3)$.
		\end{ex}

		\begin{ex}\label{Example: PR 21 Complex}
			Let $\Delta$ be the complex
			$$
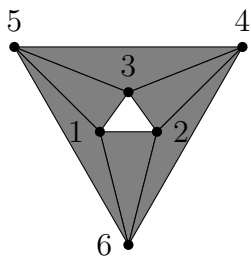
\begin{tikzpicture}[scale = 0.75]
				\tikzstyle{point}=[circle,thick,draw=black,fill=black,inner sep=0pt,minimum width=3pt,minimum height=3pt]
				\node (a)[point,label=above:$5$] at (0,3.5) {};
				\node (b)[point,label=above:$3$] at (2,2.7) {};
				\node (c)[point,label=above:$4$] at (4,3.5) {};
				\node (d)[point,label=left:$1$] at (1.5,2) {};
				\node (e)[point,label=right:$2$] at (2.5,2) {};
				\node (f)[point,label=left:$6$] at (2,0) {};	
				
				\begin{scope}[on background layer]
					\draw[fill=gray] (a.center) -- (b.center) -- (d.center) -- cycle;
					\draw[fill=gray] (b.center) -- (c.center) -- (e.center) -- cycle;
					\draw[fill=gray]   (d.center) -- (e.center) -- (f.center) -- cycle;
					
					\draw[fill=gray] (a.center) -- (b.center) -- (c.center) -- cycle;
					\draw[fill=gray] (a.center) -- (f.center) -- (d.center) -- cycle;
					\draw[fill=gray] (c.center) -- (f.center) -- (e.center) -- cycle;
				\end{scope}
			\end{tikzpicture}$$
			on vertex set $[6]$, and let $\beta=\beta(\Idstar)$. This complex has only $\nth[st]{1}$ homology. It has six edges with links of the form \begin{tikzpicture}[scale = 0.2]
				\tikzstyle{point}=[circle,thick,draw=black,fill=black,inner sep=0pt,minimum width=2pt,minimum height=2pt]
				\node (a)[point] at (0,0) {};
				\node (b)[point] at (2,0) {};
			\end{tikzpicture}, six facets of size 6, and the links of all of its other faces are acyclic. Thus using ADHF, we see that $\beta$ is equal to
			$$\begin{array}{c | ccc}
				& 0 & 1 & 2\\
				\hline
				3 & 6 & 6 & .\\
				4 & . & . & 1\\ 
			\end{array}$$
			which means $\Delta$ is PR, with shift type $(6,4,3)$ and degree type $(2,1)$.
		\end{ex}
		
		All of the above examples of PR complexes have nontrivial homology at some degree, and therefore have offset $0$. The following is an example of a PR complex with a nonzero offset.
		\begin{ex}\label{Example: PR Complex Offset 2}
			Let $\Delta$ be the complex
			$$
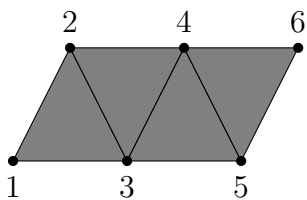
\begin{tikzpicture}[scale = 0.75]
				\tikzstyle{point}=[circle,thick,draw=black,fill=black,inner sep=0pt,minimum width=3pt,minimum height=3pt]
				\node (a)[point,label=below:$1$] at (0,0) {};
				\node (b)[point,label=above:$2$] at (1,2) {};
				\node (c)[point,label=below:$3$] at (2,0) {};
				\node (d)[point,label=above:$4$] at (3,2) {};
				\node (e)[point,label=below:$5$] at (4,0) {};	
				\node (f)[point,label=above:$6$] at (5,2) {};	
				
				\begin{scope}[on background layer]
					\draw[fill=gray] (a.center) -- (b.center) -- (c.center) -- cycle;
					\draw[fill=gray] (b.center) -- (c.center) -- (d.center) -- cycle;
					\draw[fill=gray] (c.center) -- (d.center) -- (e.center) -- cycle;
					\draw[fill=gray] (d.center) -- (e.center) -- (f.center) -- cycle;
				\end{scope}
			\end{tikzpicture}$$
			on vertex set $[6]$, and let $\beta=\beta(\Idstar)$. This complex has three edges with links of the form $\begin{tikzpicture}[scale = 0.2]
				\tikzstyle{point}=[circle,thick,draw=black,fill=black,inner sep=0pt,minimum width=2pt,minimum height=2pt]
				\node (a)[point] at (0,0) {};
				\node (b)[point] at (1.5,0) {};
			\end{tikzpicture}$ and four facets of size $3$. The links of all of its other faces are acyclic. Thus using ADHF, we see that $\beta$ is equal to
			$$\begin{array}{c | cc}
				& 0 & 1\\
				\hline
				3 & 4 & 3 \\
			\end{array}$$
			which means $\Delta$ is PR, with shift type $(4,3)$, degree type $(1)$. It also has offset $2$ because the smallest faces whose links have homology have size $2$.
		\end{ex}
		
		We end this section by considering some \textit{non}-examples. The following three complexes are \textit{almost} PR, in the sense that their corresponding Betti diagrams only have a single column with multiple entries, and hence they only fail to be PR in one aspect.
		\begin{ex}\label{Example: NOT-PR 1}
			Let $\Delta$ be the complex
			$$
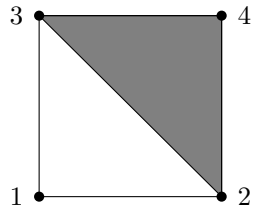
\begin{tikzpicture}[scale = 0.6]
				\tikzstyle{point}=[circle,thick,draw=black,fill=black,inner sep=0pt,minimum width=3pt,minimum height=3pt]
				\node (a)[point, label=left:\footnotesize $1$] at (0,0){};
				\node (b)[point, label=right:\footnotesize $2$] at (4,0){};
				\node (c)[point, label=right:\footnotesize $4$] at (4,4){};
				\node (d)[point, label=left:\footnotesize $3$] at (0,4){};	
				\draw[black] (a.center) -- (b.center) -- (c.center) -- (d.center) -- cycle;
				\begin{scope}[on background layer]
					\draw[fill=gray] (b.center) -- (c.center) -- (d.center) -- cycle;
				\end{scope}
			\end{tikzpicture}$$
			on vertex set $[4]$, and let $\beta=\beta(\Idstar)$.
			
			This complex is not pure, because it has one $2$-dimensional facet and two $1$-dimensional facets. Thus it cannot be PR, because there are multiple faces of different sizes whose links have $\nth[st]{(-1)}$ homology. The Betti diagram $\beta$ is equal to
			$$\begin{array}{c | ccc}
				& 0 & 1 & 2\\
				\hline
				1 & 1 & . & .\\
				2 & 2 & 3 & 1\\ 
			\end{array}$$
		\end{ex}
		
		\begin{ex}\label{Example: NOT-PR 2}
			Let $\Delta$ be the complex
			$$
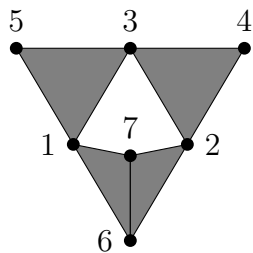
\begin{tikzpicture}[scale = 0.75]
				\tikzstyle{point}=[circle,thick,draw=black,fill=black,inner sep=0pt,minimum width=4pt,minimum height=4pt]
				\node (a)[point,label=above:$5$] at (0,3.4) {};
				\node (b)[point,label=above:$3$] at (2,3.4) {};
				\node (c)[point,label=above:$4$] at (4,3.4) {};
				\node (d)[point,label=left:$1$] at (1,1.7) {};
				\node (e)[point,label=right:$2$] at (3,1.7) {};
				\node (f)[point,label=left:$6$] at (2,0) {};
				\node (g)[point,label=above:$7$] at (2,1.5) {};	
				
				\begin{scope}[on background layer]
					\draw[fill=gray] (a.center) -- (b.center) -- (d.center) -- cycle;
					\draw[fill=gray] (b.center) -- (c.center) -- (e.center) -- cycle;
					\draw[fill=gray]   (d.center) -- (f.center) -- (g.center) -- cycle;
					\draw[fill=gray]   (e.center) -- (f.center) -- (g.center) -- cycle;
				\end{scope}
			\end{tikzpicture}$$
			on vertex set $[7]$, and let $\beta=\beta(\Idstar)$.
			
			This complex is pure, but the intersections between adjacent facets are not all of the same size: the facets $\{1,6,7\}$ and $\{2,6,7\}$ intersect at a line, while all other adjacent facets intersect at vertices. This means that $\Delta$ has one disconnected link corresponding to a face of size $2$, and three corresponding to faces of size $1$. Thus it cannot be PR, because there are multiple faces of different sizes whose links have $\nth{0}$ homology. The Betti diagram $\beta$ is equal to
			$$\begin{array}{c | ccc}
				& 0 & 1 & 2\\
				\hline
				4 & 4 & 1 & .\\
				5 & . & 3 & 1
			\end{array}$$
		\end{ex}
		
		\begin{ex}\label{Example: NOT-PR 3}
			Let $\Delta$ be the complex
			$$
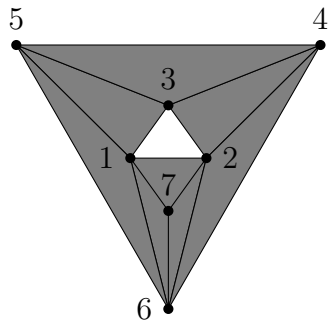
\begin{tikzpicture}[scale = 1]
				\tikzstyle{point}=[circle,thick,draw=black,fill=black,inner sep=0pt,minimum width=3pt,minimum height=3pt]
				\node (a)[point,label=above:$5$] at (0,3.5) {};
				\node (b)[point,label=above:$3$] at (2,2.7) {};
				\node (c)[point,label=above:$4$] at (4,3.5) {};
				\node (d)[point,label=left:$1$] at (1.5,2) {};
				\node (e)[point,label=right:$2$] at (2.5,2) {};
				\node (f)[point,label=left:$6$] at (2,0) {};
				\node (g)[point,label=above:$7$] at (2,1.3) {};
				
				\begin{scope}[on background layer]
					\draw[fill=gray] (a.center) -- (b.center) -- (d.center) -- cycle;
					\draw[fill=gray] (b.center) -- (c.center) -- (e.center) -- cycle;
					\draw[fill=gray]   (d.center) -- (e.center) -- (f.center) -- cycle;
					
					\draw[fill=gray] (a.center) -- (b.center) -- (c.center) -- cycle;
					\draw[fill=gray] (a.center) -- (f.center) -- (d.center) -- cycle;
					\draw[fill=gray] (c.center) -- (f.center) -- (e.center) -- cycle;
					
					\draw (d.center) -- (g.center);
					\draw (e.center) -- (g.center);
					\draw (f.center) -- (g.center);
				\end{scope}
			\end{tikzpicture}$$
			on vertex set $[7]$, and let $\beta=\beta(\Idstar)$.
			
			This complex is pure, and all disconnected links correspond to faces of the same size, but both the complex itself and the the link of the vertex $7$ have nontrivial $\nth[st]{1}$ homology, so it is not a PR complex. The Betti diagram $\beta$ is equal to
			$$\begin{array}{c | ccc}
				& 0 & 1 & 2\\
				\hline
				4 & 8 & 9 & 1\\
				5 & . & . & 1\\ 
			\end{array}$$
		\end{ex}
		
		\section{Tools for Studying PR Complexes}
		In this section we consider two tools for studying PR complexes: \textit{maximal intersections} and \textit{the link poset}.
		
		
		
		\subsection{Maximal Intersections}\label{Subsection: Maximal Intersections}
		One of the first combinatorial properties we noted about PR complexes was that they are pure (i.e. all of their facets have the same size). In fact, this property is equivalent to the condition that column $0$ of $\beta(\Idstar)$ has at most a single nonzero entry. In other words, for \textit{any} complex $\Delta$ we have that $\Delta$ is pure if and only if the total Betti number $\beta_0(\Idstar)$ is pure, where the purity of a total Betti number is defined as follows.
		
		\begin{defn}\label{Definition: Purity of Single Column of Betti Diagram}
			Let $\beta$ be a Betti diagram. We say the total Betti number $\beta_i$ is \textit{pure} if there only exists at most a single integer $c$ such that $\beta_{i,c}\neq 0$.
		\end{defn}
		
		In this section we examine the necessary and sufficient conditions for $\beta_1(\Idstar)$ to be pure, for an arbitrary fixed complex $\Delta$ (in particular, these are conditions satisfied by every PR complex). We begin with the following definition.
		
		\begin{defn}\label{Definition: Maximal Intersections}
			A \textit{maximal intersection} in $\Delta$ is a face $\sigma \in \Delta$ that is an intersection of more than one facet of $\Delta$, and is maximal with this property with respect to inclusion.
		\end{defn}
		\begin{rem}
			It is possible for a maximal intersection to be empty, if no two facets of $\Delta$ intersect.
		\end{rem}
		
		Note that if $\sigma=F_1\cap \dots \cap F_n$ is a maximal intersection, then for every $1\leq i < j\leq n$ we have $\sigma \subseteq F_i\cap F_j$, and hence by maximality $\sigma = F_i\cap F_j$. Thus every maximal intersection can always be expressed as the intersection of two facets.
		
		\begin{defn}\label{Definition: Adjacent Facets}
			Whenever $F_1\cap F_2$ is a maximal intersection we say that $F_1$ and $F_2$ are \textit{adjacent}.
		\end{defn}
		\begin{ex}\label{Example: Adjacent Facets}
			In the complex
			$$
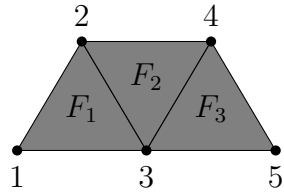
\begin{tikzpicture}[scale = 0.85]
				\tikzstyle{point}=[circle,thick,draw=black,fill=black,inner sep=0pt,minimum width=3pt,minimum height=3pt]
				\node (a)[point,label=below:$1$] at (0,0) {};
				\node (b)[point,label=above:$2$] at (1,1.7) {};
				\node (c)[point,label=below:$3$] at (2,0) {};
				\node (d)[point,label=above:$4$] at (3,1.7) {};
				\node (e)[point,label=below:$5$] at (4,0) {};
				
				\node (f)[label= {$F_1$}] at (1,0.1) {};
				\node (g)[label= {$F_2$}] at (2,0.6) {};
				\node (h)[label= {$F_3$}] at (3,0.1) {};

				\begin{scope}[on background layer]
					\draw[fill=gray, fill opacity = .8] (a.center) -- (b.center) -- (c.center) -- cycle;
					\draw[fill=gray, fill opacity = .8] (b.center) -- (c.center) -- (d.center) -- cycle;
					\draw[fill=gray, fill opacity = .8] (c.center) -- (d.center) -- (e.center) -- cycle;
				\end{scope}
			\end{tikzpicture}$$ the facets $F_1$ and $F_2$ are adjacent (because they intersect at the maximal intersection $\{2,3\}$) as are the facets $F_2$ and $F_3$ (which intersect at the maximal intersection $\{3,4\}$). The facets $F_1$ and $F_3$ intersect at the vertex $\{3\}$ but this intersection is not maximal, so these two facets are not adjacent.
		\end{ex}
		
		The following proposition gives us a second characterisation of maximal intersections.
		\begin{prop}\label{Proposition: Maximal Intersections Links}
			Let $\sigma$ be a face of $\Delta$. The following are equivalent.
			\begin{enumerate}
				\item $\sigma$ is a maximal intersection.
				\item $\lkds$ is a disjoint union of multiple simplices.
			\end{enumerate}
		\end{prop}
		\begin{proof}
			Suppose $\sigma$ is a maximal intersection, and let $F_1,\dots,F_n$ be all of the facets which contain it, so that $\sigma = F_1\cap\dots\cap F_n$. This means that $\lkds =\langle F_1 -\sigma, \dots, F_n-\sigma\rangle$. Note that for any $1\leq i<j\leq n$ we have $\sigma = F_i \cap F_j$ and hence $(F_i-\sigma)\cap(F_j-\sigma)=\emptyset$. Thus the defining facets of $\link_\Delta \sigma$ are pairwise disjoint.
			
			Now suppose that $\lkds$ is a disjoint union of multiple simplices $\tau_1,\dots,\tau_n$, and for each $1\leq i \leq n$, define $F_i=\tau_i\sqcup \sigma$. We have $\sigma = F_1\cap \dots \cap F_n$, and the only way to extend $\sigma$ to a larger face $\widetilde{\sigma}$ of $\Delta$ is to add vertices from one of the simplices $\tau_i$, in which case $\widetilde{\sigma}$ would be contained in only the facet $F_i$. Thus $\sigma$ is a maximal intersection.
		\end{proof}
		\begin{cor}
			Let $\Delta$ be a simplicial complex, and let $\sigma$ be in $\Delta$. A face $\tau$ of $L=\link_\Delta \sigma$ is a maximal intersection in $L$ if and only if $\tau \sqcup \sigma$ is a maximal intersection in $\Delta$.
		\end{cor}
		\begin{proof}
			Note that $\link_L \tau = \link_\Delta (\tau \sqcup \sigma)$, so this follows directly from Proposition \ref{Proposition: Maximal Intersections Links}.
		\end{proof}
		
		Proposition $\ref{Proposition: Maximal Intersections Links}$ shows that for any maximal intersection $\sigma$ in $\Delta$, we have that $\Hred_0(\link_\Delta \sigma)\neq 0$. In fact, in the case where $\beta_1(\Idstar)$ is pure, we have more than this.
		
		\begin{prop}\label{Proposition: beta_1 Purity Condition}
			The total Betti number $\beta_1(\Idstar)$ is pure if and only if $\Delta$ satisfies the following two conditions.
			\begin{enumerate}
				\item The only disconnected links in $\Delta$ are the links of maximal intersections.
				\item All maximal intersections in $\Delta$ have the same size.
			\end{enumerate}
		\end{prop}
		\begin{proof}
			Suppose $\beta_1(\Idstar)$ is pure, and let $\sigma \in \Delta$ be a face whose link is disconnected. In particular, $\sigma$ must be an intersection of multiple facets of $\Delta$, and is therefore contained in a maximal intersection $\tau$. If the containment were strict then we would have two faces of $\Delta$ of different sizes both of which have links with nontrivial $\nth{0}$ homology, which would contradict the purity of $\beta_1(\Idstar)$. We cannot have two maximal intersections of different sizes for the same reason.
			
			Conversely, suppose that the only faces of $\Delta$ with disconnected links are maximal intersections, and all of these have size $s$. Then by ADHF we have that $\beta_{1,d}(\Idstar)$ is nonzero if and only if $d=n-s$, where $n$ is the number of vertices of $\Delta$.
		\end{proof}
		
		\begin{prop}\label{Proposition: Every Facet Contains a Maximal Intersection}
			Suppose that $\Delta$ is a complex with more than one facet, and that $\beta_1(\Idstar)$ is pure. Let $F$ be a facet of $\Delta$ and $\tau$ a face contained in $F$ and at least one other facet. There exists some maximal intersection $M$ such that $\tau \subseteq M \subsetneqq F $. In particular, every facet of $\Delta$ contains a maximal intersection.
		\end{prop}
		\begin{proof}
			Let $G_1$ be another facet containing $\tau$, so that $\tau \subseteq F\cap G_1$. If $L=\link_\Delta (F\cap G_1)$ is disconnected, then $F\cap G_1$ must be a maximal intersection by Proposition \ref{Proposition: beta_1 Purity Condition}. Otherwise $L$ is connected, which means there must be some other facet $G_2$ containing $F\cap G_1$ such that $G_2-F\cap G_1$ has nonempty intersection with $F-F\cap G_1$. Hence we have $F\cap G_1 \subsetneqq F\cap G_2$.
			
			Continuing in this way we obtain a sequence $\tau \subseteq F\cap G_1 \subsetneqq F\cap G_2 \subsetneqq F\cap G_3 \subsetneqq \dots$, and this sequence must terminate, because $\Delta$ only has a finite number of facets. The sequence can only terminate when we have found a facet $G_m$ such that $\link_\Delta(F\cap G_m)$ is disconnected, which means $M=F\cap G_m$ is a maximal intersection satisfying $\tau \subseteq M \subsetneqq F$.
		\end{proof}
		
		Note that Proposition \ref{Proposition: Every Facet Contains a Maximal Intersection} is not true in general for arbitrary complexes. For instance, in the complex
		
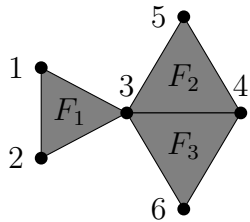
\begin{center}
			\begin{tikzpicture}[scale = 0.75]
				\tikzstyle{point}=[circle,thick,draw=black,fill=black,inner sep=0pt,minimum width=4pt,minimum height=4pt]
				
				\node (a)[point,label=left:$1$] at (-0.5,2.5) {};
				\node (b)[point,label=left:$2$] at (-0.5,0.9) {};
				\node (c)[point,label=left:$5$] at (2,3.4) {};
				\node (d)[point,label=above:$3$] at (1,1.7) {};
				\node (e)[point,label=above:$4$] at (3,1.7) {};
				\node (f)[point,label=left:$6$] at (2,0) {};	
				
				\node (g)[label= {$F_1$}] at (0,1.1) {};
				\node (h)[label= {$F_2$}] at (2,1.7) {};
				\node (i)[label= {$F_3$}] at (2,0.5) {};
				
				\begin{scope}[on background layer]
					\draw[fill=gray, fill opacity = .8] (a.center) -- (b.center) -- (d.center) -- cycle;
					\draw[fill=gray, fill opacity = .8] (c.center) -- (d.center) -- (e.center) -- cycle;
					\draw[fill=gray, fill opacity = .8] (d.center) -- (e.center) -- (f.center) -- cycle;
				\end{scope}
			\end{tikzpicture}
		\end{center}
		the facet $F_1$ contains no maximal intersections, because while it intersects with both of the other two facets at the vertex $\{3\}$, this intersection is strictly contained in the maximal intersection $\{3,4\}$. In other words, $F_1$ is not adjacent to any other facets.
		
		
		\subsection{The Link Poset}
		In the final part of this chapter we present a poset structure that has arisen naturally from our studies of PR complexes.
		
		We have found this structure to be particularly useful in developing an intuition about the PR condition. For this reason we will give examples of the link posets for a number of the complexes we look at in the next chapter.
		
		\begin{defn}\label{Definition: Link Poset}
			Let $\Delta$ be a simplicial complex. We define the \textit{link poset} $P_\Delta$ of $\Delta$ to be the set of complexes $\lkds$ for $\sigma\in \Delta$. We impose a poset structure on this set by defining $\delta_1\geq \delta_2$ whenever there is some $\tau \in \delta_1$ such that $\delta_2=\link_{\delta_1}\tau$.
		\end{defn}
		\begin{notation}\label{Notation: Link Poset}
			When representing link posets pictorially we will often use $\delta_1 \rightarrow \delta_2$ to denote the relation $\delta_1 \geq \delta_2$, and more specifically $\delta_1 \arrowX{\tau} \delta_2$ to denote the relation $\link_{\delta_1}\tau = \delta_2$. We will sometimes indicate which complexes in the chain have homology by enclosing them in a rectangle. When we do this, we will make a note of the degree(s) of the nontrivial homologies directly below the complex, and denote acyclic complexes with a $\times$ symbol. For example the following diagram
			\[
			\begin{tikzcd}[row sep=0.5em]
				|[draw=red, line width =.5mm, rectangle]|\delta_1 \arrow{r}{\tau}& \delta_2\\
				\Hred_0& \times
			\end{tikzcd}
			\]
			denotes that $\delta_2=\link_{\delta_1}\tau$, and $\Hred_0(\delta_1)\neq 0$, while $\delta_2$ is acyclic.
		\end{notation}
		\begin{rem}\label{Remark: Faces not Part of Data of Link Poset}
			In general, a complex may contain multiple faces which have the same link, and hence there may be multiple options for the face $\tau$ in the notation $\delta_1 \arrowX{\tau} \delta_2$. For example, if $\Delta$ is the complex $$
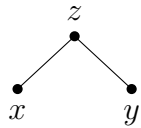
\begin{tikzpicture}[scale = 0.5]
				\tikzstyle{point}=[circle,thick,draw=black,fill=black,inner sep=0pt,minimum width=3pt,minimum height=3pt]
				
				\node (1)[point, label=below:$x$] at (9,1) {};
				\node (2)[point, label=below:$y$] at (12,1) {};
				\node (3)[point, label=above:$z$] at (10.5,2.4) {};
				
				\draw (1.center) -- (3.center) -- (2.center);
			\end{tikzpicture}$$ then the vertices $x$ and $y$ both have the same link in $\Delta$, and hence we could denote the relation
			$$
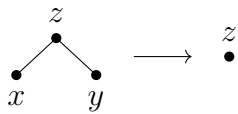
\begin{tikzpicture}[scale = 0.35]
				\tikzstyle{point}=[circle,thick,draw=black,fill=black,inner sep=0pt,minimum width=3pt,minimum height=3pt]
				
				\node (1)[point, label=below:$x$] at (9,1) {};
				\node (2)[point, label=below:$y$] at (12,1) {};
				\node (3)[point, label=above:$z$] at (10.5,2.4) {};
				
				\draw (1.center) -- (3.center) -- (2.center);
				
				\node (b1) at (13,1.7) {};
				\node (b2) at (16,1.7) {};
				\node (y)  at (14.5,1.6) {};
				\draw[->] (b1) -- (b2);
				
				\node (one)[point, label=above:$z$] at (17,1.7) {};
			\end{tikzpicture}$$
			as either
			
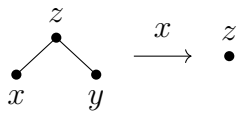
$$\begin{tikzpicture}[scale = 0.35]
				\tikzstyle{point}=[circle,thick,draw=black,fill=black,inner sep=0pt,minimum width=3pt,minimum height=3pt]
				
				\node (1)[point, label=below:$x$] at (9,1) {};
				\node (2)[point, label=below:$y$] at (12,1) {};
				\node (3)[point, label=above:$z$] at (10.5,2.4) {};
				
				\draw (1.center) -- (3.center) -- (2.center);
				
				\node (b1) at (13,1.7) {};
				\node (b2) at (16,1.7) {};
				\node (y)[label=above:$x$]  at (14.5,1.6) {};
				\draw[->] (b1) -- (b2);
				
				\node (one)[point, label=above:$z$] at (17,1.7) {};
			\end{tikzpicture}$$
			or
			
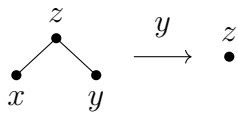
$$\begin{tikzpicture}[scale = 0.35]
				\tikzstyle{point}=[circle,thick,draw=black,fill=black,inner sep=0pt,minimum width=3pt,minimum height=3pt]
				
				\node (1)[point, label=below:$x$] at (9,1) {};
				\node (2)[point, label=below:$y$] at (12,1) {};
				\node (3)[point, label=above:$z$] at (10.5,2.4) {};
				
				\draw (1.center) -- (3.center) -- (2.center);
				
				\node (b1) at (13,1.7) {};
				\node (b2) at (16,1.7) {};
				\node (y)[label=above:$y$]  at (14.5,1.6) {};
				\draw[->] (b1) -- (b2);
				
				\node (one)[point, label=above:$z$] at (17,1.7) {};
			\end{tikzpicture}$$
		\end{rem}
		
		Note that for any complex $\Delta$, the link poset $P_\Delta$ contains both a maximal element $\Delta$ (which is $\link_\Delta \emptyset$) and a minimal element $\{\emptyset\}$ (which is $\link_\Delta F$ for any facet $F$ in $\Delta$).
		\begin{defn}\label{Definition: Maximal Chains}
			We say a chain $\Delta=\delta_0>...>\delta_m=\{\emptyset\}$ in $P_\Delta$ is \textit{maximal} if for each $1\leq i \leq m$, there is some vertex $v_i$ in $\delta_i$ such that $\delta_i= \link_{\delta_{i-1}}v_i$.
		\end{defn}
		
		By Lemma \ref{Lemma: links in links}, every chain in $P_\Delta$ can be extended to a maximal chain. Note that for any facet $F$ of $\Delta$, an ordering $v_1,\dots,v_m$ of the elements of $F$ corresponds to a maximal chain $\Delta>\link_{\Delta}\{v_1\}>\link_{\Delta}\{v_1,v_2\}>\dots>\link_{\Delta}\{v_1,...,v_m\}=\{\emptyset\}$ in $P_\Delta$. However this correspondence between ordered facets and maximal chains is not bijective, because it is possible for two distinct ordered facets to correspond to the same maximal chain. For example, if $\Delta$ is the complex presented in Remark \ref{Remark: Faces not Part of Data of Link Poset}, then the maximal chain $$
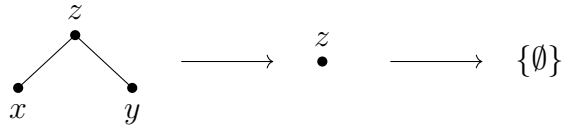
\begin{tikzpicture}[scale = 0.5]
			\tikzstyle{point}=[circle,thick,draw=black,fill=black,inner sep=0pt,minimum width=3pt,minimum height=3pt]
			
			\node (1)[point, label=below:$x$] at (9,1) {};
			\node (2)[point, label=below:$y$] at (12,1) {};
			\node (4)[point, label=above:$z$] at (10.5,2.4) {};
			
			\draw (1.center) -- (4.center) -- (2.center);
			
			\node (b1) at (13,1.7) {};
			\node (b2) at (16,1.7) {};
			\node (y)  at (14.5,1.6) {};
			\draw[->] (b1) -- (b2);
			
			\node (one)[point, label=above:$z$] at (17,1.7) {};
			
			\node (c1) at (18.5,1.7) {};
			\node (c2) at (21.5,1.7) {};
			\node (z)  at (20,1.6) {};
			\draw[->] (c1) -- (c2);
			
			\node (q)[label=right:$\{\emptyset\}$] at (21.5,1.7){};
		\end{tikzpicture}$$ corresponds to both the facets $\{x,z\}$ and $\{y,z\}$.

		Our definitions of purity, degree type and offset extend naturally to maximal chains in $P_\Delta$ as follows.
		\begin{defn}\label{Definition: Maximal Chain Purity and Degree Type}
			Let $\Delta=\delta_0>...>\delta_m=\{\emptyset\}$ be a maximal chain in $P_\Delta$. We say this chain is \textit{pure} if it satisfies the following two conditions. 
			\begin{enumerate}
				\item For each $0\leq i \leq m$, the homology index set $h(\delta_i)$ is either empty or a singleton.
				\item If $\delta_{i_p}>\dots>\delta_{i_0}$ are all the complexes in the chain with homology, the degrees of these homologies form a decreasing sequence.
			\end{enumerate}
			If the decreasing sequence of homology degrees in condition (2) is consecutive (i.e. if it is the sequence $p-1,p-2,\dots,0,-1$) we say the chain is \textit{totally pure}. In this case we define the \textit{degree type} of the chain to be  $(i_{p-1}-i_p,\dots,i_0-i_1)$, and the \textit{offset} to be $i_p$. 
		\end{defn}

		It should be noted that while every maximal chain in the link poset of a PR complex is pure, they need not all be \textit{totally} pure. Moreover, in the case where the offset of the complex is nonzero, the link poset may even contain totally pure chains of a different degree type and offset to the complex itself.
		\begin{ex}\label{Example: Link poset pure chain that is not totally pure}
			The complex in Example \ref{Example: PR Zelda Symbol} has degree type $(1,2)$ and offset $0$. Its link poset contains the maximal chain

		$$
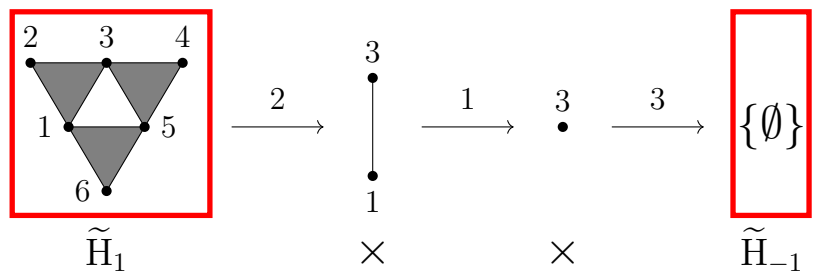
\tikz[remember picture, 
		overlay]{
			\draw[line width = .7mm, draw=red] (0,0.9) rectangle (2.6,3.6);
			\draw[line width = .7mm, draw=red] (9.48,0.9) rectangle (10.48,3.6);
		}
		\begin{tikzpicture}[scale = 0.5]
			\tikzstyle{point}=[circle,thick,draw=black,fill=black,inner sep=0pt,minimum width=3pt,minimum height=3pt]
			\node (a)[point,label=above:$2$] at (0,3.4) {};
			\node (b)[point,label=above:$3$] at (2,3.4) {};
			\node (c)[point,label=above:$4$] at (4,3.4) {};
			\node (d)[point,label=left:$1$] at (1,1.7) {};
			\node (e)[point,label=right:$5$] at (3,1.7) {};
			\node (f)[point,label=left:$6$] at (2,0) {};
			\node (h1) [label=below: \large $\Hred_1$] at (2,-0.4) {};
			
			\begin{scope}[on background layer]
				\draw[fill=gray] (a.center) -- (b.center) -- (d.center) -- cycle;
				\draw[fill=gray] (b.center) -- (c.center) -- (e.center) -- cycle;
				\draw[fill=gray]   (d.center) -- (e.center) -- (f.center) -- cycle;
			\end{scope}
			
			\node (a1) at (5,1.7) {};
			\node (a2) at (8,1.7) {};
			\node (x)[label=above:$2$]  at (6.5,1.6) {};
			\draw[->] (a1) -- (a2);
			
			\node (1)[point,label=below:$1$] at (9,0.4) {};
			\node (3)[point,label=above:$3$] at (9,3) {};
			\node (h2) [label=below: \Large $\times$] at (9,-0.65) {};
			
			\draw (1.center)--(3.center);
			
			\node (b1) at (10,1.7) {};
			\node (b2) at (13,1.7) {};
			\node (y)[label=above:$1$]  at (11.5,1.6) {};
			\draw[->] (b1) -- (b2);
			
			\node (p)[point,label=above:$3$] at (14,1.7){};
			\node (h3) [label=below: \Large $\times$] at (14,-0.65) {};
			
			\node (c1) at (15,1.7) {};
			\node (c2) at (18,1.7) {};
			\node (z)[label=above:$3$]  at (16.5,1.6) {};
			\draw[->] (c1) -- (c2);
			
			\node (q)[label=right:\Large$\{\emptyset\}$] at (18,1.7){};
			\node (h4) [label=below: \large $\Hred_{-1}$] at (19.5,-0.4) {};
		\end{tikzpicture}$$
		which is pure, because the homology index sets of its constituent complexes are either empty or singletons, and the degrees in the singleton sets form the descending sequence $1,-1$. But this sequence is not consecutive, so the chain is \textit{not} totally pure. However, if we swap the order in which we take links of the vertices $1$ and $2$ we obtain the maximal chain
		$$
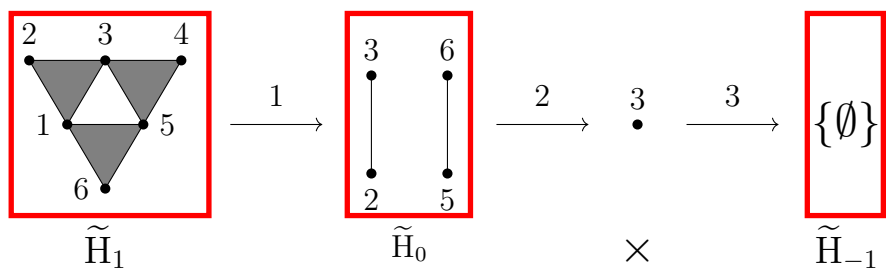
\tikz[remember picture, 
		overlay]{
			\draw[line width = .7mm, draw=red] (0,0.87) rectangle (2.6,3.55);
			\draw[line width = .7mm, draw=red] (4.42,0.87) rectangle (6.04,3.55);
			\draw[line width = .7mm, draw=red] (10.47,0.87) rectangle (11.47,3.55);
		}
		\begin{tikzpicture}[scale = 0.5]
			\tikzstyle{point}=[circle,thick,draw=black,fill=black,inner sep=0pt,minimum width=3pt,minimum height=3pt]
			\node (a)[point,label=above:$2$] at (0,3.4) {};
			\node (b)[point,label=above:$3$] at (2,3.4) {};
			\node (c)[point,label=above:$4$] at (4,3.4) {};
			\node (d)[point,label=left:$1$] at (1,1.7) {};
			\node (e)[point,label=right:$5$] at (3,1.7) {};
			\node (f)[point,label=left:$6$] at (2,0) {};
			\node (h1) [label=below: \large $\Hred_1$] at (2,-0.4) {};
			
			\begin{scope}[on background layer]
				\draw[fill=gray] (a.center) -- (b.center) -- (d.center) -- cycle;
				\draw[fill=gray] (b.center) -- (c.center) -- (e.center) -- cycle;
				\draw[fill=gray]   (d.center) -- (e.center) -- (f.center) -- cycle;
			\end{scope}
			
			\node (a1) at (5,1.7) {};
			\node (a2) at (8,1.7) {};
			\node (x)[label=above:$1$]  at (6.5,1.6) {};
			\draw[->] (a1) -- (a2);
			
			\node (2)[point,label=below:$2$] at (9,0.4) {};
			\node (3)[point,label=above:$3$] at (9,3) {};
			\node (5)[point,label=below:$5$] at (11,0.4) {};
			\node (6)[point,label=above:$6$] at (11,3) {};
			\node (h2) [label=below: $\Hred_0$] at (10,-0.4) {};
			
			\draw (2.center)--(3.center);
			\draw (5.center)--(6.center);
			
			\node (b1) at (12,1.7) {};
			\node (b2) at (15,1.7) {};
			\node (y)[label=above:$2$]  at (13.5,1.6) {};
			\draw[->] (b1) -- (b2);
			
			\node (p)[point,label=above:$3$] at (16,1.7){};
			\node (h3) [label=below: \Large $\times$] at (16,-0.65) {};
			
			\node (c1) at (17,1.7) {};
			\node (c2) at (20,1.7) {};
			\node (z)[label=above:$3$]  at (18.5,1.6) {};
			\draw[->] (c1) -- (c2);
			
			\node (q)[label=right:\Large$\{\emptyset\}$] at (20,1.7){};
			\node (h4)[label=below:\large$\Hred_{-1}$] at (21.5,-0.4){};
		\end{tikzpicture}$$
		which \textit{is} totally pure, with the same degree type and offset as the complex itself (i.e. degree type $(1,2)$ and offset $0$). We can compute the degree type of the chain by counting the number of arrows separating the complexes with homology: there is one arrow separating the complex with $\nth[st]{1}$ homology from the complex with $\nth[th]{0}$ homology, and there are two arrows separating the complex with $\nth[th]{0}$ homology from the complex with $\nth[st]{(-1)}$ homology.
	\end{ex}
	
	\begin{ex}\label{Example: Link poset max chain with different degree type and offset}
		The PR complex
		$$\begin{tikzpicture}[scale=1]
			\tikzstyle{point}=[circle,thick,draw=black,fill=black,inner sep=0pt,minimum width=3pt,minimum height=3pt]
			\node (a)[point, label=left:1] at (0,0) {};
			\node (b)[point, label=right:2] at (2,0) {};
			\node (c)[point, label=above:3] at (1,1.7) {};
			\node (d)[point, label=above left:4] at (1,0.58) {};
			
			\begin{scope}[on background layer]
				\draw[fill=gray] (a.center) -- (b.center) -- (d.center) -- cycle;
				\draw[fill=gray] (b.center) -- (c.center) -- (d.center) -- cycle;
				\draw[fill=gray] (c.center) -- (a.center) -- (d.center) -- cycle;
			\end{scope}
		\end{tikzpicture}$$
		has offset $1$ and degree type $(1,1)$. This is because the complex itself is acyclic, and all of its faces have acyclic links except for one vertex ($\{4\}$) with a link of the form \begin{tikzpicture}[scale = 0.2]
			\tikzstyle{point}=[circle,thick,draw=black,fill=black,inner sep=0pt,minimum width=2pt,minimum height=2pt]
			\node (a)[point] at (0,0) {};
			\node (b)[point] at (2,0) {};
			\node (c)[point] at (1,1.7) {};
			
			\draw (a.center) -- (b.center) -- (c.center) -- cycle;
		\end{tikzpicture}, three edges ($\{1,4\}$, $\{2,4\}$ and $\{3,4\}$) with links of the form \begin{tikzpicture}[scale = 0.2]
			\tikzstyle{point}=[circle,thick,draw=black,fill=black,inner sep=0pt,minimum width=2pt,minimum height=2pt]
			\node (a)[point] at (0,0) {};
			\node (c)[point] at (1.5,0) {};
		\end{tikzpicture}, and three facets of size $3$.
		
		However it contains the totally pure maximal chain
		$$
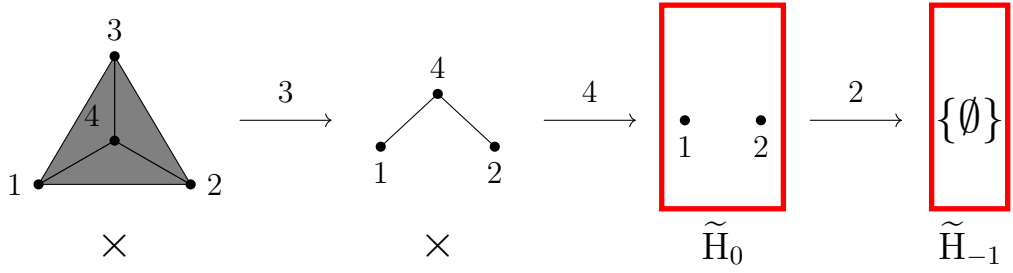
\tikz[remember picture, 
		overlay]{
			\draw[line width = .7mm, draw=red] (8.75,0.9) rectangle (10.35,3.6);
			\draw[line width = .7mm, draw=red] (12.3,0.9) rectangle (13.3,3.6);
		}
		\begin{tikzpicture}[scale = 0.5]
			\tikzstyle{point}=[circle,thick,draw=black,fill=black,inner sep=0pt,minimum width=3pt,minimum height=3pt]
			\node (a)[point, label=left:1] at (0,0) {};
			\node (b)[point, label=right:2] at (4,0) {};
			\node (c)[point, label=above:3] at (2,3.4) {};
			\node (d)[point, label=above left:4] at (2,1.16) {};
			\node (h1) [label=below: \Large $\times$] at (2,-0.65) {};
			
			\begin{scope}[on background layer]
				\draw[fill=gray] (a.center) -- (b.center) -- (d.center) -- cycle;
				\draw[fill=gray] (b.center) -- (c.center) -- (d.center) -- cycle;
				\draw[fill=gray] (c.center) -- (a.center) -- (d.center) -- cycle;
			\end{scope}
			
			\node (a1) at (5,1.7) {};
			\node (a2) at (8,1.7) {};
			\node (x)[label=above:$3$]  at (6.5,1.6) {};
			\draw[->] (a1) -- (a2);
			
			\node (1)[point, label=below:1] at (9,1) {};
			\node (2)[point, label=below:2] at (12,1) {};
			\node (4)[point, label=above:4] at (10.5,2.4) {};
			\node (h2) [label=below: \Large $\times$] at (10.5,-0.65) {};
			
			\draw (1.center) -- (4.center) -- (2.center);
			
			\node (b1) at (13,1.7) {};
			\node (b2) at (16,1.7) {};
			\node (y)[label=above:$4$]  at (14.5,1.6) {};
			\draw[->] (b1) -- (b2);
			
			\node (one)[point, label=below:1] at (17,1.7) {};
			\node (two)[point, label=below:2] at (19,1.7) {};
			\node (h3) [label=below: \large $\Hred_0$] at (18,-0.4) {};
			
			\node (c1) at (20,1.7) {};
			\node (c2) at (23,1.7) {};
			\node (z)[label=above:$2$]  at (21.5,1.6) {};
			\draw[->] (c1) -- (c2);
			
			\node (q)[label=right:\Large$\{\emptyset\}$] at (23,1.7){};
			\node (h4)[label=below:\large$\Hred_{-1}$] at (24.5,-0.4){};
			
		\end{tikzpicture}$$ which has offset $2$ and degree type $(1)$. As in Example \ref{Example: Link poset pure chain that is not totally pure}, we can compute this offset and degree type by counting arrows: there are two arrows before the first complex in the chain with homology, and one arrow separating the complex with $\nth[th]{0}$ homology from the complex with $\nth[st]{(-1)}$ homology.
	\end{ex}
	
	\begin{prop}
		Let $\Delta$ be a PR complex with degree type $\bd=(d_p,\dots,d_1)$ and offset $s$. We have the following.
		\begin{enumerate}
			\item Every maximal chain in $P_\Delta$ is pure.
			\item $P_\Delta$ contains a totally pure maximal chain of degree type $\bd$ and offset $s$.
			\item Every totally pure maximal chain in $P_\Delta$ has degree type some subsequence $\bd'$ of $\bd$ and offset $s'\geq s$ such that $s'+\sum \bd'=s+\sum\bd$.
			\item If $s=0$ then every totally pure maximal chain in $P_\Delta$ has degree type $\bd$ and offset $0$.
		\end{enumerate}
	\end{prop}
	\begin{proof}
		All of the homology index sets of the complexes in $P_\Delta$ are either empty or singletons by Corollary \ref{Corollary: PR Complex links have single homology}, and the nonempty ones form a decreasing sequence by Corollary \ref{Corollary: Homology index sets decreasing sequence}. This proves part (1). 
		
		By the definition of offset, $P_\Delta$ must contain the complex $\delta_s=\lkds$ for some face $\sigma\in \Delta$ of size $s$ such that $h(\delta_s)=\{p-1\}$. Lemma \ref{Lemma: Homology of descending links} gives us a chain of simplices $\sigma = \tau_p \subsetneqq \dots \subsetneqq \tau_0$ such that $h(\Delta, \tau_j)=\{j-1\}$. For each $0\leq j \leq p$ we define $i_j=|\tau_j|$ and $\delta_{i_j} = \lkds[\tau_j]$. We can extend the sequence $\delta_{i_p}> \dots > \delta_{i_0}$ in $P_\Delta$ to a totally pure maximal chain $\Delta=\delta_0>...>\delta_m=\{\emptyset\}$ by Lemma \ref{Lemma: links in links}, and by construction this chain has degree type $\bd=(i_{p-1}-i_p,\dots,i_0-i_1)$ and offset $s=i_p$. This proves part (2).
		
		Now suppose $\Delta=\varepsilon_0>...>\varepsilon_m=\{\emptyset\}$ is a totally pure maximal chain in $P_\Delta$ with degree type $\bd'$ and offset $s'$. Note in particular that it must have the same length as the chain $\Delta=\delta_0>...>\delta_m=\{\emptyset\}$, because all facets of $\Delta$ have the same size, and this length is equal to the sum of the offset and degree type, which gives us $s'+\sum \bd'=s+\sum \bd$. Let $\varepsilon_{i'_r}>\dots > \varepsilon_{i'_0}$ be the complexes in the chain with homology, with the degrees of these homologies forming the consecutive decreasing sequence $r-1,\dots,-1$. By the PR condition we must have $i_j=i_j'$ for each $1\leq j \leq r$, and hence $r\leq p$. Thus $\bd'$ is equal to $(i_{r-1}-i_r,\dots,i_0-i_1)$, which is a subsequence of $\bd$, and $s'$ is equal to $i_r$, which is greater than or equal to $s=i_p$. This proves part (3).
		
		Moreover, if $s=0$, then $\Delta$ itself has $\nth[st]{(p-1)}$ homology, and hence we have $r=p$ and $\varepsilon_{i_r}=\varepsilon_{i_r}=\Delta$. In particular this means $s'=i_p=0$ and $\bd'=\bd$, proving part (4). 
	\end{proof}
	
	The degree type of a PR complex can therefore be defined as the \textit{longest} degree type of the totally pure maximal chains in its link poset; and the offset as the \textit{smallest} offset of the totally pure maximal chains.

	\chapter{Some Families of PR Complexes}\label{Chapter: Families of PR Complexes}
	In this chapter we present some interesting families of highly symmetric PR complexes, along with their Betti diagrams and degree types. In particular, we introduce the family of \textit{cycle complexes}, which have degree types of projective dimension $2$; the family of \textit{intersection complexes}, whose degree types are all those increasing sequences of positive integers for which every difference sequence is also increasing; and the family of \textit{partition complexes}, which have degree types of the form $(\overbrace{1,\dots,1\underbrace{a,1,\dots,1}_{m}}^{p})$, for positive integers $a$, $p$ and $m$.
	
	\section{Motivation}\label{Subsection: Motivation PR Families}
	
	Our motivation for presenting these families is twofold. Firstly, providing more examples of PR complexes allows us to highlight some more of the combinatorial properties common to many complexes in this family, along with methods for proving that the PR condition is satisfied, all of which will (we hope) help the reader to develop an intuition about PR complexes.
	
	Secondly, they help us to investigate the following question.
	\begin{qu}\label{Question: Lower bounds on n for each degree type}
		Let $\bd$ be a sequence of positive integers. What is the lowest value of $n$ for which the cone $\Dn$ contains a pure diagram of degree type $\bd$?
	\end{qu}
	
	Our proof for Theorem \ref{Theorem: PR Complexes of Any Degree Type}, which we will present in the next chapter, is a procedure for generating PR complexes of any given degree type, and hence gives us upper bounds for $n$ in Question \ref{Question: Lower bounds on n for each degree type} for every degree type. However, in general, our procedure produces complexes on very large numbers of vertices, so we suspect these upper bounds are far greater than necessary. In contrast, the families we present in this chapter have a very low number of vertices for their given degree types. In most cases, they are the lowest we have currently found, and in a few cases, we will prove they are the lowest possible. Thus they allow us to considerably lower these bounds for $n$ in Question \ref{Question: Lower bounds on n for each degree type} for certain degree types. 
	
	In order for a PR complex to have a minimal number of vertices for its given degree type, it must have offset zero (i.e. it must have nontrivial homology at some degree). This is a consequence of Corollary \ref{Corollary: Offset 0 => Minimal}: any PR complex $\Delta$ with a nonzero offset contains a nonempty face whose link has the same degree type and offset zero, and this link has fewer vertices than $\Delta$ itself. For this reason, all of the PR complexes we present in this section will have offset zero.
	
	Finding the minimal number of vertices of a PR complex with degree type $\bd$ also tells us all the possible \textit{shift types} of PR complexes with that degree type (as discussed below). This means that an answer to Question \ref{Question: Lower bounds on n for each degree type} for any given degree type $\bd$ gives us an answer to Question \ref{Question: Shift Types} for all shift types with $\bd$ as their difference sequence; and a complete answer to Question \ref{Question: Lower bounds on n for each degree type} for \textit{every} degree type would demonstrate all of the possible shift types of PR complexes, and thus provide a full analogue of the first Boij-S\"{o}derberg conjecture for squarefree monomial ideals.
	
	To see why Questions \ref{Question: Lower bounds on n for each degree type} and \ref{Question: Shift Types} are in fact equivalent, suppose we have found a minimal PR complex $\Delta$ of degree type $\bd$, on $n(\bd)$ vertices. BY ADHF the initial shift $c_0(\bd)$ of the diagram $\beta(\Idstar)$ at degree $0$ is given by the equation $c_0(\bd)=n(\bd)-\dim \Delta - 1=n(\bd)-\sum \bd$. This is the minimal possible initial shift for a PR complex of degree type $\bd$, and along with $\bd$ it fully determines the shift type of $\Delta$.
	
	
	Recall from Remark \ref{Remark: Link of Emptyset and Facets} that the initial shift $c_0$ of a PR complex $\Delta$ is equal to its codimension. Thus if $\Delta$ is a PR complex with shift type $(c_p,\dots,c_0)$, then adding a missing vertex to $\Delta$ gives us a PR complex with shift type $(c_p+1,\dots,c_0+1)$ (because it increases the codimension of the complex by $1$). For example, the boundary of the $2$-simplex
	$$
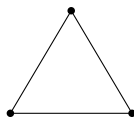
\begin{tikzpicture}[scale=0.8]
		\tikzstyle{point}=[circle,thick,draw=black,fill=black,inner sep=0pt,minimum width=2pt,minimum height=2pt]
		\node (a)[point] at (0,0) {};
		\node (b)[point] at (2,0) {};
		\node (c)[point] at (1,1.7) {};
		
		\draw (a.center) -- (b.center) -- (c.center) -- cycle;
	\end{tikzpicture}$$
	has Betti diagram
	$$\begin{array}{c | ccc}
		& 0 & 1 & 2\\
		\hline
		1 & 3 & 3 & 1 \\ 
	\end{array}$$
	and hence has shift type $(3,2,1)$. Adding a missing vertex to this complex gives us the complex
	$$
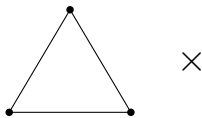
\begin{tikzpicture}[scale=0.8]
		\tikzstyle{point}=[circle,thick,draw=black,fill=black,inner sep=0pt,minimum width=2pt,minimum height=2pt]
		\node (a)[point] at (0,0) {};
		\node (b)[point] at (2,0) {};
		\node (c)[point] at (1,1.7) {};
		\node[scale=1.2] at (3,0.85) {$\times$};
		
		\draw (a.center) -- (b.center) -- (c.center) -- cycle;
	\end{tikzpicture}$$
	which has Betti diagram
	$$\begin{array}{c | ccc}
		& 0 & 1 & 2\\
		\hline
		2 & 3 & 3 & 1 \\ 
	\end{array}$$
	and hence shift type $(4,3,2)$.
	
	It follows that if there exists a PR complex of degree type $\bd$ and initial shift $c_0$, then there exists a PR complex of degree type $\bd$ and initial shift $c$ for any $c\geq c_0$. Therefore we can determine all possible shift types of PR complexes with degree type $\bd$ by finding minimal initial shift $c_0(\bd)$ for a PR complex with that degree type, or equivalently the minimal number of vertices $n(\bd)$.
	
	%
	
	
	\section{Constructing Complexes With Group Actions}\label{Subsection: Consrtucting Complexes from Groups}
	Many of the PR complexes we have seen already are highly symmetric. With this in mind, we now present a useful tool for constructing symmetric complexes using groups and group actions, along with some relevant notation.
	
	Suppose we have a group $G$ acting on a set $V$. For any subset $F\subseteq V$ and any element $g\in  G$, we define the subset $gF\subseteq V$ to be the set $\{gx:x\in F\}$. This allows us to define the following notation.
	
	\begin{notation}\label{Notation: Complexes from Groups}
		For subsets $F_1,\dots,F_m$ in $V$, and a group $G$ acting on $V$, we use $\langle F_1,\dots, F_m\rangle_G$ to denote the complex $\langle gF_i|1\leq i\leq m, g\in G\rangle$.
	\end{notation}
	
	The following lemma shows how we can exploit the symmetries of complexes constructed in this way to help us understand their links.
	
	\begin{lem}\label{Lemma: link x = link gx}
		Let $F_1,\dots,F_m \subset V$ and let $G$ be a group acting on $V$. Define $\Delta=\langle F_1,\dots,F_m \rangle_G$, and let $\sigma\in \Delta$ and $g\in G$. We have an isomorphism of complexes $\lkds \cong \link_\Delta g\sigma$.
	\end{lem}
	\begin{proof}
		If $F$ is a facet of $\Delta$ containing $\sigma$ then $gF$ is a facet of $\Delta$ containing $g\sigma$. Conversely, for any facet $F$ of $\Delta$ containing $g\sigma$, we must have that $g^{-1}F$ is a facet of $\Delta$ containing $\sigma$. Thus, $g$ gives us a bijection between the facets of $\Delta$ containing $\sigma$ and the facets of $\Delta$ containing $g\sigma$, and therefore  provides an isomorphism between $\lkds$ and $\link_\Delta g\sigma$.
	\end{proof}
	

	\section{PR Complexes of Projective Dimensions $1$ and $2$}\label{Subsection: PR projdim 1 and 2}
	
	In this section we fix two positive integers $a$ and $b$, and present PR complexes of degree type $(a)$ and $(a,b)$ with offset $0$.
	
	\subsection{Projective Dimension $1$: Disjoint Simplices}\label{Subsection: PR projdim 1}
	Suppose $\Delta$ is a PR complex of degree type $(a)$.
	
	This means that $I_{\Delta^*}$ has projective dimension $1$, and hence none of the links in $\Delta$ can have homology at a degree higher than $0$. Thus the only links in $\Delta$ which have homology are the links of maximal intersections and the links of facets. Moreover, if $\Delta$ has offset $0$, it must be disconnected itself, and hence it has only a single maximal intersection, namely the empty set $\emptyset$. This observation leads to the following result.
	\begin{prop}\label{Proposition: Disjoint Simplices are minimal PR}
		If $\Delta$ is a PR complex of degree type $(a)$ with a minimal number of vertices $n$ then it must be of the form given in Example \ref{Example: PR Disjoint Simplices}, consisting of two disjoint $(a-1)$-simplices. In particular, we have $n = 2a$.
	\end{prop}
	\begin{proof}
		In order for $\Delta$ to have a minimal number of vertices, it must have offset $0$ by Corollary \ref{Corollary: Offset 0 => Minimal}. Thus we have $\dim \Delta = a - 1$, and $\Hred_0(\Delta)\neq 0$.
		
		By Proposition \ref{Proposition: beta_1 Purity Condition}, the only disconnected links in $\Delta$ are the links of maximal intersections. Thus, because $\Delta$ itself is disconnected, the only maximal intersection in $\Delta$ is the empty set. This means that no two facets of $\Delta$ intersect, and hence $\Delta$ is a disjoint union of facets. By minimality, it is a disjoint union of exactly two facets. This proves the result.
	\end{proof}
	
	\begin{ex}
		Suppose $\Delta$ is the disjoint union of two $2$-simplices $\Delta^2+\Delta^2$, which is PR with degree type $(3)$ and offset $0$. Up to isomorphism, the maximal chains in the link poset $P_\Delta$ all look like the following.
		
		\begin{center}
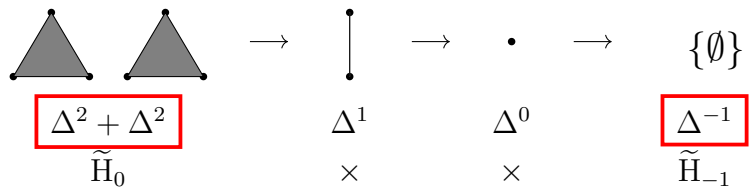
			\begin{tabular}{ c c c c c c c }
				
				\begin{tikzpicture}[scale = 0.5]
					\tikzstyle{point}=[circle,thick,draw=black,fill=black,inner sep=0pt,minimum width=2pt,minimum height=2pt]
					\node (a)[point] at (0,0) {};
					\node (b)[point] at (2,0) {};
					\node (c)[point] at (1,1.7) {};
					\node (d)[point] at (3,0) {};
					\node (e)[point] at (5,0) {};
					\node (f)[point] at (4,1.7) {};
					
					\draw[fill=gray] (a.center) -- (b.center) -- (c.center) -- cycle;
					\draw[fill=gray] (d.center) -- (e.center) -- (f.center) -- cycle;
				\end{tikzpicture}& \begin{tikzpicture}[scale = 0.5]
					\tikzstyle{point}=[circle,thick,draw=black,fill=black,inner sep=0pt,minimum width=2pt,minimum height=2pt]
					\node (a) at (0,0.75) {};
					\node (b) at (1,0.75) {};
					\node (c) at (0.5,0) {};
					
					\draw[->] (a.center) -- (b.center);
				\end{tikzpicture}   & \begin{tikzpicture}[scale = 0.5]
					\tikzstyle{point}=[circle,thick,draw=black,fill=black,inner sep=0pt,minimum width=2pt,minimum height=2pt]
					\node (a)[point] at (0,0) {};
					\node (b)[point] at (0,1.7) {};
					
					\draw[fill=gray] (a.center) -- (b.center);
				\end{tikzpicture}
				& \begin{tikzpicture}[scale = 0.5]
					\tikzstyle{point}=[circle,thick,draw=black,fill=black,inner sep=0pt,minimum width=2pt,minimum height=2pt]
					\node (a) at (0,0.75) {};
					\node (b) at (1,0.75) {};
					\node (c) at (0.5,0) {};
					
					\draw[->] (a.center) -- (b.center);
				\end{tikzpicture} & \begin{tikzpicture}[scale = 0.5]
					\tikzstyle{point}=[circle,thick,draw=black,fill=black,inner sep=0pt,minimum width=2pt,minimum height=2pt]
					\node (a)[point] at (0,0.75) {};
					\node (c) at (0,0) {};
				\end{tikzpicture}  &
				\begin{tikzpicture}[scale = 0.5]
					\tikzstyle{point}=[circle,thick,draw=black,fill=black,inner sep=0pt,minimum width=2pt,minimum height=2pt]
					\node (a) at (0,0.75) {};
					\node (b) at (1,0.75) {};
					\node (c) at (0.5,0) {};
					
					\draw[->] (a.center) -- (b.center);
				\end{tikzpicture} & \begin{tikzpicture}[scale = 0.5] \node (q)[label=right:\large$\{\emptyset\}$] at (18,1.7){};\end{tikzpicture} \\
				\begin{tikzcd}|[draw=red, line width =.5mm, rectangle]|\Delta^2 + \Delta^2\end{tikzcd}&& $\Delta^1$ && $\Delta^0$ &&\begin{tikzcd}|[draw=red, line width =.5mm, rectangle]|\Delta^{-1}\end{tikzcd}\\
				$\Hred_0$&&$\times$&&$\times$ && $\Hred_{-1}$
			\end{tabular}
		\end{center}
	\end{ex}
	
	\begin{rem}
		More generally, the PR complexes with degree type $(a)$ and offset $0$ are all disjoint unions of $(a-1)$-simplices.
	\end{rem}
	
	\subsection{Projective Dimension $2$: Cycle Complexes}\label{Subsection: PR projdim 2}
	Now we search for PR complexes of degree type $(a,b)$ with a minimal number of vertices. Once again, Corollary \ref{Corollary: Offset 0 => Minimal} tells us that any such complex must have offset $0$.
	
	Suppose $\Delta$ is such a complex. Because $\Idstar$ has projective dimension $2$ and $\Delta$ has offset $0$, we have $h(\Delta)=\{1\}$. Moreover, the only links of nonempty faces in $\Delta$ which have homology are the links of maximal intersections (which have size $a$) and the links of facets (which have size $a+b$).
	%
	%
	
	In this section we present the family of \textit{cycle complexes}, which satisfy these properties for all values of $a$ and $b$. Our basic construction is straightforward: the cycle complex of type $(a,b)$ is generated by a single face $F$ of size $a+b$, under the action of some cyclic group $H=\langle h \rangle$, chosen such that the intersection $F \cap hF$ has size $a$. More specifically, if we choose $n$ to be a sufficiently large multiple of $b$, and let $F$ denote the subset $\{1,\dots,a+b\}$ in $[n]$, then we can choose our cyclic group to be the subgroup of the symmetric group $S_n$ generated by the element $g^b$ where $g$ denotes the $n$-cycle $g=(1\dots n)$. In this case we have $F\cap g^b F = \{1+b,\dots, a+b\}$, which has size $a$ as required.
	\begin{rem}
		For this construction to work, $n$ \textit{must} be a multiple of $b$. This is because the cyclic group generated by $g^b$ in $S_n$ contains $g^{\hcf(b,n)}$, and if we have $h=\hcf(b,n)<b$ then $F\cap g^h F$ has size greater than $a$.
	\end{rem}
	
	Before we can define the cycle complexes explicitly, we will have to be more explicit about what constitutes a \textit{sufficiently large} multiple of $b$ in this context. The central constraint is that our resulting complex must have $\nth[st]{1}$ homology. 

	Suppose we have found integers $m$ and $r$ with $1\leq r \leq b$ such that $a=mb+r$. As the following lemma shows, the smallest value of $n$ for which our construction has $\nth[st]{1}$ homology turns out to be $(2m+3)b$. 
	\begin{rem}
		Note we are choosing $r$ to satisfy $1\leq r \leq b$ here, rather than $0\leq r \leq b-1$ (which is the standard practice for Euclidean division). In particular, if $a$ divides $b$, our value for $r$ would be $b$ rather than $0$; and our value for $m$ would be $\frac{a}{b}-1$ rather than $\frac{a}{b}$.
	\end{rem}
	
	\begin{lem}\label{Lemma: n = (2m+3)b for our construction}
		Suppose $a = mb + r$ for some integers $m\geq 0$ and $1\leq r \leq b$, and define $n=kb$ for some positive integer $k$. Define $g$ to be the $n$-cycle $(1 \dots n)$ in the symmetric group $S_n$. Also define $F_0$ to be the subset $\{1,\dots,a+b\}\subseteq[n]$ and $F_i$ to be the set $g^{ib}F_0$ for each integer $i$. For $\Delta =\langle F_i : i\in \ZZ \rangle$, we have $$\Hred_1(\Delta)=\begin{cases}
			0 &\text{ if } k\leq 2m+2\\
			\KK & \text{ if } k \geq 2m+3.
		\end{cases}$$
	\end{lem}
	\begin{proof}
		The permutation $g$ acts on $F_0$ by adding $1$ to each of its elements (modulo $n$). Thus for any integer $i$ we have $$F_i =\{ib+1,\dots,ib+(a+b)\}=\{ib+1,\dots,(m+i+1)b+r\}$$ where the elements of this set are read modulo $n$.
		
		In particular, we have that for any $0\leq i \leq m+1$, the set $F_i$ contains the element $a+b=(m+1)b+r$. By symmetry, any consecutive sequence of $m+2$ facets $F_i,F_{i+1},\dots,F_{i+{m+1}}$ contains a common element.
		
		Note that $\Delta$ has exactly $k$ facets, namely $F_0,\dots,F_{k-1}$. This means that if $k\leq m+2$, all of the facets of $\Delta$ contain a common element, and so $\Delta$ is acyclic. Thus we may assume that $k\geq m+3$, and decompose $\Delta$ into the union of two subcomplexes
		\begin{align*}
			A&=\langle F_0,\dots,F_{m+1}\rangle\\
			B&=\langle F_{m+2},\dots,F_{k-1} \rangle.
		\end{align*}
		
		As noted above, the facets of $A$ all contain a common element, so $A$ must be acyclic. The facets of $B$ satisfy the conditions of Lemma \ref{Lemma: Facet Sequence Deformation Retract} below, and so $B$ is also acyclic. 
		Thus the Mayer-Vietoris Sequence contains an isomorphism
		\begin{equation*}
			0\rightarrow \Hred_1(\Delta)\rightarrow \Hred_0(A\cap B) \rightarrow 0
		\end{equation*}
		so it suffices to prove that $$\Hred_0(A\cap B)=\begin{cases}
			0 &\text{ if } k\leq 2m+2\\
			\KK & \text{ if } k \geq 2m+3.
		\end{cases}$$
		
		We begin by noting that the vertices of $B$ can be arranged in the (modulo $n$) consecutive sequence $(m+2)b+1,\dots,n,1,\dots,a$. In particular, this means that $B$ contains none of the vertices $a+1,\dots,(m+2)b$, and hence $A\cap B$ also contains none of these vertices. However, $A\cap B$ does contain the faces $F_0\cap F_{k-1}=\{1,\dots,a\}$ and $F_{m+1}\cap F_{m+2}=\{(m+2)b+1,\dots,(2m+2)b+r\}$.
		
		Suppose first that $k\leq 2m+2$. In this case, we have $(2m+2)b+r> n$. This means that the faces $F_0\cap F_{k-1}$ and $F_{m+1}\cap F_{m+2}$ between them contain all the vertices of $B$, and thus all the vertices of $A\cap B$. Moreover the two faces intersect, because they both contain the vertex $1$. It follows that $A\cap B$ is connected.
		
		Now suppose that $k\geq 2m+3$. We can arrange the vertices of $A$ in the consecutive sequence $1,\dots,(2m+2)b+r$, and in particular we now have $(2m+2)b+r<n$. This means that $A$ contains none of the vertices $(2m+2)b+r+1,\dots,n$. Once again we conclude that the faces $F_0\cap F_{k-1}$ and $F_{m+1}\cap F_{m+2}$ taken together comprise all the vertices of $A\cap B$. However, in this case, these two faces are disjoint. We claim that $A\cap B$ is equal to the disjoint union of $F_0\cap F_{k-1}$ and $F_{m+1}\cap F_{m+2}$, and hence that $\Hred_0(A \cap B)=\KK$.
		
		Suppose for contradiction that $A\cap B$ is \textit{not} equal to the disjoint union of these two faces. Because $A\cap B$ contains no vertices outside of these two faces, it must contain an edge connecting a vertex $x \in F_0\cap F_{k-1}$ to a vertex $y\in F_{m+1}\cap F_{m+2}$. The edge $\{x,y\}$ must be contained both in a facet $F_i$ of $A$ for some $0\leq i \leq m+1$ and a facet $F_{m+2+j}$ of $B$ for some $0\leq j \leq k-m-3$. As before, we have \begin{align*}
			F_i &=\{ib+1,\dots,(m+i+1)b+r\}\\
			F_{m+2+j}&=\{(m+2+j)b+1,\dots,n,1,\dots,jb+r\}
		\end{align*} which gives us both $ib+1 \leq x \leq jb+r$ and $(m+2+j)b+1 \leq y \leq (m+i+1)b+r$. But this is a contradiction, because the former equation implies that $i\leq j$ and the latter implies that $i > j$.
		
		
	\end{proof}

	The above lemma does not prove that $(2m+3)b$ is the smallest possible number of vertices for a complex of degree type $(a,b)$; only that it is the smallest possible number of vertices for such a complex constructed in this particular way. However, based partly on computational evidence from the software system Macaulay2 (\cite{M2}), we strongly suspect that it is \textit{also} the minimal number needed to obtain a degree type of $(a,b)$. We prove this explicitly below for the case where $a\leq b$, by showing that a PR complex of degree type $(a,b)$ must have at least $3b$ vertices (this proves the $a\leq b$ case because in this case we have $m=0$ and hence $(2m+3)b=3b$). 
	\begin{prop}\label{Proposition: n = (2m+3)b for some cases}
		Suppose $a$ and $b$ are positive integers and let $\Delta$ be a PR complex of degree type $(a,b)$ on vertex set $[n]$. We must have $n\geq 3b$.
	\end{prop}
	\begin{proof}
		Let $F_1$ and $F_2$ be adjacent facets of $\Delta$ (i.e. facets such that the intersection $F_1\cap F_2$ is maximal). We may assume that $\Delta$ has offset $0$, and hence $\Hred_1(\Delta)\neq 0$. Because $\Delta$ has $\nth[st]{1}$ homology, it must have at least three facets, so we may choose some other facet $F_3$. We have $$n\geq |F_1|+|F_2|+|F_3|-|F_1\cap F_2|-|F_1\cap F_3|-|F_2\cap F_3| + |F_1\cap F_2 \cap F_3|.$$ We know that $|F_1|=|F_2|=|F_3|=a+b$, and by assumption we have $|F_1\cap F_2|=a$. Also, all maximal intersections in $\Delta$ have size $a$, so no intersection of facets of $\Delta$ can have size greater than $a$. This gives us																																			
		\begin{align*}
			n &
			\geq 2a + 3b - |F_1\cap F_3| - |F_2 \cap F_3| + |F_1\cap F_2\cap F_3|\\
			&\geq 3b +|F_1\cap F_2 \cap F_3|\\
			&\geq 3b.
		\end{align*}
		
	\end{proof}

	Now that we have found the minimum necessary value for $n$ for our construction to work, we can define the cycle complexes.
	\begin{defn}\label{Definition: Cycle Complexes}
		Suppose $a = mb + r$ for some integers $m\geq 0$ and $1\leq r \leq b$. We let
		\begin{itemize}
			\item $n$ denote the number $(2m+3)b$.
			\item $F$ denote the subset $\{1,\dots,a+b\}\subset [n]$.
			\item $g$ denote the $n$-cycle $(1\dots n)$ in the symmetric group $S_n$.
			\item $H$ denote the cyclic subgroup of $S_n$ generated by $g^b$ (which has a natural action on $[n]$, inherited from $S_n$).
		\end{itemize}	
		
		We define the \textit{cycle complex of type} $(a,b)$ to be the complex $$\fC_{a,b} = \langle F\rangle_{H}$$ on vertex set $[n]$.
	\end{defn}

	\begin{ex}\label{Example: Cycle Complexes}
		\begin{enumerate}
			\item The cycle complex $\fC_{1,1}$ has facets $\{1,2\}$, $\{2,3\}$ and $\{3,1\}$. This is the boundary of the 2-simplex, which is PR with degree type $(1,1)$.
			\item The cycle complex $\fC_{1,2}$ has facets $\{1,2,3\}$, $\{3,4,5\}$ and $\{5,1,2\}$. This is Example \ref{Example: PR Zelda Symbol}, which is PR with degree type $(1,2)$.
			\item The cycle complex $\fC_{1,3}$ has facets $\{1,2,3,4\}$, $\{4,5,6,7\}$ and $\{7,8,9,1\}$. This is Example \ref{Example: PR 3D Zelda Symbol}, which is PR with degree type $(1,3)$.
			\item The cycle complex $\fC_{2,1}$ is the complex
				
				$$
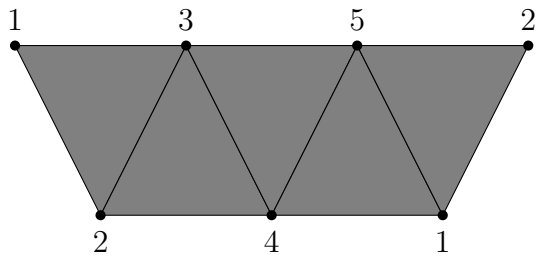
\begin{tikzpicture}[scale = 0.75]
					\tikzstyle{point}=[circle,thick,draw=black,fill=black,inner sep=0pt,minimum width=3pt,minimum height=3pt]
					\node (a)[point,label=above:$1$] at (-1.5,3) {};
					\node (b)[point,label=below:$2$] at (0,0) {};
					\node (c)[point,label=above:$3$] at (1.5,3) {};
					\node (d)[point,label=below:$4$] at (3,0) {};
					\node (e)[point,label=above:$5$] at (4.5,3) {};
					\node (f)[point,label=below:$1$] at (6,0) {};
					\node (g)[point,label=above:$2$] at (7.5,3) {};	
					
					\begin{scope}[on background layer]
						\draw[fill=gray] (a.center) -- (b.center) -- (c.center) -- cycle;
						\draw[fill=gray] (b.center) -- (c.center) -- (d.center) -- cycle;
						\draw[fill=gray] (c.center) -- (d.center) -- (e.center) -- cycle;
						\draw[fill=gray]   (d.center) -- (e.center) -- (f.center) -- cycle;
						\draw[fill=gray] (e.center) -- (f.center) -- (g.center) -- cycle;
					\end{scope}
				\end{tikzpicture}$$
				which is a triangulation of the M\"{o}bius strip, and is PR with degree type $(2,1)$. Note that it has one fewer vertex than the complex in Example \ref{Example: PR 21 Complex}.
				\item The cycle complex $\fC_{3,2}$ has facets $\{1,2,3,4,5\}$, $\{3,4,5,6,7\}$, $\{5,6,7,8,9\}$, $\{7,8,9,10,1\}$ and $\{9,10,1,2,3\}$. This complex is PR with degree type $(3,2)$.
			\end{enumerate}
		\end{ex}
		
		We want to prove that the cycle complex $\fC_{a,b}$ is always PR with degree type $(a,b)$. In order to do this, we prepare the following lemma (which we also appealed to in the proof of Lemma \ref{Lemma: n = (2m+3)b for our construction}).
		\begin{lem}\label{Lemma: Facet Sequence Deformation Retract}
			Let $\Delta$ be a complex with facets $F_0,\dots,F_k$ such that for every $0\leq i < j \leq k$ we have
			\begin{enumerate}
				\item $F_i\cap F_{i+1}\neq\emptyset$.
				\item $F_i\cap F_j \subseteq F_{j-1}$.
			\end{enumerate} There is a deformation retraction of $\Delta$ on to the acyclic complex $\langle F_0 \rangle$.
		\end{lem}
		\begin{proof}
			First assume $k>0$. By our assumption on the intersections of the facets of $\Delta$ we have $F_k\cap (\Delta-F_k)=F_k \cap F_{k-1}$. Thus all of the vertices of $F_k-F_{k-1}$ are free vertices, which means that we can deform $F_k$ on to $F_k\cap F_{k-1}$ using Lemma \ref{Lemma: Deformation Retract}. The result follows by induction on $k\geq 0$.
		\end{proof}

		\begin{thm}\label{Theorem: Cycle Complexes are PR}
			Let $a$ and $b$ be positive integers. The cycle complex $\fC_{a,b}$ is PR with degree type $(a,b)$.
		\end{thm}
		\begin{proof}
			We define the integers $m$, $r$ and $n$, the subset $F\subseteq [n]$, the permutation $g\in S_n$, and the subgroup $H\leq S_n$, all as in Definition \ref{Definition: Cycle Complexes}. The facets of $\fC_{a,b}$ are the sets $F, g^bF,\dots,g^{(2m+2)b}F$. For convenience, we use $F_i$ to denote the facet $g^{ib}F$, so we can rewrite these facets as $F_0,\dots,F_{2m+2}$ (note that we use this notation even for all integers $i$, so for instance we have $F_{-1}=F_{2m+2}$).
			
			By construction all of these facets have size $a+b$, so we have $-1 \in \hh(\Delta,a+b)$. Moreover, for any $0\leq k \leq 2m+2$, the intersection of the facets $F_k\cap F_{k+1}$ is maximal and has size $a$, so we have $0\in \hh(\Delta,a)$. We also have $\Hred_1(\Delta)\neq 0$ by Lemma \ref{Lemma: n = (2m+3)b for our construction}, so $1\in \hh(\Delta,0)$.
			
			It only remains to show that the link of any other nonempty face of $\Delta$ is acyclic: this will demonstrate that $\Delta$ satisfies the PR property, and we can conclude that it has degree type $(a,b)$ by Proposition \ref{Proposition: Alternate PR Definition With Degree Type}.
			
			To this end, suppose $\sigma$ is a nonempty simplex in $\Delta$ which is neither a facet nor a maximal intersection. Because $\sigma$ cannot be contained in every facet of $\Delta$, there must be some integer $i$ such that we have $\sigma \subseteq F_i$ and $\sigma \nsubseteq F_{i-1}$. By symmetry, we may assume that $i=0$, and hence that $\sigma$ is contained in the facet $F_0=\{1,\dots,a+b\}=\{1,\dots,(m+1)b+r\}$ but not fully contained in $F_{-1}\cap F_0=\{1,\dots,a\}=\{1,\dots,mb+r\}$.  Let $k$ be the largest integer below $2m+3$ such that $\sigma$ is contained in $F_k=\{kb+1,\dots,(m+k+1)b+r\}$. We must have that $k\leq m+1$, because otherwise the intersection $F_k \cap F_0$ would be contained in $F_{-1}\cap F_0$. Thus we have that $\sigma$ is contained in $F_0\cap F_k=\{kb+1,\dots,(m+1)b+r\}$, which is itself contained in $F_i$ for every $0\leq i \leq k$. We conclude that the facets of $\Delta$ containing $\sigma$ can be ordered consecutively as $F_0,\dots,F_k$, for some $k\leq m+1$. 
			
			
			Thus $\lkds$ is equal to the complex $\langle F_0-\sigma, \dots, F_k - \sigma \rangle$. Pick two integers $0\leq i < j \leq k$. Because $\sigma$ is not a maximal intersection, the intersection of adjacent facets $(F_i-\sigma)\cap (F_{i+1}-\sigma)$ must be nonempty; and the intersection $(F_i-\sigma)\cap (F_j-\sigma)$ is equal to $\{jb+1,\dots,(m+i+1)b+r\}-\sigma$, which is contained in $F_{j-1}-\sigma$. Hence the facets of $\lkds$ satisfy the conditions of Lemma \ref{Lemma: Facet Sequence Deformation Retract}, which means that $\lkds$ is acyclic.
			
			
		\end{proof}
		
		We can compute the Betti diagrams of cycle complexes exactly, as shown below.
		\begin{cor}\label{Corollary: Cycle Complex Betti Diagrams}
			Let $a$ and $b$ be positive integers, with $a = mb + r$ for some integers $m\geq 0$ and $1\leq r \leq b$, and let $n=(2m+3)b$. The nonzero entries of the Betti diagram $\beta=\beta(I^*_{\fC_{a,b}})$ are $\beta_{0,n-a-b}=\beta_{1,n-b}=2m+3$ and $\beta_{2,n}=1$.
		\end{cor}
		\begin{proof}
			The cycle complex $\fC_{a,b}$ has $2m+3$ facets (each of size $a+b$), and $2m+3$ maximal intersections (each of size $a$, and each contained in only two facets, ensuring that their links have $\nth{0}$ homology of dimension $1$). It also has $\nth[st]{1}$ homology of dimension $1$.
		\end{proof}
		
		\section{Intersection Complexes}\label{Subsection: Intersection Complexes}
		In this section we construct another infinite family of PR complexes. The key motivation for this construction comes from the following lemma.
		\begin{lem}\label{Lemma: link facet intersections}
			Let $\Delta$ be a simplicial complex. If $\sigma\in \Delta$ is such that $\lkds$ has homology, then $\sigma$ is an intersection of facets in $\Delta$.
		\end{lem}
		\begin{proof}
			We prove the contrapositive. Suppose that $\sigma$ is not an intersection of facets of $\Delta$ and let $F_1,\dots,F_m$ be all the facets of $\Delta$ that contain $\sigma$. The intersection $\bigcap_{i=1}^m F_i$ must contain $\sigma$. Because $\sigma$ is not an intersection of facets, there must be some vertex $v\in \bigcap_{i=1}^m F_i - \sigma$. This means that $\lkds$ is a cone over $v$, and is therefore acyclic.
		\end{proof}
		In particular, Lemma \ref{Lemma: link facet intersections} tells us that the only faces of a complex $\Delta$ that contribute towards the Betti numbers of $\Idstar$ are the ones that are intersections of facets. 
		Due to this observation, we may expect many PR complexes to exhibit some kind of symmetry around the points where their facets intersect. The following definition gives us one such symmetry condition on facet intersections.
		\begin{defn}\label{Definition: FIS}
			Let $\Delta$ be a complex with facets $F_1,\dots,F_n$. We say that $\Delta$ is \textit{intersectionally symmetric} if for any $1\leq k\leq n$ and any permutation $\alpha$ in $S_n$, we have
			\begin{enumerate}
				\item $|F_1\cap \dots \cap F_k|=|F_{\alpha(1)}\cap \dots \cap F_{\alpha(k)}|$
				\item $\link_\Delta (F_1\cap \dots \cap F_k)\cong \link_\Delta (F_{\alpha(1)}\cap \dots \cap F_{\alpha(k)})$.
			\end{enumerate}
		\end{defn}
		
		Intersectional symmetry is a very strict property. None-the-less it is a property exhibited by many of the examples we saw in Section \ref{Subsection: PR Examples}, including Examples \ref{Example: PR Simplex}, \ref{Example: PR Disjoint Simplices}, \ref{Example: PR Zelda Symbol} and \ref{Example: PR 3D Zelda Symbol}.
		
		Prima facie, there is good reason to hope that many intersectionally symmetric complexes are also PR complexes. Indeed, in order for a complex $\Delta$ to be PR we must have that for any intersection of its facets $\sigma = F_1\cap \dots \cap F_k$, the degrees of the homologies of $\lkds$ are dependent on the size of $\sigma$; and for complexes with intersectional symmetry, both of these statistics are functions of a third factor - the value of $k$.
		
		In fact, as we will show in this section, all complexes with intersectional symmetry turn out to be PR complexes. We begin by presenting an explicit construction of intersectionally symmetric complexes.
		
		\subsection{Defining the Complexes}
		In this section, we define the \textit{intersection complexes}, the family of all complexes with intersectional symmetry, along with some examples.
		
		\begin{defn}\label{Definition: Intersection Complexes}
			Let $\bm=(m_1,\dots,m_n)$ be a sequence of nonnegative integers.
			We define the \emph{intersection complex} $\calI(\bm)$ as follows.
			\begin{enumerate}
				\item The vertices of $\calI(\bm)$ are all symbols of the form $v_S^r$ where $S$ is a subset of $[n]$ and $1\leq r \leq m_{|S|}$.
				\item The facets of $\calI(\bm)$ are the sets $F_1,\dots,F_n$ where 
				$$F_j=\left \{v_S^r \,:\, S\subseteq [n], j\in S, 1 \leq r \leq m_{|S|} \right\}.$$
			\end{enumerate}
		\end{defn}
		
		\begin{rem}\label{Remark: Intersection Complex Observations}
			\begin{enumerate}
				\item The integer $m_i$ is equal to the number of vertices contained in the intersection $F_1\cap\dots \cap F_i$ (or any other intersection of $i$ facets) that are not contained in an intersection of $i+1$ facets.
				\item For any subset $S\subseteq[n]$ of size $j$ and any $1\leq i \leq n$, the vertices $v_S^1,\dots,v_S^{m_j}$ are in the facet $F_i$ if and only if we have $i\in S$.
			\end{enumerate}
		\end{rem}
		
		\begin{rem}
			Note that we could also define intersection complexes in group notation as in Notation \ref{Notation: Complexes from Groups}. Specifically, $S_n$ acts on the vertex $v^r_S$ via $\alpha(v^r_S)=v^r_{\alpha S}$, and hence we may define the intersection complex $\calI(\bm)$ to be the complex $\langle F_1\rangle_{S_n}$.
		\end{rem}
		
		By construction, all intersection complexes have intersectional symmetry. To see that the converse also holds, suppose that $\Delta$ is a complex with intersectional symmetry and let $F_1,...,F_n$ be the facets of $\Delta$. Define $m_n$ to be the number of vertices in all $n$ facets, and label these vertices $v_{[n]}^1,\dots, v_{[n]}^{m_n}$. Now  define $m_{n-1}$ to be the number of vertices in $F_1\cap \dots \cap F_{n-1}$ but not in $F_n$, and label these vertices $v_{[n-1]}^1,\dots, v_{[n-1]}^{m_{n-1}}$. By symmetry we know that any intersection of $n-1$ facets contains $m_{n-1}$ vertices outside of $F_1\cap \dots \cap F_n$, and we can label each of these accordingly. Now we define $m_{n-2}$ to be the number of vertices in the intersection $F_1\cap \dots \cap F_{n-2}$ which are not in the intersection of a larger number of facets. Proceeding in this way we can find a sequence $(m_1,\dots,m_n)$ of nonnegative integers such that $\Delta$ is isomorphic to the intersection complex $\calI(m_1,\dots,m_n)$.
		
		Thus the family of intersection complexes as defined above is the family of all complexes with intersectional symmetry.
		
		\begin{ex}\label{Example: Intersection Complex (0,0,1,0)}
			The boundary of the 3-simplex in Example \ref{Example: PR Tetrahedron} can be thought of as the intersection complex $\calI(0,0,1,0)$, as shown below.
			
			\begin{center}
				\begin{tikzpicture}[line join = round, line cap = round]
					
					\coordinate [label=above:{$v_{\{1,2,3\}}^1$}] (4) at (0,{sqrt(2)},0);
					\coordinate [label=left:{$v_{\{1,2,4\}}^1$}] (3) at ({-.5*sqrt(3)},0,-.5);
					\coordinate [label=below:{$v_{\{1,3,4\}}^1$}] (2) at (0,0,1);
					\coordinate [label=right:{$v_{\{2,3,4\}}^1$}] (1) at ({.5*sqrt(3)},0,-.5);
					
					\begin{scope}
						\draw (1)--(3);
						\draw[fill=lightgray,fill opacity=.5] (2)--(1)--(4)--cycle;
						\draw[fill=gray,fill opacity=.5] (3)--(2)--(4)--cycle;
						\draw (2)--(1);
						\draw (2)--(3);
						\draw (3)--(4);
						\draw (2)--(4);
						\draw (1)--(4);
					\end{scope}
				\end{tikzpicture}
			\end{center}
		\end{ex}
		\begin{rem}\label{Remark: Intersection Complex Boundary of $p$-simplex}
			In fact, Example \ref{Example: Intersection Complex (0,0,1,0)} is a special case of a more general observation: for any $p$, the boundary of the $p$-simplex $\partial \Delta^p$ is equal to the the intersection complex $\calI(\underbrace{0,\dots,0,1}_p,0)$.
		\end{rem}
		
		\begin{ex}\label{Example: Intersection Complex (1,1,0)}
			The $2$-dimensional complex in Example \ref{Example: PR Zelda Symbol} can be thought of as the intersection complex $\calI(1,1,0)$, as shown below.
			$$
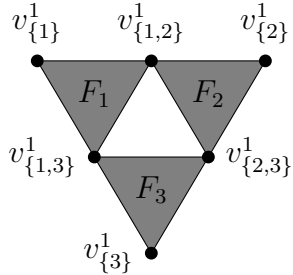
\begin{tikzpicture}[scale = 0.75]
				\tikzstyle{point}=[circle,thick,draw=black,fill=black,inner sep=0pt,minimum width=4pt,minimum height=4pt]
				\node (a)[point,label=above:{$v_{\{1\}}^1$}] at (0,3.4) {};
				\node (b)[point,label=above:{$v_{\{1,2\}}^1$}] at (2,3.4) {};
				\node (c)[point,label=above:{$v_{\{2\}}^1$}] at (4,3.4) {};
				\node (d)[point,label=left:{$v_{\{1,3\}}^1$}] at (1,1.7) {};
				\node (e)[point,label=right:{$v_{\{2,3\}}^1$}] at (3,1.7) {};
				\node (f)[point,label=left:{$v_{\{3\}}^1$}] at (2,0) {};	
				
				\node (g)[label=above: {$F_1$}] at (1,2.2) {};
				\node (h)[label=above: {$F_2$}] at (3,2.2) {};
				\node (i)[label=above: {$F_3$}] at (2,0.5) {};
				
				\begin{scope}[on background layer]
					\draw[fill=gray, fill opacity = .8] (a.center) -- (b.center) -- (d.center) -- cycle;
					\draw[fill=gray, fill opacity = .8] (b.center) -- (c.center) -- (e.center) -- cycle;
					\draw[fill=gray, fill opacity = .8]   (d.center) -- (e.center) -- (f.center) -- cycle;
				\end{scope}
			\end{tikzpicture}$$
		\end{ex}

		\begin{ex}\label{Example: Intersection Complex (2,1,0)}
			The $3$-dimensional complex in Example \ref{Example: PR 3D Zelda Symbol} can be thought of as the intersection complex $\calI(2,1,0)$, as shown below.
			$$\begin{tikzpicture}[scale = 1]
				\tikzstyle{point}=[circle,thick,draw=black,fill=black,inner sep=0pt,minimum width=4pt,minimum height=4pt]
				\node (a)[point,label=above:$v_{\{1\}}^1$] at (0,3.4) {};
				\node (b)[point,label=above:$v_{\{1,2\}}^1$] at (2,3.4) {};
				\node (c)[point,label=above:$v_{\{2\}}^1$] at (4,3.4) {};
				\node (d)[point,label=left:$v_{\{1,3\}}^1$] at (0.8,1.7) {};
				\node (e)[point,label=right:$v_{\{2,3\}}^1$] at (2.8,1.7) {};
				\node (f)[point,label=left:$v_{\{3\}}^1$] at (1.8,0) {};
				
				\node (g)[point,label=above:$v_{\{1\}}^2$] at (1,3.9) {};
				\node (h)[point,label=above:$v_{\{2\}}^2$] at (3,3.9) {};
				\node (i)[point,label=above:$v_{\{3\}}^2$] at (2,2.2) {};
				
				\begin{scope}[on background layer]
					\draw [fill=lightgray, fill opacity=.5] (a.center) -- (b.center) -- (d.center) -- cycle;
					\draw [fill=lightgray, fill opacity=.5] (b.center) -- (c.center) -- (e.center) -- cycle;
					\draw [fill=lightgray, fill opacity=.5] (d.center) -- (e.center) -- (f.center) -- cycle;
					
					\draw[fill=gray, fill opacity=.8] (a.center)--(d.center)--(g.center)--cycle;
					\draw[fill=lightgray,fill opacity=.5] (b.center)--(d.center)--(g.center)--cycle;
					
					\draw[fill=gray,fill opacity=.8] (b.center)--(e.center)--(h.center)--cycle;
					\draw[fill=lightgray,fill opacity=.5] (c.center)--(e.center)--(h.center)--cycle;
					
					\draw[fill=gray,fill opacity=.8] (d.center)--(f.center)--(i.center)--cycle;
					\draw[fill=lightgray,fill opacity=.5] (e.center)--(f.center)--(i.center)--cycle;
				\end{scope}
			\end{tikzpicture}$$
		\end{ex}
		
		\begin{ex}\label{Example: Intersection Complex Cycle Complex}
			Examples \ref{Example: Intersection Complex (1,1,0)} and \ref{Example: Intersection Complex (2,1,0)} are in fact special cases of the following result: for any positive integers $a$ and $b$ with $a\leq b$, the cycle complex $\fC_{a,b}$ has exactly three facets, and is isomorphic to the intersection complex $\calI(b-a,a,0)$.
		\end{ex}
		
		
		\begin{ex}\label{Example: Intersection Complex (0,1,0,0)}
			The intersection complex $\calI(0,1,0,0)$ is
			$$
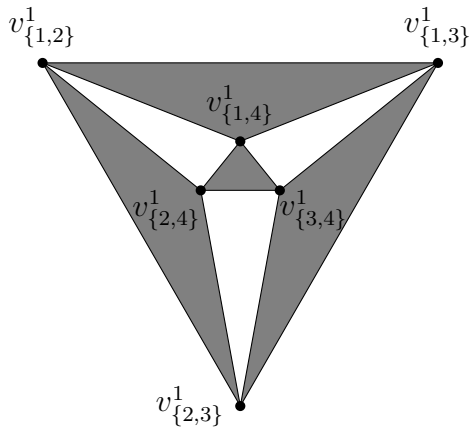
\begin{tikzpicture}[scale = 1.3]
				\tikzstyle{point}=[circle,thick,draw=black,fill=black,inner sep=0pt,minimum width=3pt,minimum height=3pt]
				\node (a)[point,label=above:$v^1_{\{1,2\}}$] at (0,3.5) {};
				\node (b)[point,label=above:$v^1_{\{1,4\}}$] at (2,2.7) {};
				\node (c)[point,label={[label distance = 0mm]above:$v^1_{\{1,3\}}$}] at (4,3.5) {};
				\node (d)[point,label={[label distance = -2.7mm]below left:$v^1_{\{2,4\}}$}] at (1.6,2.2) {};
				\node (e)[point,label={[label distance = -2.8mm]below right:$v^1_{\{3,4\}}$}] at (2.4,2.2) {};
				\node (f)[point,label=left:$v^1_{\{2,3\}}$] at (2,0) {};
				
				\begin{scope}[on background layer]
					\draw[fill=gray, fill opacity = .8] (d.center) -- (e.center) -- (b.center) -- cycle;
					\draw[fill=gray, fill opacity = .8] (a.center) -- (b.center) -- (c.center) -- cycle;
					\draw[fill=gray, fill opacity = .8] (a.center) -- (f.center) -- (d.center) -- cycle;
					\draw[fill=gray, fill opacity = .8] (c.center) -- (f.center) -- (e.center) -- cycle;
				\end{scope}
			\end{tikzpicture}$$
		\end{ex}
		
		\begin{ex}\label{Example: Intersection Complex Trivial Cases}
			There are some trivial cases of intersection complexes where the defining facets $F_1,\dots,F_n$ are all equal. Namely, for any $m\geq 0$, the complex $\Delta=\calI(\underbrace{0,\dots,0}_{n-1},m)$ contains only the $m$ vertices $v_{[n]}^1,...,v_{[n]}^m$, each of which is in every set $F_1,\dots,F_n$. Thus $\Delta$ is the full simplex on these $m$ vertices.
			
			As a special case of this, the complex $\Delta=\calI(\underbrace{0,\dots,0}_n)$ contains no vertices, so the sets $F_1,\dots,F_n$ are all empty, which means $\Delta$ is the irrelevant complex $\{\emptyset\}$.
		\end{ex}
		
		We can find the size of the intersection of any $i$ facets of an intersection complex as follows.
		
		\begin{lem}\label{Lemma: size of sigma_i}
			Let $\bm=(m_1,\dots,m_n)$ be a sequence of nonnegative integers, and let $\Delta = \calI(\bm)$ be the corresponding intersection complex with facets $F_1,\dots, F_n$. Suppose $\sigma_i = F_1 \cap \dots \cap F_i$ for some $1\leq i \leq n$. Then $|\sigma_i|=\sum_{j=0}^{n-i} \binom{n-i}{j} m_{i+j}$.
		\end{lem}
		\begin{proof}
			Let $S$ be a subset of $[n]$ and let $1\leq r\leq m_{|S|}$. The vertex $v_S^r$ is contained in the intersection $F_1\cap \dots \cap F_i$ if and only if we have $S\supseteq [i]$. For each $0\leq j \leq n-i$, there are ${n-i \choose j}$ subsets of $[n]$ containing $[i]$ of size $i+j$ (one for each choice of $j$ elements from the set $\{i+1,...,n\}$). The result follows.
		\end{proof}
		
		We now state our key theorems about intersection complexes, concerning the purity of their corresponding Betti diagrams, and their degree types and Betti numbers. In what follows, we fix a sequence of nonnegative integers $\bm=(m_1,\dots,m_n)$ such that $m_n= 0$ but $\bm\neq 0$. We also let $p$ denote the maximum value of $1\leq i\leq n-1$ for which $m_i\neq 0$. We define $\Delta=\calI(\bm)$ and $\beta=\beta(\Idstar)$.
		
		Note that if $m_n > 0$, then every facet of $\calI(\bm)$ contains the vertices $v_{[n]}^1,\dots,v_{[n]}^{m_n}$, which means $\calI(\bm)$ is a multi-cone over $\calI(m_1,\dots,m_{n-1},0)$, and so the two complexes have the same Betti diagram. The condition $m_n=0$ is therefore harmless.
		
		We wish to prove the following two theorems.
		
		\begin{thm}\label{Theorem: Intersection Complexes Degree Type}
			%
			Let $\bm=(m_1,\dots,m_n)$ be a sequence of nonnegative integers with $m_n=0$, and define $p=\max\{j\in [n] : m_j \neq 0 \}$. The intersection complex $\Delta=\calI(\bm)$ is PR with degree type $(d_p,\dots,d_1)$ where for each $1\leq i\leq p$, $d_i = \sum_{j=0}^{n-i-1} {n-i-1 \choose j} m_{i+j}$.
		\end{thm}
		
		In particular, the degree types of these intersection complexes are all the positive integer sequences $\bs$ for which every difference sequence of $\bs$ is monotonically increasing.
		
		\begin{thm}\label{Theorem: Intersection Complexes Betti Numbers}
			Let $\Delta$ and $\beta$ be as in Theorem \ref{Theorem: Intersection Complexes Degree Type}, and suppose $\beta$ has nonzero Betii numbers $\beta_{0,c_0},\dots.,\beta_{p,c_p}$. We have the following result. 
			\begin{equation*}
				\beta_{i,c_i}=\begin{cases}
					{n \choose i+1}& \text{if } 1\leq i \leq p-1\\
					{n-1 \choose p}& \text{if } i=p\, . 
				\end{cases}
			\end{equation*}
		\end{thm}
		
		Our proof for these theorems proceeds as follows. First, we make use of some deformation retractions to allow us to restrict our attention to intersection complexes of a particularly simple form. Next we show that these simple intersection complexes each have only a single nontrivial homology group. And then we assemble these pieces together to show that all intersection complexes are PR complexes with the desired degree types and Betti numbers.
		
		\subsection{Deformation Retractions and Links}
		For the rest of this section we let $\eseq{n}{i}$ denote the sequence $(\overbrace{\underbrace{0,\dots,0, 1}_{i},0,\dots,0}^n)$ (that is, the sequence of length $n$ whose only nonzero term is a $1$ at position $i$). We will also use $\eseq{n}{0}$ to denote the zero sequence of length $n$.
		
		The following result shows that all intersection complexes deformation retract on to an intersection complex of the form $\calI(\eseq{j}{i})$ for some $i$ and $j$.
		
		\begin{prop}\label{Proposition: Def Retract of Intersection Complex}
			Let $\bm=(m_1,\dots,m_n)$ be a nonzero sequence in $\ZZ_{\geq 0}^n$ with $m_n=0$, and define $p=\max \{i\in [n-1]: m_i \neq 0\}$. We have a deformation retraction $\calI(\bm)\leadsto \calI(\eseq{n}{p})$.
		\end{prop}
		\begin{proof}
			Let $v_S^r$ be a vertex of $\calI(\bm)$, with $S$ a subset of $[n]$ of size less than or equal to $p$, and $1 \leq r \leq m_{|S|}$. Choose any subset $S'\subseteq [n]$ of size $p$ which contains $S$. The vertex $v_{S'}^1$ lies inside $\calI(\eseq{n}{p})$, and any facet of $\calI(\bm)$ containing $v_S^r$ must also contain $v_{S'}^1$. Thus by Corollary \ref{Corollary: Deformation Retract} there is a deformation retraction $\calI(\bm) \leadsto \calI(\eseq{n}{p})$ obtained by identifying every vertex $v_S^r$ in $\calI(\bm)$ with an appropriate vertex $v_{S'}^1$ in $\calI(\eseq{n}{p})$.
		\end{proof}
		\begin{ex}\label{Example: Intersection complex def retract}
			For the intersection complex $\calI(2,1,0)$ in Example \ref{Example: Intersection Complex (2,1,0)} we have the deformation retraction
			\begin{center}
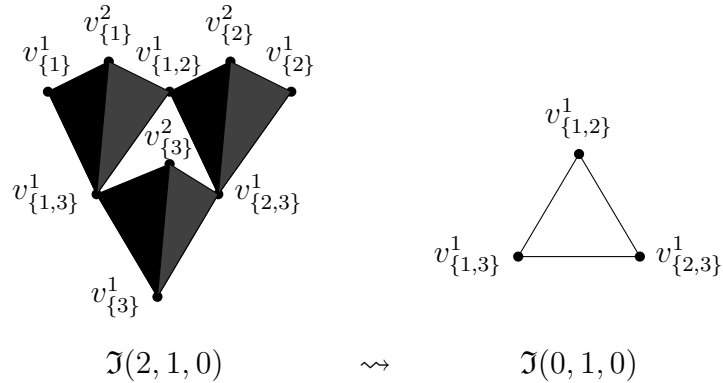
				\begin{tabular}{ccc}
					\begin{tikzpicture}[scale = 0.8]
						\tikzstyle{point}=[circle,thick,draw=black,fill=black,inner sep=0pt,minimum width=3pt,minimum height=3pt]
						\node (a)[point,label=above:$v_{\{1\}}^1$] at (0,3.4) {};
						\node (b)[point,label=above:$v_{\{1,2\}}^1$] at (2,3.4) {};
						\node (c)[point,label=above:$v_{\{2\}}^1$] at (4,3.4) {};
						\node (d)[point,label=left:$v_{\{1,3\}}^1$] at (0.8,1.7) {};
						\node (e)[point,label=right:$v_{\{2,3\}}^1$] at (2.8,1.7) {};
						\node (f)[point,label=left:$v_{\{3\}}^1$] at (1.8,0) {};
						
						\node (g)[point,label=above:$v_{\{1\}}^2$] at (1,3.9) {};
						\node (h)[point,label=above:$v_{\{2\}}^2$] at (3,3.9) {};
						\node (i)[point,label={[label distance = -1.5mm]above:$v_{\{3\}}^2$}] at (2,2.2) {};
						
						\draw [fill=lightgray, fill opacity=.5] (a.center) -- (b.center) -- (d.center) -- cycle;
						\draw [fill=lightgray, fill opacity=.5] (b.center) -- (c.center) -- (e.center) -- cycle;
						\draw [fill=lightgray, fill opacity=.7] (d.center) -- (e.center) -- (f.center) -- cycle;
						
						\draw[fill=black, fill opacity=1] (a.center)--(d.center)--(g.center)--cycle;
						\draw[fill=darkgray,fill opacity=.7] (b.center)--(d.center)--(g.center)--cycle;
						
						\draw[fill=black,fill opacity=1] (b.center)--(e.center)--(h.center)--cycle;
						\draw[fill=darkgray,fill opacity=.7] (c.center)--(e.center)--(h.center)--cycle;
						
						\draw[fill=black,fill opacity=1] (d.center)--(f.center)--(i.center)--cycle;
						\draw[fill=darkgray,fill opacity=.7] (e.center)--(f.center)--(i.center)--cycle;
					\end{tikzpicture} & & \begin{tikzpicture}[scale=0.8]
						\tikzstyle{point}=[circle,thick,draw=black,fill=black,inner sep=0pt,minimum width=3pt,minimum height=3pt]
						\node (a)[point, label=left:$v_{\{1,3\}}^1$] at (0,0) {};
						\node (b)[point, label=right:$v_{\{2,3\}}^1$] at (2,0) {};
						\node (c)[point, label=above:$v_{\{1,2\}}^1$] at (1,1.7) {};
						
						\node at (0,-1) {};
						
						\draw (a.center) -- (b.center) -- (c.center) -- cycle;
					\end{tikzpicture}\\ 
					$\calI(2,1,0)$& $\rightsquigarrow$ & $\calI(0,1,0)$\\
				\end{tabular}
			\end{center}
		\end{ex}
		
		In fact, not only does the intersection complex $\calI(\bm)$ itself deformation retract onto a complex of the form $\calI(\eseq{j}{i})$, but the link of any intersections of facets in $\calI(\bm)$ also deformation retracts onto a complex of this form. By symmetry (or more explicitly, by Lemma \ref{Lemma: link x = link gx}) it suffices to consider facet intersections of the form $\sigma_i=F_1\cap \dots \cap F_i$.
		\begin{prop}\label{Proposition: Def Retract of Intersection Complex Link}
			Let $\bm=(m_1,\dots,m_n)$ be a nonzero sequence in $\ZZ_{\geq 0}^n$ with $m_n=0$, and let $\Delta = \calI(\bm)$ be the corresponding intersection complex with facets $F_1,\dots, F_n$. Suppose $\sigma_i = F_1 \cap \dots \cap F_i$ for some $1\leq i \leq p$. We have a deformation retraction $\lkds_i\leadsto \calI(\eseq{i}{i-1})$.
		\end{prop}
		\begin{proof}
			The complex $\calI(\eseq{i}{i-1})$ contains precisely those vertices of the form $v^1_{[i]-\{j\}}$ for $1 \leq j \leq i$. All of these vertices are contained in exactly $i-1$ of the facets $F_1,\dots,F_i$, which means none of them is contained in $\sigma_i$ and all are therefore contained in $\lkds_i = \langle F_1 - \sigma_i,\dots, F_i-\sigma_i \rangle$.
			
			Let $v^r_S$ be any vertex in $\lkds_i$, with $S$ a subset of $[n]$ and $1 \leq r \leq m_{|S|}$. Because $v^r_S$ does not lie in $\sigma_i$, there must be some $1\leq j \leq i$ for which $v^r_S$ does not lie in the facet $F_j$. This means that we have $j\notin S$, and hence $S\cap [i] \subseteq [i]-\{j\}$. Thus every facet of $\lkds_i$ containing $v^r_S$ must also contain $v^1_{[i]-\{j\}}$.
			
			Thus by Corollary \ref{Corollary: Deformation Retract} there is a deformation retraction $\lkds_i\leadsto \calI(\eseq{i}{i-1})$ given by identifying every vertex $v_S^r$ in $\lkds_i$ with an appropriate vertex $v^1_{[i]-\{j\}}$ in $\calI(\eseq{i}{i-1})$.
		\end{proof}
		\begin{ex}\label{Example: Intersection complex link def retract}
			For the intersection complex $\Delta=\calI(2,1,0)$ in Example \ref{Example: Intersection Complex (2,1,0)} we have $\sigma_2=F_1\cap F_2 = \{v^1_{1,2}\}$, and we get the following deformation retraction.
			\begin{center}
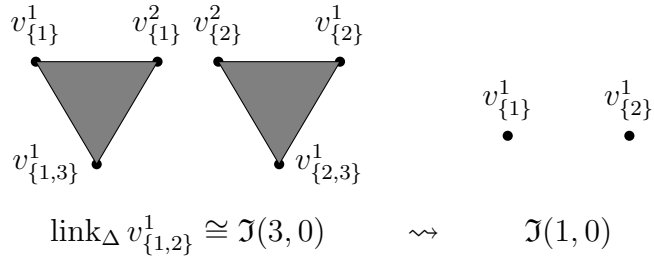
				\begin{tabular}{ccc}
					\begin{tikzpicture}[scale = 0.8]
						\tikzstyle{point}=[circle,thick,draw=black,fill=black,inner sep=0pt,minimum width=3pt,minimum height=3pt]
						\node (a)[point,label=above:$v_{\{1\}}^1$] at (0,3.4) {};
						\node (g)[point,label=above:$v_{\{1\}}^2$] at (2,3.4) {};
						\node (d)[point,label=left:$v_{\{1,3\}}^1$] at (1,1.7) {};
						
						\node (c)[point,label=above:$v_{\{2\}}^1$] at (5,3.4) {};
						\node (h)[point,label=above:$v_{\{2\}}^2$] at (3,3.4) {};
						\node (e)[point,label=right:$v_{\{2,3\}}^1$] at (4,1.7) {};
						
						\draw[fill=gray] (a.center)--(d.center)--(g.center)--cycle;
						\draw[fill=gray] (c.center)--(e.center)--(h.center)--cycle;
					\end{tikzpicture}  & & \begin{tikzpicture}[scale=0.8]
						\tikzstyle{point}=[circle,thick,draw=black,fill=black,inner sep=0pt,minimum width=3pt,minimum height=3pt]
						\node (a)[point, label=above:$v_{\{1\}}^1$] at (1,2.5) {};
						\node (b)[point, label=above:$v_{\{2\}}^1$] at (3,2.5) {};
						
						\node at (1,1.7) {};
						
					\end{tikzpicture}\\ 
					$\lkds[v^1_{\{1,2\}}]\cong \calI(3,0)$& $\rightsquigarrow$ & $\calI(1,0)$\\
				\end{tabular}
			\end{center}
		\end{ex}
		
		These two deformation retractions allow us to restrict our attention solely to the homologies of the complexes $\calI(\eseq{n}{p})$ for $p=1,\dots,n-1$. To find the homologies, we will make use of the following lemma.
		
		Note that the vertex set of $\calI(\eseq{n}{p})$ contains exactly one vertex $v_S^1$ for each subset $S$ in $[n]$ of size $p$. For ease of notation, we label this vertex as $v_S$ rather than $v_S^1$.
		\begin{lem}\label{Lemma: subcomplexes e^n_p}
			Let $\Delta$ be the complex $\calI(\eseq{n}{p})$ for some $1\leq p \leq n-1$, with facets $F_1,...,F_n$.
			
			Define $A$ to be the subcomplex of $\Delta$ generated by $F_1,...,F_{n-1}$ and $B$ to be the subcomplex of $\Delta$ generated by $F_n$. We have the following results.
			\begin{enumerate}
				\item $A$ deformation retracts on to $\calI(\eseq{n-1}{p})$.
				\item $A\cap B$ is homeomorphic to $\calI(\eseq{n-1}{p-1})$.
			\end{enumerate}
		\end{lem}
		\begin{proof}
			For part (1), note that $A$ contains the complex $\calI(\eseq{n-1}{p})$, and the only vertices in $A$ outside of $\calI(\eseq{n-1}{p})$ are those of the form $v_S$ where $S$ is a subset in $[n]$ of size $p$ containing $n$. For any such $S$, we may choose some $l\in [n] - S$. All of the facets $F_1,...,F_{n-1}$ containing the vertex $v_S$ must also contain the vertex $v_{S\cup\{l\}-\{n\}}$, which lies inside $\calI(\eseq{n-1}{p})$, and hence we may identify the vertex $v_S$ with the vertex $v_{S\cup\{l\}-\{n\}}$. This gives us a deformation retract $\Delta$ on to its subcomplex $\calI(\eseq{n-1}{p})$.
			
			For part (2), note that $A \cap B = \langle F_1\cap F_n,....,F_{n-1} \cap F_n \rangle$. Hence the vertices in $A\cap B$ are all those of the form $v_S$ where $S$ is a subset in $[n]$ of size $p$ containing $n$. Thus there is a bijection from the vertex set of $A\cap B$ to the vertex set of $\calI(\eseq{n-1}{p-1})$ given by $v_S\mapsto v_{S-\{n\}}$. This bijection takes each facet $F_i\cap F_n$ to a facet of $\calI(\eseq{n-1}{p-1})$, and hence it is a homeomorphism.
		\end{proof}
		
		Using Lemma \ref{Lemma: subcomplexes e^n_p}, we can find the homology of the complex $\calI(\eseq{n}{p})$.
		\begin{prop}\label{Proposition: homology of e^n_p}
			For all $0\leq p\leq n-1$, the complex $\calI(\eseq{n}{p})$ has only $\nth[st]{(p-1)}$ homology, of dimension ${n-1 \choose p}$.
		\end{prop}
		\begin{proof}
			We proceed by induction on $n\geq 1$ and $p\geq 0$. For the base cases, if $p=0$, the complex $\calI(\eseq{n}{p})=\{\emptyset\}$ has only $\nth[st]{(-1)}$ homology, of dimension $1$. Similarly if $n=1$, then we must have $p=0$, so the same reasoning applies.
			
			Now suppose $n\geq 2$ and $1\leq p \leq n-1$, and let $\Delta=\calI(\eseq{n}{p})$, with facets $F_1,...,F_n$. We define the subcomplexes $A$ and $B$ of $\Delta$ as in Lemma \ref{Lemma: subcomplexes e^n_p}. Note that $B$ consists of a single facet and is therefore acyclic.
			
			By Lemma \ref{Lemma: subcomplexes e^n_p}, we have that $A$ deformation retracts on to $\calI(\eseq{n-1}{p})$. If $p=n-1$, then $\calI(\eseq{n-1}{p})$ is a full simplex and is therefore acyclic. Otherwise, by the inductive hypothesis, it has only $(p-1)^\text{st}$ homology, of dimension ${n-2 \choose p}$. Either way we have that $\dim_\KK \Hred_{p-1} (\calI(\eseq{n-1}{p})) = {n-2\choose p}$ and all other homologies are zero.
			
			We also have that $A\cap B$ is homeomorphic to $\calI(\eseq{n-1}{p-1})$. By the inductive hypothesis, this has only $(p-2)^\text{nd}$ homology, of dimension ${n-2 \choose p-1}$.
			
			Thus the Mayer-Vietoris Sequence yields an exact sequence
			$$0\rightarrow \Hred_{p-1}(A)\rightarrow \Hred_{p-1}(\Delta)\rightarrow\Hred_{p-2}(A\cap B)\rightarrow 0$$
			which means that $\Delta$ has only $(p-1)^\text{st}$ homology, and this homology has dimension ${n-2 \choose p}+{n-2\choose p-1}={n-1 \choose p}$.
		\end{proof}
		
		%
		Lemma \ref{Lemma: link facet intersections} has allowed us to restrict our attention to the links of intersection of facets; however, before we move on to the proof of Theorems \ref{Theorem: Intersection Complexes Degree Type} and \ref{Theorem: Intersection Complexes Betti Numbers}, it is worth taking a moment to consider the links of \textit{arbitrary} faces in intersection complexes. Notably, it turns out that every link in an intersection complex is also an intersection complex, making this the only family of complexes presented in this chapter which is closed under the operation of taking links. To prove this, it suffices to consider the links of vertices, by Lemma \ref{Lemma: links in links}; and by symmetry considerations we need only consider those vertices of the form $v^{m_i}_{[i]}$ for some integer $1\leq i\leq n$.
		
		\begin{lem}\label{Lemma: Intersection complex link of arbitrary face}
			Let $\bm=(m_1,\dots,m_n)$ be a nonzero sequence in $\ZZ_{\geq 0}^n$. Fix some integer $1\leq i\leq n$, and let $v$ denote vertex $v^{m_i}_{[i]}$ in the intersection complex $\Delta=\calI(\bm)$. There is an isomorphism of complexes $$\lkds[v]\cong\calI(k_1,\dots,k_i)$$ where for each $1\leq l \leq i$ we define $$k_l=\begin{cases}
				\sum_{j=0}^{n-i}{n -i \choose j}m_{j+l} & \text{ if } 1\leq l \leq i-1\\
				\sum_{j=0}^{n-i}{n -i \choose j}m_{j+l} -1& \text{ if } l=i.
			\end{cases}$$.
		\end{lem}
		\begin{proof}
			The facets of $\Delta$ containing $v$ are $F_1,\dots,F_i$, and hence we have $L=\lkds[v]=\langle F_1-v,\dots,F_i-v\rangle$. For any subset $S$ of $[i]$, the size of the intersection $\bigcap_{j\in S}(F_j-v)$ is $|\bigcap_{j\in S_1}F_j|-1$, which is wholly dependent on the size of $S$; and the link of $\bigcap_{j\in S}(F_j -v)$ in $L$ is equal to the link of $\bigcap_{j\in S_1}F_j$ in $\Delta$, which is also wholly dependent on the size of $S$. This shows that $L$ is intersectionally symmetric, and thus must be an intersection complex $\calI(\bk)$ for some sequence $\bk(k_1,\dots,k_i)$ in $\ZZ^n_{\geq 0}$.
			
			Pick some $1\leq l\leq i$. As observed in Remark \ref{Remark: Intersection Complex Observations}, the value $k_l$ is the number of vertices contained in the intersection $(F_1\cap \dots \cap F_l)-v$ which are not contained in an intersection of $l+1$ facets of $L$. These are the vertices of the form $v^r_S$ for subsets $S\subseteq [n]$ such that $S\cap [i]=[l]$.
			
			For each $0\leq j \leq n-i$, there are ${n-i \choose j}$ such subsets of $[n]$ of size $l+j$ (one for each choice of $j$ elements from the set $\{i+1,...,n\}$). Each such subset $S$ has $m_{j+l}$ corresponding vertices $v^1_S,\dots,v_S^{m_{j+l}}$ in $L$, except in the case where $l=i$ and $S=[i]$, in which case $L$ contains only the $m_i-1$ vertices $v_S^1,\dots,v_{[i]}^{m_i-1}$ (because the vertex $v=v_{[i]}^{m_i}$ is not in $L$). The result follows.
		\end{proof}
		
		In particular, every complex in the link poset of an intersection complex is another intersection complex. We give a few examples of the link posets of intersection complexes below. All the complexes $\calI(m_1,\dots,m_n)$ in the chain with $m_n\neq 0$ are cones, and therefore acyclic; and by Propositions \ref{Proposition: Def Retract of Intersection Complex} and \ref{Proposition: homology of e^n_p}, the rest of the complexes have homology index set $h(\calI(\bm))=\{p-1\}$ where $p=\max\{i\in [n-1] : m_i\neq 0\}$.
		\begin{ex}\label{Example: Intersection Complex (0,1,0,0) Link Poset}
			Consider the intersection complex $\Delta=\calI(0,1,0,0)$ in Example \ref{Example: Intersection Complex (0,1,0,0)}. The maximal chains in the link poset $P_\Delta$ all look like the following.
			\begin{center}
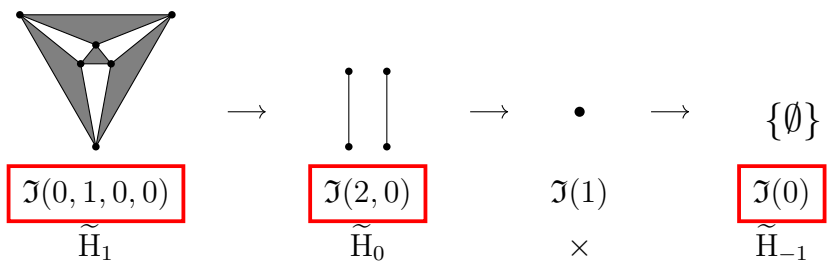
				\begin{tabular}{ c c c c c c c }
					\begin{tikzpicture}[scale = 0.5]
						\tikzstyle{point}=[circle,thick,draw=black,fill=black,inner sep=0pt,minimum width=2pt,minimum height=2pt]
						\node (a)[point] at (0,3.5) {};
						\node (b)[point] at (2,2.7) {};
						\node (c)[point] at (4,3.5) {};
						\node (d)[point] at (1.6,2.2) {};
						\node (e)[point] at (2.4,2.2) {};
						\node (f)[point] at (2,0) {};
						
						\begin{scope}[on background layer]
							\draw[fill=gray, fill opacity = .8] (d.center) -- (e.center) -- (b.center) -- cycle;
							\draw[fill=gray, fill opacity = .8] (a.center) -- (b.center) -- (c.center) -- cycle;
							\draw[fill=gray, fill opacity = .8] (a.center) -- (f.center) -- (d.center) -- cycle;
							\draw[fill=gray, fill opacity = .8] (c.center) -- (f.center) -- (e.center) -- cycle;
						\end{scope}
					\end{tikzpicture}& \begin{tikzpicture}[scale = 0.5]
						\tikzstyle{point}=[circle,thick,draw=black,fill=black,inner sep=0pt,minimum width=2pt,minimum height=2pt]
						\node (a) at (0,0.75) {};
						\node (b) at (1,0.75) {};
						\node (c) at (0.5,0) {};
						
						\draw[->] (a.center) -- (b.center);
					\end{tikzpicture} & \begin{tikzpicture}[scale = 0.5]
						\tikzstyle{point}=[circle,thick,draw=black,fill=black,inner sep=0pt,minimum width=2pt,minimum height=2pt]
						\node[point] (a) at (0,0) {};
						\node[point] (b) at (0,2) {};
						\node[point] (c) at (1,0) {};
						\node[point] (d) at (1,2) {};
						
						\draw (a.center) -- (b.center);
						\draw (c.center) -- (d.center);
					\end{tikzpicture}  & \begin{tikzpicture}[scale = 0.5]
						\tikzstyle{point}=[circle,thick,draw=black,fill=black,inner sep=0pt,minimum width=2pt,minimum height=2pt]
						\node (a) at (0,0.75) {};
						\node (b) at (1,0.75) {};
						\node (c) at (0.5,0) {};
						
						\draw[->] (a.center) -- (b.center);
					\end{tikzpicture} & \begin{tikzpicture}[scale = 0.5]
						\tikzstyle{point}=[circle,thick,draw=black,fill=black,inner sep=0pt,minimum width=3pt,minimum height=3pt]
						\node (a)[point] at (0,0.75) {};
						\node (c) at (0,0) {};
					\end{tikzpicture}  &
					\begin{tikzpicture}[scale = 0.5]
						\tikzstyle{point}=[circle,thick,draw=black,fill=black,inner sep=0pt,minimum width=2pt,minimum height=2pt]
						\node (a) at (0,0.75) {};
						\node (b) at (1,0.75) {};
						\node (c) at (0.5,0) {};
						
						\draw[->] (a.center) -- (b.center);
					\end{tikzpicture} & \begin{tikzpicture}[scale = 0.5] \node (q)[label=right:\large$\{\emptyset\}$] at (,){};\end{tikzpicture}\\
					\begin{tikzcd}
						|[draw=red, line width =.5mm, rectangle]|\calI(0,1,0,0)\end{tikzcd}&& \begin{tikzcd}
						|[draw=red, line width =.5mm, rectangle]|\calI(2,0)\end{tikzcd} && $\calI(1)$ && \begin{tikzcd}
						|[draw=red, line width =.5mm, rectangle]|\calI(0)\end{tikzcd}\\
					$\Hred_1$&&$\Hred_0$&&$\times$&&$\Hred_{-1}$
				\end{tabular}
			\end{center}
			This diagram demonstrates that $\calI(0,1,0,0)$ has degree type $(1,2)$.
		\end{ex}
		\begin{ex}\label{Example: Intersection Complex (1,1,1,0) Link Poset}
			Consider the intersection complex $\Delta=\calI(0,1,1,0)$. Without needing to draw the complex, we can use Lemma \ref{Lemma: Intersection complex link of arbitrary face} to find all the possible chains in the link poset $P_\Delta$. In the below diagram we take $\delta_1 \arrowX{i} \delta_2$ to denote that $\delta_1=\link_{\delta_1} v^j_S$ for some set $S$ of size $i$ (i.e. it denotes taking the link of a vertex which lies inside exactly $i$ facets).
			
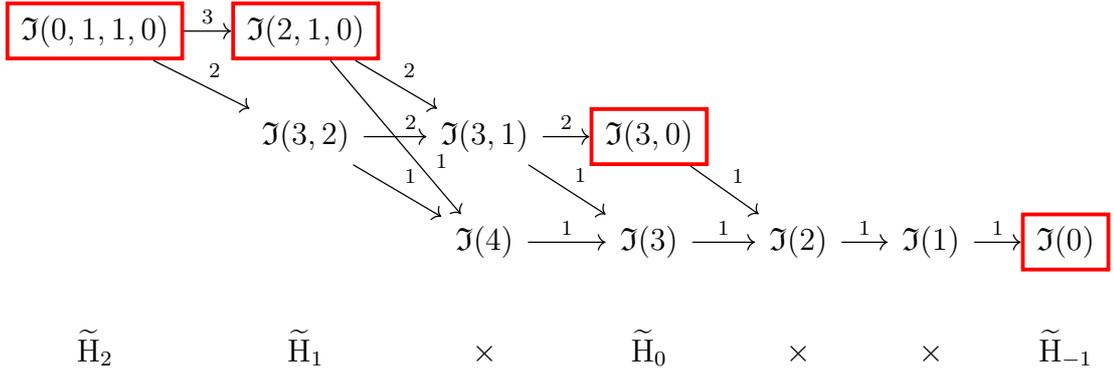
\[
			\begin{tikzcd}[row sep=1.5em,column sep=1.5em]
				|[draw=red, line width =.5mm, rectangle]|\calI(0,1,1,0)  \arrow{dr}{2}\arrow{r}{3} &|[draw=red, line width =.5mm, rectangle]|\calI(2,1,0)\arrow[to=Z, near end, "1"]\arrow{dr}{2}&&&&&\\
				& \calI(3,2)\arrow{dr}{1}\arrow[to=Y, near end, "2"] & |[alias=Y]|\calI(3,1)\arrow{dr}{1}\arrow{r}{2} & |[draw=red, line width =.5mm, rectangle]|\calI(3,0)\arrow{dr}{1} & & & & \\
				& &|[alias=Z]|\calI(4)\arrow{r}{1} & \calI(3)\arrow{r}{1}& \calI(2)\arrow{r}{1}& \calI(1)\arrow{r}{1}& |[draw=red, line width =.5mm, rectangle]|\calI(0)\\
				\Hred_2& \Hred_1 &\times& \Hred_0&\times&\times& \Hred_{-1}
			\end{tikzcd}
			\]
			This diagram demonstrates that $\calI(0,1,1,0)$ has degree type $(1,2,3)$.
		\end{ex}
		
		\subsection{Proving Theorems \ref{Theorem: Intersection Complexes Degree Type} and \ref{Theorem: Intersection Complexes Betti Numbers}}
		
		We now have all the ingredients we need to prove Theorems \ref{Theorem: Intersection Complexes Degree Type} and \ref{Theorem: Intersection Complexes Betti Numbers}. We prove both theorems together below.
		
		\begin{proof}[Proof of Theorems \ref{Theorem: Intersection Complexes Degree Type} and \ref{Theorem: Intersection Complexes Betti Numbers}]
			As in Proposition \ref{Proposition: Def Retract of Intersection Complex Link}, we define $\sigma_i$ to be the intersection of facets $F_1\cap \dots \cap F_i$ for $1\leq i \leq p$. We also define $\sigma_{p+1}=F_1\cap \dots \cap F_{p+1}=\emptyset$. Propositions \ref{Proposition: Def Retract of Intersection Complex}, \ref{Proposition: Def Retract of Intersection Complex Link} and \ref{Proposition: homology of e^n_p} show that for $1\leq i\leq p+1$, the link $\lkds_i$ has only $(i-1)^\text{st}$ homology.
			
			We wish to show that $\Delta=\calI(\bm)$ is a PR complex with the desired degree type and Betti numbers.
			
			To see that $\Delta$ is a PR complex, suppose we have two faces $\tau_1$ and $\tau_2$ of $\Delta$ whose links both have nontrivial homology at the same degree. By Lemma \ref{Lemma: link facet intersections}, $\tau_1$ and $\tau_2$ must be intersections of facets. By the intersectional symmetry of $\Delta$ we may assume that $\tau_1=\sigma_{j_1}$ and $\tau_2 = \sigma_{j_2}$ for some $1\leq j_1,j_2\leq p+1$, and thus their links have $(j_1-1)^\text{st}$ homology and $(j_2-1)^\text{st}$ homology respectively. This means $j_1=j_2$, and in particular $|\tau_1|=|\tau_2|$.
			
			We now move on to the degree type. For any $1\leq i \leq p$, the value $d_i$ as in Definition \ref{Definition: PR Complex Degree Types} is given by $d_i = |\sigma_i|-|\sigma_{i+1}|$. Proposition \ref{Lemma: size of sigma_i} tells us that $|\sigma_i| = \sum_{j=0}^{n-i} {n-i \choose j} m_{i+j}$, and hence we have
			\begin{align*}
				d_i &= \sum_{j=0}^{n-i} {n-i \choose j} m_{i+j} - \sum_{j=0}^{n-i-1} {n-i-1 \choose j} m_{i+j+1}\\
				&= \sum_{j=0}^{n-i} {n-i \choose j} m_{i+j} - \sum_{j=0}^{n-i} {n-i-1 \choose j-1} m_{i+j}\\
				&= \sum_{j=0}^{n-i} ({n-i \choose j}-{n-i-1 \choose j-1}) m_{i+j}\\
				&= \sum_{j=0}^{n-i} {n-i-1 \choose j} m_{i+j}
			\end{align*}
			
			Finally we look at the Betti numbers $\beta_{0,c_0},\dots,\beta_{p,c_p}$. First we note that $\beta_{p,c_p}$ is equal to $\dim_{\KK}\Hred_{p-1}(\Delta)$, which is ${n-1\choose p}$ by Propositions \ref{Proposition: Def Retract of Intersection Complex} and \ref{Proposition: homology of e^n_p}. Let $0\leq i \leq p-1$. For any face $\tau$ of $\Delta$, the complex $\link_\Delta \tau$ has nontrivial $(i-1)^\text{st}$ homology if and only if it is an intersection of $i+1$ facets, in which case this homology has dimension ${i-1\choose i-1}=1$ by Propositions \ref{Proposition: Def Retract of Intersection Complex Link} and \ref{Proposition: homology of e^n_p}. Thus, by Theorem \ref{Theorem: ADHF}, the Betti number $\beta_{i,c_i}$ is equal to the number of intersections of $i+1$ facets in $\Delta$, which is ${n\choose i+1}$.
		\end{proof}
		
		\begin{rem}\label{Remark: Suspected Minimality of Intersection Complexes}
			In general, intersection complexes do \textit{not} have minimal vertex sets for their given degree type. This can be seen from the fact that many distinct intersection complexes have the \textit{same} degree type, despite having different numbers of vertices. For example both the intersection complex $\calI(2,1,0)$ from Example \ref{Example: Intersection Complex (2,1,0)} and the intersection complex $\calI(1,1,0,0)$ (which can be obtained from the complex $\calI(0,1,0,0)$ in Example \ref{Example: Intersection Complex (0,1,0,0)} by adding an additional free vertex to every facet) have degree type $(3,1)$, but the latter complex has $10$ vertices while the former has only $9$. 
			
			However, it seems extremely likely that intersection complexes of the form $\calI(m_1,\dots,m_{n-1},0)$ for which $m_{n-1}\neq 0$ \textit{are} minimal for their given degree type. These are the intersection complexes which deformation retract onto the boundary of a simplex, by Proposition \ref{Proposition: Def Retract of Intersection Complex}. Certainly these complexes have a minimal number of \textit{facets} for their given degree type; indeed, $\calI(m_1,\dots,m_{n-1},0)$ has the same number of facets as the $(n-1)$-simplex, which is the minimum number required for a complex to have homology at degree $n-2$.
		\end{rem}

		\section{Partition Complexes}\label{Subsection: Partition Complexes}
		In this section we construct another infinite family of PR complexes, which we call \textit{partition complexes}. Just as every intersection complex has a corresponding sequence of nonnegative integers $\bm$, every partition complex has three corresponding integers $a$, $p$ and $m$, and we denote them accordingly as $\calP(a,p,m)$. We show that when $a\geq 2$, and $1\leq m\leq p$ the partition complex $\calP(a,p,m)$ has degree type $(\overbrace{1,\dots,1\underbrace{a,1,\dots,1}_{m}}^{p})$.
			
			Just like intersection complexes, partition complexes can be seen as generalisations of the boundary complexes of simplices. Specifically, we saw in the last section that for any sequence of nonnegative integers $\bm=(m_1,\dots,m_{p-1},0)$ with $m_{p-1}\neq 0$, the intersection complex $\calI(\bm)$ deformation retracts onto the boundary of the $p$-simplex. The same is also true for the partition complex $\calP(a,p,m)$ for any integers $a\geq 2$ and $1\leq m \leq p$.
			
			
			Unlike most of the complexes we have seen so far in this chapter, however, we do \textit{not} suspect partition complexes of being minimal examples of their given degree type in general. In fact, in a number of cases we know explicitly that they are \textit{not} minimal (for example, the partition complex of degree type $(2,1)$ is the complex in Example \ref{Example: PR 21 Complex}, which has six vertices, whereas the Cycle Complex $\fC_{2,1}$ in Example \ref{Example: Cycle Complexes} has only five vertices).
			
			However we include their construction anyway, for two reasons. The first is that, in \textit{most} cases, they are still the smallest PR complexes of their given degree type that we have been able to find so far, and so they allow us to lower the bounds for $n$ in Question \ref{Question: Lower bounds on n for each degree type}. The second is that the construction of partition complexes bears a strong similarity to the key construction we will present in Chapter \ref{Chapter: Generating Degree Types} to prove Theorem \ref{Theorem: PR Complexes of Any Degree Type}, and thus may help to illuminate some of the most important steps in that proof.
			
			\subsection{Defining the Complexes}
			In this section we define the family of partition complexes along with some examples, and present some preliminary results about them. For any integer $p\geq 0$ the partition complex $\calP(a,p,m)$ admits a natural symmetry under the action of the symmetric group $\Sym\{0,\dots,p\}$ (which is isomorphic to $S_{p+1}$). For notational convenience we will denote the group $\Sym\{0,\dots,p\}$ by $S^0_p$.
			
			It is worth noting that our main theorem on the degree types of the partition complex $\calP(a,p,m)$ (Theorem \ref{Theorem: Partition Complex Degree Type}) holds only for integers $a\geq 2$ and $1\leq m \leq p$. However, we define our complexes below slightly more broadly, to include the additional cases $p\in \{-1,0\}$ and $m\in \{0,p+1\}$. We do not care about these fringe cases of partition complexes for their own sake, but their construction will be crucial to our proof of Theorem \ref{Theorem: Partition Complex Degree Type}, because they occur as links in the partition complexes that we do care about.
			
			\begin{defn}\label{Definition: Partition Complex Vertex Set}
				For two integers $a\geq 2$ and $p\geq -1$. We define the vertex set $V_p^a$ to be the set consisting of vertices of the form $x_i$ and $y_i^j$ for $0\leq i\leq p$ and $1\leq j \leq a-1$. For convenience, we often write $y_i^1$ simply as $y_i$.
				
				We will sometimes partition $V_p^a$ into subsets $X_p \sqcup Y_p^a$ with $X_p=\{x_i\}_{0\leq i \leq p}$ and $Y_p^a =\{y_i^j\}_{0\leq i \leq p, 1\leq j \leq a-1}$. For reasons that will become apparent we refer to the vertices in $X_p$ as \textit{boundary vertices} and the vertices in $Y_p^a$ as \textit{partition vertices}. We sometimes make a further distinction between those partition vertices $y_i^j$ for which $j\geq 2$ and those for which $j=1$, by referring to them respectively as \textit{upper} and \textit{lower} partition vertices.
			
			For $p\geq 0$ the symmetry group $S^0_p=\Sym\{0,\dots,p\}$ acts on $V_p^a$ via the action
			\begin{align*}
				\sigma(x_i)&=x_{\sigma(i)}\\
				\sigma(y_i^j)&=y_{\sigma(i)}^j\\
			\end{align*}
		\end{defn}
		
		\begin{defn}\label{Definition: Partition Complex Generating Sets}
			Let $a$, $p$ and $i$ be integers with $a\geq 2$, $p\geq -1$ and $1\leq i \leq p+1$.
			
			Let $\PP(a,i)$ denote the set of partitions $\lambda = (\lambda_1,...,\lambda_i)$ of $a+i-2$ into $i$ parts (i.e. $\lambda_1\geq \dots \geq \lambda_i>0$ and $\sum_{j=1}^i\lambda_j=a+i-2$).
			
			For $\lambda = (\lambda_1,...,\lambda_i) \in \PP(a,i)$, we define the $\lambda$-\textit{generating set} $G^\lambda_{p,i}$ as follows.
			\begin{itemize}
				\item $G_{p,i}=\{x_i,\dots,x_p\}\subset X_p$.
				\item $G_{p,i}^\lambda = G_{p,i}\sqcup \{y_r^j:0\leq r\leq i-1, 1\leq j\leq \lambda_{r+1}\}\subset V_p^a$.
			\end{itemize}
		\end{defn}
		
		\begin{defn}\label{Definition: Partition Complexes}
			Let $a$ and $p$ be integers with $a\geq 2$, $p\geq -1$. For a third integer $1\leq m \leq p+1$ we define the \textit{partition complex} $\calP(a,p,m)$ on vertex set $V_p^a$ to be
			$$\calP(a,p,m) = \left \langle G^\lambda_{p,i} : 1\leq i \leq m, \lambda \in \PP(a,i) \right \rangle_{S^0_p}.$$
			We also define the partition complex $\calP(a,p,0)$ to be the irrelevant complex $\{\emptyset\}$ on vertex set $V_p^a$.
		\end{defn}
		
		By definition, partition complexes are symmetric under the action of $S^0_p$.
		
		\begin{rem}\label{Remark: Partitions of a-2}
			For every $1\leq i \leq m$ and every partition $\lambda\in \PP(a,i)$, the $\lambda$-generating set $G^\lambda_{p,i}$ contains the lower partition vertices $y^1_1,\dots,y^1_i$. This means that it must also contain exactly $a-2$ upper partition vertices. Thus for $m\geq 1$ and $p\geq 0$, every facet of the partition complex $\calP(a,p,m)$ contains exactly $a-2$ upper partition vertices.
		\end{rem}
		
		We now consider some special cases and examples to help illustrate the construction. We begin with the cases for small values of $p$.
		\begin{ex}
			\underline{\textbf{The case $p=-1$:}}
			
			The vertex set $V^a_{-1}$ is empty, and we have $\calP(a,-1,0)=\{\emptyset\}$, by definition.
		\end{ex} 
		\begin{ex}
			\underline{\textbf{The case $p=0$:}}
			
			The vertex set $V^a_0$ is equal to $\{x_0,y_0^1,\dots,y_o^{a-1}\}$. The complex  $\calP(a,0,0)$ is the irrelevant complex just as above, and we also have $$\calP(a,0,1)=\langle G^{(a-1)}_{0,1} \rangle = \langle \{y_0^1,\dots,y_0^{a-1}\}\rangle.$$
		\end{ex}
		
		\begin{ex}\label{Example: Partition Complex p=m=1}
			\underline{\textbf{The case $p=m=1$:}} 
			
			The partition complex $\calP(a,1,1)$ has precisely two facets, namely the facets $\{y_0^1,\dots,y_0^{a-1},x_1\}$ and $\{y_1^1,\dots,y_1^{a-1},x_0\}$. Thus it comprises of two disjoint $(a-1)$-simplices, so it is PR with degree type $(a)$.
		\end{ex}
		
		In the above example, the complex $\calP(a,1,1)$ deformation retracts on to the boundary of the full simplex on $X_1=\{x_0,x_1\}$, $\partial \langle X_1 \rangle$. In fact we will see in the next section that for \textit{any} $1\leq m \leq p$, the partition complex $\calP(a,p,m)$ always deformation retracts onto the boundary of the full simplex on $X_p$, $\partial \langle X_p \rangle$.
		\begin{ex}\label{Example: Partition Complex a=2}
			\underline{\textbf{The case $a=2$:}}
			
			The vertex set $V^2_p$ is the set $\{x_0,\dots,x_p,y_0,\dots,y_p\}$. For any positive integer $i$, there is only one partition of $2+i-2=i$ into $i$ parts, namely the partition $\underbrace{(1,\dots,1)}_i$. Thus for any integers $m$ and $p$ with $1 \leq m\leq p+1$, the generating facets of $\calP(2,p,m)$ are the sets
			\begin{align*}
				G_{p,1}^{(1)}&= \{y_0,x_1,\dots,x_p\}\\
				G_{p,2}^{(1,1)}&= \{y_0,y_1,x_2\dots,x_p\}\\
				\vdots\\
				G_{p,m}^{(1,\dots,1)}&=\{y_0,\dots,y_{m-1},x_m,\dots,x_p\}.
			\end{align*}
			The facets generated by these sets under the action of $S^0_p$ are all the sets $F=\{\varepsilon_0,\dots, \varepsilon_p\}$ such that $\varepsilon_i\in \{x_i,y_i\}$ for each $0\leq i \leq p$ and the number of partition vertices in $F$ is less than or equal to $m$.
			
			We have seen two examples of this already in Section \ref{Subsection: PR Examples}. Namely, the complex in Examples \ref{Example: PR Zelda Symbol} and \ref{Example: PR 21 Complex} are, respectively, the partition complexes $\calP(2,2,1)$ and $\calP(2,2,2)$, as shown below.
			
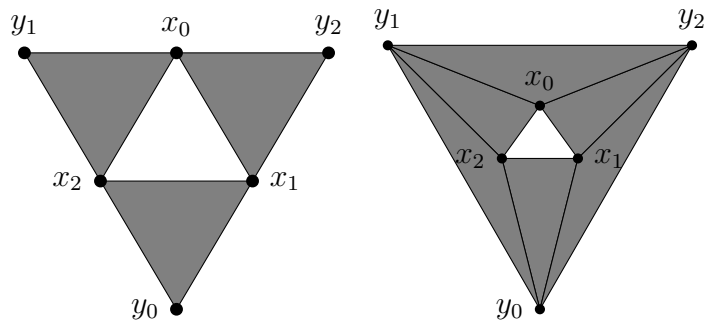
\begin{center}
				\begin{tikzpicture}[scale = 1]
					\tikzstyle{point}=[circle,thick,draw=black,fill=black,inner sep=0pt,minimum width=4pt,minimum height=4pt]
					\node (a)[point,label=above:$y_1$] at (0,3.4) {};
					\node (b)[point,label=above:$x_0$] at (2,3.4) {};
					\node (c)[point,label=above:$y_2$] at (4,3.4) {};
					\node (d)[point,label=left:$x_2$] at (1,1.7) {};
					\node (e)[point,label=right:$x_1$] at (3,1.7) {};
					\node (f)[point,label=left:$y_0$] at (2,0) {};	
					
					\begin{scope}[on background layer]
						\draw[fill=gray] (a.center) -- (b.center) -- (d.center) -- cycle;
						\draw[fill=gray] (b.center) -- (c.center) -- (e.center) -- cycle;
						\draw[fill=gray]   (d.center) -- (e.center) -- (f.center) -- cycle;
					\end{scope}
				\end{tikzpicture}
				\begin{tikzpicture}[scale = 1]
					\tikzstyle{point}=[circle,thick,draw=black,fill=black,inner sep=0pt,minimum width=3pt,minimum height=3pt]
					\node (a)[point,label=above:$y_1$] at (0,3.5) {};
					\node (b)[point,label=above:$x_0$] at (2,2.7) {};
					\node (c)[point,label=above:$y_2$] at (4,3.5) {};
					\node (d)[point,label=left:$x_2$] at (1.5,2) {};
					\node (e)[point,label=right:$x_1$] at (2.5,2) {};
					\node (f)[point,label=left:$y_0$] at (2,0) {};	
					
					\begin{scope}[on background layer]
						\draw[fill=gray] (a.center) -- (b.center) -- (d.center) -- cycle;
						\draw[fill=gray] (b.center) -- (c.center) -- (e.center) -- cycle;
						\draw[fill=gray]   (d.center) -- (e.center) -- (f.center) -- cycle;
						
						\draw[fill=gray] (a.center) -- (b.center) -- (c.center) -- cycle;
						\draw[fill=gray] (a.center) -- (f.center) -- (d.center) -- cycle;
						\draw[fill=gray] (c.center) -- (f.center) -- (e.center) -- cycle;
					\end{scope}
				\end{tikzpicture}
			\end{center}
		\end{ex}

		\begin{ex}\label{Example: Partition Complexes and Intersection Complexes}
			For a sequence $\bm = (\underbrace{a-1,0,\dots,0,1}_{p},0)$, the intersection complex $\calI(\bm)$ is equal to the partition complex $\calP(a,p,1)$. If $p=2$, this is also equal to the cycle complex $\fC_{1,a}$.
		\end{ex}
		
		%
		\begin{ex}\label{Example: Partition Complex m=p+1}
			
			\underline{\textbf{The case $m=p+1$:}}
			
			As mentioned at the start of this section, we are not interested in this case for its own sake; we include it only because complexes of this form occur as links in other partition complexes.
			
			Unlike the cases where $1\leq m \leq p$ we will see in the next section that the complex $\calP(a,p,p+1)$ is always acyclic. It may be helpful to think of $\calP(a,p,p+1)$ as the complex $\calP(a,p,p)$ but with some additional facets which `fill in' the homology.
			
			For example, as we saw in Example \ref{Example: Partition Complex p=m=1}, the complex $\calP(3,1,1)$ consists of two disjoint $2$-simplices.
			$$
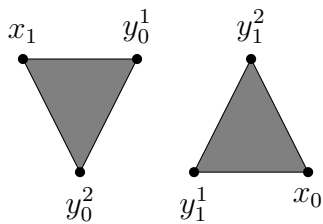
\begin{tikzpicture}[scale = 0.5]
				\tikzstyle{point}=[circle,thick,draw=black,fill=black,inner sep=0pt,minimum width=3pt,minimum height=3pt]
				\node (a)[point,label=above:$x_1$] at (-1.5,3) {};
				\node (b)[point,label=below:$y^2_0$] at (0,0) {};
				\node (c)[point,label=above:$y^1_0$] at (1.5,3) {};
				\node (d)[point,label=below:$y^1_1$] at (3,0) {};
				\node (e)[point,label=above:$y^2_1$] at (4.5,3) {};
				\node (f)[point,label=below:$x_0$] at (6,0) {};
				
				\begin{scope}[on background layer]
					\draw[fill=gray] (a.center) -- (b.center) -- (c.center) -- cycle;
					\draw[fill=gray]   (d.center) -- (e.center) -- (f.center) -- cycle;
				\end{scope}
			\end{tikzpicture}$$
			Meanwhile, the complex $\calP(3,1,2)$ has two additional facets containing only partition vertices, which `bridge the gap' between these two disjoint simplices, thus making the resulting complex acyclic.
			$$
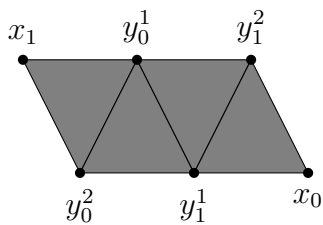
\begin{tikzpicture}[scale = 0.5]
				\tikzstyle{point}=[circle,thick,draw=black,fill=black,inner sep=0pt,minimum width=3pt,minimum height=3pt]
				\node (a)[point,label=above:$x_1$] at (-1.5,3) {};
				\node (b)[point,label=below:$y^2_0$] at (0,0) {};
				\node (c)[point,label=above:$y^1_0$] at (1.5,3) {};
				\node (d)[point,label=below:$y^1_1$] at (3,0) {};
				\node (e)[point,label=above:$y^2_1$] at (4.5,3) {};
				\node (f)[point,label=below:$x_0$] at (6,0) {};
				
				\begin{scope}[on background layer]
					\draw[fill=gray] (a.center) -- (b.center) -- (c.center) -- cycle;
					\draw[fill=gray] (b.center) -- (c.center) -- (d.center) -- cycle;
					\draw[fill=gray] (c.center) -- (d.center) -- (e.center) -- cycle;
					\draw[fill=gray]   (d.center) -- (e.center) -- (f.center) -- cycle;
				\end{scope}
			\end{tikzpicture}$$
			
			Similarly, the complex $\calP(2,2,2)$ can be drawn as
			$$
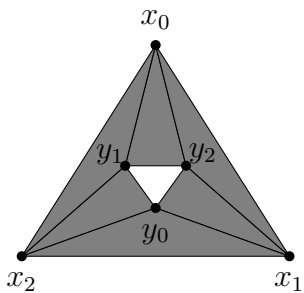
\begin{tikzpicture}[scale = 0.8]
				\tikzstyle{point}=[circle,thick,draw=black,fill=black,inner sep=0pt,minimum width=3pt,minimum height=3pt]
				\node (a)[point,label=below:$x_2$] at (-.2,0) {};
				\node (b)[point,label={[label distance = 0mm] below:$y_0$}] at (2,0.8) {};
				\node (c)[point,label=below:$x_1$] at (4.2,0) {};
				\node (d)[point,label={[label distance = -2.5mm] above left:$y_1$}] at (1.5,1.5) {};
				\node (e)[point,label={[label distance = -2.2mm] above right:$y_2$}] at (2.5,1.5) {};
				\node (f)[point,label=above:$x_0$] at (2,3.5) {};	
				
				\begin{scope}[on background layer]
					\draw[fill=gray] (a.center) -- (b.center) -- (d.center) -- cycle;
					\draw[fill=gray] (b.center) -- (c.center) -- (e.center) -- cycle;
					\draw[fill=gray]   (d.center) -- (e.center) -- (f.center) -- cycle;
					
					\draw[fill=gray] (a.center) -- (b.center) -- (c.center) -- cycle;
					\draw[fill=gray] (a.center) -- (f.center) -- (d.center) -- cycle;
					\draw[fill=gray] (c.center) -- (f.center) -- (e.center) -- cycle;
					
				\end{scope}
			\end{tikzpicture} $$
			while the complex $\calP(2,2,3)$ contains the additional facet $\{y_1,y_2,y_3\}$ which `fills in' the hole at the centre, once again making the complex acyclic.
			$$
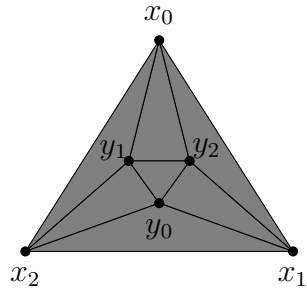
\begin{tikzpicture}[scale = 0.8]
				\tikzstyle{point}=[circle,thick,draw=black,fill=black,inner sep=0pt,minimum width=3pt,minimum height=3pt]
				\node (a)[point,label=below:$x_2$] at (-.2,0) {};
				\node (b)[point,label={[label distance = 0mm] below:$y_0$}] at (2,0.8) {};
				\node (c)[point,label=below:$x_1$] at (4.2,0) {};
				\node (d)[point,label={[label distance = -2.5mm] above left:$y_1$}] at (1.5,1.5) {};
				\node (e)[point,label={[label distance = -2.2mm] above right:$y_2$}] at (2.5,1.5) {};
				\node (f)[point,label=above:$x_0$] at (2,3.5) {};	
				
				\begin{scope}[on background layer]
					\draw[fill=gray] (a.center) -- (b.center) -- (d.center) -- cycle;
					\draw[fill=gray] (b.center) -- (c.center) -- (e.center) -- cycle;
					\draw[fill=gray]   (d.center) -- (e.center) -- (f.center) -- cycle;
					
					\draw[fill=gray] (a.center) -- (b.center) -- (c.center) -- cycle;
					\draw[fill=gray] (a.center) -- (f.center) -- (d.center) -- cycle;
					\draw[fill=gray] (c.center) -- (f.center) -- (e.center) -- cycle;
					
					\draw[fill=gray] (b.center) -- (d.center) -- (e.center) -- cycle;
				\end{scope}
			\end{tikzpicture}$$
			(Note we have adjusted our earlier picture of $\calP(2,2,2)$ in Example \ref{Example: Partition Complex a=2} here, by drawing the boundary vertices on the \textit{outside} of the complex; both depictions of this complex will be useful to us, and we frequently switch between them).
		\end{ex}
		
		The complex $\calP(a,p,p+1)$ can be obtained from $\calP(a,p,p)$ by adding faces consisting entirely of partition vertices. In a similar way, the following construction adds a face consisting entirely of boundary vertices.
		\begin{defn}\label{Definition: Closed Partition Complexes}
			Let $a$, $p$ and $m$ be integers with $a\geq 2$ and $0\leq m\leq p+1$. We define the \textit{closed partition complex} $\barP(a,p,m)$ on vertex set $V_p^a$ to be the complex $\calP(a,p,m)\cup \langle X_p\rangle$.
		\end{defn}
		\begin{rem}\label{Remark: Closed Partition Complex m=0 Cases}
			In the case $m=0$ the closed partition complex $\barP(a,p,0)$ is equal to the full simplex on the set $X_p$. Note that this is acyclic in all cases except for the case $p=-1$. In the case $p=-1$ the set $X_{-1}$ is empty so the full simplex on $X_{-1}$ is simply the irrelevant complex $\{\emptyset\}$.
		\end{rem}
		
		Just as with partition complexes of the form $\calP(a,p,p+1)$, we are interested in closed partition complexes only because they occur in the links of regular partition complexes, and our theorem on the degree types of partition complexes (Theorem \ref{Theorem: Partition Complex Degree Type}) does not extend to them. In fact, closed partition complexes are not even PR in general, except in the case where $a=2$. In all other cases the additional facet $X_p=\{x_0,\dots,x_p\}$ is of lower dimension than all the other facets of $\calP(a,p,m)$, and so $\barP(a,p,m)$ is not pure.
		
		We will see in the next section that the closure operation essentially acts as a kind of `switch' for the homology of a partition complex. Indeed, for every $1\leq m\leq p$ the partition complex $\calP(a,p,m)$ deformation retracts onto the boundary $\partial \langle X_p\rangle$, which means it has homology; whereas the closed partition complex $\barP(a,p,m)$ contains $X_p$ itself as a face, which means it deformation retracts on to the full simplex $\langle X_p\rangle$, and is hence acyclic. Meanwhile the partition complex $\calP(a,p,p+1)$ is acyclic, while its closure $\barP(a,p,p+1)$ has homology. 
		\begin{ex}\label{Example: Closed partition complex example}
			In the case $a=p=2$ we have the following partition complexes and closed partition complexes.
			\begin{center}
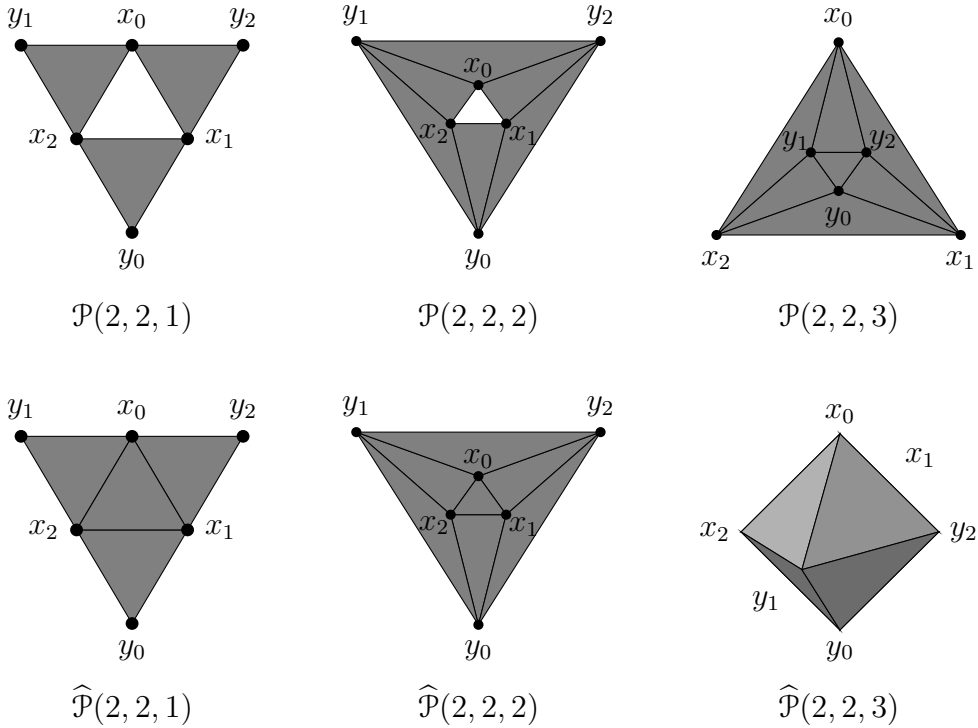
				\begin{tabular}{ c c c c c }
					\begin{tikzpicture}[scale = 0.73]
						\tikzstyle{point}=[circle,thick,draw=black,fill=black,inner sep=0pt,minimum width=4pt,minimum height=4pt]
						\node (a)[point,label=above:$y_1$] at (0,3.4) {};
						\node (b)[point,label=above:$x_0$] at (2,3.4) {};
						\node (c)[point,label=above:$y_2$] at (4,3.4) {};
						\node (d)[point,label=left:$x_2$] at (1,1.7) {};
						\node (e)[point,label=right:$x_1$] at (3,1.7) {};
						\node (f)[point,label=below:$y_0$] at (2,0) {};	
						
						\begin{scope}[on background layer]
							\draw[fill=gray] (a.center) -- (b.center) -- (d.center) -- cycle;
							\draw[fill=gray] (b.center) -- (c.center) -- (e.center) -- cycle;
							\draw[fill=gray]   (d.center) -- (e.center) -- (f.center) -- cycle;
						\end{scope}
					\end{tikzpicture}
					&& \begin{tikzpicture}[scale = 0.73]
						\tikzstyle{point}=[circle,thick,draw=black,fill=black,inner sep=0pt,minimum width=3pt,minimum height=3pt]
						\node (a)[point,label=above:$y_1$] at (-.2,3.5) {};
						\node (b)[point,label={[label distance = -1mm] above:$x_0$}] at (2,2.7) {};
						\node (c)[point,label=above:$y_2$] at (4.2,3.5) {};
						\node (d)[point,label={[label distance = -2.5mm] below left:$x_2$}] at (1.5,2) {};
						\node (e)[point,label={[label distance = -2.2mm] below right:$x_1$}] at (2.5,2) {};
						\node (f)[point,label=below:$y_0$] at (2,0) {};	
						
						\begin{scope}[on background layer]
							\draw[fill=gray] (a.center) -- (b.center) -- (d.center) -- cycle;
							\draw[fill=gray] (b.center) -- (c.center) -- (e.center) -- cycle;
							\draw[fill=gray]   (d.center) -- (e.center) -- (f.center) -- cycle;
							
							\draw[fill=gray] (a.center) -- (b.center) -- (c.center) -- cycle;
							\draw[fill=gray] (a.center) -- (f.center) -- (d.center) -- cycle;
							\draw[fill=gray] (c.center) -- (f.center) -- (e.center) -- cycle;
						\end{scope}
					\end{tikzpicture} && \begin{tikzpicture}[scale = 0.73]
						\tikzstyle{point}=[circle,thick,draw=black,fill=black,inner sep=0pt,minimum width=3pt,minimum height=3pt]
						\node (a)[point,label=below:$x_2$] at (-.2,0) {};
						\node (b)[point,label={[label distance = 0mm] below:$y_0$}] at (2,0.8) {};
						\node (c)[point,label=below:$x_1$] at (4.2,0) {};
						\node (d)[point,label={[label distance = -2.5mm] above left:$y_1$}] at (1.5,1.5) {};
						\node (e)[point,label={[label distance = -2.2mm] above right:$y_2$}] at (2.5,1.5) {};
						\node (f)[point,label=above:$x_0$] at (2,3.5) {};	
						
						\begin{scope}[on background layer]
							\draw[fill=gray] (a.center) -- (b.center) -- (d.center) -- cycle;
							\draw[fill=gray] (b.center) -- (c.center) -- (e.center) -- cycle;
							\draw[fill=gray]   (d.center) -- (e.center) -- (f.center) -- cycle;
							
							\draw[fill=gray] (a.center) -- (b.center) -- (c.center) -- cycle;
							\draw[fill=gray] (a.center) -- (f.center) -- (d.center) -- cycle;
							\draw[fill=gray] (c.center) -- (f.center) -- (e.center) -- cycle;
							
							\draw[fill=gray] (b.center) -- (d.center) -- (e.center) -- cycle;
						\end{scope}
					\end{tikzpicture} \\
					$\calP(2,2,1)$&& $\calP(2,2,2)$ && $\calP(2,2,3)$\\
					&& &&\\
					\begin{tikzpicture}[scale = 0.73]
						\tikzstyle{point}=[circle,thick,draw=black,fill=black,inner sep=0pt,minimum width=4pt,minimum height=4pt]
						\node (a)[point,label=above:$y_1$] at (0,3.4) {};
						\node (b)[point,label=above:$x_0$] at (2,3.4) {};
						\node (c)[point,label=above:$y_2$] at (4,3.4) {};
						\node (d)[point,label=left:$x_2$] at (1,1.7) {};
						\node (e)[point,label=right:$x_1$] at (3,1.7) {};
						\node (f)[point,label=below:$y_0$] at (2,0) {};	
						
						\begin{scope}[on background layer]
							\draw[fill=gray] (a.center) -- (b.center) -- (d.center) -- cycle;
							\draw[fill=gray] (b.center) -- (c.center) -- (e.center) -- cycle;
							\draw[fill=gray]   (d.center) -- (e.center) -- (f.center) -- cycle;
							
							\draw[fill=gray] (b.center) -- (d.center) -- (e.center) -- cycle;
						\end{scope}
					\end{tikzpicture}
					&& \begin{tikzpicture}[scale = 0.73]
						\tikzstyle{point}=[circle,thick,draw=black,fill=black,inner sep=0pt,minimum width=3pt,minimum height=3pt]
						\node (a)[point,label=above:$y_1$] at (-.2,3.5) {};
						\node (b)[point,label={[label distance = -1mm] above:$x_0$}] at (2,2.7) {};
						\node (c)[point,label=above:$y_2$] at (4.2,3.5) {};
						\node (d)[point,label={[label distance = -2.5mm] below left:$x_2$}] at (1.5,2) {};
						\node (e)[point,label={[label distance = -2.2mm] below right:$x_1$}] at (2.5,2) {};
						\node (f)[point,label=below:$y_0$] at (2,0) {};	
						
						\begin{scope}[on background layer]
							\draw[fill=gray] (a.center) -- (b.center) -- (d.center) -- cycle;
							\draw[fill=gray] (b.center) -- (c.center) -- (e.center) -- cycle;
							\draw[fill=gray]   (d.center) -- (e.center) -- (f.center) -- cycle;
							
							\draw[fill=gray] (a.center) -- (b.center) -- (c.center) -- cycle;
							\draw[fill=gray] (a.center) -- (f.center) -- (d.center) -- cycle;
							\draw[fill=gray] (c.center) -- (f.center) -- (e.center) -- cycle;
							
							\draw[fill=gray] (b.center) -- (d.center) -- (e.center) -- cycle;
						\end{scope}
					\end{tikzpicture} &&	\begin{tikzpicture}[scale=1.3][line join=bevel,z=-5.5]
						\coordinate [label={[label distance=3mm]above right:$x_1$}] (A1) at (0,0,-1);
						\coordinate [label=left:$x_2$] (A2) at (-1,0,0);
						\coordinate [label={[label distance=2mm]below left: $y_1$}] (A3) at (0,0,1);
						\coordinate [label=right:$y_2$] (A4) at (1,0,0);
						\coordinate [label=above:$x_0$] (B1) at (0,1,0);
						\coordinate [label=below:$y_0$] (C1) at (0,-1,0);
						
						\begin{scope}[on background layer]
							\draw (A1) -- (A2) -- (B1) -- cycle;
							\draw (A4) -- (A1) -- (B1) -- cycle;
							\draw (A1) -- (A2) -- (C1) -- cycle;
							\draw (A4) -- (A1) -- (C1) -- cycle;
							\draw [fill opacity=0.7,fill=lightgray!80!gray] (A2) -- (A3) -- (B1) -- cycle;
							\draw [fill opacity=0.7,fill=gray!70!lightgray] (A3) -- (A4) -- (B1) -- cycle;
							\draw [fill opacity=0.7,fill=darkgray!30!gray] (A2) -- (A3) -- (C1) -- cycle;
							\draw [fill opacity=0.7,fill=darkgray!30!gray] (A3) -- (A4) -- (C1) -- cycle;
						\end{scope}
					\end{tikzpicture}\\
					$\barP(2,2,1)$&& $\barP(2,2,2)$ && $\barP(2,2,3)$
				\end{tabular}
			\end{center}
		\end{ex}
		We now introduce some key notation and terminology which will help us to discuss the faces of partition complexes going forward.
		
		\begin{notation}\label{Notation: sigma-X and sigma-Y and support}
			Let $a$, $p$ and $m$ be integers with $a\geq 2$ and $0\leq m \leq p+1$, and let $\sigma$ be a subset of $V^a_p$. We use
			\begin{enumerate}
				\item $\sigma_X$ to denote the intersection $\sigma \cap X_p$.
				\item $\sigma_Y$ to denote the intersection $\sigma \cap Y^a_p$.
			\end{enumerate}
			We also define the \textit{support} of $\sigma$ to be the set $$\Supp(\sigma) = \left \{i\in \{0,\dots, p\} : x_i \in \sigma \text{ or } y_i^j \in \sigma \text{ for some } 1\leq j\leq a-1\right \}.$$
		\end{notation}
		
		\begin{defn}\label{Definition: Totally Separated and partition complete}
			Let $a$, $p$ and $m$ be integers with $a\geq 2$ and $0\leq m \leq p+1$, and let $F$ be a subset of the vertex set $V^a_p$. We say
			\begin{enumerate}
				\item $F$ is \textit{partition complete} if for every partition vertex $y_i^j$ in $\sigma$, the vertices $y_i^1,\dots,y_i^{j-1}$ are also in $\sigma$.
				\item $F$ is \textit{separated}  if for each $0\leq i \leq p$, $F$ contains at most one of the vertices $x_i$ or $y_i$. We say it is \textit{totally separated} if for each $0\leq i \leq p$ it contains exactly one of the vertices $x_i$ or $y_i$. 
			\end{enumerate}
		\end{defn}
		
		By construction every face of a partition complex is separated, and every facet is both totally separated and partition complete. In fact we have more than this.
		\begin{lem}\label{Lemma: Partition Complex Description of Facets}
			Let $a$, $p$ and $m$ be integers with $a\geq 2$ and $0\leq m \leq p+1$. A subset $F\subseteq V^a_p$ is a facet of the partition complex $\calP(a,p,m)$ if and only if it satisfies the following four conditions.
			\begin{enumerate}
				\item $F$ is partition complete.
				\item $F$ is totally separated.
				\item $|F|=a+p-1$.
				\item $1\leq |\Supp(F_Y)|\leq m$.
			\end{enumerate}
		\end{lem}
		\begin{proof}
			Every facet of $\calP(a,p,m)$ must satisfy these four conditions, because the generating facets $G_{p,i}^\lambda$ satisfy them by construction, and all four conditions are invariant under the action of $S^0_p$.
			
			Conversely, suppose $F$ is a totally separated, partition complete subset of $V^a_p$ of size $a+p-1$ with $|\Supp(F_Y)|=i$ for some $1\leq i \leq m$. We can permute the vertices of $F$ to ensure that $\Supp(F_Y)=\{0,\dots,i-1\}$ for some $g\in S^0_p$. Because $F$ is totally separated we must therefore have $F_X=\{x_i,\dots,x_p\}=G_{p,i}$.
			
			This leaves us with a total of $(a+p-1)-(p+1-i)=a+i-2$ partition vertices in $F$. For each $0\leq r \leq i-1$ we let $j_r$ denote the maximum integer between $1$ and $a-1$ such that the partition vertex $y_r^{j_r}$ is in $F$. Once again, we can permute the vertices of $F$ to ensure that the sequence $\lambda = (j_0,\dots,j_{i-1})$ is monotonically decreasing. The sequence $\lambda$ is a partition of $a+i-2$ into $i$ parts, and because $F$ is partition complete we have (after our permutations of vertices) that $F$ is equal to the generating facet $G_{p,i}^\lambda$.
		\end{proof}
		
		Now that we have demonstrated the construction of partition complexes, and built up the tools we will need to talk about them, we present our main result.
		\begin{thm}\label{Theorem: Partition Complex Degree Type}
			Let $a$, $p$ and $m$ be positive integers with $a\geq 2$ and $1 \leq m\leq p$. The partition complex $\calP(a,p,m)$ is PR with degree type $(\overbrace{1,\dots,1\underbrace{a,1,\dots,1}_{m}}^{p})$.
		\end{thm}
		
		In the next section we will make use of some deformation retractions to find the homology of both partition complexes and their closures. We also show how all of the links in partition complexes with homology can be built out of smaller partition complexes and closed partition complexes. Then in Section \ref{Subsection: Proving Partition Complex Theorem} we will put these results together to prove Theorem \ref{Theorem: Partition Complex Degree Type}.
		
		\subsection{Deformation Retractions and Links}
		Just as with intersection complexes, we can find the homology of the links in partition complexes by using some deformation retractions. Once again, our main tools for this are Lemma \ref{Lemma: Deformation Retract} and Corollary \ref{Corollary: Deformation Retract}.
		
		\begin{prop}\label{Proposition: Deformation D(a,p,m) to D(2,p,m)}
			Let $a$, $p$ and $m$ be integers with $a\geq 2$ and $1\leq m\leq p+1$. There is a deformation retraction $\calP(a,p,m)\rightsquigarrow \calP(2,p,m)$ given by the vertex maps $y_i^j\mapsto y^1_i$. There is a similar deformation retraction $\barP(a,p,m)\rightsquigarrow \barP(2,p,m)$.
		\end{prop}
		\begin{proof}
			Let $y_i^j$ be a partition vertex of $\calP(a,p,m)$ for some $j\geq 2$. By construction, every facet of $\calP(a,p,m)$ which contains the vertex $y_i^j$ also contains the vertex $y_i^1$. Thus Corollary \ref{Corollary: Deformation Retract} gives us a deformation retraction of $\calP(a,p,m)$ on to the complex obtained by deleting every partition vertex $y_i^j$ for $j\geq 2$, by identifying each of these vertices with the corresponding vertex $y^1_i$. The complex obtained from these deletions is $\calP(2,p,m)$. The proof for $\barP(a,p,m)$ is identical.
		\end{proof}
		
		\begin{prop}\label{Proposition: Deformation D(2,p,m) to D(1,p,m)}
			Let $p$ and $m$ be positive integers with $1\leq m\leq p$. There is a deformation retraction $\calP(2,p,m)\rightsquigarrow \partial \langle X_p \rangle$. There is a similar deformation retraction $\barP(2,p,m)\rightsquigarrow \langle X_p \rangle$.
		\end{prop}
		\begin{proof}
			As noted in Example \ref{Example: Partition Complex a=2}, all the facets of $\calP(2,p,m)$ contain no more than $m$ partition vertices. Thus if $F$ is any facet with exactly $m$ partition vertices $y_{i_1},\dots,y_{i_m}$, it must be the \textit{only} facet containing the face $\{y_{i_1},\dots,y_{i_m}\}$. Because $m\leq p$, we know $F$ \textit{strictly} contains $\{y_{i_1},\dots,y_{i_m}\}$, and hence Lemma \ref{Lemma: Deformation Retract} allows us to delete the face $\{y_{i_1},\dots,y_{i_m}\}$ from the complex. The complex obtained by deleting \textit{all} faces consisting of $m$ partition vertices from $\calP(2,p,m)$ is $\calP(2,p,m-1)$.
			
			By induction on $m\geq 1$ we obtain a series of deformation retractions $\calP(2,p,m)\rightsquigarrow \calP(2,p,m-1) \rightsquigarrow \dots \rightsquigarrow \calP(2,p,1)$. The facets of $\calP(2,p,1)$ are all the totally separated subsets of $V^2_p$ with only a single partition vertex, and so once again we may use Lemma \ref{Lemma: Deformation Retract} to delete these partition vertices. This leaves us with a complex with facets of the form $\{x_0,\dots,x_p\}-\{x_i\}$ for $0\leq i \leq p$. This is $\partial \langle X_p\rangle$.
			
			The proof for $\barP(2,p,m)$ is identical, except because $\barP(2,p,m)$ also contains the face $X_p$, it deformation retracts onto the full simplex $\langle X_p\rangle$.
			
			
		\end{proof}
		\begin{ex}\label{Example: D(2,2,2) def retract}
			In the case $p=m=2$ we have the following deformation retraction.
			\begin{center}
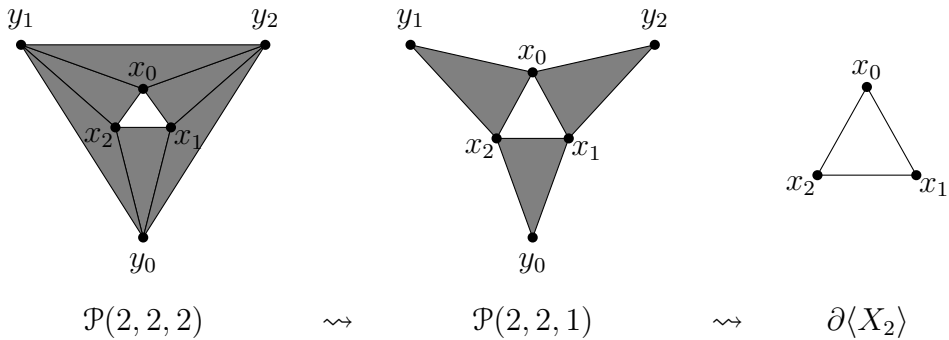
				\begin{tabular}{ c c c c c }
					\begin{tikzpicture}[scale = 0.73]
						\tikzstyle{point}=[circle,thick,draw=black,fill=black,inner sep=0pt,minimum width=3pt,minimum height=3pt]
						\node (a)[point,label=above:$y_1$] at (-.2,3.5) {};
						\node (b)[point,label={[label distance = -1mm] above:$x_0$}] at (2,2.7) {};
						\node (c)[point,label=above:$y_2$] at (4.2,3.5) {};
						\node (d)[point,label={[label distance = -2.5mm] below left:$x_2$}] at (1.5,2) {};
						\node (e)[point,label={[label distance = -2.2mm] below right:$x_1$}] at (2.5,2) {};
						\node (f)[point,label=below:$y_0$] at (2,0) {};	
						
						\begin{scope}[on background layer]
							\draw[fill=gray] (a.center) -- (b.center) -- (d.center) -- cycle;
							\draw[fill=gray] (b.center) -- (c.center) -- (e.center) -- cycle;
							\draw[fill=gray]   (d.center) -- (e.center) -- (f.center) -- cycle;
							
							\draw[fill=gray] (a.center) -- (b.center) -- (c.center) -- cycle;
							\draw[fill=gray] (a.center) -- (f.center) -- (d.center) -- cycle;
							\draw[fill=gray] (c.center) -- (f.center) -- (e.center) -- cycle;
						\end{scope}
					\end{tikzpicture} & 
					&  \begin{tikzpicture}[scale = 0.73]
						\tikzstyle{point}=[circle,thick,draw=black,fill=black,inner sep=0pt,minimum width=3pt,minimum height=3pt]
						\node (a)[point,label=above:$y_1$] at (-.2,3.5) {};
						\node (b)[point,label={[label distance = -1mm] above:$x_0$}] at (2,3) {};
						\node (c)[point,label=above:$y_2$] at (4.2,3.5) {};
						\node (d)[point,label={[label distance = -2.5mm] below left:$x_2$}] at (1.35,1.8) {};
						\node (e)[point,label={[label distance = -2.2mm] below right:$x_1$}] at (2.65,1.8) {};
						\node (f)[point,label=below:$y_0$] at (2,0) {};	
						
						\begin{scope}[on background layer]
							\draw[fill=gray] (a.center) -- (b.center) -- (d.center) -- cycle;
							\draw[fill=gray] (b.center) -- (c.center) -- (e.center) -- cycle;
							\draw[fill=gray]   (d.center) -- (e.center) -- (f.center) -- cycle;
							
						\end{scope}
					\end{tikzpicture} & 
					& \begin{tikzpicture}[scale = 1.3]
						\tikzstyle{point}=[circle,thick,draw=black,fill=black,inner sep=0pt,minimum width=3pt,minimum height=3pt]
						\node (b)[point,label={[label distance = -1mm] above:$x_0$}] at (2,2.9) {};
						\node (d)[point,label={[label distance = -2.5mm] below left:$x_2$}] at (1.5,2) {};
						\node (e)[point,label={[label distance = -2.2mm] below right:$x_1$}] at (2.5,2) {};
						\node at (2,1) {};
						
						\draw (b.center) -- (d.center) -- (e.center) -- cycle;
						
						
					\end{tikzpicture}\\
					$\calP(2,2,2)$& $\rightsquigarrow$  & $\calP(2,2,1)$ & $\rightsquigarrow$  & $\partial \langle X_2 \rangle$
				\end{tabular}
			\end{center}	
		\end{ex}
		
		\begin{prop}\label{Proposition: Deformation D(2,p,p+1) to point}
			For any integer $p\geq 0$, the partition complex $\calP(2,p,p+1)$ is acyclic.
		\end{prop}
		\begin{proof}
			Our proof for this result is very similar to our proof for Proposition \ref{Proposition: Deformation D(2,p,m) to D(1,p,m)} above. In that proof, we used Lemma \ref{Lemma: Deformation Retract} to remove all of the partition vertices $y_0,\dots,y_p$ from the complex. In this case, we use the same lemma to remove the boundary vertices $x_0,\dots,x_p$.
			
			Note that while $\Delta=\calP(2,p,p+1)$ contains the facet $\{y_0,\dots,y_p\}$ (which has $p+1$ partition vertices), no facet of $\Delta$ contains more than $p$ boundary vertices. Thus if $F$ is a facet of $\Delta$ containing $p$ boundary vertices $x_{i_1},\dots,x_{i_p}$, it must be the \textit{only} facet containing all of those vertices; and once again, because it \textit{strictly} contains them, Lemma \ref{Lemma: Deformation Retract} allows us to delete the face $\{x_{i_1},\dots,x_{i_p}\}$ from the complex. Proceeding in this way, we may remove all faces consisting of boundary vertices from $\Delta$, beginning with the ones of size $p$ and continuing in decreasing order of size. Thus we obtain a deformation retraction of $\calP(2,p,p+1)$ on to the complex generated by the single facet $\{y_0,\dots,y_p\}$, which is acyclic.
		\end{proof}
		\begin{ex}\label{Example: D(2,2,3) def retract}
			In the case $p=2$, $m=3$ we have the following deformation retraction.
			\begin{center}
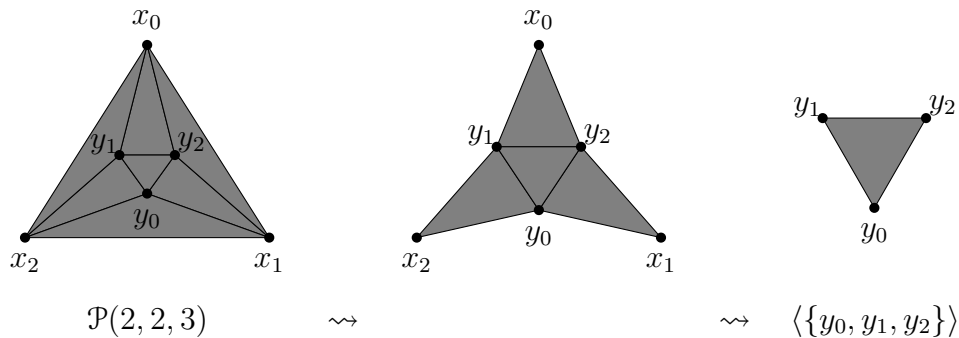
				\begin{tabular}{ c c c c c }
					\begin{tikzpicture}[scale = 0.73]
						\tikzstyle{point}=[circle,thick,draw=black,fill=black,inner sep=0pt,minimum width=3pt,minimum height=3pt]
						\node (a)[point,label=below:$x_2$] at (-.2,0) {};
						\node (b)[point,label={[label distance = 0mm] below:$y_0$}] at (2,0.8) {};
						\node (c)[point,label=below:$x_1$] at (4.2,0) {};
						\node (d)[point,label={[label distance = -2.5mm] above left:$y_1$}] at (1.5,1.5) {};
						\node (e)[point,label={[label distance = -2.2mm] above right:$y_2$}] at (2.5,1.5) {};
						\node (f)[point,label=above:$x_0$] at (2,3.5) {};	
						
						\begin{scope}[on background layer]
							\draw[fill=gray] (a.center) -- (b.center) -- (d.center) -- cycle;
							\draw[fill=gray] (b.center) -- (c.center) -- (e.center) -- cycle;
							\draw[fill=gray]   (d.center) -- (e.center) -- (f.center) -- cycle;
							
							\draw[fill=gray] (a.center) -- (b.center) -- (c.center) -- cycle;
							\draw[fill=gray] (a.center) -- (f.center) -- (d.center) -- cycle;
							\draw[fill=gray] (c.center) -- (f.center) -- (e.center) -- cycle;
							
							\draw[fill=gray] (b.center) -- (d.center) -- (e.center) -- cycle;
						\end{scope}
					\end{tikzpicture} & 
					&  \begin{tikzpicture}[scale = 0.73]
						\tikzstyle{point}=[circle,thick,draw=black,fill=black,inner sep=0pt,minimum width=3pt,minimum height=3pt]
						\node (a)[point,label=below:$x_2$] at (-.2,0) {};
						\node (b)[point,label={[label distance = 0mm] below:$y_0$}] at (2,0.5) {};
						\node (c)[point,label=below:$x_1$] at (4.2,0) {};
						\node (d)[point,label={[label distance = -2.5mm] above left:$y_1$}] at (1.24,1.65) {};
						\node (e)[point,label={[label distance = -2.2mm] above right:$y_2$}] at (2.76,1.65) {};
						\node (f)[point,label=above:$x_0$] at (2,3.5) {};	
						
						\begin{scope}[on background layer]
							\draw[fill=gray] (a.center) -- (b.center) -- (d.center) -- cycle;
							\draw[fill=gray] (b.center) -- (c.center) -- (e.center) -- cycle;
							\draw[fill=gray]   (d.center) -- (e.center) -- (f.center) -- cycle;
							
							
							\draw[fill=gray] (b.center) -- (d.center) -- (e.center) -- cycle;
						\end{scope}
					\end{tikzpicture} & 
					& \begin{tikzpicture}[scale = 1.7]
						\tikzstyle{point}=[circle,thick,draw=black,fill=black,inner sep=0pt,minimum width=3pt,minimum height=3pt]
						\node (b)[point,label={[label distance = 0mm] below:$y_0$}] at (2,0.8) {};
						\node (d)[point,label={[label distance = -2.5mm] above left:$y_1$}] at (1.6,1.5) {};
						\node (e)[point,label={[label distance = -2.2mm] above right:$y_2$}] at (2.4,1.5) {};
						\node at (2,0.3) {};
						
						\begin{scope}[on background layer]
							
							
							\draw[fill=gray] (b.center) -- (d.center) -- (e.center) -- cycle;
						\end{scope}
					\end{tikzpicture}\\
					$\calP(2,2,3)$& $\rightsquigarrow$  &  & $\rightsquigarrow$  & $\langle \{y_0,y_1,y_2\} \rangle$
				\end{tabular}
			\end{center}	
		\end{ex}

		\begin{prop}\label{Proposition: Partition Complex Cross Polytope}
			Let $p\geq -1$ be an integer. The closed partition complex $\barP(2,p,p+1)$ is isomorphic to the cross-polytope $O^p$.
		\end{prop}
		\begin{proof}
			The minimal nonfaces of the cross polytope $O^p$ consist of $p$ pairwise disjoint edges. Meanwhile, the faces of $\barP(2,p,p+1)$ are all the separated subsets of $V^2_p=\{x_0,\dots,x_p,y_0,\dots,y_p\}$ (including both the sets $\{x_0,\dots,x_p\}$ and $\{y_0,\dots,y_p\}$), and so its minimal nonfaces are the $p$ pairwise disjoint edges $\{x_0,y_0\},\dots,\{x_p,y_p\}$. The result follows.
		\end{proof}
		
		Taken together these results are enough to give us the homology of all partition complexes and closed partition complexes.
		\begin{cor}\label{Corollary: Homology of Partition Complexes}
			Let $a$, $p$ and $m$ be integers with $a\geq 2$, $p\geq -1$ and $0\leq m\leq p+1$. Let $\Delta = \calP(a,p,m)$ and $\barDelta = \barP(a,p,m)$.
			\begin{enumerate}
				\item $h(\Delta) = \begin{cases}
					\{-1\} &\text{ if } m=0\\
					\{p-1\} &\text{ if } 1\leq m\leq p\\
					\emptyset & \text{ if } m=p+1 \text{ and } p\neq -1
				\end{cases}$
				
				and in the first two cases, the dimension of the nontrivial homology is $1$.
				\item $h(\barDelta) = \begin{cases}
					\emptyset &\text{ if } m=0 \text{ and } p\neq -1\\
					\emptyset & \text{ if } 1\leq m\leq p\\
					\{p\} &\text{ if } m=p+1
				\end{cases}$
				
				and in the final case, the dimension of the nontrivial homology is $1$.
			\end{enumerate}
		\end{cor}
		\begin{proof}
			We consider the $m=0$ cases together first. The partition complex $\calP(a,p,0)$ is defined to be the irrelevant complex $\{\emptyset\}$ which has only $\nth[st]{(-1)}$ homology. Meanwhile the closed partition complex $\barP(a,p,0)$ is the full $p$-simplex $\langle X_p\rangle$, which is acyclic except in the case where $p=-1$. We now proceed to the other cases.
			
			For part (1), if $1\leq m\leq p$, then Propositions \ref{Proposition: Deformation D(a,p,m) to D(2,p,m)} and \ref{Proposition: Deformation D(2,p,m) to D(1,p,m)} give us a deformation retraction from $\calP(a,p,m)$ on to $\partial \langle X_p \rangle$, which is the boundary of a $p$-simplex, and hence has only $\nth[st]{(p-1)}$ homology of dimension 1. If $m=p+1$ and $p\neq -1$, then Proposition \ref{Proposition: Deformation D(a,p,m) to D(2,p,m)} gives us a deformation retraction from $\calP(a,p,p+1)$ on to $\calP(2,p,p+1)$, which is acyclic by Proposition \ref{Proposition: Deformation D(2,p,p+1) to point}.
			
			For part (2), if $1\leq m\leq p$, then Propositions \ref{Proposition: Deformation D(a,p,m) to D(2,p,m)} and \ref{Proposition: Deformation D(2,p,m) to D(1,p,m)} give us a deformation retraction from $\barP(a,p,m)$ to $\langle X_p \rangle$, which is a full $p$-simplex, and is hence acyclic. If $m=p+1$, then Proposition \ref{Proposition: Deformation D(a,p,m) to D(2,p,m)} gives us a deformation retraction from $\barP(a,p,m)$ on to $\barP(2,p,m)$. By Proposition \ref{Proposition: Partition Complex Cross Polytope} this is isomorphic to the cross polytope $O^p$, which has only $\nth[th]{p}$ homology, of dimension $1$.
		\end{proof}
		
		We now examine the links in partition complexes. Suppose $\sigma$ is a face of a partition complex $\calP(a,p,m)$. We can partition $\sigma$ into $\sigma_X\sqcup \sigma_Y$ (as defined in Notation \ref{Notation: sigma-X and sigma-Y and support}). We first consider the case $\sigma_Y = \emptyset$. In fact, it suffices to consider the subcase where $\sigma$ is equal to the boundary vertex $x_p$.
		
		\begin{lem}\label{Lemma: Partition Complex Link of x}
			Let $a$, $p$ and $m$ be integers with $a\geq 2$ and $1\leq m\leq p+1$, and let $\Delta=\calP(a,p,m)$. We have
			$$\link_\Delta x_p = \begin{cases}
				\calP(a,p-1,m)& \text{ if } 1\leq m \leq p\\
				\calP(a,p-1,p)& \text{ if } m = p+1.
			\end{cases}$$
		\end{lem}
		\begin{proof}
			We begin with the case where $1\leq m \leq p$. Let $F$ be a facet of $\Delta$ containing $x_p$. We know $F$ is of the form $gG_{p,i}^\lambda$ for some permutation $g\in S^0_p$, some integer $1\leq i \leq m$ and some partition $\lambda \in \PP(a,i)$.
			
			Because $x_p$ is in $F$ we may assume that $g(x_p)=x_p$ (if not, we may decompose $g$ into cycles and remove $p+1$ from its cycle without affecting $gG_{p,i}^\lambda$). Thus we have
			\begin{align*}
				F-x_p&= gG_{p,i}^\lambda - x_p\\
				&=g(G_{p,i}^\lambda - x_p)\\
				&=gG_{p-1,i}^\lambda
			\end{align*}
			which is a facet of $\calP(a,p-1,m)$.
			
			Conversely, for any generating facet $G^\lambda_{p-1,i}$ of $\calP(a,p-1,m)$ and any permutation $h$ in $S^0_{p-1}$ we can view $h$ as a permutation in $S^0_p$ which fixes $p$, and hence we have $hG^\lambda_{p-1,i}\sqcup \{x_p\}=hG^\lambda_{p,i}$, which is a facet of $\calP(a,p,m)$.
			
			For the $m=p+1$ case note that the set $G_{p,p+1}$ is empty, and therefore for every partition $\lambda \in \PP(a,p+1)$ the generating facet $G^\lambda_{p,p+1}$ contains no partition vertices. Thus the only facets of $\calP(a,p,p+1)$ which contain $x_p$ are those of the form $gG^\lambda_{p,i}$ for $1\leq i \leq p$ and permutations $g\in S^0_p$. The remainder of the proof is identical to the earlier case.
		\end{proof}
		
		\begin{cor}\label{Corollary: Partition Complex Link of sigma_X}
			Let $a$, $p$ and $m$ be integers with $a\geq 2$ and $1\leq m\leq p+1$, and let $\Delta=\calP(a,p,m)$. Let $\sigma$ be a face of $\Delta$ of size $\alpha$ contained entirely inside $X_p$. We have an isomorphism of complexes
			$$\lkds \cong \begin{cases}
				\calP(a,p-\alpha,m)& \text{ if } 0\leq \alpha \leq p - m\\
				\calP(a,p-\alpha,p+1-\alpha)& \text{ otherwise.}
			\end{cases}$$
		\end{cor}
		\begin{proof}
			We proceed by induction on $\alpha\geq 0$. The base case $\alpha=0$ is immediate.
			
			Now suppose $\alpha \geq 1$. By symmetry we may assume that $x_p\in \sigma$: indeed, if it is not we may choose some permutation $g\in S^0_p$ such that $x_p\in g\sigma$, and we have that $\lkds\cong \lkds[g\sigma]$ by Lemma \ref{Lemma: link x = link gx}.
			
			We consider the two cases separately. First we assume that $1\leq \alpha \leq p-m$. In particular this means that $m\leq p-1$, and hence Lemma \ref{Lemma: Partition Complex Link of x} tells us that $\lkds[x_p] = \calP(a,p-1,m)$. Thus $\lkds$ is equal to $\link_{\calP(a,p-1,m)}(\sigma-x_p)$. Note that $0\leq \alpha -1 \leq (p-1)-m$, and so the inductive hypothesis gives us an isomorphism $\link_{\calP(a,p-1,m)}(\sigma-x_p)\cong \calP(a,p-\alpha,m)$.
			
			Now suppose $\alpha > p-m$. By Lemma \ref{Lemma: Partition Complex Link of x}, we know that $\lkds$ is equal to either $\link_{\calP(a,p-1,m)}(\sigma-x_p)$ or $\link_{\calP(a,p-1,p)}(\sigma-x_p)$. We know that $\alpha-1$ is greater than both $(p-1)-m$ and $(p-1)-p$, so in either case the inductive hypothesis gives us an isomorphism $\lkds\cong \calP(a,p-\alpha,p+1-\alpha)$.
		\end{proof}
		
		Corollary \ref{Corollary: Partition Complex Link of sigma_X} allows us to restrict to the case where $\sigma_X=\emptyset$, because it shows us that the link of $\sigma_X$ is also a partition complex, and thus the link of $\sigma$ in $\calP(a,p,m)$ is isomorphic to the link of $\sigma_Y$ in a smaller partition complex.
		
		%
		
		The following lemma allow us to restrict our attention even further.
		
		\begin{lem}\label{Lemma: Partition incomplete => acyclic}
			Let $a$, $p$ and $m$ be positive integers with $m\leq p+1$, and let $\Delta = \calP(a,p,m)$. If $\sigma$ is a face of $\Delta$ which is not partition complete then $\lkds$ is acyclic.
		\end{lem}
		\begin{proof}
			Suppose $\sigma\in \Delta$ is not partition complete. This means there exist some $1\leq i \leq p$ and $1\leq j<j'\leq a-1$ such that the parition vertex $y_i^{j'}$ is in $\sigma$ but $y_i^j$ is not. Because every facet of $\Delta$ is partition complete, every facet of $\Delta$ containing $\sigma$ also contains $y_i^j$. Thus $\lkds$ is a cone over $y_i^j$.
		\end{proof}
		
		It now only remains for us to consider the case where $\sigma$ is a nonempty partition complete face of $\Delta$ contained entirely in $Y_p^a$.
		
		\begin{lem}\label{Lemma: Partition Vertices Deformation Retract}
			Let $a$, $p$ and $m$ be positive integers with $m\leq p+1$, and let $\Delta = \calP(a,p,m)$. Fix some $\sigma \in \Delta$. For each $1\leq i\leq p+1$ we define $\varphi(\sigma,i)=\min \{1\leq j \leq a-1 : y_i^j \notin \sigma\}$. There is a deformation retraction $\Phi$ of $\lkds$ onto a complex $\Gamma$ obtained by identifying every partition vertex $y_i^j$ with the partition vertex $y_i^{\varphi(\sigma,i)}$.
		\end{lem}
		\begin{proof}
			Let $y_i^j$ be a partition vertex in $\lkds$. For any facet $F$ of $\lkds$ containing $y_i^j$, we have that $F\sqcup \sigma$ is a facet of $\Delta$. This means $F\sqcup \sigma$ must be partition complete, and must therefore also contain the vertex $y_i^{\varphi(\sigma,i)}$. But $y_i^{\varphi(\sigma,i)}\notin \sigma$ by definition, and so it must be in $F$. The result follows from Corollary \ref{Corollary: Deformation Retract}. 
		\end{proof}
		
		\begin{prop}\label{Proposition: Partition Complex Skeleton Link}
			Let $a$, $p$ and $m$ be positive integers with $m\leq p+1$, and let $\Delta = \calP(a,p,m)$. Let $\sigma$ be a nonempty partition complete face of $\Delta$ contained entirely inside $Y_p^a$, and set $\beta = |\{y_i^j \in \sigma : j = 1\}|$ and $\gamma = |\{y_i^j \in \sigma : j > 1\}|$. We have a deformation retraction
			$$\lkds \rightsquigarrow \barP(2,p-\beta,m-\beta)\ast \Skel_{a-\gamma-3}([\beta]).$$
		\end{prop}
		\begin{proof}
			We begin by decomposing $\lkds$ into the disjoint union of two subcomplexes $A$ and $B$ defined as
			\begin{align*}
				A &= \{ \tau \in \lkds : \Supp(\tau) \cap \Supp(\sigma) = \emptyset\}\\
				B &= \{ \tau \in \lkds : \Supp(\tau) \subseteq \Supp(\sigma)\}.
			\end{align*}
			In particular note that $|\Supp(\sigma)|=\beta$ because $\sigma$ is partition complete and thus must contain the partition vertex $y^1_i$ for every $i\in \Supp(\sigma)$. Without loss of generality we will assume that $\Supp(\sigma)=\{p-\beta+1,\dots,p\}$.
			
			We apply the deformation retraction $\Phi$ from Lemma \ref{Lemma: Partition Vertices Deformation Retract} on to the complex $\lkds$ to get a new complex $\Gamma$ whose only partition vertices are the vertices $y_0^{\varphi(0,\sigma)}, \dots, y_p^{\varphi(p,\sigma)}$, and denote the images of $A$ and $B$ under $\Phi$ by $A'$ and $B'$ respectively. Note that we have $\Gamma = A'\sqcup B'$.
			
			We aim to show the following.
			\begin{enumerate}
				\item $A' = \barP(2,p-\beta,m-\beta)$.
				\item $B' = \Skel_{a-\gamma-3}(S)$ where $S$ denotes the set $\{y_i^{\varphi(i,\sigma)}: i \in \Supp(\sigma)\}$.
				\item $\Gamma = A' \ast B'$.
			\end{enumerate}
			
			For part (1), note first that by our assumption on $\Supp(\sigma)$, we have $\varphi(i,\sigma)=1$ for each $0\leq i \leq p-\beta$. Thus, the only partition vertices of $A'$ are the vertices $y^1_0,\dots,y^1_{p-\beta}$. Hence the complex $A'$ has vertex set $V^2_{p-\beta}$, so it is equal to the induced subcomplex of $A$ on this vertex set. Suppose $\tau$ is a face of the complex $A$. Because $\sigma \sqcup \tau$ is a face of $\Delta$, we must have $|\Supp(\sigma_Y\sqcup \tau_Y)|\leq m$. We know $|\Supp(\sigma_Y)|=\beta$ so this means $|\Supp(\tau_Y)|\leq m-\beta$. Thus the faces of $A'$ are all the separated subsets $\tau$ of $V^2_{p-\beta}$ such that $0\leq |\Supp(\tau_Y)|\leq m-\beta$ (including the subset $\{x_0,\dots,x_{p-\beta}\}$). These are precisely the faces of $\barP(2,p-\beta,m-\beta)$.
			
			Now we move on to part (2). Because $\sigma$ consists entirely of partition vertices, the faces of $B$ must also consist entirely of partition vertices (otherwise their union with $\sigma$ would not be separated). In particular, this means that every face of $B'$ must be a subset of $S$, and so $B'$ is equal to the induced subcomplex $B|_S$. From Remark \ref{Remark: Partitions of a-2} we know $B$ must contain all subsets of $S$ whose union with the upper partition vertices of $\sigma$ has size up to $a-2$. Thus $B'$ contains every subset of $S$ of size up to $a-\gamma - 2$ (i.e. dimension at most $a-\gamma-3$).
			
			Finally we consider part (3). For any two faces $\tau\in A'$ and $\rho\in B'$, the disjoint union $\tau\sqcup \rho \sqcup \sigma$ is partition complete and separated. From the above discussion we also have that $|\tau_Y|\leq m-\beta$ and $|\rho_Y|\leq a-\gamma - 2$. This means both that $|\tau_Y \sqcup \rho_Y \sqcup \sigma|\leq a+m-2$ and $|\Supp(\tau_Y \sqcup \rho_Y \sqcup \sigma)|\leq m$. We conclude that the disjoint union $\tau\sqcup \rho \sqcup \sigma$ is a face of $\Delta$, which shows that $\tau \sqcup \rho$ is a face of $\Gamma$. The result follows.
		\end{proof}
		
		\begin{ex}\label{Example: Partition Complex D(2,p,m) Link Poset} 
			The complex $\calP(2,p,m)$ contains no upper partition vertices so in this case the Skeleton complex $\Skel_{a-\gamma-2}([\beta])$ as given in the above proposition would be $\Skel_{-1}([\beta])$, which is the irrelevant complex $\{\emptyset\}$. It follows that the links in $\calP(2,p,m)$ are all either partition complexes or closed partition complexes.
			
			The following diagram shows all possible maximal chains in the link poset of $\calP(2,3,2)$. We use $\arrowX{x}$ to denote taking the links of boundary vertices and $\arrowX{y}$ to denote taking the links of (lower) partition vertices.
			
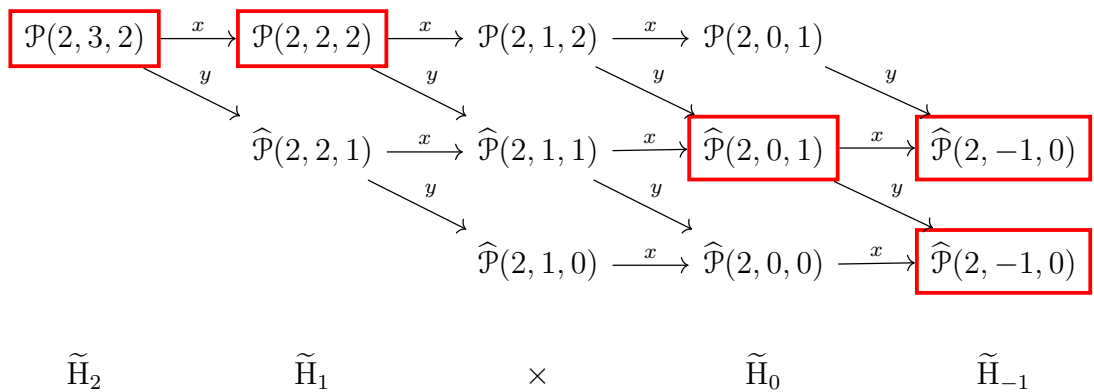
\[
			\begin{tikzcd}[row sep=1.5em]
				|[draw=red, line width =.5mm, rectangle]|\calP(2,3,2) \arrow{dr}{y} \arrow{r}{x} & |[draw=red, line width =.5mm, rectangle]|\calP(2,2,2)\arrow{dr}{y}  \arrow{r}{x}  &  \calP(2,1,2)\arrow{dr}{y} \arrow{r}{x}& \calP(2,0,1)\arrow{dr}{y} & \\
				&\barP(2,2,1)\arrow{dr}{y} \arrow{r}{x}& \barP(2,1,1)\arrow{dr}{y} \arrow{r}{x} & |[draw=red, line width =.5mm, rectangle]|\barP(2,0,1)\arrow{dr}{y} \arrow{r}{x}& |[draw=red, line width =.5mm, rectangle]|\barP(2,-1,0)\\
				&& \barP(2,1,0)\arrow{r}{x} & \barP(2,0,0) \arrow{r}{x}& |[draw=red, line width =.5mm, rectangle]|\barP(2,-1,0)\\
				\Hred_2&\Hred_1&\times&\Hred_0&\Hred_{-1}
			\end{tikzcd}
			\]
			The homologies of these complexes can be computed from Corollary \ref{Corollary: Homology of Partition Complexes}. We can see from the diagram that the complex $\calP(2,3,2)$ is PR with degree type $(1,2,1)$.
		\end{ex}

		\subsection{Proving Theorem \ref{Theorem: Partition Complex Degree Type}}\label{Subsection: Proving Partition Complex Theorem}
		We now have all the ingredients we need to prove Theorem \ref{Theorem: Partition Complex Degree Type}. We start by proving the following corollaries to our results about the links of $\Delta$ in the last section.
		
		\begin{cor}\label{Corollary: Partition Complex Homology of Link of sigma_X}
			Let $a$, $p$ and $m$ be integers with $a\geq 2$ and $1\leq m\leq p+1$, and let $\Delta=\calP(a,p,m)$. Let $\sigma$ be a face of $\Delta$ contained entirely inside $X_p$. We have
			$$h(\Delta,\sigma) = \begin{cases}
				\{p-|\sigma|-1\}& \text{ if } 0\leq |\sigma| \leq p - m\\
				\emptyset& \text{ otherwise.}
			\end{cases}$$
		\end{cor}
		\begin{proof}
			We set $\alpha = |\sigma|$. By corollary \ref{Corollary: Partition Complex Link of sigma_X}, $\lkds$ is isomorphic to $\calP(a,p-\alpha,m)$ if $0\leq \alpha \leq p-m$ or $\calP(a,p-\alpha,p+1-\alpha)$ otherwise. By Corollary \ref{Corollary: Homology of Partition Complexes}, the former of these has $\nth[st]{(p-\alpha- 1)}$ homology while the latter is acyclic.
		\end{proof}
		
		\begin{cor}\label{Corollary: Partition Complex Homology of Skeleton Link}
			Let $a$, $p$ and $m$ be integers with $a\geq 2$ and $1\leq m\leq p+1$, and let $\Delta = \calP(a,p,m)$. Let $\sigma$ be a partition complete face of $\Delta$ with $\sigma_Y\neq \emptyset$. We have
			$$h(\Delta,\sigma)=\begin{cases}
				\{a+p - |\sigma| - 2\} & \text{ if } |\sigma_X| \geq p-m + 1 \text{ and } |\sigma_Y| \geq a-1 \\
				\emptyset & \text{ otherwise.}
			\end{cases}$$
		\end{cor}
		\begin{proof}
			We set $\alpha = |\sigma_X|$, $\beta = |\{y_i^j \in \sigma_Y : j = 1\}|$ and $\gamma = |\{y_i^j \in \sigma_Y : j > 1\}|$. Because $\sigma$ is partition complete and $\sigma_Y\neq \emptyset$, we know $\beta > 0$.
			
			First suppose that $0\leq \alpha \leq p-m$. Corollary \ref{Corollary: Partition Complex Link of sigma_X} tells us that $\lkds[\sigma_X]$ is isomorphic to $\calP(a,p-\alpha,m)$. Therefore, by Proposition \ref{Proposition: Partition Complex Skeleton Link}, there is a deformation retraction
			$$\lkds \rightsquigarrow \barP(2,p-\alpha -\beta,m-\beta)\ast \Skel_{a-\gamma-3}([\beta]).$$
			The complex $\barP(2,p-\alpha -\beta,m-\beta)$ is acyclic by Corollary \ref{Corollary: Homology of Partition Complexes}, and thus so is $\lkds$.
			
			Now suppose that $\alpha \geq p-m+1$. Corollary \ref{Corollary: Partition Complex Link of sigma_X} tells us that $\lkds[\sigma_X]$ is isomorphic to $\calP(a,p-\alpha,p+1 - \alpha)$, and hence by Proposition \ref{Proposition: Partition Complex Skeleton Link}, there is a deformation retraction
			$$\lkds \rightsquigarrow \barP(2,p-\alpha -\beta,p+1-\alpha -\beta)\ast \Skel_{a-\gamma-3}([\beta]).$$
			The complex $\barP(2,p-\alpha -\beta,p+1-\alpha-\beta)$ has only $\nth[th]{(p-\alpha-\beta)}$ homology, by Corollary \ref{Corollary: Homology of Partition Complexes}. Meanwhile the complex $\Skel_{a-\gamma-3}([\beta])$ has only $\nth[rd]{(a-\gamma-3)}$ homology so long as  $a- \gamma - 3 < \beta - 1$ (otherwise it is equal to the full simplex on $[\beta]$, which is acyclic). Thus, $\lkds$ has nontrivial homology if and only if $\beta + \gamma \geq a - 1$, and in this case we have (by Corollary \ref{Corollary: homology index set of joins}) $h(\Delta,\sigma) = \{p - \alpha - \beta\} + \{a - \gamma - 3\} + \{1\} = \{a + p -|\sigma| - 2\}$.
		\end{proof}
		
		With these two corollaries in our toolkit, we now proceed to the proof of Theorem \ref{Theorem: Partition Complex Degree Type}.
		\begin{proof}[Proof of Theorem \ref{Theorem: Partition Complex Degree Type}]
			Let $a$, $p$ and $m$ be positive integers with $m\leq p$ and set $\Delta=\calP(a,p,m)$. Fix a face $\sigma \in \Delta$.
			
			By Lemma \ref{Lemma: Partition incomplete => acyclic}, $\lkds$ is acyclic unless $\sigma$ is partition complete. In the case where $\sigma$ is partition complete, Corollaries \ref{Corollary: Partition Complex Homology of Link of sigma_X} and \ref{Corollary: Partition Complex Homology of Skeleton Link} tell us that
			\begin{equation*}\label{Equation: Homology Index Sets Partition Complexes}
				h(\Delta,\sigma)=\begin{cases}
					\{p-|\sigma| - 1\} & \text{ if } 0\leq |\sigma_X|\leq p-m \text{ and } |\sigma_Y| = 0\\
					\{a+p - |\sigma| - 2\} & \text{ if } |\sigma_X| \geq p-m + 1 \text{ and } |\sigma_Y| \geq a-1 \\
					\emptyset & \text{ otherwise.}
				\end{cases}
			\end{equation*}
			
			In particular, this means that for any integer $k$, we have
			\begin{equation*}\label{Equation: Complete Homology Index Sets Partition Complexes}\hh(\Delta,k)=\begin{cases}
					\{p- k - 1\} & \text{ if } 0\leq k\leq p-m\\
					\{a+p - k - 2\} & \text{ if } a + p - m \leq k \leq a+ p - 1 \\
					\emptyset & \text{ otherwise.}
				\end{cases}
			\end{equation*}
			This proves $\Delta$ is PR with degree type $(\overbrace{1,\dots,1\underbrace{a,1,\dots,1}_{m}}^{p})$, by Proposition \ref{Proposition: Alternate PR Definition With Degree Type}.
		\end{proof}
		
	
	
	\chapter{Pure Resolutions of Any Degree Type}\label{Chapter: Generating Degree Types}
	This chapter is devoted to the proof of Theorem \ref{Theorem: PR Complexes of Any Degree Type}. Our method for this proof is focussed around finding operations on simplicial complexes which preserve the PR property while altering the degree types of PR complexes in specified ways. This reduces the problem of generating PR complexes of an arbitrary degree type $\bd$ to the problem of generating PR complexes with degree types which can be altered to $\bd$ under these operations.
	
	We begin by presenting two examples of such operations to illustrate the general principle. We then introduce an infinite family of PR-preserving operations $\{\phi_i: i \in \ZZ^+\}$ which allow us to generate PR complexes of any given degree type.
	
	\section{Operations on Degree Types}\label{Subsection: Operations on Degree Types}
	In this section we present two examples of operations on simplicial complexes which preserve the PR property while altering degree types. We begin with the \textit{scalar multiple operation} $f^\lambda$ as defined below.
	
	\begin{defn}\label{Definition: f-lambda}
		Let $\Delta$ be a complex on vertex set $V$ with facets $F_1,\dots,F_m$, and let $\lambda$ be a positive integer.
		
		We define the vertex set $V^\lambda$ to be the set   $\{u_x^1, \dots, u_x^{\lambda} : x \in V \}$. There is a natural map
		\begin{align*}
			f^\lambda : \calP(V) &\rightarrow \calP(V^\lambda)\\
			S& \mapsto \{u_x^1,\dots,u_x^\lambda : x \in S\}. 
		\end{align*}
		We define $f^\lambda(\Delta)$ to be the complex on $V^\lambda$ with facets $f^\lambda (F_1),\dots f^\lambda(F_m)$.
		
		
		
	\end{defn}
	
	We claim that if $\Delta$ is PR with degree type $\bd$, then $f^\lambda(\Delta)$ is also PR, with degree type $\lambda \bd$. This is a consequence of the following three-part lemma.
	\begin{lem}\label{Lemma: f-lambda 3-part lemma}
		Let $\Delta$ be a simplicial complex on vertex set $V$, and define $\tDelta = f^\lambda (\Delta)$ on vertex set $V^\lambda$.
		\begin{enumerate}
			\item $\tDelta$ deformation retracts on to a complex isomorphic to $\Delta$.
			\item For any face $\sigma\in \Delta$ we have $\lkds[f^\lambda (\sigma)][\tDelta]=f^\lambda(\lkds)$.
			\item For any face $\tau \in \tDelta$ which is not of the form $f^\lambda (\sigma)$ for some $\sigma \in \Delta$, the complex $\lkds[\tau][\tDelta]$ is acyclic.
		\end{enumerate}
	\end{lem}
	\begin{proof}
		\begin{enumerate}
			\item We identify $\Delta$ with the induced subcomplex of $f^\lambda (\Delta)$ on vertex set $U=\{u_x^1 \in \tV : x \in V\}\subset V^\lambda$. For every $x\in V$ and every $2\leq i\leq \lambda$, the vertex $u^i_x\in \tV$ is contained in exactly the same facets of $f^\lambda(\Delta)$ as the vertex $u_x^1$. Thus Corollary \ref{Corollary: Deformation Retract} gives us a deformation retraction $\tDelta\rightsquigarrow \tDelta|_U\cong \Delta$.
			\item If $F_{i_1},\dots,F_{i_k}$ are the facets of $\Delta$ containing $\sigma$, then the facets of $\tDelta$ containing $f^\lambda(\sigma)$ are $f^\lambda(F_{i_1})$,$\dots$,$f^\lambda(F_{i_k})$. The result follows.
			\item Because $\tau$ is not of the form $f^\lambda (\sigma)$ for some $\sigma \in \Delta$, there must be some $x\in V$ and some $1\leq i\neq j \leq \lambda$ such that $u_x^i \in \tau$ but $u_x^j\notin \tau$. This means that every facet of $f^\lambda (\Delta)$ containing $\tau$ must also contain $u_x^j$, and hence $\lkds[\tau][f^\lambda(\Delta)]$ is a cone over $u_x^j$.
		\end{enumerate}
	\end{proof}
	
	\begin{cor}\label{Corollary: f-lambda preserves PR}
		Suppose $\Delta$ is a PR complex on vertex set $V$ with degree type $\bd=(d_p,\dots,d_1)$. Let $\lambda$ be a positive integer, and let the operation $f^\lambda$ be as in Definition \ref{Definition: f-lambda}. The complex $f^\lambda (\Delta)$ is PR with degree type $\lambda \bd$.
	\end{cor}
	\begin{proof}
		Let $\tau$ be a face of $\tDelta = f^\lambda(\Delta)$. By part (3) of Lemma \ref{Lemma: f-lambda 3-part lemma}, the complex $\lkds[\tau][\tDelta]$ is acyclic unless we have $\tau=f^\lambda(\sigma)$ for some face $\sigma\in \Delta$. In this case, we have $|\tau|=\lambda |\sigma|$, and the complex $\lkds[\tau][\tDelta]$ is equal to $f^\lambda(\lkds)$ (by part (2)), which deformation retracts onto a complex isomorphic to $\lkds$ (by part (1)). Thus for any integer $m\geq 0$, we have $\hh(f^\lambda(\Delta),\lambda m) = \hh(\Delta,m)$, and if $\lambda$ does not divide $m$ we have $\hh(f^\lambda(\Delta),m)=\emptyset$. The result follows from Proposition \ref{Proposition: Alternate PR Definition With Degree Type}.
	\end{proof}
	\begin{rem}
		It should be clear from this proof that the operation $f^\lambda$ does not affect the \textit{values} in each column of the Betti diagram $\beta(I_\Delta)^*$, only the shifts at which those values occur. In other words, for any integer $i$ and any complex $\Delta$ we have $\beta_i(I^*_{f^\lambda(\Delta)})=\beta_i(I_\Delta^*)$.
	\end{rem}

	The operation $f^\lambda$ allows us to take a degree type and scale it up by a factor of $\lambda$. This reduces the problem of finding PR complexes of arbitrary degree types to the problem of finding PR complexes of degree type $(d_p,\dots,d_1)$ where $\hcf (d_p,\dots,d_1) = 1$.
	
	We now consider a second PR-preserving operation, the \textit{free vertex operation} $f^{\free}$, as defined below.
	
	\begin{defn}\label{Definition: f-free}
		Let $\Delta$ be a complex on vertex set $V$ with facets $F_1,\dots,F_m$.
		
		We define the vertex set $\tV$ to be the set $V\cup \{u_{F_1},\dots,u_{F_m}\}$, and we define $f^{\free }(\Delta)$ to be the complex on $\tV$ with facets $F_1\cup\{u_{F_1}\},\dots, F_m \cup \{u_{F_m}\}$.
	\end{defn}
	
	The operation $f^{\free }$ acts on a complex by adding an additional free vertex $u_F$ to each of its facets $u_F$. We claim that if $\Delta$ is PR with degree type $(d_p,\dots,d_1)$, then $f^{\free}(\Delta)$ is also PR, with degree type $(d_p,\dots,d_2,d_1+1)$. Once again, our proof requires three key elements, as laid out in the following lemma.
	\begin{lem}\label{Lemma: f-free 3-part lemma}
		Let $\Delta$ be a simplicial complex on vertex set $V$, and define $\tDelta = f^{\free } (\Delta)$ on vertex set $\tV$.
		\begin{enumerate}
			\item $\tDelta$ deformation retracts on to $\Delta$, unless we have $\Delta=\{\emptyset\}$.
			\item For any face $\sigma\in \Delta$ we have an ismorphism $\lkds[\sigma][\tDelta]\cong f^{\free }(\lkds)$.
			\item For any face $\tau \in \tDelta$ containing a free vertex $u_F$ for some facet $F$ of $\Delta$, either $\tau$ is a facet of $\tDelta$ or $\lkds[\tau][\tDelta]$ is acyclic.
		\end{enumerate}
	\end{lem}
	\begin{proof}
		\begin{enumerate}
			\item Each vertex $u_F\in \tV - V$ is contained in only a single facet of $\tDelta$. If we have $\Delta\neq \{\emptyset\}$, then we know this facet of $\tDelta$ must contain at least one other vertex aside from $u_F$. Thus Corollary \ref{Corollary: Deformation Retract} allows us to delete all of the vertex $u_F$ from $\tDelta$. Deleting all such vertices gives us a deformation retraction $\tDelta\rightsquigarrow \Delta$.
			\item If $F_{i_1},\dots,F_{i_k}$ are the facets of $\Delta$ containing $\sigma$, then the facets of $\tDelta$ containing $\sigma$ are $F_{i_1}\cup \{u_{F_{i_1}}\}\dots,F_{i_k}\cup \{u_{F_{i_k}}\}$. The result follows.
			\item The only facet of $\tDelta$ containing the vertex $u_F$ is $F\cup \{u_F\}$, and hence $\tau$ must also be contained in this facet. Thus the complex $\lkds[\tau][\tDelta]$ is equal to $\langle F\cup \{u_f\} - \tau \rangle$, which is acyclic unless $\tau$ is equal to $F\cup \{u_F\}$ (in which case it is the irrelevant complex $\{\emptyset\}$).
		\end{enumerate}
	\end{proof}
	
	\begin{cor}\label{Corollary: f-free preserves PR}
		Suppose $\Delta$ is a PR complex on vertex set $V$ with degree type $\bd=(d_p,\dots,d_1)$. Let $\lambda$ be a positive integer, and let the operation $f^{\free}$ be as in Definition \ref{Definition: f-free}. The complex $f^{\free }(\Delta)$ is PR with degree type $(d_p,\dots,d_2,d_1+1)$.
	\end{cor}
	\begin{proof}
		Let $\sigma$ be a face of $\tDelta = f^{\free }(\Delta)$. By part (3) of Lemma \ref{Lemma: f-free 3-part lemma}, the complex $\lkds[\sigma][\tDelta]$ is acyclic unless $\sigma$ is either a face of $\Delta$ or a facet of $\tDelta$. In the former case, the complex $\lkds[\tau][\tDelta]$ is isomorphic to $f^{\free }(\lkds)$ by part (2). If $\sigma$ is a facet of $\Delta$, then $\lkds=\{\emptyset\}$ and $f^{\free }(\lkds)=\{u_\sigma\}$, which is acyclic. Otherwise, $f^{\free }(\lkds)$ deformation retracts onto $\lkds$ by part (1). Thus for any simplex $\sigma\in \tDelta$ we have $$h(\tDelta, \sigma)=\begin{cases}
			h(\Delta,\sigma)  & \text{ if } \sigma\in \Delta\text{ but } \sigma \text{ is not a facet of } \Delta\\
			\{-1\}& \text{ if } \sigma \text{ is a facet of } \tDelta\\
			\emptyset &\text{ otherwise.}
		\end{cases}$$ and in particular, the size of the facets of $\tDelta$ is one greater than the size of the facets of $\Delta$. The result follows from Proposition \ref{Proposition: Alternate PR Definition With Degree Type}.
	\end{proof}
	\begin{rem}
		Once again, this proof also demonstrates that the operation $f^{\free }$ does not affect the total Betti numbers of the ideal $I_\Delta^*$. That is, for any integer $i$ and any complex $\Delta$ we have $\beta_i(I^*_{f^{\free }(\Delta)})=\beta_i(I^*_\Delta)$.
	\end{rem}
	
	The operation $f^{\free }$ allows us to take a degree type and add $1$ to its final value. Repeated application of this operation allows us to add any positive integer to the final value of a degree type. This reduces the problem of finding PR complexes of arbitrary degree types to the problem of finding PR complexes of degree type $(d_p,\dots,d_2,1)$.
	
	The $\phi_i$ operations we are about to introduce in the next section have a similar effect on degree types to $f^{\free }$ (in fact, as we will discuss, the operation $\phi_1$ acts identically to $f^{\free }$ on PR complexes).
	
	Before we proceed to the next section, we will look briefly at how the families of complexes we studied in Chapter \ref{Chapter: Families of PR Complexes} behave under the operations $f^\lambda$ and $f^{\free }$.
	
	\underline{\textbf{The operation $f^\lambda$:}}
	\begin{enumerate}
		\item \textbf{$f^\lambda$ preserves group symmetries.}
		
		Specifically, if $\Delta$ is symmetric under the action of a group $G$, then we can extend this group action to $f^\lambda(\Delta)$ by setting $g(u_x^i)=u_{gx}^i$.
		
		\item \textbf{The family of disjoint simplices is closed under $f^\lambda$.}
		
		Specifically, if $\Delta$ is the complex consisting of two disjoint simplices of size $a$, then $f^\lambda(\Delta)$ is the complex consisting of two disjoint simplices of size $\lambda a$.
		
		\item \textbf{The family of intersection complexes is closed under $f^\lambda$.}
		
		Specifically, we have $f^\lambda(\calI(\bm))=\calI(\lambda \bm)$.
		
		\item \textbf{The family of cycle complexes is closed under $f^\lambda$.}
		
		Specifically, we have $f^\lambda(\fC_{a,b})=\fC_{\lambda a, \lambda b}$.
		
		\item\textbf{ The family of partition complexes is \textit{not} closed under $f^\lambda$ (for $\lambda > 1$).}
		
		We can see this because for any partition complex $\calP(a,p,m)$ the complex $f^\lambda(\calP(a,p,m))$ has a degree type of form $(\overbrace{\lambda,\dots,\lambda\underbrace{\lambda a,\lambda,\dots,\lambda}_{m}}^{p})$, which cannot be degree type of a partition complex except in the case where $\lambda = 1$.
	\end{enumerate}
	
	\underline{\textbf{The operation $f^{\free }$:}}
	\begin{enumerate}
		\item \textbf{$f^{\free }$ preserves group symmetries.}
		
		Specifically, if $\Delta$ is symmetric under the action of a group $G$, then we can extend this group action to $f^{\free }(\Delta)$ by setting $g(u_F)=u_{gF}$.
		
		\item \textbf{The family of disjoint simplices is closed under $f^{\free }$.}
		
		Specifically, if $\Delta$ is the complex consisting of two disjoint simplices of size $a$, then $f^{\free }(\Delta)$ is the complex consisting of two disjoint simplices of size $a+1$.
		
		\item \textbf{The family of intersection complexes is closed under $f^{\free }$.}
		
		Specifically, we have $f^{\free }(\calI(m_1,\dots,m_n))=\calI(m_1+1,m_2,\dots,m_n)$.
		
		\item \textbf{The family of cycle complexes is \textit{not} closed under $f^{\free }$.}
		
		If $a\leq b$ (i.e. in the case where $\fC_{a,b}$ is also an intersection complex) we have $f^{\free }(\fC_{a,b})=\fC_{a,b+1}$, but this is \textit{not} the case when $a>b$. For example, the complex $\fC_{2,1}$ which we saw in Example \ref{Example: Cycle Complexes} has five facets, while the complex $\fC_{2,2}$ has only three facets. Thus $\fC_{2,2}$ cannot be equal to $f^{\free }(\fC_{2,1})$.
		
		\item \textbf{The family of partition complexes is \textit{not} closed under $f^{\free }$.}
		
		In the case where $m=1$ (i.e. when $\calP(a,p,m)$ is also an intersection complex) we have $f^{\free }(\calP(a,p,1))=\calP(a+1,p,1)$, but this is not the case for $m>1$. We can see this from the fact that in all other cases, $f^{\free }(\calP(a,p,m))$ has degree type  $(\overbrace{1,\dots,1\underbrace{a,1,\dots,1,2}_{m}}^{p})$, which cannot be the degree type of a partition complex.
	\end{enumerate}
	
	
	\section{The $\phi_i$ Operations}
	In this section we introduce a family of  PR-preserving operations which we can use to prove Theorem \ref{Theorem: PR Complexes of Any Degree Type}.
	
	Specifically, we will construct a family of operations on simplicial complexes $\{\phi_i : i\in \ZZ^+\}$ with the property that for each $i\in \ZZ^+$, and any PR complex $\Delta$ with degree type $(d_p,\dots,d_i,\underbrace{1,\dots,1}_{i-1})$ for some $p\geq i$, the complex $\phi_i(\Delta)$ is PR with degree type $(d_p,\dots,d_{i-1},d_i+1,\underbrace{1,\dots,1}_{i-1})$. If $\Delta$ is PR with a degree type of a different form, then $\phi_i(\Delta)$ will not be PR.
	
	To see why the existence of such a family is sufficient to prove Theorem \ref{Theorem: PR Complexes of Any Degree Type}, suppose we wish to construct a PR complex of degree type $\bd=(d_p,\dots,d_1)$. We can do so by taking a PR complex $\Delta$ of degree type $(\underbrace{1,\dots,1}_{p})$ (such as the boundary of the $p$-simplex in Example \ref{Example: PR Simplex}), and applying each of the operations $\phi_i$ to $\Delta$ a total of $(d_i-1)$ times.
	
	In other words, if $\Delta$ is the boundary of the $p$-simplex, then the complex $\phi_1^{d_1-1}\dots \phi_p^{d_p-1} (\Delta)$ is a PR complex of degree type $(d_p,\dots,d_1)$.
	
	We define the operation $\phi_i$ below.
	
	\begin{defn}\label{Definition: Phi_i}
	Let $\Delta$ be a simplicial complex on vertex set $V$, and fix a positive integer $i$. The complex $\phi_i(\Delta)$ is defined as follows.
	\begin{enumerate}
		\item $S^i_\Delta:=\{u_{\sigma}: \sigma \in \Delta, |\sigma|\geq \dim \Delta + 2 -i\}$.
		\item $\phi_i(\Delta)$ is the complex on vertex set $V\sqcup S^i_\Delta$ obtained by adding to $\Delta$ all those faces $\sigma \sqcup \{u_{\tau_1},\dots u_{\tau_r}\}$ for each sequence $\sigma \subseteq \tau_1 \subset \dots \subset \tau_r$ in $\Delta$ for which the vertices $u_{\tau_1},\dots, u_{\tau_r}$ are in $S^i_{\Delta}$.
	\end{enumerate}
	\end{defn}
	\begin{rem}\label{Remark: Phi_i = Phi_j}
	If $m= \dim \Delta + 2$ then for any $i>m$ we have $\phi_i(\Delta)=\phi_m(\Delta)$. Note also that $m$ is the smallest integer for which the set $S^m_\Delta$ contains the vertex $u_\emptyset$.
	\end{rem}
	
	\begin{ex}
	The complex $\phi_1(\Delta)$ is obtained from $\Delta$ by adding the faces $F\cup \{u_F\}$ for each facet $F$ in $\Delta$ of maximal dimension. In other words, $\phi_1$ acts on a complex by adding an additional free vertex $u_F$ to each facet $F$ of maximal dimension. In particular, if $\Delta$ is PR then all of its facets have maximal dimension by Lemma \ref{Lemma: PR complexes are pure}, so $\phi_1$ acts on $\Delta$ identically to the free vertex operation $f^{\free }$ from Definition \ref{Definition: f-free}.
	\end{ex}
	
	\begin{ex}\label{Example: phi_i(Delta)}
	Let $\Delta$ be the boundary of the $2$-simplex on vertex set $\{x,y,z\}$:
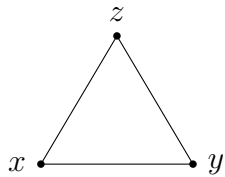
	\begin{center}
		\begin{tikzpicture}[scale = 1]
			\tikzstyle{point}=[circle,thick,draw=black,fill=black,inner sep=0pt,minimum width=2pt,minimum height=2pt]
			\node (a)[point, label=left:$x$] at (0,0) {};
			\node (b)[point, label=right:$y$] at (2,0) {};
			\node (c)[point, label=above:$z$] at (1,1.7) {};
			
			\draw (a.center) -- (b.center) -- (c.center) -- cycle;
		\end{tikzpicture}
	\end{center}
	This is PR with degree type $(1,1)$.
	
	The complex $\phi_1(\Delta)$ is
	
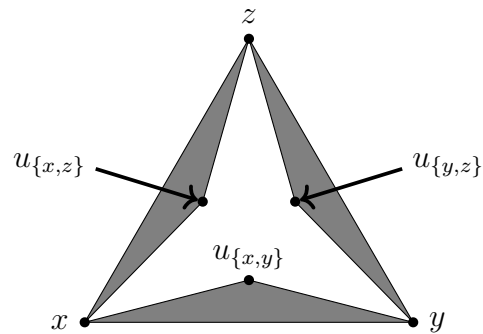
\begin{center}
		\begin{tikzpicture}[scale = 0.8]
			\tikzstyle{point}=[circle,thick,draw=black,fill=black,inner sep=0pt,minimum width=3pt,minimum height=3pt]
			\node (x)[point, label=left:$x$] at (-1.7,-1) {};
			\node (y)[point, label=right:$y$] at (3.7,-1) {};
			\node (z)[point, label=above:$z$] at (1,3.7) {};
			
			
			\node (d)[point] at (1,-0.3) {}; 
			\node (e)[point] at (1.76,1) {}; 
			\node (f)[point] at (0.24,1) {}; 
			
			
			
			\node (uxy)[label = {[label distance=-1.5mm] above:$u_{\{x,y\}}$}] at (1,-0.3) {};
			\node (uyz)[label = {[label distance=-3mm] right:$u_{\{y,z\}}$}] at (3.7,1.6) {};
			\node (uxz)[label = {[label distance=-3mm] left:$u_{\{x,z\}}$}] at (-1.7,1.6) {};
			
			
			\begin{scope}[on background layer]
				
				\draw[fill=gray,fill opacity=.8] (x.center) -- (y.center) -- (d.center) -- cycle;
				\draw[fill=gray,fill opacity=.8] (y.center) -- (z.center) -- (e.center) -- cycle;
				\draw[fill=gray,fill opacity=.8] (x.center) -- (z.center) -- (f.center) -- cycle;
				
			\end{scope}
			
			\draw[->, line width=1.4pt] (uyz)  -- (e);
			\draw[->, line width=1.4pt] (uxz)  -- (f);
			
		\end{tikzpicture}
	\end{center}
	which is PR with degree type $(1,2)$.
	
	The complex $\phi_2(\Delta)$ is
	
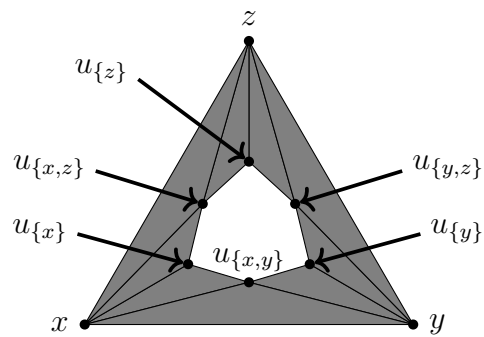
\begin{center}
		\begin{tikzpicture}[scale = 0.8]
			\tikzstyle{point}=[circle,thick,draw=black,fill=black,inner sep=0pt,minimum width=3pt,minimum height=3pt]
			\node (x)[point, label=left:$x$] at (-1.7,-1) {};
			\node (y)[point, label=right:$y$] at (3.7,-1) {};
			\node (z)[point, label=above:$z$] at (1,3.7) {};
			
			\node (a)[point] at (0,0) {}; 
			\node (b)[point] at (2,0) {}; 
			\node (c)[point] at (1,1.7) {}; 
			
			\node (d)[point] at (1,-0.3) {}; 
			\node (e)[point] at (1.76,1) {}; 
			\node (f)[point] at (0.24,1) {}; 
			
			
			\node (ux)[color=red, label = {[label distance=-3mm] left:$u_{\{x\}}$}] at (-2,0.55) {};
			\node (uy)[label = {[label distance=-3mm] right:$u_{\{y\}}$}] at (4,0.55) {};
			\node (uz)[label = {[label distance=-3mm] left:$u_{\{z\}}$}] at (-1,3.2) {};

			\node (uxy)[label = {[label distance=-1.5mm] above:$u_{\{x,y\}}$}] at (1,-0.3) {};
			\node (uyz)[label = {[label distance=-3mm] right:$u_{\{y,z\}}$}] at (3.7,1.6) {};
			\node (uxz)[label = {[label distance=-3mm] left:$u_{\{x,z\}}$}] at (-1.7,1.6) {};
			
			
			\begin{scope}[on background layer]
				\draw[fill=gray,fill opacity=.8] (x.center) -- (a.center) -- (d.center) -- cycle;
				\draw[fill=gray,fill opacity=.8] (x.center) -- (a.center) -- (f.center) -- cycle;
				\draw[fill=gray,fill opacity=.8] (y.center) -- (b.center) -- (d.center) -- cycle;
				\draw[fill=gray,fill opacity=.8] (y.center) -- (b.center) -- (e.center) -- cycle;
				\draw[fill=gray,fill opacity=.8] (z.center) -- (c.center) -- (e.center) -- cycle;
				\draw[fill=gray,fill opacity=.8] (z.center) -- (c.center) -- (f.center) -- cycle;
				
				\draw[fill=gray,fill opacity=.8] (x.center) -- (y.center) -- (d.center) -- cycle;
				\draw[fill=gray,fill opacity=.8] (y.center) -- (z.center) -- (e.center) -- cycle;
				\draw[fill=gray,fill opacity=.8] (x.center) -- (z.center) -- (f.center) -- cycle;
				
			\end{scope}
			
			\draw[->, line width=1.4pt] (ux)   -- (a);
			\draw[->, line width=1.4pt] (uy)   -- (b);
			\draw[->, line width=1.4pt] (uz)   -- (c);
			\draw[->, line width=1.4pt] (uyz)  -- (e);
			\draw[->, line width=1.4pt] (uxz)  -- (f);
			
		\end{tikzpicture}
	\end{center}
	which is PR with degree type $(2,1)$.
	
	And the complex $\phi_3(\Delta)$ is
	
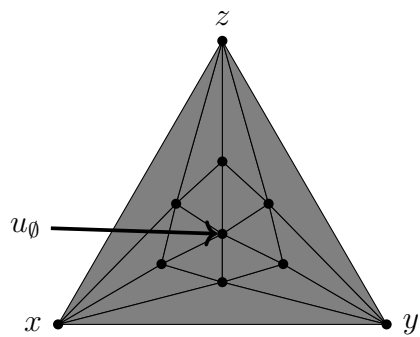
\begin{center}
		\begin{tikzpicture}[scale = 0.8]
			\tikzstyle{point}=[circle,thick,draw=black,fill=black,inner sep=0pt,minimum width=3pt,minimum height=3pt]
			\node (x)[point, label=left:$x$] at (-1.7,-1) {};
			\node (y)[point, label=right:$y$] at (3.7,-1) {};
			\node (z)[point, label=above:$z$] at (1,3.7) {};
			
			\node (a)[point] at (0,0) {}; 
			\node (b)[point] at (2,0) {}; 
			\node (c)[point] at (1,1.7) {}; 
			
			\node (d)[point] at (1,-0.3) {}; 
			\node (e)[point] at (1.76,1) {}; 
			\node (f)[point] at (0.24,1) {}; 
			
			\node (g)[point, label] at (1,0.5) {}; 
			
			
			
			\node (uemp)[label = {[label distance=-3mm] left:$u_{\emptyset}$}] at (-2,0.6) {};
			
			\begin{scope}[on background layer]
				\draw[fill=gray,fill opacity=.8] (x.center) -- (a.center) -- (d.center) -- cycle;
				\draw[fill=gray,fill opacity=.8] (x.center) -- (a.center) -- (f.center) -- cycle;
				\draw[fill=gray,fill opacity=.8] (y.center) -- (b.center) -- (d.center) -- cycle;
				\draw[fill=gray,fill opacity=.8] (y.center) -- (b.center) -- (e.center) -- cycle;
				\draw[fill=gray,fill opacity=.8] (z.center) -- (c.center) -- (e.center) -- cycle;
				\draw[fill=gray,fill opacity=.8] (z.center) -- (c.center) -- (f.center) -- cycle;
				
				\draw[fill=gray,fill opacity=.8] (x.center) -- (y.center) -- (d.center) -- cycle;
				\draw[fill=gray,fill opacity=.8] (y.center) -- (z.center) -- (e.center) -- cycle;
				\draw[fill=gray,fill opacity=.8] (x.center) -- (z.center) -- (f.center) -- cycle;
				
				\draw[fill=gray,fill opacity=.8] (a.center) -- (d.center) -- (g.center) -- cycle;
				\draw[fill=gray,fill opacity=.8] (a.center) -- (f.center) -- (g.center) -- cycle;
				\draw[fill=gray,fill opacity=.8] (b.center) -- (d.center) -- (g.center) -- cycle;
				\draw[fill=gray,fill opacity=.8] (b.center) -- (e.center) -- (g.center) -- cycle;
				\draw[fill=gray,fill opacity=.8] (c.center) -- (e.center) -- (g.center) -- cycle;
				\draw[fill=gray,fill opacity=.8] (c.center) -- (f.center) -- (g.center) -- cycle;
			\end{scope}
			
			\draw[->, line width=1.4pt] (uemp) -- (g);
			
		\end{tikzpicture}
	\end{center}
	Note that the addition of the facets containing $u_\emptyset$ makes this complex acyclic.
	\end{ex}
	
	\begin{rem}\label{Remark: Phi-i and Partition Complexes are Similar}
	It may help the reader to consider the parallels between the $\phi_i$ operations defined above and the partition complex construction.
	
	Recall that we can obtain $\calP(2,p,1)$ by introducing some additional partition vertices and adding additional facets to the boundary of the $p$-simplex on the vertex set $X_p=\{x_0,\dots,x_p\}$, $\partial \langle X_p\rangle$ (in fact, $\calP(2,p,1)$ is the result of applying the operation $\phi_1$ to $\partial \langle X_p \rangle$). Moreover, we can obtain $\calP(2,p,2)$ from $\calP(2,p,1)$ by adding additional facets, and we can obtain $\calP(2,p,3)$ from $\calP(2,p,2)$ similarly, and so on. Thus we get an inclusion of partition complexes $\partial \langle X_p\rangle\subseteq \calP(2,p,1) \subseteq \calP(2,p,2)\subseteq \dots\subseteq \calP(2,p,p+1)$, with the final complex in this chain being acyclic, and all the others having $\nth[st]{(p-1)}$ homology.
	
	This is analogous to how we can obtain $\phi_1(\Delta)$ from $\Delta$ by introducing the vertices in $S^1_{\Delta}$ and adding additional facets, and then $\phi_2(\Delta)$ from $\phi_1(\Delta)$ by introducing more vertices and facets, and so on. This gives us an inclusion of complexes $\Delta\subseteq \phi_1(\Delta)\subseteq \phi_2(\Delta)\subseteq \dots \subseteq \phi_{\dim\Delta+2}(\Delta)$, which mirrors the inclusion of partition complexes above.
	
	Indeed, as topological spaces, the complexes $\phi_1(\Delta)$, $\phi_2(\Delta)$ and $\phi_3(\Delta)$ in Example \ref{Example: phi_i(Delta)} are homeomorphic to the partition complexes $\calP(2,2,1)$, $\calP(2,2,2)$ and $\calP(2,2,3)$ (shown in Example \ref{Example: Closed partition complex example}). And more generally, if $\Delta$ is the boundary of a $p$-simplex, then for any $1\leq i \leq p+1$, the complex $\phi_i(\Delta)$ is homeomorphic to the partition complex $\calP(2,p,i)$.
	
	The benefit that the $\phi_i$ operations have over partition complexes, as we will see, is that they can be applied to \textit{any} PR complex of degree type $(d_p,\dots,d_i,1,\dots,1)$ and still preserve the PR property (in contrast, partition complexes all use the boundaries of simplices as their starting point). This is a significant benefit (as discussed above, we can exploit it to prove Theorem \ref{Theorem: PR Complexes of Any Degree Type}), but it comes at the cost of adding a very large number of additional vertices. In general, partition complexes have far fewer vertices than the corresponding complex of their degree type obtained from repeated use of the $\phi_i$ operations.\\
\end{rem}

We wish to prove the following theorem about the operation $\phi_i$.

\begin{thm}\label{Theorem: Phi_i Operations Degree Type}
Let $\Delta$ be a PR complex on vertex set $V$, with degree type of the form $(d_p,\dots,d_i,\underbrace{1,\dots,1}_{i-1})$ for some $p\geq i\geq 1$. The complex $\tDelta = \phi_i(\Delta)$ on vertex set $V\sqcup S$ is a PR complex with degree type $(d_p,\dots,d_i+1,\underbrace{1,\dots,1}_{i-1})$
\end{thm}

The following lemma will turn out to be particularly crucial, because it shows us that $\phi_i$ commutes with taking links of faces in $\Delta$.

\begin{lem}\label{Lemma: Phi_i and Link commute}
Let $\Delta$ be a simplicial complex on vertex set $V$, let $i\geq 1$ and let $\sigma \in \Delta$. We have an isomorphism of complexes $\link_{\phi_i(\Delta)} \sigma \cong \phi_i(\lkds)$.
\end{lem}
\begin{proof}



We claim that the map of vertices $u_{\tau}\mapsto u_{\tau-\sigma}$ gives a well-defined bijection between the facets of $\link_{\phi_i(\Delta)} \sigma$ and the facets of $\phi_i(\lkds)$.

The facets of $\link_{\phi_i(\Delta)} \sigma$ are all of the form $G-\sigma$ for some facet $G$ of $\phi_i(\Delta)$ containing $\sigma$. Let $G = \tau_1 \cup \{u_{\tau_1},\dots, u_{\tau_r}\}$ be one such facet of $\phi_i(\Delta)$, for some $\tau_1 \subset \dots \subset \tau_r$ in $\Delta$ with $|\tau_j|\geq \dim \Delta + 2 - i$ for each $1\leq j \leq r$. Because $G$ contains $\sigma$, we must have $\sigma \subseteq \tau_1$. Hence, for each $1\leq j \leq r$, the simplex $\tau_j-\sigma$ is a face of $\lkds$, and we have
\begin{align*}
|\tau_j - \sigma| &= |\tau_j| - |\sigma|\\
&\geq \dim \Delta + 2 - i - |\sigma|\\
&= \dim(\lkds) + 2 - i
\end{align*}
which shows that $u_{\tau_j-\sigma}$ is a vertex of $S^i_{\lkds}$. Thus the map of vertices $u_{\tau}\mapsto u_{\tau-\sigma}$ gives us a unique corresponding facet $\widetilde{G}=(\tau_1 - \sigma) \cup \{u_{\tau_1-\sigma},\dots, u_{\tau_r-\sigma}\}$ of $\phi_i(\lkds)$.

Conversely, suppose  $H = \rho_1 \cup \{u_{\rho_1},\dots, u_{\rho_r}\}$ is a facet of $\phi_i(\lkds)$, for some $\rho_1 \subset \dots \subset \rho_r$ in $\lkds$ with $|\rho_j|\geq \dim (\lkds) + 2 - i$ for each $1\leq j \leq r$. By the definition of $\lkds$ we know that for each $1\leq j \leq r$,  $\rho_j\sqcup\sigma$ is a face of $\Delta$. Moreover, we have
\begin{align*}
|\rho_j \sqcup \sigma| &= |\rho_j| + |\sigma|\\
&\geq \dim (\lkds) + 2 - i + |\sigma|\\
&= \dim(\Delta) + 2 - i
\end{align*}
which shows that $u_{\rho_j-\sigma}$ is a vertex of $S^i_\Delta$. Thus the map of vertices $u_{\rho}\mapsto u_{\rho\sqcup\sigma}$ gives us a unique corresponding facet $\widehat{H}=\rho_1 \cup \{u_{\rho_1\sqcup\sigma},\dots, u_{\rho_r\sqcup\sigma}\}$ of $\link_{\phi_i(\Delta)}\sigma$.
\end{proof}
\begin{rem}\label{Remark: Phi-i and Partition Complexes Links are Similar}
Lemma \ref{Lemma: Phi_i and Link commute} can be seen as an analogue to Corollary \ref{Corollary: Partition Complex Link of sigma_X} about partition complexes, which told us that the link of any face in a partition complex consisting entirely of boundary vertices $x_0,\dots,x_p$ is a smaller partition complex.
\end{rem}

In the following sections we work towards proving Theorem \ref{Theorem: Phi_i Operations Degree Type} (and hence also Theorem \ref{Theorem: PR Complexes of Any Degree Type}). The structure will be as follows.

Section \ref{Subsection: Barycentric Subdivision} introduces the barycentric subdivision of $\Delta$, a standard combinatorial construction which is an important subcomplex of $\phi_i(\Delta)$ in the case where $i\geq \dim \Delta +1$. Section \ref{Subsection: Deformation Retractions} examines some deformation retractions of $\phi_i(\Delta)$, to help us in finding its homology. Section \ref{Subsection: Links in BDelta} is devoted to the homologies of links in the barycentric subdivision of $\Delta$. And finally Section \ref{Subsection: Proving Phi_i Theorem} assembles all of these results together to prove the theorem.

\section{Barycentric Subdivision}\label{Subsection: Barycentric Subdivision}

Let $\Delta$ be a simplicial complex on vertex set $\{v_1,\dots,v_n\}$, and let $X_\Delta$ be its geometric realization in the space $\RR^n$ with canonical basis $\{e_1,\dots,e_n\}$. Recall (e.g. from the proof of Lemma \ref{Lemma: Deformation Retract}) that we define $X_\Delta$ to be the union of the sets $X_\sigma$ for each nonempty face $\sigma = \{v_{i_1}\dots,v_{i_r}\}$, each of which is given by $X_\sigma = \left\{\sum_{j=1}^r \lambda_i e_{j_i} : \lambda_1,\dots,\lambda_r > 0, \sum_{j=1}^r \lambda_j = 1 \right\}.$

For a face $\sigma=\{v_{i_1},\dots,v_{i_r}\}$ in $\Delta$, we define the \textit{barycenter} $\bfb_\sigma$ of $X_\sigma$ to be the vector $\sum_{j=1}^r \frac{1}{r} e_{i_j}$ (this is the unique vector in $X_\sigma$ whose nonzero coordinates are all equal). We may use these barycenters to divide $X_\Delta$ up into subsimplices in a process known as \textit{barycentric subdivision}. Specifically, we can view $X_\Delta$ as the union of all the convex hulls of vertex sets of the form $\{\bfb_{\sigma_1},\dots,\bfb_{\sigma_r}\}$ for some inclusion of faces $\sigma_1\subset \dots\subset \sigma_r$ in $\Delta$.

For example if $\Delta$ is the $2$-simplex, its geometric realization in $\RR^3$ is
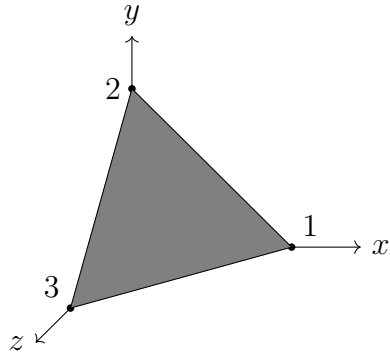
\begin{center}
\begin{tikzpicture}[scale=0.7][line join = round, line cap = round]
\tikzstyle{point}=[circle,thick,draw=black,fill=black,inner sep=0pt,minimum width=2pt,minimum height=2pt]

\coordinate [label=above right:$1$] (1) at (3,0,0);
\coordinate [label=left:$2$] (2) at (0,3,0);
\coordinate [label=above left:$3$] (3) at (0,0,3);		
\node [point] at (3,0,0) {};
\node [point] at (0,3,0) {};
\node [point] at (0,0,3) {};

\coordinate [label=left: $O$](O) at (0,0,0);
\coordinate [label=right:$x$] (x) at (4.3,0,0);
\coordinate [label=above:$y$] (y) at (0,4,0);
\coordinate [label=left:$z$] (z) at (0,0,4.7);

\begin{scope}
	\draw [->] (O)--(x);
	\draw [->] (O)--(y);
	\draw [->] (O)--(z);
	\draw[fill=gray,fill opacity=.6] (1)--(2)--(3)--cycle;
\end{scope}

\end{tikzpicture}
\end{center}
and its barycentric subdvision is
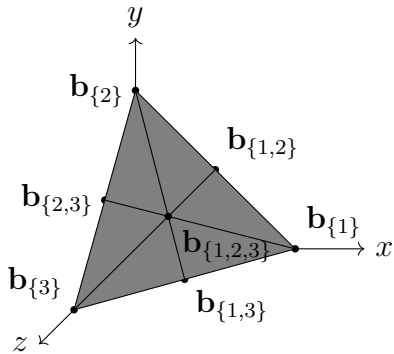
\begin{center}
\begin{tikzpicture}[scale=0.7][line join = round, line cap = round]
\tikzstyle{point}=[circle,thick,draw=black,fill=black,inner sep=0pt,minimum width=2pt,minimum height=2pt]

\coordinate [label=above right:$\bfb_{\{1\}}$] (1) at (3,0,0);
\coordinate [label=left:$\bfb_{\{2\}}$] (2) at (0,3,0);
\coordinate [label=above left:$\bfb_{\{3\}}$] (3) at (0,0,3);		
\node [point] at (3,0,0) {};
\node [point] at (0,3,0) {};
\node [point] at (0,0,3) {};

\coordinate [label=above right:$\bfb_{\{1,2\}}$] (4) at (1.5,1.5,0);
\coordinate [label=below right:$\bfb_{\{1,3\}}$] (5) at (1.5,0,1.5);
\coordinate [label=left:$\bfb_{\{2,3\}}$] (6) at (0,1.5,1.5);
\coordinate (7) at (1,1,1);
\node [point] at (1.5,1.5,0) {};
\node [point] at (1.5,0,1.5) {};
\node [point] at (0,1.5,1.5) {};

\coordinate [label=right:$x$] (x) at (4.3,0,0);
\coordinate [label=above:$y$] (y) at (0,4,0);
\coordinate [label=left:$z$] (z) at (0,0,4.7);

\begin{scope}
	\draw [->] (1)--(x);
	\draw [->] (2)--(y);
	\draw [->] (3)--(z);
	\draw[fill=gray,fill opacity=.8] (1)--(2)--(3)--cycle;
	\draw (1) -- (7);
	\draw (2) -- (7);
	\draw (3) -- (7);
	\draw (4) -- (7);
	\draw (5) -- (7);
	\draw (6) -- (7);
\end{scope}

\node [point, label=below right:$\bfb_{\{1,2,3\}}$] at (1,1,1) {};
\end{tikzpicture}
\end{center}

It is also possible to define this barycentric subdivision entirely combinatorially (see e.g. the introduction to \cite{bary}), on the complex $\Delta$ itself rather than on the geometric realization of $\Delta$. In the combinatorial context, instead of defining the barycenter of $\sigma$ as a vector $\bfb_\sigma$ in $\RR^n$, we simply introduce a new vertex $u_\sigma$. This gives us the following definition.
\begin{defn}\label{Definition: BDelta}
We define the \textit{barycentric subdivision} of $\Delta$ to be the complex $\BDelta$ on vertex set $S_\Delta = \{u_{\sigma}: \sigma \in \Delta-\{\emptyset\}\}$, with faces $\{u_{\sigma_1},\dots, u_{\sigma_r}\}$ whenever $\sigma_1 \subset \dots \subset \sigma_r$.
\end{defn}

Significantly, note that barycentric subdivision does not affect the topology of a simplicial complex. In other words, the complexes $\BDelta$ and $\Delta$ are homeomorphic as topological spaces. In particular, this means that for any integer $j\geq -1$ we have $$\Hred_j\left (\BDelta\right )=\Hred_j(\Delta).$$
This fact will be particularly useful to us.

Note that if $i$ is chosen such that the vertex set $S^i_\Delta$ given in Definition \ref{Definition: Phi_i} contains vertices $u_\sigma$ corresponding to every simplex $\sigma$ in $\Delta$ except $\emptyset$, then the induced subcomplex $\phi_i(\Delta)|_{S^i_\Delta}$ is equal to $\BDelta$. This happens when $i = \dim \Delta + 1$ (and when $i\geq \dim \Delta+2$, the induced subcomplex $\phi_i(\Delta)|_{S^i_\Delta}$ is equal to $\BDelta\ast u_\emptyset$).

For example, if $\Delta$ is the boundary of the $2$-simplex, as in Example \ref{Example: phi_i(Delta)}, its barycentric subdivision $\BDelta$ is
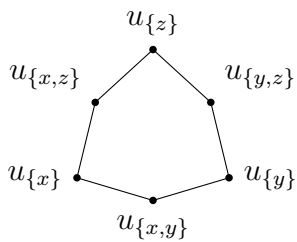
\begin{center}
\begin{tikzpicture}[scale = 1]
\tikzstyle{point}=[circle,thick,draw=black,fill=black,inner sep=0pt,minimum width=2pt,minimum height=2pt]
\node (a)[point, label=left:$u_{\{x\}}$] at (0,0) {};
\node (b)[point, label=right:$u_{\{y\}}$] at (2,0) {};
\node (c)[point, label=above:$u_{\{z\}}$] at (1,1.7) {};

\node (d)[point, label=below:$u_{\{x,y\}}$] at (1,-0.3) {};
\node (e)[point, label=above right:$u_{\{y,z\}}$] at (1.76,1) {};
\node (f)[point, label=above left:$u_{\{x,z\}}$] at (0.24,1) {};

\draw (a.center) -- (d.center) -- (b.center) -- (e.center) -- (c.center) -- (f.center) -- cycle;
\end{tikzpicture}
\end{center}
which is an induced subcomplex of the complex $\phi_2(\Delta)$.
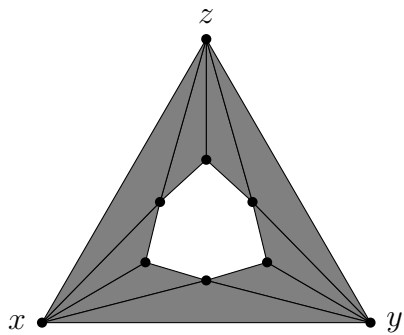
\begin{center}
\begin{tikzpicture}[scale = 0.8]
\tikzstyle{point}=[circle,thick,draw=black,fill=black,inner sep=0pt,minimum width=3pt,minimum height=3pt]
\node (x)[point, label=left:$x$] at (-1.7,-1) {};
\node (y)[point, label=right:$y$] at (3.7,-1) {};
\node (z)[point, label=above:$z$] at (1,3.7) {};

\node (a)[point] at (0,0) {};
\node (b)[point] at (2,0) {};
\node (c)[point] at (1,1.7) {};

\node (d)[point] at (1,-0.3) {}; 
\node (e)[point] at (1.76,1) {}; 
\node (f)[point] at (0.24,1) {}; 

\begin{scope}[on background layer]
	\draw[fill=gray] (x.center) -- (a.center) -- (d.center) -- cycle;
	\draw[fill=gray] (x.center) -- (a.center) -- (f.center) -- cycle;
	\draw[fill=gray] (y.center) -- (b.center) -- (d.center) -- cycle;
	\draw[fill=gray] (y.center) -- (b.center) -- (e.center) -- cycle;
	\draw[fill=gray] (z.center) -- (c.center) -- (e.center) -- cycle;
	\draw[fill=gray] (z.center) -- (c.center) -- (f.center) -- cycle;
	
	\draw[fill=gray] (x.center) -- (y.center) -- (d.center) -- cycle;
	\draw[fill=gray] (y.center) -- (z.center) -- (e.center) -- cycle;
	\draw[fill=gray] (x.center) -- (z.center) -- (f.center) -- cycle;
\end{scope}
\end{tikzpicture}
\end{center}
Thus for $i=\dim \Delta +1$ we can view the operation $\phi_i$ as a kind of prism operator, with $\Delta$ at one end of the prism and $\BDelta$ at the other end (and for $i>\dim \Delta +1$, the complex $\phi_i(\Delta)$ is a prism with $\Delta$ at one end, and the cone $\BDelta \ast u_{\emptyset}$ at the other end).

\section{Deformation Retractions}\label{Subsection: Deformation Retractions}
To find the homologies of the links in $\phi_i(\Delta)$, we will make use of some deformation retractions.

In particular, we use Lemma \ref{Lemma: Deformation Retract} to obtain two deformation retractions of $\tDelta$: one ``\textit{vertex-first}'' deformation from $\tDelta$ on to $\Delta$, and one ``\textit{facet-first}'' deformation from $\tDelta$ on to $\BDelta$. The latter deformation only holds in the specific case where $i=\dim \Delta+1$. For each deformation we provide an example before detailing the general result. We begin with the vertex-first deformation.
\begin{ex}\label{Example: vertex-first DR}
Let $\Delta$ be the boundary of the $2$-simplex on vertex set $\{x,y,z\}$ as in Example \ref{Example: phi_i(Delta)}. We show that there is a deformation retraction $\phi_2(\Delta) \rightsquigarrow \Delta$.

Note that every facet of $\phi_2(\Delta)$ which contains $u_{\{x\}}$ also contains $x$. Thus, if we set $g=\{u_{\{x\}}\}$ and $f = \{x,u_{\{x\}}\}$, Lemma \ref{Lemma: Deformation Retract} allows us to 	remove the vertex $u_{\{x\}}$ from $\phi_2(\Delta)$. Similarly we may remove the vertices $u_{\{y\}}$ and $u_{\{z\}}$. This gives us a deformation retraction $\phi_2(\Delta)\rightsquigarrow \phi_1(\Delta)$.

The same reasoning now allows us to remove the vertices $u_{\{x,y\}}$, $u_{\{x,z\}}$ and $u_{\{y,z\}}$ from $\phi_1(\Delta)$ to obtain a deformation retraction $\phi_1(\Delta)\rightsquigarrow \Delta$.

Diagrammatically, we have the deformation retractions:
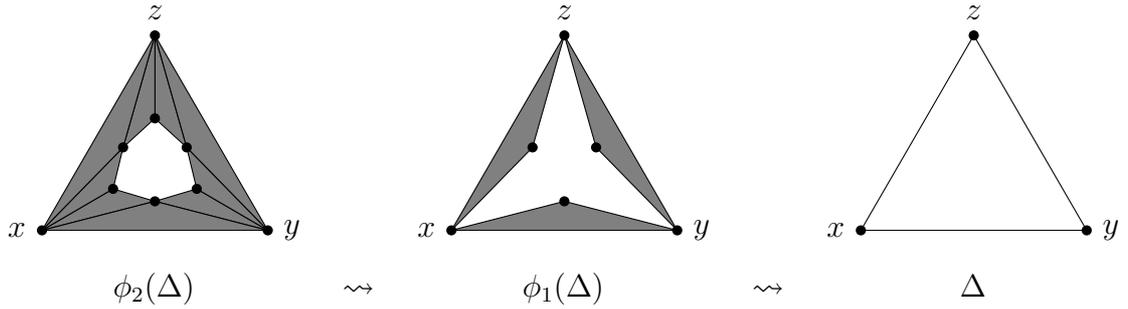
\begin{center}
\begin{tabular}{ c c c c c }
	\begin{tikzpicture}[scale = 0.55]
		\tikzstyle{point}=[circle,thick,draw=black,fill=black,inner sep=0pt,minimum width=3pt,minimum height=3pt]
		\node (x)[point, label=left:$x$] at (-1.7,-1) {};
		\node (y)[point, label=right:$y$] at (3.7,-1) {};
		\node (z)[point, label=above:$z$] at (1,3.7) {};
		
		\node (a)[point] at (0,0) {};
		\node (b)[point] at (2,0) {};
		\node (c)[point] at (1,1.7) {};
		
		\node (d)[point] at (1,-0.3) {}; 
		\node (e)[point] at (1.76,1) {}; 
		\node (f)[point] at (0.24,1) {}; 
		
		\begin{scope}[on background layer]
			\draw[fill=gray] (x.center) -- (a.center) -- (d.center) -- cycle;
			\draw[fill=gray] (x.center) -- (a.center) -- (f.center) -- cycle;
			\draw[fill=gray] (y.center) -- (b.center) -- (d.center) -- cycle;
			\draw[fill=gray] (y.center) -- (b.center) -- (e.center) -- cycle;
			\draw[fill=gray] (z.center) -- (c.center) -- (e.center) -- cycle;
			\draw[fill=gray] (z.center) -- (c.center) -- (f.center) -- cycle;
			
			\draw[fill=gray] (x.center) -- (y.center) -- (d.center) -- cycle;
			\draw[fill=gray] (y.center) -- (z.center) -- (e.center) -- cycle;
			\draw[fill=gray] (x.center) -- (z.center) -- (f.center) -- cycle;
		\end{scope}
	\end{tikzpicture} & 
	& \begin{tikzpicture}[scale = 0.55]
		\tikzstyle{point}=[circle,thick,draw=black,fill=black,inner sep=0pt,minimum width=3pt,minimum height=3pt]
		\node (x)[point, label=left:$x$] at (-1.7,-1) {};
		\node (y)[point, label=right:$y$] at (3.7,-1) {};
		\node (z)[point, label=above:$z$] at (1,3.7) {};
		
		
		\node (d)[point] at (1,-0.3) {}; 
		\node (e)[point] at (1.76,1) {}; 
		\node (f)[point] at (0.24,1) {}; 
		
		\begin{scope}[on background layer]
			
			\draw[fill=gray] (x.center) -- (y.center) -- (d.center) -- cycle;
			\draw[fill=gray] (y.center) -- (z.center) -- (e.center) -- cycle;
			\draw[fill=gray] (x.center) -- (z.center) -- (f.center) -- cycle;
		\end{scope}
	\end{tikzpicture} & 
	& \begin{tikzpicture}[scale = 0.55]
		\tikzstyle{point}=[circle,thick,draw=black,fill=black,inner sep=0pt,minimum width=3pt,minimum height=3pt]
		\node (x)[point, label=left:$x$] at (-1.7,-1) {};
		\node (y)[point, label=right:$y$] at (3.7,-1) {};
		\node (z)[point, label=above:$z$] at (1,3.7) {};
		
		\draw (x.center) -- (y.center) -- (z.center) -- cycle;
	\end{tikzpicture}\\
	$\phi_2(\Delta)$& $\rightsquigarrow$  & $\phi_1(\Delta)$ & $\rightsquigarrow$  & $\Delta$
\end{tabular}
\end{center}
where each deformation retraction is obtained by identifying the vertices $u_\sigma$ for which $|\sigma|$ is minimal with the face $\sigma$ in $\Delta$.
\end{ex}

This example generalises as follows:
\begin{lem}\label{Lemma: DR Phi_i to Delta}
Let $\Delta$ be a pure simplicial complex on vertex set $V$, $i\leq \dim \Delta + 1$, and $\tDelta=\phi_i(\Delta)$ on vertex set $V\sqcup S^i_\Delta$. There is a deformation retraction $\tDelta \rightsquigarrow \tDelta|_V = \Delta$.
\end{lem}
\begin{proof}

Let $\sigma$ be a minimally sized face of $\Delta$ such that $u_\sigma$ is in $S^i_\Delta$. Note that $|\sigma|\geq \dim\Delta + 2 - i \geq \dim\Delta + 2 - (\dim\Delta + 1) = 1$, so in particular $\sigma$ is not empty.

By construction, because $|\sigma|$ is minimal, every facet of $\Delta$ that contains $u_\sigma$ also contains $\sigma$. Thus, setting $g = \{u_\sigma\}$ and $f=\sigma \cup \{u_\sigma\}$, Lemma \ref{Lemma: Deformation Retract} gives us a deformation retraction $\tDelta \rightsquigarrow \tDelta - \{u_\sigma\}$.

Continuing in this way we may remove every vertex $u_\sigma$ in $S^i_\Delta$ from $\tDelta$ - in increasing order of the size of $\sigma$ - and thus obtain a deformation retraction $\tDelta \rightsquigarrow \tDelta|_V = \Delta$.
\end{proof}

We now proceed to the facet-first deformation.
\begin{ex}\label{Example: facet-first DR}
Let $\Delta$ be the boundary of the $2$-simplex on vertex set $\{x,y,z\}$ as in Example \ref{Example: phi_i(Delta)}. We show that there is a deformation retraction $\phi_2(\Delta) \rightsquigarrow \BDelta$.

Note that the edge $\{x,y\}$ of $\Delta$ occurs only in the facet $\{x,y,u_{\{x,y\}}\}$. Thus, if we set $g=\{x,y\}$ and $f = \{x,y,u_{\{x,y\}}\}$, Lemma \ref{Lemma: Deformation Retract} allows us to remove the edge $\{x,y\}$ from $\phi_2(\Delta)$. Similarly we may remove the edges $\{x,z\}$ and $\{y,z\}$.

The same reasoning now allows us to remove the vertices $x$, $y$ and $z$ from $\phi_2(\Delta)$ to obtain a deformation retraction $\phi_2(\Delta)\rightsquigarrow \phi_2(\Delta)|_{S^2_\Delta}= \BDelta$.

Diagrammatically, we have the deformation retractions:
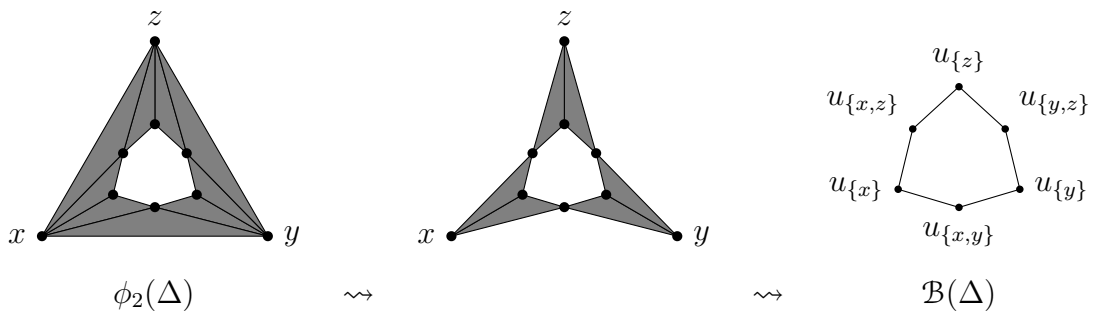
\begin{center}
\begin{tabular}{ c c c c c }
	\begin{tikzpicture}[scale = 0.55]
		\tikzstyle{point}=[circle,thick,draw=black,fill=black,inner sep=0pt,minimum width=3pt,minimum height=3pt]
		\node (x)[point, label=left:$x$] at (-1.7,-1) {};
		\node (y)[point, label=right:$y$] at (3.7,-1) {};
		\node (z)[point, label=above:$z$] at (1,3.7) {};
		
		\node (a)[point] at (0,0) {};
		\node (b)[point] at (2,0) {};
		\node (c)[point] at (1,1.7) {};
		
		\node (d)[point] at (1,-0.3) {}; 
		\node (e)[point] at (1.76,1) {}; 
		\node (f)[point] at (0.24,1) {}; 
		
		\begin{scope}[on background layer]
			\draw[fill=gray] (x.center) -- (a.center) -- (d.center) -- cycle;
			\draw[fill=gray] (x.center) -- (a.center) -- (f.center) -- cycle;
			\draw[fill=gray] (y.center) -- (b.center) -- (d.center) -- cycle;
			\draw[fill=gray] (y.center) -- (b.center) -- (e.center) -- cycle;
			\draw[fill=gray] (z.center) -- (c.center) -- (e.center) -- cycle;
			\draw[fill=gray] (z.center) -- (c.center) -- (f.center) -- cycle;
			
			\draw[fill=gray] (x.center) -- (y.center) -- (d.center) -- cycle;
			\draw[fill=gray] (y.center) -- (z.center) -- (e.center) -- cycle;
			\draw[fill=gray] (x.center) -- (z.center) -- (f.center) -- cycle;
		\end{scope}
	\end{tikzpicture}& 
	& 		\begin{tikzpicture}[scale = 0.55]
		\tikzstyle{point}=[circle,thick,draw=black,fill=black,inner sep=0pt,minimum width=3pt,minimum height=3pt]
		\node (x)[point, label=left:$x$] at (-1.7,-1) {};
		\node (y)[point, label=right:$y$] at (3.7,-1) {};
		\node (z)[point, label=above:$z$] at (1,3.7) {};
		
		\node (a)[point] at (0,0) {};
		\node (b)[point] at (2,0) {};
		\node (c)[point] at (1,1.7) {};
		
		\node (d)[point] at (1,-0.3) {}; 
		\node (e)[point] at (1.76,1) {}; 
		\node (f)[point] at (0.24,1) {}; 
		
		\begin{scope}[on background layer]
			\draw[fill=gray] (x.center) -- (a.center) -- (d.center) -- cycle;
			\draw[fill=gray] (x.center) -- (a.center) -- (f.center) -- cycle;
			\draw[fill=gray] (y.center) -- (b.center) -- (d.center) -- cycle;
			\draw[fill=gray] (y.center) -- (b.center) -- (e.center) -- cycle;
			\draw[fill=gray] (z.center) -- (c.center) -- (e.center) -- cycle;
			\draw[fill=gray] (z.center) -- (c.center) -- (f.center) -- cycle;
			
		\end{scope}
	\end{tikzpicture}  & 
	& 	\begin{tikzpicture}[scale = 0.8]
		\tikzstyle{point}=[circle,thick,draw=black,fill=black,inner sep=0pt,minimum width=2pt,minimum height=2pt]
		\node (a)[point, label=left:$u_{\{x\}}$] at (0,0) {};
		\node (b)[point, label=right:$u_{\{y\}}$] at (2,0) {};
		\node (c)[point, label=above:$u_{\{z\}}$] at (1,1.7) {};
		
		\node (d)[point, label=below:$u_{\{x,y\}}$] at (1,-0.3) {};
		\node (e)[point, label=above right:$u_{\{y,z\}}$] at (1.76,1) {};
		\node (f)[point, label=above left:$u_{\{x,z\}}$] at (0.24,1) {};
		
		\draw (a.center) -- (d.center) -- (b.center) -- (e.center) -- (c.center) -- (f.center) -- cycle;
	\end{tikzpicture}\\
	$\phi_2(\Delta)$& $\rightsquigarrow$  &  & $\rightsquigarrow$  & $\BDelta$
\end{tabular}
\end{center}
where each deformation retraction is obtained by identifying the faces $\sigma$ of $\Delta$ for which $|\sigma|$ is maximal with the vertex $u_\sigma$ in $\BDelta$.
\end{ex}

Once again this example admits a generalisation:
\begin{lem}\label{Lemma: DR Phi_i to BDelta}
Let $\Delta$ be a simplicial complex on vertex set $V$, $i = \dim \Delta + 1$, and $\tDelta=\phi_i(\Delta)$ on vertex set $V\sqcup S^i_\Delta$. There is a deformation retraction $\tDelta \rightsquigarrow \tDelta|_{S^i_\Delta} = \BDelta$.
\end{lem}
\begin{proof}
The condition on the value of $i$ here means that $S^i_\Delta$ contains vertices $u_\sigma$ for every face $\sigma$ of $\Delta$ such that $|\sigma|\geq \dim \Delta + 2 - (\dim \Delta + 1) = 1$. In other words, \textit{every nonempty} face $\sigma$ of $\Delta$ has a corresponding vertex $u_\sigma$ in $S^i_\Delta$, but $u_\emptyset$ is not in $S^i_\Delta$. In particular this means that $\tDelta|_{S^i_\Delta} = \BDelta$.

Let $\sigma$ be a maximally sized face of $\Delta$ (i.e. a facet of dimension $\dim \Delta$). By construction, because $|\sigma|$ is maximal, every facet of $\tDelta$ that contains $\sigma$ also contains $u_\sigma$. Thus, setting $g = \sigma$ and $f=\sigma \cup \{u_\sigma\}$, Lemma \ref{Lemma: Deformation Retract} gives us a deformation retraction $\tDelta \rightsquigarrow \tDelta - \sigma$.

Because every nonempty face $\sigma$ of $\Delta$ has a corresponding vertex in $S^i_\Delta$, then we may continue in this way to remove every nonempty face $\sigma$ of $\Delta$ from $\tDelta$ - in decreasing order of the size of $\sigma$ - and thus obtain a deformation retraction $\tDelta \rightsquigarrow \tDelta|_{S^i_\Delta} = \BDelta$.
\end{proof}

The deformation in Lemma \ref{Lemma: DR Phi_i to BDelta} has the following important corollary.

\begin{cor}\label{Corollary: DR Phi_i Acyclic}
Let $\Delta$ be a simplicial complex on vertex set $V$, and $i > \dim \Delta + 1$. The complex $\tDelta=\phi_i(\Delta)$ is acyclic.
\end{cor}
\begin{proof}
The condition on the value of $i$ here means that \textit{every} face $\sigma$ of $\Delta$ has a corresponding vertex $u_\sigma$ in $S^i_\Delta$, including the empty set. By Remark \ref{Remark: Phi_i = Phi_j} we also have that $\tDelta$ is equal to $\phi_m(\Delta)$ where $m = \dim \Delta + 2$.

We can decompose $\tDelta$ into those facets $F$ which contain $u_\emptyset$ and those which do not. Note that we have $u_\emptyset \notin F$ if and only if $F$ is a facet of $\phi_{m-1}(\Delta)$; and $u_\emptyset \in F$ if and only if $F-\{u_\emptyset\}$ is a facet of $\BDelta$. Thus $\tDelta$ may be expressed as the union of $\phi_{m-1}(\Delta)$ and $\BDelta \ast \{u_\emptyset\}$, and these two subcomplexes intersect at $\BDelta$.

By Lemma \ref{Lemma: DR Phi_i to BDelta} we have a deformation retraction $\phi_{m-1}(\Delta)\rightsquigarrow \BDelta$, which extends to a deformation retraction $\tDelta \rightsquigarrow \BDelta \ast \{u_\emptyset\}$. The complex $\BDelta \ast \{u_\emptyset\}$ is a cone over $u_\emptyset$, and is thus acyclic.
\end{proof}

\begin{ex}\label{Example: phi3 def retract}
Once again, let $\Delta$ be the boundary of the $2$-simplex on vertex set $\{x,y,z\}$ as in Example \ref{Example: phi_i(Delta)}. Corollary \ref{Corollary: DR Phi_i Acyclic} gives us a deformation retraction $\phi_3(\Delta) \rightsquigarrow \BDelta$.
\begin{center}
\begin{tabular}{ c c c c c }
	\begin{tikzpicture}[scale = 0.6]
		\tikzstyle{point}=[circle,thick,draw=black,fill=black,inner sep=0pt,minimum width=3pt,minimum height=3pt]
		\node (x)[point, label=left:$x$] at (-1.7,-1) {};
		\node (y)[point, label=right:$y$] at (3.7,-1) {};
		\node (z)[point, label=above:$z$] at (1,3.7) {};
		
		\node (a)[point] at (0,0) {};
		\node (b)[point] at (2,0) {};
		\node (c)[point] at (1,1.7) {};
		
		\node (d)[point] at (1,-0.3) {}; 
		\node (e)[point] at (1.76,1) {}; 
		\node (f)[point] at (0.24,1) {}; 
		
		\node (g)[point] at (1,0.5) {};
		
		\begin{scope}[on background layer]
			\draw[fill=gray] (x.center) -- (a.center) -- (d.center) -- cycle;
			\draw[fill=gray] (x.center) -- (a.center) -- (f.center) -- cycle;
			\draw[fill=gray] (y.center) -- (b.center) -- (d.center) -- cycle;
			\draw[fill=gray] (y.center) -- (b.center) -- (e.center) -- cycle;
			\draw[fill=gray] (z.center) -- (c.center) -- (e.center) -- cycle;
			\draw[fill=gray] (z.center) -- (c.center) -- (f.center) -- cycle;
			
			\draw[fill=gray] (x.center) -- (y.center) -- (d.center) -- cycle;
			\draw[fill=gray] (y.center) -- (z.center) -- (e.center) -- cycle;
			\draw[fill=gray] (x.center) -- (z.center) -- (f.center) -- cycle;
			
			\draw[fill=gray] (a.center) -- (d.center) -- (g.center) -- cycle;
			\draw[fill=gray] (a.center) -- (f.center) -- (g.center) -- cycle;
			\draw[fill=gray] (b.center) -- (d.center) -- (g.center) -- cycle;
			\draw[fill=gray] (b.center) -- (e.center) -- (g.center) -- cycle;
			\draw[fill=gray] (c.center) -- (e.center) -- (g.center) -- cycle;
			\draw[fill=gray] (c.center) -- (f.center) -- (g.center) -- cycle;
		\end{scope}
	\end{tikzpicture} && 
	\begin{tikzpicture}[scale = 0.8]
		\tikzstyle{point}=[circle,thick,draw=black,fill=black,inner sep=0pt,minimum width=3pt,minimum height=3pt]
		\node (x) at (-1.7,-1) {};
		\node (y) at (3.7,-1) {};
		\node (z) at (1,3.7) {};
		
		\node (a)[point,label=left:$u_{\{x\}}$] at (0,0) {};
		\node (b)[point,label=right:$u_{\{y\}}$] at (2,0) {};
		\node (c)[point,label=above:$u_{\{z\}}$] at (1,1.7) {};
		
		\node (d)[point,label=below:$u_{\{x,y\}}$] at (1,-0.3) {}; 
		\node (e)[point,label=right:$u_{\{y,z\}}$] at (1.76,1) {}; 
		\node (f)[point,label=left:$u_{\{x,z\}}$] at (0.24,1) {}; 
		
		\node (g)[point] at (1,0.5) {};
		
		\begin{scope}[on background layer]
			\draw[fill=gray] (a.center) -- (d.center) -- (g.center) -- cycle;
			\draw[fill=gray] (a.center) -- (f.center) -- (g.center) -- cycle;
			\draw[fill=gray] (b.center) -- (d.center) -- (g.center) -- cycle;
			\draw[fill=gray] (b.center) -- (e.center) -- (g.center) -- cycle;
			\draw[fill=gray] (c.center) -- (e.center) -- (g.center) -- cycle;
			\draw[fill=gray] (c.center) -- (f.center) -- (g.center) -- cycle;
		\end{scope}
	\end{tikzpicture}
	&&  \begin{tikzpicture}[scale = 0.8]
		\tikzstyle{point}=[circle,thick,draw=black,fill=black,inner sep=0pt,minimum width=3pt,minimum height=3pt]
		\node (x) at (-1.5,-1) {};
		\node (y) at (3.7,-1) {};
		\node (z) at (1,3.7) {};
		
		\node (g)[point,label=below:$u_{\emptyset}$] at (1,0.5) {};
	\end{tikzpicture}\\
	$\phi_3(\Delta)$& $\rightsquigarrow$  & $\BDelta\ast \{u_\emptyset\}$ & $\rightsquigarrow$  & $\{u_\emptyset\}$
\end{tabular}
\end{center}
\end{ex}
\begin{rem}\label{Remark: Phi-i and Partition Complex Deformation Retractions are Similar}
The deformation retractions given in these last few lemmas can be seen as analogues to the deformation retractions given in Propositions \ref{Proposition: Deformation D(2,p,m) to D(1,p,m)} and \ref{Proposition: Deformation D(2,p,p+1) to point} for partition complexes $\calP(2,p,m)$. 
\end{rem}

\begin{lem}\label{Lemma: DR Phi_i link of new vertices acyclic}
Let $\Delta$ be a simplicial complex on vertex set $V$, $i\leq \dim \Delta + 1$, and $\tDelta=\phi_i(\Delta)$ on vertex set $V\sqcup S^i_\Delta$. For any nonempty face $\sigma$ of $\tDelta$ contained entirely in $S^i_\Delta$, the complex $\link_{\tDelta} \sigma$ is acyclic.
\end{lem}
\begin{proof}
Suppose $\sigma = \{u_{\tau_1},\dots,u_{\tau_r}\}$ for some faces $\tau_1 \subset \dots \subset \tau_r$ of $\Delta$. The condition on $i$ implies that $\tau_1 \neq \emptyset$.

We must have $\tau_1 \in \link_{\tDelta} \sigma$ because $\tau_1 \sqcup \{u_{\tau_1},\dots,u_{\tau_r}\}$ is a face of $\tDelta$. Also, no vertex in $V-\tau_1$ can be contained in $\link_{\tDelta} \sigma$, because by construction, for any facet $G$ of $\tDelta$ containing $u_{\tau_1}$, we have $G\cap V \subseteq \tau_1$. We claim that there is a deformation retraction $\link_{\tDelta} \sigma \rightsquigarrow \link_{\tDelta} \sigma|_V = \langle\tau_1\rangle$, which is acyclic.

We start by removing every vertex $u_\rho$ in $\link_{\tDelta} \sigma$ and $S^i_\Delta$ for which $\rho \subset \tau_1$. Suppose $u_\rho$ is any such vertex with $|\rho|$ minimal. By the minimality of $|\rho|$, we know that every facet of $\tDelta$ containing $u_\rho$ must also contain $\rho$. Thus every facet of $\link_{\tDelta} \sigma$ containing $u_\rho$ must also contain $\rho$. Setting $g=\{u_\rho\}$ and $f = \{u_\rho\} \cup \rho$, Lemma \ref{Lemma: Deformation Retract} gives us a deformation retraction $\link_{\tDelta} \sigma \rightsquigarrow \link_{\tDelta} \sigma - \{u_\rho\}$. Continuing in this way we may remove every vertex $u_\rho$ in $S^i_\Delta$ from $\link_{\tDelta} \sigma$, in increasing order of the size of $\rho$.

Now we remove the vertices $u_\rho$ in $\link_{\tDelta} \sigma$ and $S^i_\Delta$ for which $\rho \supset \tau_1$. Suppose $u_\rho$ is any such vertex. Because $\rho$ contains $\tau_1$, every facet of $\tDelta$ containing $u_\rho$ and $\sigma$ must also contain $\tau_1$. Thus every facet of $\link_{\tDelta} \sigma$ containing $u_\rho$ must also contain $\tau_1$. Setting $g=\{u_\rho\}$ and $f = \{u_\rho\} \cup \tau_1$, Lemma \ref{Lemma: Deformation Retract} allows us to remove $u_\rho$ from $\link_{\tDelta} \sigma$.
\end{proof}

\section{Links in $\BDelta$}\label{Subsection: Links in BDelta}
Let $\Delta$ be a simplicial complex and $i > \dim \Delta + 1$, and set $\tDelta = \phi_i(\Delta)$. The condition on $i$ ensures that $u_\emptyset$ is a vertex in $S^i_\Delta$.

In this section, we examine the links of those faces $\sigma$ of $\tDelta$ for which $\emptyset \neq \sigma \subseteq S^i_\Delta$ (i.e. the nonempty faces in the induced subcomplex $\tDelta|_{S^i_\Delta}$). We begin by showing that for any such $\sigma$, we can express the homology of $\link_{\tDelta} \sigma$ in terms of the homology of $\link_{\tDelta} (\sigma\cup \{u_\emptyset\})$. As we will explain, this allows us to restrict our attention to the links in the barycentric subdivision complex $\BDelta$.

\begin{prop}\label{Proposition: link of new vertices may as well contain u-emptyset}
Let $\Delta$ be a simplicial complex and $i > \dim \Delta + 1$, and set $\tDelta = \phi_i(\Delta)$. Suppose $\emptyset \neq \sigma \subseteq S^i_\Delta$ is a face of $\tDelta$ with $u_{\emptyset}\notin \sigma$. For every $j\geq -1$ we have an isomorphism $\Hred_j(\link_{\tDelta} \sigma)\cong \Hred_{j-1}(\link_{\tDelta} (\sigma\cup \{u_{\emptyset}\}))$.
\end{prop}
\begin{proof}
As in the proof of Corollary \ref{Corollary: DR Phi_i Acyclic}, we may decompose $\link_{\tDelta} \sigma$ into a subcomplex $A$ consisting of facets which contain $u_\emptyset$ and a subcomplex $B$ consisting of those which do not. The intersection of these subcomplexes consists of those faces $f$ in $\link_{\tDelta} \sigma$ for which $u_{\emptyset}$ is not in $f$ but $f\sqcup \{u_\emptyset\}$ is a face of $\tDelta$. In other words, we have $A\cap B = \link_{\tDelta} (\sigma\cup \{u_{\emptyset}\})$.

The subcomplex $A$ is a cone over $u_\emptyset$, and is therefore acyclic. For every face $f$ in $B$, the intersection of $f$ and $S^i_{\Delta}$ contains only vertices $u_\tau$ for which $\tau$ is nonempty. All of these are faces of $\phi_m(\Delta)$ where $m=\dim\Delta +1$, and hence we have $B=\link_{\phi_m(\Delta)} \sigma$, which is also acyclic by Lemma \ref{Lemma: DR Phi_i link of new vertices acyclic}.

Thus for every $j\geq -1$, the Mayer-Vietoris Sequence gives us an exact sequence
\begin{equation*}
0 \rightarrow \Hred_j(\link_{\tDelta} \sigma)\rightarrow \Hred_{j-1}(\link_{\tDelta} (\sigma\cup \{u_{\emptyset}\}))\rightarrow 0 
\end{equation*}
as required.
\end{proof}

Proposition \ref{Proposition: link of new vertices may as well contain u-emptyset} allows us to restrict our attention to those faces of $\tDelta|_{S^i_\Delta}$ which contain $u_\emptyset$. Note we have $\link_{\tDelta} u_\emptyset = \BDelta$, and hence the link of any face of $\tDelta$ which contains $u_\emptyset$ must be a link in $\BDelta$. For this reason, we devote the rest of this section to investigating the links of $\BDelta$.

\begin{prop}\label{Proposition: Links in BDelta}
Let $\sigma = \{u_{\tau_1},\dots, u_{\tau_r}\}\in \BDelta$ for some faces $\tau_1 \subset \dots \subset \tau_r$ of $\Delta$. We have an isomorphism of complexes
\begin{equation*}
\link_{\BDelta}\sigma \cong \mathcal{B}(\lkds[\tau_r])\ast \mathcal{B}(\link_{\partial \tau_r} \tau_{r-1})\ast \dots \ast \mathcal{B}(\link_{\partial \tau_2} \tau_1)\ast \mathcal{B}(\partial \tau_1).
\end{equation*}
\end{prop}
\begin{proof}

For notational convenience, we set $\tau_0 = \emptyset$, so that $\mathcal{B}(\partial \tau_1)$ may be rewritten as $\mathcal{B}(\link_{\partial \tau_1} \tau_0)$.

Let $A_r$ denote the induced subcomplex of $\BDelta$ on vertices of the form $u_f$ where $f$ contains $\tau_r$; and for each $0\leq j \leq r-1$, let $A_j$ denote the induced subcomplex of $\BDelta$ on vertices of the form $u_f$ for which we have $\tau_j \subset f\subset \tau_{j+1}$. Note that $A_r,\dots,A_0$ are pairwise disjoint subcomplexes of $\BDelta$.

We claim that $\link_{\BDelta}\sigma = A_r\ast A_{r-1}\ast \dots\ast A_0$. This is sufficient to prove our proposition because the complex $A_r$ is isomorphic to $\mathcal{B}(\lkds[\tau_r])$ via the vertex map $u_f\mapsto u_{f-\tau_r}$, and for each $0\leq j \leq r-1$, the complex $A_j$ is isomorphic to $\mathcal{B}(\link_{\partial \tau_{j+1}} \tau_j)$ via the vertex map $u_f\mapsto u_{f-\tau_j}$.

Let $\rho = \rho_r\sqcup \dots \sqcup \rho_0$ be a face of $A_r\ast \dots\ast A_0$, with $\rho_j\in A_j$ for each $0\leq j \leq r$. For each $0\leq j \leq r-1$, the face $\rho_j$ must be of the form $\{u_{f_1},\dots,u_{f_m}\}$ for some sequence of faces $\tau_j \subset f_1\subset \dots \subset f_m\subset \tau_{j+1}$ of $\Delta$. Similarly the face $\rho_r$ must be of the form $\{u_{f_1},\dots,u_{f_m}\}$ for some sequence of faces $\tau_r \subset f_1\subset \dots \subset f_m$ of $\Delta$. In particular, none of the vertices $u_{\tau_1},\dots,u_{\tau_r}$ are contained in $\rho$, so we have $\rho \cap \sigma = \emptyset$. Moreover, the faces of $\Delta$ corresponding to vertices in $\rho\sqcup \sigma$ may be arranged in a strict sequence by inclusion, which means that $\rho\sqcup \sigma\in \BDelta$. Thus $\rho$ is a face of $\link_{\BDelta}\sigma$.

Conversely, suppose $\rho=\{u_{f_1},\dots,u_{f_m}\}$ is any face of $\link_{\BDelta} \sigma$. We may decompose $\rho$ into the disjoint union $\rho_r\sqcup \dots\sqcup \rho_0$, where for each $0\leq j\leq r-1$  the face $\rho_j$ contains all the vertices $u_f$ in $\rho$ for which we have $\tau_j\subset f \subset \tau_{j+1}$, making $\rho_j$ a face of $A_j$; and the face $\rho_r$ contains all the vertices $u_f$ in $\rho$ for which $f$ contains $\tau_r$, making $\rho_r$ a face of $A_r$. Thus $\rho$ is a face of $A_r\ast \dots \ast A_0$.
\end{proof}

In particular, Proposition \ref{Proposition: Links in BDelta} has the following important corollaries.
\begin{cor}\label{Corollary: Homology of Links in BDelta}
Let $\Delta$ and $\sigma=\{u_{\tau_1},\dots,u_{\tau_r}\}\in \BDelta$ be as in Proposition \ref{Proposition: Links in BDelta}. We have $h(\BDelta,\sigma)=h(\Delta,\tau_r)+\{|\tau_r|-r\}$.
\end{cor}
\begin{proof}
Just as in the proof of Proposition \ref{Proposition: Links in BDelta}, we set $\tau_0=\emptyset$ for notational convenience. Using this proposition and Corollary \ref{Corollary: homology index set of joins}, we can compute the homology of $\link_{\BDelta}\sigma$ from the homologies of $\calB(\link_\Delta \tau_r), \mathcal{B}(\link_{\partial \tau_r} \tau_{r-1}), \dots, \mathcal{B}(\link_{\partial \tau_1} \tau_0)$.

To compute these homologies, first recall from Section \ref{Subsection: Barycentric Subdivision} that we have, for any integer $i\geq -1$,
\begin{equation}\label{Equation: h(BDelta)=h(Delta)}
\Hred_i(\BDelta)=\Hred_i(\Delta).
\end{equation}
Next, note that for each $0\leq j \leq r-1$, the complex $\partial \tau_{j+1}$ is the boundary of the $(|\tau_{j+1}|-1)$-simplex. As observed in Example \ref{Example: PR Simplex}, the link of any face $\tau$ in the boundary of the $p$-simplex is the boundary of the $(p-|\tau|)$-simplex, which has homology only at degree $p-|\tau|-1$. Thus we have
\begin{equation}\label{Equation: h(link of tau_j)}
h(\link_{\partial \tau_{j+1}}\tau_j)=\{|\tau_{j+1}|-|\tau_j|-2\}.
\end{equation}
Putting these results together, we find
\begin{align*}
h(\BDelta, \sigma) &= h(\mathcal{B}(\lkds[\tau_r])\ast \circledast_{j=0}^{r-1} \mathcal{B}(\link_{\partial \tau_{j+1}}\tau_j)) &\text{by Prop. \ref{Proposition: Links in BDelta}}\\
&= h(\mathcal{B}(\lkds[\tau_r]))+ \sum_{j=0}^{r-1} h(\mathcal{B}(\link_{\partial \tau_{j+1}}\tau_j))+ \{r\} &\text{by Cor. \ref{Corollary: homology index set of joins}}\\
&= h(\lkds[\tau_r])+ \sum_{j=0}^{r-1}h(\link_{\partial \tau_{j+1}}\tau_j)+ \{r\} &\text{by Equ. \ref{Equation: h(BDelta)=h(Delta)}}\\
&=h(\Delta,\tau_r) + \sum_{j=0}^{r-1} \{|\tau_{j+1}|-|\tau_j|-2\}+\{r\}&\text{by Equ. \ref{Equation: h(link of tau_j)}}\\
&=h(\Delta, \tau_r)+\{|\tau_r|-r\}.&
\end{align*}
\end{proof}
\begin{cor}\label{Corollary: Homology of Links in BDelta PR (1,..,1)}
Let $\Delta$ be a PR complex with degree type $(1,\dots,1)$. For any $\sigma\in \BDelta$ we have that $h(\BDelta,\sigma)$ is either empty or equal to $\{\dim \Delta-|\sigma|\}$	

\end{cor}
\begin{proof}
For $\sigma=\emptyset$ we have $h(\BDelta,\emptyset)=h(\BDelta)=h(\Delta)$, because $\BDelta$ is homeomorphic to $\Delta$. Thus $h(\BDelta,\emptyset)$ is empty, unless $\Delta$ has homology, in which case it is equal to $\{\dim \Delta\}$ by Corollary \ref{Corollary: PR (1...1) Definition}.

Now assume $\sigma=\{u_{\tau_1},\dots,u_{\tau_r}\}$ for some $\tau_1\subset \dots \subset \tau_r$ in $\Delta$. From Corollary \ref{Corollary: Homology of Links in BDelta}, we know that $h(\BDelta,\sigma)=h(\Delta,\tau_r)+\{|\tau_r|-|\sigma|\}$. If $h(\Delta,\tau_r)$ is empty (i.e. $\lkds[\tau_r]$ is acyclic) then this sum is empty. Otherwise, by Corollary \ref{Corollary: PR (1...1) Definition} we have $h(\Delta,\tau_r)=\{\dim \Delta-|\tau_r|\}$, and hence $h(\BDelta,\sigma)=\{\dim \Delta-|\tau_r|\}+\{|\tau_r|-|\sigma|\}=\{\dim \Delta-|\sigma|\}$.
\end{proof}

It follows from Corollary \ref{Corollary: Homology of Links in BDelta PR (1,..,1)} that if $\Delta$ is Cohen-Macaulay (i.e. PR with degree type $(1,\dots,1)$), then $\BDelta$ is also Cohen-Macaulay. In fact, this turns out to be the \textit{only} condition under which $\BDelta$ is PR, as the following proposition demonstrates. This proposition will not be strictly necessary for our proof of Theorem \ref{Theorem: Phi_i Operations Degree Type}, but it helps to illuminate why the operation $\phi_i$ preserves the PR property for PR complexes of degree type $(d_p,\dots,d_i,1,\dots,1)$, and why it fails to do so for PR complexes of other degree types.

\begin{prop}\label{Proposition: When BDelta is PR}
Let $\Delta$ be a simplicial complex. The following are equivalent.
\begin{enumerate}
\item $\BDelta$ is PR.
\item  $\BDelta$ is Cohen-Macaulay (i.e. PR with degree type $(1,\dots,1)$).
\item $\Delta$ is Cohen-Macaulay (i.e. PR with degree type $(1,\dots,1)$).
\end{enumerate}
\end{prop}
\begin{proof}
(3)$\Rightarrow$(2) follows from Corollary \ref{Corollary: Homology of Links in BDelta PR (1,..,1)}, and (2)$\Rightarrow$(1) is immediate. To prove (1)$\Rightarrow$(3), we show the contrapositive.

First assume that $\Delta$ is not PR. This means that $\Delta$ has two faces $\tau_1$ and $\tau_2$ of different sizes such that the intersection $h(\Delta,\tau_1)\cap h(\Delta,\tau_2)$ is nonempty. Suppose $\iota$ is an index in both $h(\Delta,\tau_1)$ and $h(\Delta,\tau_2)$. By Corollary \ref{Corollary: Homology of Links in BDelta}, we have $h(\BDelta,u_{\tau_1})=h(\Delta,\tau_1)+\{|\tau_1|-1\}$ and $h(\BDelta,u_{\tau_2})=h(\Delta,\tau_2)+\{|\tau_2|-1\}$. Thus the complete homology index set $\hh(\BDelta,1)$ contains both $\iota+|\tau_1|-1$ and $\iota+|\tau_2|-1$, and therefore cannot be a singleton. This means $\BDelta$ is not PR by Corollary \ref{Corollary: PR Complex links have single homology}.

Now assume that $\Delta$ is PR of degree type $(d_p,\dots,d_1)$ where $d_j > 1$ for some $1\leq j\leq p$. By Proposition \ref{Proposition: Alternate PR Definition With Degree Type}, $\Delta$ must have two faces $\tau_1$ and $\tau_2$ such that $|\tau_2|=|\tau_1|+d_j$ and $h(\Delta,\tau_1)=\{\iota\}$ while $h(\Delta,\tau_2)=\{\iota-1\}$ for some index $\iota$. By Corollary \ref{Corollary: Homology of Links in BDelta}, we have $h(\BDelta,u_{\tau_1})=\{\iota\}+\{|\tau_1|-1\}=\{\iota+|\tau_1|-1\}$, and $h(\BDelta,u_{\tau_2})=\{\iota-1\}+\{|\tau_2|-1\}=\{\iota + |\tau_2|+d_j-2\}$. In particular, because $d_j>1$, we have $d_j-2>-1$ and hence these two sets are not equal. Thus, once again, the complete homology index set $\hh(\BDelta,1)$ is not a singleton, and so $\BDelta$ cannot be PR by by Corollary \ref{Corollary: PR Complex links have single homology}.
\end{proof}

\section{Proving Theorem \ref{Theorem: Phi_i Operations Degree Type}}\label{Subsection: Proving Phi_i Theorem}
In this section, we bring together all of the results of the previous sections to prove Theorem \ref{Theorem: Phi_i Operations Degree Type}.

\begin{proof}[Proof of Theorem \ref{Theorem: Phi_i Operations Degree Type}]
Let $\Delta$ be a PR complex with offset $s$ and degree type $(d_p,\dots,d_i,\underbrace{1,\dots,1}_{i-1})$ for some $p\geq i \geq 1$, and set $\tDelta =\phi_i(\Delta)$. We aim to prove that $\tDelta$ is PR with degree type $(d_p,\dots,d_i+1,1,\dots,1)$.

By Proposition \ref{Proposition: Alternate PR Definition With Degree Type}, the only nonempty complete homology index sets are the ones given in Table 6.1 below. By the same proposition, we can show that $\tDelta$ is PR with the desired degree type by proving that its only nonempty complete homology index sets are the ones given in Table 6.2.

\begin{center}
\begin{tabular}{c c}
	\textbf{Table 6.1:} \textit{Homology} & \textbf{Table 6.2:} \textit{Required Homology}\\
	\textit{Index Sets of} $\Delta$ &\textit{Index Sets of} $\tDelta$\\
	\begin{tabular}{ c | l }
		$\hh(\Delta,m)$ & $m$\\
		\hline
		$\{p-1\}$ & $s$\\
		$\{p-2\}$ & $s+d_p$\\
		$\vdots$ & $\vdots$\\
		$\{i-1\}$ & $s+\sum_{j=i+1}^p d_j$ \\
		$\{i-2\}$ & $s+\sum_{j=i}^p d_j$ \\
		$\{i-3\}$ & $s+\sum_{j=i}^p d_j + 1$ \\
		$\vdots$ & $\vdots$ \\
		$\{0\}$ & $s+\sum_{j=i}^p d_j + (i-2)$ \\
		$\{-1\}$ & $s+\sum_{j=i}^p d_j + (i-1)$\\
	\end{tabular}
	&
	\begin{tabular}{ c | l }
		$\hh(\tDelta,m)$ & $m$\\
		\hline
		$\{p-1\}$ & $s$\\
		$\{p-2\}$ & $s+d_p$\\
		$\vdots$ & $\vdots$\\
		$\{i-1\}$ & $s+\sum_{j=i+1}^p d_j$ \\
		$\{i-2\}$ & $s+\sum_{j=i}^p d_j + 1$ \\
		$\{i-3\}$ & $s+\sum_{j=i}^p d_j + 2$ \\
		$\vdots$ & $\vdots$ \\
		$\{0\}$ & $s+\sum_{j=i}^p d_j + (i-1)$ \\
		$\{-1\}$ & $s+\sum_{j=i}^p d_j + i$\\
	\end{tabular}
\end{tabular}
\end{center}
Thus it suffices to show that for any natural number $0\leq m \leq s + \sum_{j=i}^p d_j + i$, we have
\begin{equation}\label{Equation: homology index sets}
\hh(\tDelta,m) = \begin{cases}
	\hh(\Delta,m)& \text{ if } m < s + \sum_{j=i}^p d_j\\
	\emptyset & \text{ if } m = s + \sum_{j=i}^p d_j\\
	\hh(\Delta,m-1)& \text{ if } m > s + \sum_{j=i}^p d_j.
\end{cases}
\end{equation}

To this end, we fix a face $\sigma$ of $\tDelta$ of size $|\sigma|=m$, and investigate the homology index set $h(\tDelta,\sigma)$.

We start by decomposing $\sigma$ into $\sigma = \sigma_V \sqcup \sigma_S$, where $\sigma_V$ is a subset of $V$ (and is thus in $\Delta$) and $\sigma_S$ is a subset of $S^i_\Delta$ (and is thus in $\BDelta$).

Note that $\link_{\tDelta} \sigma = \link_{\link_{\tDelta} \sigma_V} \sigma_S$. Thus, using the isomorphism in Lemma \ref{Lemma: Phi_i and Link commute}, we may view $\link_{\tDelta} \sigma$ as a link in the complex $\phi_i(\lkds[\sigma_V])$, and then apply the results of Sections \ref{Subsection: Deformation Retractions} and \ref{Subsection: Links in BDelta} to this link.

To work out which results to apply, we will need to determine whether $i$ is less than, equal to, or greater than $\dim(\lkds[\sigma_V])+1$. By Proposition \ref{Proposition: Dimension of PR Complexes} we know $\dim \Delta = s + \sum_{j=i}^p d_j + i - 2$, which means we have
\begin{equation}\label{Equation: dim lkds[sigma_V]}
\dim (\lkds[\sigma_V])+1= s+ \sum_{j=i}^p d_j + i - 1 - |\sigma_V|.
\end{equation}

We now examine a number of different cases.

\begin{itemize}
\item \textbf{Case 1:} Assume $\sigma_S = \emptyset$ (i.e. $\sigma \in \Delta$). By Lemma \ref{Lemma: Phi_i and Link commute}, we have an isomorphism of complexes $\link_{\tDelta} \sigma \cong\phi_i(\lkds)$.

\subitem \textbf{Case 1.1:} If $|\sigma_V| < s + \sum_{j=i}^p d_j$, then by Equation (\ref{Equation: dim lkds[sigma_V]}) above, we have $i \leq \dim (\lkds) + 1$. Thus Lemma \ref{Lemma: DR Phi_i to Delta} gives us a deformation retraction $\phi_i(\lkds)\rightsquigarrow \lkds$, which means we have $h(\tDelta, \sigma) = h(\Delta,\sigma) \subseteq \hh(\Delta,m)$.

\subitem \textbf{Case 1.2:} If $|\sigma_V| \geq s + \sum_{j=i}^p d_j$, then by Equation (\ref{Equation: dim lkds[sigma_V]}) above, we have $i > \dim (\lkds) + 1$. By Corollary \ref{Corollary: DR Phi_i Acyclic}, the complex $\phi_i(\lkds)$ is acyclic, so we have $h(\tDelta, \sigma) =\emptyset$.

\item \textbf{Case 2:} Assume $\sigma_S \neq \emptyset$ (i.e. $\sigma \notin \Delta$). As mentioned above we use Lemma \ref{Lemma: Phi_i and Link commute} to reinterpret $\link_{\tDelta} \sigma$ as a link in $\phi_i(\lkds[\sigma_V])$. Specifically we have $\link_{\tDelta}\sigma = \link_{\phi_i(\lkds[\sigma_V])} \tau$ for some nonempty face $\tau$ contained in $S^i_{\lkds[\sigma_V]}$, obtained by relabelling the vertices of $\sigma_S$ under the isomorphism given in Lemma \ref{Lemma: Phi_i and Link commute}. In particular we have $|\sigma_S|=|\tau|$.

\subitem \textbf{Case 2.1:} If $|\sigma_V| < s+ \sum_{j=i}^p d_j$, then by Equation (\ref{Equation: dim lkds[sigma_V]}) above, we have $i \leq \dim (\lkds[\sigma_V]) + 1$. By Lemma \ref{Lemma: DR Phi_i link of new vertices acyclic}, the complex $\link_{\phi_i(\lkds[\sigma_V])} \tau$ is acyclic, so we have $h(\tDelta, \sigma) = \emptyset$.

\subitem \textbf{Case 2.2:} Now let $|\sigma_V| \geq s + \sum_{j=i}^p d_j$. By Equation (\ref{Equation: dim lkds[sigma_V]}) above, we have $i > \dim (\lkds[\sigma_V]) + 1$ which means $u_\emptyset \in S^i_{\lkds[\sigma_V]}$. We also have that $\lkds[\sigma_V]$ is PR with degree type $(\underbrace{1,\dots,1}_j)$ for some $0\leq j \leq i-1$.

\subitem First we assume $u_\emptyset \in \tau$. In this case we have
\begin{align*}
	h(\tDelta, \sigma) &= h(\phi_i(\lkds[\sigma_V]), \tau) &\text{by Lemma \ref{Lemma: Phi_i and Link commute}}\\
	&= h(\mathcal{B}(\lkds[\sigma_V]),\tau-\{u_\emptyset\}) &\link_{\phi_i(\Gamma)} u_{\emptyset}=\mathcal{B}(\Gamma)\\
	&\subseteq \{\dim (\lkds[\sigma_V])-|\tau|+1\}&\text{by Corollary \ref{Corollary: Homology of Links in BDelta PR (1,..,1)}}\\
	&=  \{\dim \Delta -|\sigma_V|-|\sigma_S|+1\}&|\tau|=|\sigma_S|\\
	&= \{\dim\Delta - m+1\}&\text{by definition of }m\\
	&= \hh(\Delta,m-1)&\text{by Remark \ref{Remark: PR (dp,...,di,1,...,1) Definition}.}
\end{align*}

\subitem Now we assume $u_\emptyset\notin \tau$. In this case we have
\begin{align*}
	h(\tDelta, \sigma) &= h(\phi_i(\lkds[\sigma_V]), \tau) &\text{by Lemma \ref{Lemma: Phi_i and Link commute}}\\
	&= h(\phi_i(\lkds[\sigma_V]), \tau\sqcup \{u_\emptyset\}) + \{1\} &\text{by Prop. \ref{Proposition: link of new vertices may as well contain u-emptyset}}\\
	&\subseteq \hh(\Delta,m) + \{1\} &\text{by the } u_\emptyset \in \tau \text{ case}\\
	&= \{\dim \Delta - m\}+\{1\} &\text{by Remark \ref{Remark: PR (dp,...,di,1,...,1) Definition}.}\\
	&= \hh(\Delta,m-1) &\text{by Remark \ref{Remark: PR (dp,...,di,1,...,1) Definition}.}
\end{align*}


This exhausts all the possible cases for $\sigma$, and it is sufficent to prove that Equation (\ref{Equation: homology index sets}) is satisfied, because
\begin{itemize}
	\item For $m< s + \sum_{j=i}^p d_j$, any face of $\tDelta$ of size $m$ falls under case 1.1 or 2.1, and hence $\hh(\tDelta,m)=\hh(\Delta,m)$.
	\item For $m= s + \sum_{j=i}^p d_j$, any face of $\tDelta$ of size $m$ falls under case 1.2 or 2.1, and hence $\hh(\tDelta,m)=\emptyset$.
	\item For $n> s + \sum_{j=i}^p d_j$, any face of $\tDelta$ of size $m$ falls under case 1.2, 2.1 or 2.2, and hence $\hh(\tDelta,m)=\hh(\Delta,m-1)$.
\end{itemize}
\end{itemize}


\end{proof}


	
	\chapter{Future Directions}\label{Chapter: Future Directions}
	
	We end by briefly considering a number of future avenues of research arising from the topics we have presented.
	
	\section{Minimal Numbers of Vertices and Shift Types}\label{Subsection: Shift Types and Bounds for n}
	Question \ref{Question: Lower bounds on n for each degree type} still remains unsolved for most degree types. That is, for an arbitrary given degree type $\bd$ we do not know the minimum value of $n$ for which there exists a PR complex of degree type $\bd$ on $n$ vertices.
	
	Recall (from the discussion in Section \ref{Subsection: Motivation PR Families}) that this is in fact equivalent to Question \ref{Question: Shift Types}, because an answer to this question for any given degree type $\bd$ also tells us all of the possible shift types of pure resolutions arising from Stanley-Reisner ideals that have $\bd$ as their difference sequence. Thu, answering this question entirely would provide a full analogue of the first Boij-S\"{o}derberg conjecture for squarefree monomial ideals.
	
	Our procedure in Chapter \ref{Chapter: Generating Degree Types} gives us upper bounds for $n$ for any given degree type, but in general (as noted already), these bounds seem to be far greater than necessary. As an illustration of quite how excessive these bounds are, consider using the procedure to generate a PR complex of degree type $(3,1)$. To do this, we start with the boundary of the $2$-simplex $\Delta$ which has degree type $(1,1)$, and apply the operation $\phi_2$ to $\Delta$ twice. The complex $\phi_2(\Delta)$, as shown below,
	
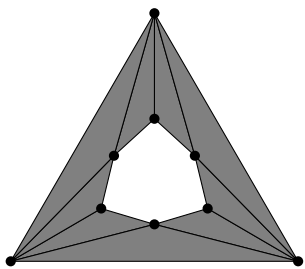
\begin{center}
		\begin{tikzpicture}[scale = 0.7]
			\tikzstyle{point}=[circle,thick,draw=black,fill=black,inner sep=0pt,minimum width=3pt,minimum height=3pt]
			\node (x)[point] at (-1.7,-1) {};
			\node (y)[point] at (3.7,-1) {};
			\node (z)[point] at (1,3.7) {};
			
			\node (a)[point] at (0,0) {};
			\node (b)[point] at (2,0) {};
			\node (c)[point] at (1,1.7) {};
			
			\node (d)[point] at (1,-0.3) {}; 
			\node (e)[point] at (1.76,1) {}; 
			\node (f)[point] at (0.24,1) {}; 
			
			\begin{scope}[on background layer]
				\draw[fill=gray] (x.center) -- (a.center) -- (d.center) -- cycle;
				\draw[fill=gray] (x.center) -- (a.center) -- (f.center) -- cycle;
				\draw[fill=gray] (y.center) -- (b.center) -- (d.center) -- cycle;
				\draw[fill=gray] (y.center) -- (b.center) -- (e.center) -- cycle;
				\draw[fill=gray] (z.center) -- (c.center) -- (e.center) -- cycle;
				\draw[fill=gray] (z.center) -- (c.center) -- (f.center) -- cycle;
				
				\draw[fill=gray] (x.center) -- (y.center) -- (d.center) -- cycle;
				\draw[fill=gray] (y.center) -- (z.center) -- (e.center) -- cycle;
				\draw[fill=gray] (x.center) -- (z.center) -- (f.center) -- cycle;
			\end{scope}
		\end{tikzpicture}
	\end{center}
	has $9$ vertices. It also has $9$ facets and $9$ edges, each of which has a corresponding vertex in the set $S^2_{\phi_2(\Delta)}$. Thus the set $S^2_{\phi_2(\Delta)}$ contains $18$ vertices, and the complex $\phi_2^2(\Delta)$ on vertex set $V\sqcup S^2_\Delta \sqcup S^2_{\phi_2(\Delta)}$ has a total of $27$ vertices. Meanwhile, the cycle complex $\fC_{3,1}$ has only $7$ vertices.
	
	The families of PR complexes presented in Chapter \ref{Chapter: Families of PR Complexes} lower the bounds for $n$ for their given degree types; and in a few simple cases they even resolve Question \ref{Question: Lower bounds on n for each degree type} entirely. As noted in Sections \ref{Subsection: PR projdim 2} and \ref{Subsection: Intersection Complexes}, we strongly suspect that cycle complexes and intersection complexes have minimal numbers of vertices for their given degree types; we hope to find proofs for these statements in due course. 
	
	Answering Question \ref{Question: Lower bounds on n for each degree type} in general for arbitrary degree types may be currently out of reach; however answers for some other specific cases should be manageable. For example we could restrict our attention to degree types of projective dimension $3$; or attempt to improve the bounds for $n$ given by partition complexes, for degree types of the form $(\overbrace{1,\dots,1,\underbrace{a,1,\dots,1}_{m}}^{p})$.
	
	\section{Ideals Generated in Degree $d$}\label{Subsection: Pure Resolutions in Degree d}
	Edge ideals are Stanley-Reisner ideals generated in degree $2$. It may be of interest to consider whether results about their Betti diagrams can be generalised to the case of Stanley-Reisner ideals generated in some arbitrary degree $d$; or at least to see if similar results can be obtained for specific higher degrees.
	
	For example, we could search for analogues of Theorems \ref{Theorem: dimCn} and \ref{Theorem: dimCnh} on the dimensions of the cones $\Cn$ and $\Cnh$, for cones generated by squarefree monomial ideals generated in higher degrees. It is possible that the formulae given in Theorems \ref{Theorem: dimCn} and \ref{Theorem: dimCnh} may turn out to be instances of a more general result.
	
	Another result about edge ideals which might admit a generalisation is the classification of all edge ideals with pure resolutions, due to Fr\"{o}berg, Bruns and Hibi. To state it, we require the following terminology.
	\begin{defn}\label{Definition: Chordal Graphs}
		Let $C$ be a cycle of edges $\{v_1,v_2\},\dots,\{v_{m-1},v_m\},\{v_m,v_1\}$ in a graph $G$. A \textit{chord} for $C$ is an edge of $G$ of the form $\{v_i,v_j\}$ for some non-consecutive $1\leq i < j \leq m$. We say $G$ is \textit{chordal} if all of its cycles of length at least $4$ have a chord.
	\end{defn}
	
	The classification is as follows. The first of these three cases comes from Theorem 1 in \cite{Fro}; the latter two come from Theorem 2.1 in \cite{Bruns-Hibi}. Note that the requirement that $G$ has no isolated vertices is harmless, by Corollary \ref{Corollary: Betti Graphs With Isolated Vertices}.
	\begin{thm}\label{Theorem: Edge Ideals with Pure Resolutions}
		Let $G$ be a graph with no isolated vertices. The edge ideal $I(G)$ has a pure resolution if and only if the complement $G^c$ is of one of the following three forms.
		\begin{enumerate}
			\item $G^c$ is chordal, in which case $I(G)$ has a pure resolution with shift type degree type $(1,\dots,1)$.
			\item $G^c$ is cyclic, in which case $I(G)$ has a pure resolution with degree type $(2,1,\dots,1)$.
			\item $G^c$ is the $1$-skeleton of a cross polytope, in which case $I(G)$ has a pure resolution with degree type $(2,\dots,2)$.
		\end{enumerate}
	\end{thm}
	\begin{rem}
		We also know all of the possible Betti diagrams arising from the three families of graphs presented in Theorem \ref{Theorem: Edge Ideals with Pure Resolutions}. For graphs of the first form, \cite{E-S} proves that all corresponding Betti diagrams can be obtained from threshold graphs (Theorem 4.4) and enumerated accordingly (Proposition 4.11). We have seen the Betti diagram of the complement of the cyclic graph $C_m$ already in Proposition \ref{Proposition: Betti-Cn}; and by repeated application of Lemma \ref{Lemma: Betti S-Delta} we can compute the Betti numbers of the cross polytope $O^m$ as $\beta_{i,2i+2}(O^m)={m+1 \choose i+1}$ for $0\leq i \leq m$. 
	\end{rem}
	
	This classification raises a natural question.
	\begin{qu}\label{Question: Pure resolutions in degree d}
		For a given positive integer $d$, what are the possible pure resolutions of Stanley-Reisner ideals generated in degree $d$?
	\end{qu}
	
	Recall that the degree in which such an ideal $I$ is generated is equal to the value of $c_0$ for which $\beta_{0,c_0}(I)$ is nonzero. Thus Question \ref{Question: Pure resolutions in degree d} is similar to Question \ref{Question: Lower bounds on n for each degree type}, but in reverse: instead of fixing a degree type of a pure resolution and asking what the possible values of $n$ and $c_0$ could be for PR complexes of that degree type, we fix a value for $c_0$ and ask which pure resolutions we can get which have that value of $c_0$ as their initial shift. Due to this similarity it may be useful to attack these two problems in tandem.
	
	As observed in Remark \ref{Remark: Link of Emptyset and Facets}, the minimum degree of the generators of $I_\Delta^*$ is equal to the codimension of $\Delta$. Thus our problem can be reframed as the task of finding all PR complexes of a given codimension.
	
	If we are to reframe the problem in this light, it would be useful to obtain a restatement of Theorem \ref{Theorem: Edge Ideals with Pure Resolutions} in terms of PR complexes. This would require finding the Alexander duals of the complexes corresponding to the three cases of graphs above. We present the duals of cross polytopes below.
	
	\begin{lem}
		The dual of the $d$-dimensional cross polytope $O^d$ can be obtained by applying the scalar multiple operation $f^2$ from Definition \ref{Definition: f-lambda} to the boundary of the $d$-simplex $\partial \Delta^d$. That is, $$(O^d)^*=f^2(\partial \Delta^d).$$
	\end{lem}
	\begin{proof}
		The minimal nonfaces of the cross polytope $O^d$ consist of $d+1$ disjoint edges. Thus we can think of $O^d$ as the complex on vertex set $V=\{u_i^1,u_i^2 : 1 \leq i \leq d+1 \}$ whose minimal nonfaces are the edges $e_i=\{u_i^1,u_i^2\}$ for each $i\in [d+1]$. The facets of its Aleander dual are the faces $V-e_i$ for each $i\in [d+1]$. Meanwhile the facets of the complex $\partial \Delta^d$ on vertex set $[d+1]$ are the sets $[d+1]-\{i\}$ for each $i\in [d+1]$, which are mapped to $V-e_i$ under the operation $f^2$.
	\end{proof}
	
	We do not currently know of a concrete description for the Alexander duals of clique complexes of cyclic and chordal graphs.
	
	\section{Homology-preserving Subcomplexes of $\BDelta$}\label{Subsection: Subcomplexes of BDelta}
	Let $\Delta$ be a simplicial complex. Recall from the discussion in Section \ref{Subsection: Barycentric Subdivision} that $\Delta$ is homeomorphic to its barycentric subdivision $\BDelta$ and hence both complexes have identical homology. Recall also (from the same section) that $\BDelta$ is equal to the induced subcomplex $\phi_m(\Delta)|_{S_\Delta}$ when $m=\dim \Delta+1$.
	
	While investigating the $\phi_i$ operations as given in Definition \ref{Definition: Phi_i}, it came to our attention that even for smaller values of $i$, the induced subcomplex $\phi_i(\Delta)|_{S^i_\Delta}$ seems to have identical homologies to $\Delta$ up to a certain degree. This leads to a natural question.
	\begin{qu}
		What is the smallest subcomplex of $\BDelta$ that preserves the homological data of $\Delta$ up to a given degree?
	\end{qu}
	
	Proposition \ref{Proposition: Links in BDelta} demonstrates that $\BDelta$ preserves a lot of information about the homology of the links in $\Delta$. This observation led us to consider the following subcomplexes of $\BDelta$.
	
	\begin{defn}\label{Definition: Li-Delta}
		Let $i\geq -1$ be an integer, and let $\LiDelta$ be the induced subcomplex of $\BDelta$ on vertex set $\{u_\sigma \in S_\Delta : \Hred_j(\lkds)\neq 0 \text{ for some } j\leq i\}$.
	\end{defn}
	
	\begin{defn}\label{Definition: L-Delta}
		Let $\LDelta$ be the induced subcomplex of $\BDelta$ on vertex set $\{u_\sigma \in S_\Delta : h(\Delta,\sigma)\neq \emptyset\}$ (i.e. $\LDelta=\bigcup_i \LiDelta$).
	\end{defn}
	
	Based on some experimental data from the software system Macaulay2 (\cite{M2}), our belief is that the homology of the complex $\LiDelta$ agrees with the homology of $\Delta$ up to degree $i$. In other words, we make the following conjecture.
	\begin{conj}\label{LiDelta Conjecture}
		For each $-1 \leq j \leq i$ we have $\Hred_j\left (\LiDelta\right) = \Hred_j(\Delta)$.
	\end{conj}
	
	In particular, $\LDelta$ is equal to $\LiDelta$ for some sufficiently large $i$, so a proof of Conjecture \ref{LiDelta Conjecture} would also prove the following slightly weaker conjecture.
	\begin{conj}\label{LDelta Conjecture}
		For every $j\geq -1$ we have $\Hred_j\left(\LDelta\right) = \Hred_j(\Delta)$.
	\end{conj}
	
	\section{Extremal rays and Defining Halfspaces}\label{Subsection: Extremal Rays and Defining Halfspaces}
	A complete description of the cone $\Dn$, or any of its subcones, would require a classification either of extremal rays or of defining halfspaces. We hope that the results in Chapter \ref{Chapter: Dimensions}, on the dimensions of these cones and their minimal ambient vector spaces, might prove useful in attempts to establish defining halfspaces.
	
	Classifying the extremal rays of $\Dn$ for arbitrary $n$ is a difficult task. As noted in Section \ref{Subsection: Motivation PR}, these rays are \textit{not} all given by pure diagrams, nor is every pure diagram in $\Dn$ extremal.
	
	\begin{ex}\label{Example: non-pure extremal ray} 
		If $\Delta$ is the complex 	$$
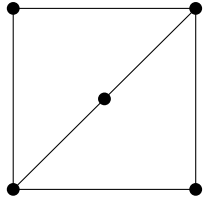
\begin{tikzpicture}[scale = 0.6]
			\tikzstyle{point}=[circle,thick,draw=black,fill=black,inner sep=0pt,minimum width=4pt,minimum height=4pt]
			\node (a)[point] at (0,0){};
			\node (b)[point] at (4,0){};
			\node (c)[point] at (4,4){};
			\node (d)[point] at (0,4){};	
			\node (e)[point] at (2,2){};	
			\draw[black] (a.center) -- (b.center) -- (c.center) -- (d.center) -- cycle;
			\draw[black] (a.center) -- (e.center) -- (c.center);
		\end{tikzpicture}$$ on $5$ vertices, then the diagram $\beta=\beta(I_\Delta)$ of its Stanley-Reisner ideal is $$\begin{array}{c | ccc}
			& 0 & 1 & 2 \\
			\hline
			2 & 4 & 2 & .\\
			3 & . & 3 & 2
		\end{array}$$
		which is not pure. However, using the `\textit{Polyhedra}' (\cite{M2-Poly}) and `\textit{BoijSoederberg}' (\cite{M2-BS}) packages on the software system Macaulay2 we found that this diagram \textit{is} in fact an extremal ray of $\Dn$. We will not prove this explicitly here; instead we show simply that it cannot be decomposed into a rational sum of pure diagrams in $\Dn$, which demonstrates that $\Dn$ must contain at least \textit{some} non-pure extremal ray.
		
		Suppose for contradiction that our diagram $\beta$ can be decomposed into a rational sum of pure diagrams $\sum_i q_i \alpha^i$. In particular, $\beta_{2,5}$ is nonzero, so at least one of the pure diagrams $\alpha$ in this decomposition must satisfy $\alpha_{2,5}\neq 0$. Our two options for the shift type of $\alpha$ are $(5,4,2)$ and $(5,3,2)$. The first option would require a pure diagram of degree type $(1,2)$, but we saw in Section \ref{Subsection: PR projdim 2} (Proposition \ref{Proposition: n = (2m+3)b for some cases}) that the smallest PR complex of degree type $(1,2)$ is the cycle complex $\fC_{1,2}$ on $6$ vertices. So $\alpha$ must instead have shift type $(5,4,2)$, and the only option is the diagram corresponding to the the cycle complex $\fC_{2,1}$, which is
		$$\begin{array}{c | ccc}
			& 0 & 1 & 2 \\
			\hline
			2 & 5 & 5 & .\\
			3 & . & . & 1
		\end{array}$$
		Because we have $\beta_{2,5}=2 = 2 \alpha_{2,5}$, the coefficient of $\alpha$ in the decomposition must be $2$. But this is a contradiction because $2\alpha_{0,2}=10$, which is greater than $\beta_{0,2}$.
		
		Thus $\Dn[5]$ contains some extremal rays which are not pure.
	\end{ex}
	
	\begin{ex}\label{Example: non-extremal pure diagram}
		If $\Delta$ is the complex $$
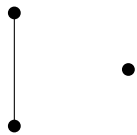
\begin{tikzpicture}[scale = 0.75]
			\tikzstyle{point}=[circle,thick,draw=black,fill=black,inner sep=0pt,minimum width=4pt,minimum height=4pt]
			
			\node (a)[point] at (0,1) {};
			\node (b)[point] at (0,-1) {};
			\node (c)[point] at (2,0) {};
			
			\draw (a.center) -- (b.center);
		\end{tikzpicture}$$ on $3$ vertices, then the diagram $\beta=\beta(I_\Delta)$ of its Stanley-Reisner ideal is $$\begin{array}{c | cc}
			& 0 & 1 \\
			\hline
			2 & 2 & 1 
		\end{array}$$ This diagram is pure, but it can be expressed as a sum of smaller pure diagrams in $\Dn[3]$. Specifically, using the notation $\beta(\Delta)$ to stand in for $\beta(I_\Delta)$ as in Chapter \ref{Chapter: Dimensions}, we have
		\begin{center}
			\begin{tabular}{ c c c c c }
				$\beta\left (\begin{tikzpicture}[scale = 0.4]
					\tikzstyle{point}=[circle,thick,draw=black,fill=black,inner sep=0pt,minimum width=4pt,minimum height=4pt]
					
					\node (a)[point] at (0,0) {};
					\node (b)[point] at (0,1.7) {};
					\node (c)[point] at (2,0.85) {};
					\node at (0,1.8) {};
					
					\draw (a.center) -- (b.center);
				\end{tikzpicture}\right )$ & $=$ & $\frac{1}{2} \beta\left (\begin{tikzpicture}[scale = 0.4]
					\tikzstyle{point}=[circle,thick,draw=black,fill=black,inner sep=0pt,minimum width=4pt,minimum height=4pt]
					
					\node (a)[point] at (0,0) {};
					\node (b)[point] at (2,0) {};
					\node at (0,1.8) {};
					
				\end{tikzpicture}\right )$ & $+$ & $\frac{1}{2} \beta\left (\begin{tikzpicture}[scale = 0.4]
					\tikzstyle{point}=[circle,thick,draw=black,fill=black,inner sep=0pt,minimum width=4pt,minimum height=4pt]
					
					\node (a)[point] at (0,0) {};
					\node (b)[point] at (2,0) {};
					\node (b)[point] at (1,1.7) {};
					\node at (0,1.8) {};
					
				\end{tikzpicture}\right )$ \\
				$\begin{array}{c | cc}
					& 0 & 1 \\
					\hline
					2 & 2 & 1 
				\end{array}$ & $=$ & $\frac{1}{2}\left (\begin{array}{c | cc}
					& 0 \\
					\hline
					2 & 1 
				\end{array}\right )$ & $+$ & $\frac{1}{2} \left (\begin{array}{c | cc}
					& 0 & 1\\
					\hline
					2 & 3 & 2
				\end{array}\right )$
			\end{tabular}
		\end{center}
		
		Thus $\beta$ is a pure diagram in $\Dn[3]$ which is not extremal.
	\end{ex}
	
	Of course, as a consequence of the Boij-S\"{o}derberg conjectures, all pure diagrams in $\Dn$ which correspond to Cohen-Macaulay complexes are extremal rays of the wider Betti cone generated by all $R$-modules, and hence of $\Dn$ itself. This gives us the following proposition.
	\begin{prop}\label{Proposition: Pure Extremal Rays of Dn}
		If $\Delta$ is a complex of any of the following forms, then the Betti diagram $\beta(I_\Delta)$ lies on an extremal ray of the cone $\Dn$.
		\begin{enumerate}
			\item The skeleton complex $\Skel_r([m])$ for some $-1\leq r < m \leq n$.
			\item The $d$-dimensional cross polytope $O^d$ for some $d\leq \frac{n}{2}-1$.
			\item The $1$-dimensional cycle $C_m$ for some $m\leq n$.
		\end{enumerate} 
	\end{prop}
	\begin{proof}
		For part (1), the complex $\Skel_r([m])$ has a pure Betti diagram by Lemma \ref{Lemma: Betti Skeleton Complexes}, and it satisfies Reisner's criterion for Cohen-Macaulay complexes (this can be shown by an inductive argument on $r\geq -1$, using the fact that the link of any vertex in $\Skel_r([m])$ is equal to $\Skel_{r-1}([m-1])$).
		
		For parts (2) and (3), we saw that both the complexes $O^d$ and $C_m$ have pure Betti diagrams in Theorem \ref{Theorem: Edge Ideals with Pure Resolutions}, and they are both boundaries of simplicial spheres, which means they are Cohen-Macaulay by Remark \ref{Remark: Simplicial Spheres are CM}.
	\end{proof}
	\begin{rem}\label{Remark: All extremal pure diagrams of this shape accounted for}
		The above proposition is not an exhaustive list of the pure extremal rays of the cone $\Dn$. However, these diagrams \textit{do} account for all pure extremal rays of $\Dn$ of their particular shift types. To see why, suppose (for example) that $\beta$ is a pure diagram in $\Dn$ with the same shift type $\bc=(c_p,\dots,c_0)$ as the diagram of a skeleton complex. The second Boij-S\"{o}derberg Theorem allows us to decompose the ray $t\cdot\beta$ into a sum of rays of the form $t\cdot\pi(\bc')$ for subsequences $\bc'=(c_r,\dots,c_0)$ of $\bc$; and for every such subsequence $\bc'$ there is a corresponding skeleton complex on up to $n$ vertices whose diagram lies on the ray $t\cdot \pi(\bc')$, which means that all of these rays lie in $\Dn$. A similar argument holds for pure diagrams with the same shift type as cross polytopes or cyclic graphs.
		
		Thus the diagrams of skeleton complexes are the only linear extremal rays in $\Dn$; and the diagrams of cross polytopes and cyclic graphs are, respectively, the only pure extremal rays with shift types of the form $(2r,2r-2,\dots,4,2)$ and $(m+2,m,m-1,\dots,3,2)$. In particular, this accounts for all the pure extremal rays in $\Cn$, by Theorem \ref{Theorem: Edge Ideals with Pure Resolutions}.
	\end{rem}
	
	Using Macaulay2 we have found the extremal rays of $\Dn$ for low values of $n$, but we have so far failed to establish any patterns for the non-pure examples. We hope that further investigation with Macaulay2 might yield some examples of non-pure rays that occur for arbitrary values of $n$ (although a complete description of the extremal rays of $\Dn$ is probably too ambitious).
	

	\backmatter
	\phantomsection
	\addcontentsline{toc}{chapter}{Bibliography}
	
	
		
	\printunsrtglossary[title=\listofsymbolsname,type=symbols,style=long-name-desc,nogroupskip]
	
\end{document}